\documentclass{amsart}
\usepackage{amssymb,amscd,amsxtra}
\usepackage[dvips,%
bookmarks=true,%
bookmarksopen=true,%
setpagesize=false,%
bookmarksnumbered=true,%
colorlinks=false,%
pdftitle={Grothendieck duality and Q-Gorenstein morphisms},%
pdfauthor={Yongnam Lee and Noboru Nakayama}%
]{hyperref}
\normalsize

\theoremstyle{plain}
    \newtheorem{thm}{Theorem}[section]
    
    \newtheorem{lem}[thm]{Lemma}%
    \newtheorem{cor}[thm]{Corollary}
    \newtheorem{prop}[thm]{Proposition}

\theoremstyle{definition}
    \newtheorem{dfn}[thm]{Definition}
    
    \newtheorem{nota}[thm]{Notation}

\theoremstyle{remark}
    \newtheorem{rem}[thm]{Remark}
    \newtheorem*{remn}{Remark}
    \newtheorem{fact}[thm]{Fact}
    \newtheorem*{factn}{Fact}
    \newtheorem{exam}[thm]{Example}
    \newtheorem*{examn}{Example}
    \newtheorem{ppty}[thm]{Property}
    \newtheorem{sit}[thm]{Situation}

\numberwithin{equation}{section}

\newcommand{\BAA}{\mathbb{A}}
\newcommand{\BCC}{\mathbb{C}}

\newcommand{\BHH}{\mathbb{H}}
\newcommand{\BPP}{\mathbb{P}}
\newcommand{\BQQ}{\mathbb{Q}}
\newcommand{\BVV}{\mathbb{V}}
\newcommand{\BZZ}{\mathbb{Z}}

\newcommand{\SA}{\mathcal{A}}
\newcommand{\SB}{\mathcal{B}}
\newcommand{\SC}{\mathcal{C}}
\newcommand{\SE}{\mathcal{E}}
\newcommand{\SF}{\mathcal{F}}
\newcommand{\SG}{\mathcal{G}}
\newcommand{\SH}{\mathcal{H}}
\newcommand{\SJ}{\mathcal{J}}
\newcommand{\SK}{\mathcal{K}}
\newcommand{\SL}{\mathcal{L}}
\newcommand{\SM}{\mathcal{M}}
\newcommand{\SN}{\mathcal{N}}
\newcommand{\SO}{\mathcal{O}}

\newcommand{\SR}{\mathcal{R}}
\newcommand{\ScS}{\mathcal{S}}

\newcommand{\SX}{\mathcal{X}}

\newcommand{\ep}{\varepsilon}

\newcommand{\Ga}{\mathfrak{a}}
\newcommand{\GD}{\mathfrak{D}}
\newcommand{\GM}{\mathfrak{m}}
\newcommand{\Gp}{\mathfrak{p}}

\newcommand{\ttt}{\mathtt{t}}
\newcommand{\xtt}{\mathtt{x}}
\newcommand{\ytt}{\mathtt{y}}
\newcommand{\ztt}{\mathtt{z}}

\newcommand{\SAA}{\mathsf{A}}

\newcommand{\bfD}{\mathbf{D}}
\newcommand{\bfL}{\mathbf{L}}
\newcommand{\bfR}{\mathbf{R}}
\newcommand{\bfS}{\mathbf{S}}

\newcommand{\Ann}{\operatorname{Ann}} 
\newcommand{\Ass}{\operatorname{Ass}} 
\newcommand{\chara}{\operatorname{char}}
\newcommand{\CM}{\operatorname{CM}}
\newcommand{\Codim}{\operatorname{codim}}

\newcommand{\Coker}{\operatorname{Coker}}
\newcommand{\Cone}{\operatorname{Cone}}
\newcommand{\Div}{\operatorname{div}}
\newcommand{\depth}{\operatorname{depth}}
\newcommand{\codepth}{\operatorname{codepth}}

\newcommand{\Ext}{\operatorname{Ext}}
\newcommand{\SExt}{\operatorname{\mathcal{E}\mathit{xt}}}
\newcommand{\BExt}{\operatorname{\mathbb{E}xt}}
\newcommand{\SBExt}{\operatorname{\boldsymbol{\mathcal{E}}\mspace{-1mu}\mathit{xt}}}
\newcommand{\Fl}{\operatorname{Fl}} 
\newcommand{\Gor}{\operatorname{Gor}}
\newcommand{\Gm}{\mathbb{G}_{\mathtt{m}}}
\newcommand{\Hom}{\operatorname{Hom}}
\newcommand{\RHom}{\operatorname{\mathbf{R}Hom}}
\newcommand{\SHom}{\operatorname{\mathcal{H}\mathit{om}}}
\newcommand{\SRHom}{\operatorname{\mathbf{R}\mathcal{H}\mathit{om}}}
\newcommand{\id}{\operatorname{id}}
\newcommand{\Ker}{\operatorname{Ker}}
\newcommand{\Mod}{\mathsf{Mod}}
\newcommand{\OH}{\operatorname{H}}
\newcommand{\Pic}{\operatorname{Pic}}
\newcommand{\pro}{\operatorname{pro}} 
\newcommand{\Proj}{\operatorname{Proj}}
\newcommand{\QCoh}{\mathsf{QCoh}}
\newcommand{\rad}{\operatorname{rad}}
\newcommand{\Set}{\mathbf{Set}}
\newcommand{\Sch}{\mathsf{Sch}}
\newcommand{\LNSch}{\mathsf{LNSch}} 
\newcommand{\Spec}{\operatorname{Spec}}
\newcommand{\SSpec}{\operatorname{\mathcal{S}\mathit{pec}}}

\newcommand{\Supp}{\operatorname{Supp}}

\newcommand{\Tor}{\operatorname{Tor}}
\newcommand{\STor}{\operatorname{\mathcal{T}\!\mathit{or}}}
\newcommand{\transdeg}{\operatorname{tr.deg}}

\newcommand{\isom}{\simeq}
\newcommand{\injmap}{\hookrightarrow}

\newcommand{\coh}{\mathrm{coh}}
\newcommand{\op}{\mathrm{op}}
\newcommand{\qcoh}{\mathrm{qcoh}}
\newcommand{\qis}{\mathrm{qis}}

\begin{document}

\title[Grothendieck duality and $\BQQ$-Gorenstein morphisms]{Grothendieck duality and $\BQQ$-Gorenstein morphisms}
\author{Yongnam Lee}
\address{%
\textsc{Department of Mathematical Sciences} \endgraf
\textsc{Korea Advanced Institute of Science and Technology} \endgraf 
\textsc{291, Daehak-ro, Yuseong-gu, Daejeon 34141 South Korea}}
\email{ynlee@kaist.ac.kr}

\author{Noboru Nakayama}
\address{%
\textsc{Research Institute for Mathematical Sciences} \endgraf 
\textsc{Kyoto University, Kyoto 606-8502 Japan}}
\email{nakayama@kurims.kyoto-u.ac.jp}

\begin{abstract}
The notions of \( \BQQ \)-Gorenstein scheme and of \( \BQQ \)-Gorenstein morphism 
are introduced for locally Noetherian schemes by dualizing complexes and (relative) canonical sheaves. 
These cover all the previously known notions of \( \BQQ \)-Gorenstein algebraic 
variety and of \( \BQQ \)-Gorenstein deformation satisfying Koll\'ar condition, over a field.  
By studies on relative \( \bfS_{2} \)-condition and base change properties, 
valuable results are proved for \( \BQQ \)-Gorenstein morphisms, which include  
infinitesimal criterion, valuative criterion, 
\( \BQQ \)-Gorenstein refinement, and so forth. 
\end{abstract}


\thanks{The first named author is partly supported by
the NRF of Korea funded by the Korean government(MSIP)(No.2013006431).}

\maketitle

\setcounter{tocdepth}{1}
\tableofcontents
\setcounter{tocdepth}{2}

\section{Introduction}
\label{sect:intro}

The notion of \( \BQQ \)-Gorenstein variety
is important for the minimal model theory
of algebraic varieties in characteristic zero:
A normal algebraic variety \( X \) defined over a field of
any characteristic is said to be \( \BQQ \)-Gorenstein if
\( rK_{X} \) is Cartier for some positive integer \( r \), where \( K_{X} \) stands for
the canonical divisor of \( X \).
In some papers, \( X \) is additionally required to be Cohen--Macaulay.
Reid used this notion essentially to define the canonical singularity in
\cite[Def.~(1.1)]{Reid},
and he named the notion ``\( \BQQ \)-Gorenstein''
in \cite[(0.12.e)]{ReidMM}, where the Cohen--Macaulay condition is required.
The notion without the Cohen--Macaulay condition appears
in \cite{KMM} for example.
In the minimal model theory of algebraic varieties of dimension more than two,
we must deal with varieties with mild singularities such as terminal, canonical,
log-terminal, and log-canonical (cf.\ \cite[\S 0-2]{KMM} for the definition).
The notion of \( \BQQ \)-Gorenstein is hence frequently used in studying
the higher dimensional birational geometry.

The notion of \( \BQQ \)-Gorenstein deformation is also
popular in the study of degenerations of normal algebraic varieties
in characteristic zero related to
the minimal model theory and the moduli theory
since the paper \cite{KSh} by Koll\'ar and Shepherd-Barron.
Roughly speaking, a \( \BQQ \)-Gorenstein deformation \( \SX \to C \)
of a \( \BQQ \)-Gorenstein normal algebraic variety \( X \)
is considered as a flat family of algebraic varieties over a smooth curve \( C \)
with a closed fiber being isomorphic to \( X \) such that
\( rK_{\SX/C} \) is Cartier and \( rK_{\SX/C}|_{X} \sim rK_{X} \)
for some \( r > 0 \), where \( K_{\SX/C} \) stands for
the relative canonical divisor.  We call
such a deformation ``naively \( \BQQ \)-Gorenstein'' (cf.\ Definition~\ref{dfn:QGorMor} below).
This is said to be ``weakly \( \BQQ \)-Gorenstein'' in \cite[\S 3]{Hacking}, 
or satisfying \emph{Viehweg's condition} (cf.\ Property \( \mathbf{V}^{[N]} \) in \cite[\S 2]{HaKo}). 
We say that \( \SX \to C \) is a \( \BQQ \)-Gorenstein deformation if
\[ \SO_{\SX}(mK_{\SX/C}) \otimes_{\SO_{\SX}} \SO_{X} \isom
\SO_{X}(mK_{X})  \]
for any integer \( m \).
This additional condition seems to be considered first by Koll\'ar
\cite[2.1.2]{KolProj}, and it is called the \emph{Koll\'ar condition}; 
A similar condition is named as Property \( \mathbf{K} \) in  \cite[\S 2]{HaKo} for example.
A typical example of \( \BQQ \)-Gorenstein deformation appears as a deformation
of the weighted projective plane \( \BPP(1, 1, 4) \):
Its versal deformation space has two irreducible components, in which
the one-dimensional component corresponds to the \( \BQQ \)-Gorenstein deformation and
its general fibers are \( \BPP^{2} \) (cf.\ \cite[\S 8]{Pinkham}).
The \( \BQQ \)-Gorenstein deformation is also used in constructing
some simply connected surfaces of general type over the complex number field \( \BCC \) in \cite{LeePark}.
The authors have succeeded in generalizing the construction to the positive characteristic case 
in \cite{LeeNakayama}, where a special case of \( \BQQ \)-Gorenstein deformation
over a mixed characteristic 
base scheme is considered.

During the preparation of the joint paper \cite{LeeNakayama}, the authors began
generalizing the notion of \( \BQQ \)-Gorenstein morphism to
the case of morphisms between locally Noetherian schemes.
The purpose of this article is to give good definitions of
\emph{\( \BQQ \)-Gorenstein scheme}
and \emph{\( \BQQ \)-Gorenstein morphism}:
We define the notion of ``\( \BQQ \)-Gorenstein'' for
locally Noetherian schemes admitting dualizing complexes
(cf.\ Definition~\ref{dfn:QGorSch} below) and
define the notion of ``\( \BQQ \)-Gorenstein'' for
flat morphisms locally of finite type between
locally Noetherian schemes (cf.\ Definition~\ref{dfn:QGorMor} below).
So, we try to define the notion of ``\( \BQQ \)-Gorenstein'' as general as possible.
We do not require the Cohen--Macaulay condition for fibers, which is assumed in
most articles on \( \BQQ \)-Gorenstein deformations, 
and we allow all locally Noetherian schemes as the base scheme of a \( \BQQ \)-Gorenstein morphism.

\subsection*{\texorpdfstring{$\BQQ$}{Q}-Gorenstein schemes and \texorpdfstring{$\BQQ$}{Q}-Gorenstein morphisms}

The definition of \( \BQQ \)-Gorenstein scheme in Definition~\ref{dfn:QGorSch} below 
is interpreted as follows (cf.\ Lemma\ \ref{lem:QGorSch3}\eqref{lem:QGorSch3:3}): 
A locally Noetherian scheme \( X \) is said to be \( \BQQ \)-Gorenstein if and only if 
\begin{itemize}
\item it satisfies Serre's condition \( \bfS_{2} \), 
\item it is Gorenstein in codimension one, 
\item there exists a dualizing sheaf \( \SL \) locally on \( X \), and 
the double dual of \( \SL^{\otimes r} \) is invertible for some integer \( r > 0 \) locally on \( X \). 
\end{itemize}
Here, we consider a dualizing sheaf (cf.\ Definition~\ref{dfn:ordinaryDC+DS}) 
as the zero-th cohomology \( \SH^{0}(\SR^{\bullet}) \) 
of an ordinary dualizing complex \( \SR^{\bullet} \), which is 
a dualizing complex of special type and exists for 
any \emph{locally equi-dimensional} 
(cf.\ Definition~\ref{dfn:equi-dim}\eqref{dfn:equi-dim:locally}) 
and locally Noetherian schemes admitting dualizing complexes 
(cf.\ Lemma~\ref{lem:ordinaryDC}). 
The dualizing sheaf of the Gorenstein locus \( U = \Gor(X) \) is isomorphic to 
\( \SL|_{U} \) up to tensor product with invertible sheaves, 
and \( \SL \) is a reflexive \( \SO_{X} \)-module with an isomorphism 
\( \SL \isom j_{*}(\SL|_{U}) \) for the open immersion \( j \colon U \injmap X \). 
This definition generalizes the usual definition of \( \BQQ \)-Gorenstein normal 
algebraic varieties over a field (cf.\ Example~\ref{exam:dfn:QGorScheme:3}).

On the other hand, in order to define \( \BQQ \)-Gorenstein morphisms, 
we need to discuss the \emph{relative canonical sheaf} of an \( \bfS_{2} \)-morphism. 
An \( \bfS_{2} \)-morphism is defined as a flat morphism of locally Noetherian schemes 
which is locally of finite type and every fiber satisfies Serre's condition \( \bfS_{2} \) 
(cf.\ Definition~\ref{dfn:SkCMmorphism}). 
By Conrad \cite[Sect.\ 3.5]{Conrad} and Sastry \cite{Sastry}, 
we have a good notion of the relative canonical sheaf for Cohen--Macaulay morphisms. 
For an \( \bfS_{2} \)-morphism \( f \colon Y \to T \) of locally Noetherian schemes, 
the relative Cohen--Macaulay locus \( Y^{\flat} = \CM(Y/T) \) is an open subset 
(cf.\ Definition~\ref{dfn:RelSkCMlocus} and Fact~\ref{fact:dfn:RelSkCMlocus}), and 
we define the relative canonical sheaf \( \omega_{Y/T} \) 
as the direct image by 
the open immersion \( Y^{\flat} \subset Y \) 
of the relative canonical sheaf \( \omega_{Y^{\flat}/T} \) of the Cohen--Macaulay morphism 
\( Y^{\flat} \to T \) (cf.\ Definition~\ref{dfn:relcanosheaf}). 
The sheaf \( \omega_{Y/T} \) is coherent (cf.\ Proposition~\ref{prop:BC-S2CM}), 
and it is reflexive when every fiber is Gorenstein in codimension one 
(cf.\ Proposition~\ref{prop:BCGor}).
We set \( \omega^{[m]}_{Y/T} \) to be the double dual of \( \omega^{\otimes m}_{Y/T} \) 
for \( m \in \BZZ \), 
and we define \( \BQQ \)-Gorenstein morphisms as follows: 
A flat morphism \( f \colon Y \to T \) locally of finite type between locally Noetherian schemes 
is called a \( \BQQ \)-Gorenstein morphism (cf.\ Definition~\ref{dfn:QGorMor}) if and only if 
\begin{itemize}
\item  every fiber is a \( \BQQ \)-Gorenstein scheme, and 

\item \( \omega^{[m]}_{Y/T} \) satisfies relative \( \bfS_{2} \) over \( T \) for any \( m \in \BZZ \). 
\end{itemize} 
Note that \( f \) is an \( \bfS_{2} \)-morphism and every fiber is Gorenstein in codimension one, 
by the first condition. The second condition corresponds to the Koll\'ar condition. 
A weaker notion: naively \( \BQQ \)-Gorenstein morphism 
is defined by replacing the second condition to: 
\begin{itemize}
\item  \( \omega^{[m]}_{Y/T} \) is invertible for some \( m \) locally on \( Y \). 
\end{itemize}
This condition corresponds to Viehweg's condition. 

By our definition above, 
we can consider \( \BQQ \)-\emph{Gorenstein deformations} 
\( f \colon Y \to T \) of a \( \BQQ \)-Gorenstein scheme \( X \) defined 
over a field \( \Bbbk \). Here, \( f \) is a \( \BQQ \)-Gorenstein morphism, 
\( T \) contains a point \( o \) with residue field \( \Bbbk \), 
and the fiber \( Y_{o} = f^{-1}(o)\) 
is isomorphic to \( X \) over \( \Bbbk \). 
The scheme \( X \) is not necessarily assumed to be reduced nor normal, 
and \( f \) is not necessarily a morphism of \( \Bbbk \)-schemes. 
The \( \BQQ \)-Gorenstein deformations of non-normal schemes have been treated 
in articles in some special cases: 
Hacking \cite{Hacking} and Tziolas \cite{Tz} consider 
\( \BQQ \)-Gorenstein deformations of slc surfaces, which are not normal in general, 
over \( \BCC \). 
The work of Abramovich--Hassett \cite{AH} covers also non-normal reduced Cohen--Macaulay algebraic schemes
over a fixed field. 
By further studies of \( \BQQ \)-Gorenstein morphisms, we may have 
a well-defined theory of infinitesimal \( \BQQ \)-Gorenstein deformations, 
which is now in progress in the authors' joint work. 

\subsection*{Notable results on \texorpdfstring{$\BQQ$}{Q}-Gorenstein morphisms}

Some expected properties on \( \BQQ \)-Gorenstein morphisms 
can be verified by standard methods. For example, we prove that 
\( \BQQ \)-Gorenstein morphisms are stable under base change and composition 
(cf.\ Propositions~\ref{prop:QGormor}\eqref{prop:QGormor:1} 
and \ref{prop:QGorCompo}\eqref{prop:QGorCompo:3}). 
Besides such elementary properties, 
we have notable results in the topics below, 
which show that our definition of \( \BQQ \)-Gorenstein morphism is reasonable 
and widely applicable: 

\begin{enumerate}
\item \label{serious:1} 
A sufficient condition for a virtually \( \BQQ \)-Gorenstein morphism to be \( \BQQ \)-Gorenstein; 

\item \label{serious:2} 
Infinitesimal and valuative criteria for a morphism to be \( \BQQ \)-Gorenstein; 

\item \label{serious:3} 
Some conditions on fibers related to Serre's \( \bfS_{3} \)-condition  
which are sufficient for a morphism to be \( \BQQ \)-Gorenstein; 

\item \label{serious:4} 
The existence of \( \BQQ \)-Gorenstein refinement.
\end{enumerate}

We shall explain results on these topics briefly. 

\eqref{serious:1}: The virtually \( \BQQ \)-Gorenstein morphism is introduced in 
Section~\ref{subsect:virQGormor} as a weak form of \( \BQQ \)-Gorenstein morphism (cf.\ Definition~\ref{dfn:vQGorMor}). 
This is inspired by the definition  \cite[Def.~3.1]{Hacking} by Hacking on \( \BQQ \)-Gorenstein deformation
of an slc surface in characteristic zero:
His definition is generalized to the notion of Koll\'ar family of \( \BQQ \)-line bundles
in \cite{AH}.
Hacking defines
the \( \BQQ \)-Gorenstein deformation by the property that it locally lifts to
an equivariant deformation of an index-one cover.
This definition essentially coincides with our definition of virtually \( \BQQ \)-Gorenstein morphism
(cf.\ Lemma~\ref{lem:Hacking} and Remark~\ref{rem:Hacking}).
A \( \BQQ \)-Gorenstein morphism is always a virtually \( \BQQ \)-Gorenstein morphism. 
The converse holds if every fiber satisfies \( \bfS_{3} \); It  
is proved as a part of Theorem~\ref{thm:wQGvsQG}. 
This theorem is derived from Theorems~\ref{thm:invExt} and \ref{thm:S2S3crit} 
on criteria for certain sheaves to be invertible, and from 
a study of the relative canonical dualizing complex in Section~\ref{subsect:RelomegaS2}. 
By Theorem~\ref{thm:wQGvsQG}, we can study infinitesimal \( \BQQ \)-Gorenstein deformations   
of a \( \BQQ \)-Gorenstein algebraic scheme over a field 
satisfying \( \bfS_{3} \) via the equivariant 
deformations of the index one cover. 

\eqref{serious:2}: 
The infinitesimal criterion says that, for a given flat morphism \( f \colon Y \to T \) 
locally of finite type between locally Noetherian schemes, it is a \( \BQQ \)-Gorenstein morphism 
if and only if 
the base change \( f_{A} \colon Y_{A} = Y \times_{T} \Spec A \to \Spec A \) is a \( \BQQ \)-Gorenstein morphism 
for any morphism 
\( \Spec A \to T \) for any Artinian local ring \( A \). 
The valuative criterion is similar but \( T \) is assumed to be reduced and \( A \) 
is replaced with any discrete valuation ring. 
These criteria and some variants are proved in Theorems~\ref{thm:InfQGorCrit} and \ref{thm:valcritQGor} 
and  Corollaries~\ref{cor:QGorCrit} and \ref{cor:NaiveQGorCrit}. 
The proof of these criteria uses Proposition~\ref{prop:inf+val} 
on infinitesimal and valuative criteria for a reflexive sheaf on \( Y \) 
to satisfy relative \( \bfS_{2} \) over \( T \).  

\eqref{serious:3}: 
Theorem~\ref{thm:S3Gor2} proves that, for  
a morphism \( f \colon Y \to T \) in \eqref{serious:2} above, 
it is \( \BQQ \)-Gorenstein along a fiber \( Y_{t} = f^{-1}(t)\) 
if \( Y_{t} \)  is \( \BQQ \)-Gorenstein and Gorenstein 
in codimension two and if 
\[  \omega^{[m]}_{Y_{t}/\Bbbk(t)}=\omega^{[m]}_{Y_{t}/\Spec\Bbbk(t)} \] 
satisfies \( \bfS_{3} \) for any \( m \in \BZZ \). Here, \( \Bbbk(t)\) denotes the residue field of 
\( \SO_{T, t}\).

\eqref{serious:4}: 
For an \( \bfS_{2} \)-morphism \( f \colon Y \to T \) of locally Noetherian schemes 
such that every fiber is \( \BQQ \)-Gorenstein, the \emph{\( \BQQ \)-Gorenstein refinement} 
is defined as a monomorphism \( S \to T \) satisfying the following universal property 
(cf.\ Definition~\ref{dfn:QGorRefinement}): 
For a morphism 
\( T' \to T \) from another locally Noetherian scheme \( T' \), 
the base change \( Y \times_{T} T' \to T' \) is a \( \BQQ \)-Gorenstein morphism if and only if 
\( T' \to T \) factors through \( S \to T \). 
Theorem~\ref{thm:QGorRef} proves the existence of 
\( \BQQ \)-Gorenstein refinement, for example, in the case where \( f \) is proper 
and \( \omega^{[m]}_{Y_{t}/\Bbbk(t)} \) is invertible for a constant \( m > 0 \) 
for any fiber \( Y_{t} \). 
In this case, \( S \to T \) is shown to be separated and locally of finite type. 
Similar results are given as Theorem~\ref{thm:QGorRefLocal} for a local version 
and as Theorem~\ref{thm:dec naive} for naively \( \BQQ \)-Gorenstein morphisms. 
Koll\'ar's result \cite[Cor.\ 25]{KollarHusk} is stronger 
than Theorem~\ref{thm:QGorRef} when \( f \) is a projective morphism.

\subsection*{The role of our key proposition}

The deep results above on the topics \eqref{serious:1}--\eqref{serious:4} 
and some basic properties of \( \bfS_{2} \)-morphisms 
and \( \BQQ \)-Gorenstein morphisms are 
obtained by applying our key proposition (= Proposition~\ref{prop:key}). 
It proves that, for a flat morphism \( Y \to T \) of locally Noetherian schemes and 
for an exact sequence 
\[ 0 \to \SF \to \SE^{0} \to \SE^{1} \to \SG \to 0 \]  
of coherent \( \SO_{Y} \)-modules 
satisfying suitable conditions, 
the relative \( \bfS_{2} \)-condition for \( \SF \) over \( T \) 
is equivalent to the flatness of \( \SG \) over \( T \).  
For example, 
if \( f \) is an \( \bfS_{2} \)-morphism, then a reflexive sheaf \( \SF \) on \( Y \) 
admits such an exact sequence locally on \( Y \)  when \( \SF \) is locally free in 
codimension one on each fiber (cf.\ Lemma~\ref{lem:SurjFlat(reflexive)}).  
Therefore, the relative \( \bfS_{2} \)-condition for \( \SF \) over \( T \) can be studied by 
the flatness of another sheaf \( \SG \) defined locally on \( Y \). 
This is useful, since the sheaves \( \omega^{[m]}_{Y/T} \) are reflexive. 
For the relative canonical sheaf \( \omega_{Y/T} \) of an \( \bfS_{2} \)-morphism \( Y \to T \), 
we can show in the proof of Proposition~\ref{prop:BC-S2CM} that 
\( Y \) is locally embedded into an affine smooth \( T \)-scheme \( P \) as a closed subscheme 
and \( \omega_{Y/T} \) admits such an exact sequence as \( \SF \) on \( P \).

Applying the local criterion of flatness (cf.\ Proposition~\ref{prop:LCflat}) 
and the valuative criterion of flatness (cf.\ \cite[IV, Th.\ (11.8.1)]{EGA}) for \( \SG \), 
we have infinitesimal and valuative criteria for \( \SF \) 
to satisfy relative \( \bfS_{2} \) over \( T \) in Proposition~\ref{prop:inf+val}. 
This is applied to the reflexive sheaf \( \SF = \omega^{[m]}_{Y/T} \) 
in the topic \eqref{serious:2} above. 
The flattening stratification theorem by Mumford in \cite[Lect.\ 8]{Mumford} 
and the representability theorem of unramified functors by Murre \cite{Murre} applied to 
\( \SG \) yield Theorem~\ref{thm:dd dec} 
on the \emph{relative} \( \bfS_{2} \) \emph{refinement} for \( \SF \) 
(cf.\ Definition~\ref{dfn:relS2refinement}). 
This is defined as a monomorphism \( S \to T \) satisfying the following universal property: 
For a morphism \( T' \to T \) from a locally Noetherian scheme \( T' \) and 
for the induced morphisms \( p \colon Y' \to Y \) and \( Y' \to T' \) from 
the fiber product \( Y' = Y \times_{T} T' \), 
the double dual of \( p^{*}\SF \) satisfies relative \( \bfS_{2} \) over \( T' \) 
if and only if \( T' \to T \) factors through \( S \to T \). 
When \( Y \to T \) is projective, 
Theorem~\ref{thm:dd dec} is similar to Koll\'ar's result \cite[Th.~2]{KollarHusk} on ``hulls and husks.'' 
We also have a ``local version'' of the relative \( \bfS_{2} \) refinement for \( \SF \) 
as Theorem~\ref{thm:enhancement} by 
applying to \( \SG \) theorems on local universal flattening functor by 
Raynaud--Gruson \cite[Part 1, Th.\ (4.1.2)]{RG} or Raynaud \cite[Ch.\ 3, Th.\ 1]{Raynaud}. 
Applying the results on relative \( \bfS_{2} \) refinement for \( \SF = \omega^{[m]}_{Y/T} \), 
we have theorems on \( \BQQ \)-Gorenstein refinement mentioned in the topic \eqref{serious:4} above.  

The implication \eqref{prop:key:condBB} \( \Rightarrow \) \eqref{prop:key:condB} in 
Proposition~\ref{prop:key} is important. 
It is essential in the proof of Corollary~\ref{cor:SurjFlat(new)|Kollar}, 
and it is used in the proofs of Theorem~\ref{thm:invExt} 
on a criterion for a certain sheaf to be invertible, 
mentioned in the explanation of \eqref{serious:1} above 
and of Theorem~\ref{thm:S3Gor2} in \eqref{serious:3} above. 
Corollary~\ref{cor:SurjFlat(new)|Kollar} is similar to a special case of 
\cite[Th.\ 12]{KollarFlat}, but to which we have found a counterexample 
(cf.\ Example~\ref{exam:8-3}). 

\subsection*{Various other results and remarks}
Most parts of Sections~\ref{sect:Serre} and \ref{sect:GD} are surveys:  
Section~\ref{sect:Serre} discusses 
basic properties on \( \bfS_{k} \)-conditions and relative \( \bfS_{k} \)-conditions, 
and Section~\ref{sect:GD} 
discusses the dualizing complex and the relative dualizing complex in a little detail. 
Even in the surveys, we present some results and remarks, 
which seem to be new or not well known. These are listed as follows: 
\begin{itemize}

\item Some general properties on reflexive modules \( \SF \) over a locally Noetherian scheme \( X \) 
are presented in Lemmas~\ref{lem:j*reflexive}, \ref{lem:US2add}, \ref{lem:bc reflexive}, and 
Corollary~\ref{cor:prop:S1S2:reflexive}. 
These are related to the \( \bfS_{2} \)-conditions and the relative \( \bfS_{2} \)-conditions 
for \( \SF \) and for \( X \). Similar properties can be found in \cite[\S3]{HaKo}. 
Note that reflexive modules are well understood over an integral domain 
(cf.\ \cite[VII, \S4]{BAC}).

\item \emph{Every} \( \bfS_{2} \)-morphism (cf.\ Definition~\ref{dfn:SkCMmorphism}) 
\emph{of locally Noetherian schemes locally has pure relative dimension} 
(cf.\ Lemma~\ref{lem:S2Codim2}\eqref{lem:S2Codim2:1}).

\item \emph{For a locally Noetherian scheme \( X \) admitting a dualizing complex and for a coherent 
sheaf \( \SF \) on it, the \( \bfS_{k} \)-locus \( \bfS_{k}(\SF) \), 
the Cohen--Macaulay locus \( \CM(\SF) \), 
and the Gorenstein locus \( \Gor(X) \) are open} (cf.\ Proposition~\ref{prop:CMlocus}).  

\item Let \( X \) be a locally equi-dimensional and 
locally Noetherian scheme admitting a dualizing complex. 
In this case, \( X \) has an \emph{ordinary} dualizing complex \( \SR^{\bullet} \) 
and the dualizing sheaf \( \SL = \SH^{0}(\SR^{\bullet}) \) 
in the sense of Definition~\ref{dfn:ordinaryDC+DS} (cf.\ Lemma~\ref{lem:ordinaryDC}). 
\emph{Then, \( \SL \) satisfies \( \bfS_{2} \), and \( \SHom_{\SO_{X}}(\SL, \SL) \) 
is an \( \SO_{X} \)-algebra defining the \( \bfS_{2} \)-ification of \( X \)}
(cf.\ Proposition~\ref{prop:DSSk} and its Remark).

\item (cf.\ Corollary~\ref{cor:BC}) 
\emph{For a flat separated morphism \( Y \to T \) of finite type of Noetherian schemes 
and for the canonical inclusion morphism \( \psi_{t} \colon Y_{t} \to Y \) 
from the fiber \( Y_{t} = f^{-1}(t) \) over a point \( t \in T \), there is a quasi-isomorphism
\[ \bfL \psi_{t}^{*} (f^{!}\SO_{T}) \isom_{\qis} \omega^{\bullet}_{Y_{t}/\Bbbk(t)}  \]
of complexes of \( \SO_{Y_{t}} \)-modules, where  
\( f^{!}\SO_{T} \) is the} twisted inverse image 
\emph{of \( \SO_{T} \) by \( f \)} (cf.\ Example~\ref{exam:DeligneVerdier}), 
\emph{and \( \omega^{\bullet}_{Y_{t}/\Bbbk(t)} \) is the} canonical dualizing complex 
\emph{of the algebraic scheme \( Y_{t} \) over the residue field \( \Bbbk(t) \) of \( \SO_{T, t} \)}
(cf.\ Definition~\ref{dfn:canosheaf}). 
When \( \SO_{T, t} \) is regular, the assertion has been proved by \cite[Prop.\ 3.3(1)]{Pa}.

\end{itemize}

In the other sections \ref{sect:flat}, \ref{sect:Relcan}, \ref{sect:QGorSch}, and \ref{sect:QGormor}, 
we can find the following interesting results 
and remarks which are not listed above as notable ones:

\begin{itemize}
\item A counterexample of Koll\'ar's theorem \cite[Th.\ 12]{KollarFlat} 
is given in Example~\ref{exam:8-3}. 
Using another example produced there,  in Remark~\ref{rem:exam:8-3}, we explain 
a wrong assertion on the commutativity of projective limit and 
the direct image by open immersions for quasi-coherent sheaves.

\item \emph{For an \( \bfS_{2} \)-morphism \( f \colon Y \to T \) of locally Noetherian schemes 
and for the open immersion \( j \colon Y^{\flat} \injmap T \) from the relative Cohen--Macaulay locus 
\( Y^{\flat} = \CM(Y/T) \), 
the pushforward \( j_{*}\omega_{Y^{\flat}/T} \) is coherent, and it is isomorphic to 
\( \SH^{-d}(f^{!}\SO_{T}) \) locally on \( Y \) for the relative dimension \( d \) 
and the twisted inverse image \( f^{!}\SO_{T} \)} (cf.\ Proposition~\ref{prop:BC-S2CM}). 
\emph{Moreover, this sheaf is reflexive if every fiber is Gorenstein in codimension one} 
(cf.\ Proposition~\ref{prop:BCGor}). 

\item In Section~\ref{subsect:cone}, we discuss affine cones \( X \) 
of connected polarized projective schemes \( (S, \SA) \) over a field \( \Bbbk \). 
Here, the projective scheme \( S \) is not assumed to be reduced nor irreducible. 
In the sequel, we give useful conditions for  \( \omega^{[r]}_{X/\Bbbk} \) 
to satisfy \( \bfS_{k} \) and for \( X \) to be \( \BQQ \)-Gorenstein, 
in several situations of \( (S, \SA) \) (cf.\ Proposition~\ref{prop:coneGor}, 
Corollaries~\ref{cor:coneGor} and \ref{cor:coneGor2}). 

\item By Lemma~\ref{lem:Lee} and Example~\ref{exam:KummerType}, we present 
a new example of naively \( \BQQ \)-Gorenstein morphisms which is not \( \BQQ \)-Gorenstein 
in the case of morphisms of algebraic varieties over an algebraically closed field of characteristic zero. 
For known examples, see Fact~\ref{fact:Pa2}. 

\item For a naively \( \BQQ \)-Gorenstein morphism \( f \colon Y \to T \) and a point \( t \in T\), 
the relative Gorenstein index of \( f \) along the fiber \( Y_{t} = f^{-1}(t) \) 
coincides with the Gorenstein index of \( Y_{t} \) 
under suitable conditions (cf.\ Proposition~\ref{prop:GorIndexQGorMor}). 
This covers \cite[Lem.~3.16]{KSh}, but whose proof has problems explained in 
Remark~\ref{rem:KSh1}. 

\item In Remark~\ref{rem:naiveQGorCharP}, applying Corollary~\ref{cor:NaiveQGorCrit}, 
we verify 
the unboundedness of \( \{r_{n}\} \) for Koll\'ar's example 
of naively \( \BQQ \)-Gorenstein morphisms 
over the spectra of Artinian rings, 
explained in \cite[14.7]{HacKov} and \cite[Exam.\ 7.6]{Kovacs}. 
There, the proof is left to the reader, 
but we are afraid that the expected proof might have a problem similar to 
the second problem in Remark~\ref{rem:KSh1}. 

\item For a famous example (Example~\ref{exam:P114}) 
of deformations of the weighted projective plane \( \BPP(1, 1, 4) \) 
which is not \( \BQQ \)-Gorenstein, its \( \BQQ \)-Gorenstein refinement 
is determined in Example~\ref{exam:P114Part2} by using Lemma~\ref{lem:P114}. 

\end{itemize}

\subsection*{Organization of this article} 

In Section~\ref{sect:Serre}, we recall some basic notions and properties 
related to Serre's \( \bfS_{k} \)-condition. 
Section~\ref{subsect:SerreBasics} recalls basic properties on dimension, depth, 
and the \( \bfS_{k} \)-condition. 
The relative \( \bfS_{k} \)-condition is explained in Section~\ref{subsect:Rel}.
In Section~\ref{sect:flat}, we proceed the study of relative \( \bfS_{2} \)-condition 
and give some criteria for this condition. 
Section~\ref{subsect:Resthom} is devoted to prove the key proposition (Proposition~\ref{prop:key}) 
and related properties. 
Some applications of Proposition~\ref{prop:key} are given 
in Section~\ref{subsect:AppResthom}: Theorem~\ref{thm:invExt} on a criterion 
for a certain sheaf to be invertible, 
Proposition~\ref{prop:inf+val} on infinitesimal and valuative criteria, 
Theorem~\ref{thm:dd dec} on the relative \( \bfS_{2} \) refinement, and 
its local version: Theorem~\ref{thm:enhancement}. 

The theory of Grothendieck duality 
is surveyed briefly in Section~\ref{sect:GD} with a few original results. 
In Sections~\ref{subsect:dualizingcpx} and \ref{subsect:ordinaryDC}, 
we recall some well-known properties 
on the dualizing complex based on arguments in \cite{ResDual} and \cite{Conrad}. 
The twisted inverse image functor
is explained in Section~\ref{subsect:twisted inverse} with the famous Grothendieck duality theorem for
proper morphisms (cf.\ Theorem~\ref{thm:DualityProper}).
The base change theorem for the relative dualizing sheaf for a Cohen--Macaulay morphism is
mentioned in Section~\ref{subsect:CMmorGormor}.
In Section~\ref{sect:Relcan}, 
we give some technical base change results for the relative canonical sheaf of
an \( \bfS_{2} \)-morphism. 
Section~\ref{subsect:bcS3} is devoted to proving Theorem~\ref{thm:S2S3crit} 
on a criterion for a certain sheaf related to the relative canonical sheaf to be invertible. 
This technical theorem also gives sufficient conditions for 
the base change homomorphism of the relative canonical sheaf to the fiber to be an isomorphism, 
and it is applied to the proof of Theorem~\ref{thm:wQGvsQG} 
on virtually \( \BQQ \)-Gorenstein morphism (cf.\ the topic \eqref{serious:1} above).  

In Section~\ref{sect:QGorSch}, we study \( \BQQ \)-Gorenstein schemes. 
The definition and its basic properties are given in Section~\ref{subsect:QGorSch}. 
As an example of \( \BQQ \)-Gorenstein schemes, in Section~\ref{subsect:cone},
we consider the case of affine cones over polarized projective schemes over a field. 
In Section~\ref{sect:QGormor}, we study \( \BQQ \)-Gorenstein morphisms, and 
two variants: naively \( \BQQ \)-Gorenstein morphisms and virtually \( \BQQ \)-Gorenstein morphisms. 
The \( \BQQ \)-Gorenstein morphism and the naively \( \BQQ \)-Gorenstein morphism are defined in
Section~\ref{subsect:QGormor}, and their basic properties are discussed. 
The virtually \( \BQQ \)-Gorenstein morphism 
is defined in Section~\ref{subsect:virQGormor}, and we prove 
Theorem~\ref{thm:wQGvsQG} on a criterion 
for a virtually \( \BQQ \)-Gorenstein morphism to be \( \BQQ \)-Gorenstein 
(cf.\ the topic \eqref{serious:1} above). 
In Section~\ref{subsect:propQGormor}, several basic properties including base change 
of \( \BQQ\)-Gorenstein morphisms and of their variants are discussed. 
Theorems mentioned in the topics 
\eqref{serious:2}--\eqref{serious:4} above are proved in Section~\ref{subsect:thmsQGormor}.

Some elementary facts on local criterion of flatness and base change isomorphisms are
explained in Appendix~\ref{sect:Basics} for the readers' convenience. 
In this article, we try to cite references kindly as much as possible for the readers' convenience 
and for the authors' assurance. We also try to refer to the original article if possible.

\subsection*{Acknowledgements}
The first named author would like to thank Research Institute for Mathematical Sciences (RIMS) 
in Kyoto University for their support and hospitality. 
This work was initiated during his stay at RIMS in 2010 and 
is continued to his next stay in 2016. 
The second named author expresses his thanks to Professor Yoshio Fujimoto 
for his advices and continuous encouragement. 
Both authors are grateful to Professors J\'anos Koll\'ar and S\'andor Kov\'acs 
for useful information. 

\subsection*{Notation and conventions}

\begin{enumerate}

\item \label{nc:truncshift}
For a complex
\( K^{\bullet} = [\cdots \to K^{i} \xrightarrow{d^{i}} K^{i+1} \to \cdots]\) in
an abelian category and
for an integer \( q \), we denote by \( \tau^{\leq q}(K^{\bullet}) \)
(resp.\ \( \tau^{\geq q}(K^{\bullet}) \))
the ``truncation'' of \( K^{\bullet} \), which
is defined as the complex
\begin{align*}
&[\cdots \to K^{q-2} \xrightarrow{d^{q-2}} K^{q-1} \to \Ker(d^{q}) \to 0 \to \cdots] \\
(\text{resp.  }
&[\cdots \to 0 \to \Coker(d^{q-1}) \to K^{q+1} \xrightarrow{d^{q+1}} K^{q+2} \to \cdots])
\end{align*}
(cf.\ \cite[D\'ef.~1.1.13]{DeligneSGA4-17}).
The complex \( K^{\bullet}[m] \) shifted by an integer \( m \)
is defined as the complex
\( L^{\bullet} = [\cdots \to L^{i} \xrightarrow{d_{L}^{i}} L^{i+1} \to \cdots]\)
such that \( L^{i} = K^{i + m} \) and \( d_{L}^{i} = (-1)^{m}d^{i + m} \) for any \( i \in \BZZ \). 
It is known that the mapping cone of the natural morphism
\( \tau^{\leq q}(K^{\bullet}) \to K^{\bullet} \) is quasi-isomorphic to \( \tau^{\geq q+1}(K^{\bullet}) \) 
for any \( q \in \BZZ \).

\item  For a complex
\( K^{\bullet} \) in an abelian category
(resp.\ for an object \( K^{\bullet} \) of the derived category),
the \( i \)-th cohomology of \( K^{\bullet} \)
is denoted usually by \( \OH^{i}(K^{\bullet}) \).
For a complex \( \SK^{\bullet} \) of sheaves on a scheme,
the \( i \)-th cohomology is a sheaf and is denoted by \( \SH^{i}(\SK^{\bullet}) \).

\item  The derived category of an abelian category \( \SAA \) is denoted by \( \bfD(\SAA) \).
Moreover, we write \( \bfD^{+}(\SAA) \) (resp.\ \( \bfD^{-}(\SAA) \), resp.\ \( \bfD^{b}(\SAA) \))
for the full subcategory consisting of bounded below (resp.\ bounded above, resp.\ bounded)
complexes.

\item An \emph{algebraic scheme} over a field \( \Bbbk \)
means a \( \Bbbk \)-scheme of finite type.
An \emph{algebraic variety} over \( \Bbbk \)
is an integral separated algebraic scheme over \( \Bbbk \).

\item \label{conv:sch:sheaf}
For a scheme \( X \), a sheaf of \( \SO_{X} \)-modules is called
an \( \SO_{X} \)-\emph{module} for simplicity.
A coherent (resp.\ quasi-coherent) sheaf on \( X \)
means a coherent (resp.\ quasi-coherent) \( \SO_{X} \)-module.
The (abelian) category of \( \SO_{X} \)-modules (resp.\ quasi-coherent \( \SO_{X} \)-modules)
is denoted by \( \Mod(\SO_{X}) \) (resp.\ \( \QCoh(\SO_{X}) \)).

\item For a scheme \( X \) and a point \( x \in X \),
the maximal ideal (resp.\ the residue field)
of the local ring \( \SO_{X, x} \) is denoted by \( \GM_{X, x} \) (resp.\ \( \Bbbk(x) \)).
The stalk of a sheaf \( \SF \) on \( X \) at \( x \)
is denoted by \( \SF_{x} \).

\item For a morphism \( f \colon Y \to T \) of schemes and for a point \( t \in T \), 
the fiber \( f^{-1}(t) \) over \( t \) is defined as \( Y \times_{T} \Spec \Bbbk(t) \) and 
is denoted by \( Y_{t} \). For an \( \SO_{Y} \)-module \( \SF \), 
the restriction \( \SF \otimes_{\SO_{Y}} \SO_{Y_{t}}  \) to the fiber \( Y_{t} \) is denoted by \( \SF_{(t)} \) 
(cf.\ Notation~\ref{nota:F_(t)}). 

\item \label{nota:sch:D+(X)}
The derived category of a scheme \( X \) is defined as the derived category of
\( \Mod(\SO_{X}) \), and is denoted by \( \bfD(X) \).
The full subcategory consisting of complexes with quasi-coherent (resp.\ coherent)
cohomology is denoted by \( \bfD_{\qcoh}(X) \) (resp.\ \(\bfD_{\coh}(X)\)).
For \( * = + \), \( - \), \( b \) and for \( \dag = \qcoh \), \( \coh \), we set
\[ \bfD^{*}(X) = \bfD^{*}(\Mod(\SO_{X})) \quad \text{and} \quad
\bfD^{*}_{\dag}(X) = \bfD^{*}(X) \cap \bfD_{\dag}(X).\]

\item \label{nota:sch:localcohomology}
For a sheaf \( \SF \) on a scheme \( X \) and for a closed subset \( Z \),
the \( i \)-th local cohomology sheaf of \( \SF \) with support in \( Z \)
is denoted by \( \SH^{i}_{Z}(\SF) \) (cf.\ \cite{LC}).

\item
For a morphism \( X \to Y \) of schemes,
\( \varOmega^{1}_{X/Y} \) denotes the sheaf of relative one-forms.
When \( X \to Y \) is smooth, \( \varOmega^{p}_{X/Y} \) denotes the \( p \)-th exterior power 
\( \bigwedge^{p} \! \varOmega^{1}_{X/Y} \) for integers \( p \geq 0 \).

\end{enumerate}


\section{Serre's \texorpdfstring{$\bfS_{k}$}{Sk}-condition}
\label{sect:Serre}

We shall recall several fundamental properties on locally Noetherian schemes, which
are indispensable for understanding the explanation of dualizing complex and
Grothendieck duality in Section~\ref{sect:GD}
as well as the discussion of relative canonical sheaves and \( \BQQ \)-Gorenstein morphisms
in Sections~\ref{sect:Relcan} and \ref{sect:QGorSch}, respectively.
In Section~\ref{subsect:SerreBasics}, we recall basic properties on
dimension, depth, Serre's \( \bfS_{k} \)-condition especially for \( k = 1 \) and \( 2 \), 
and reflexive sheaves. 
The relative \( \bfS_{k} \)-condition is discussed in Section~\ref{subsect:Rel}.


\subsection{Basics on Serre's condition}
\label{subsect:SerreBasics}

The 
\( \bfS_{k} \)-condition is defined by
``depth'' and ``dimension.''
We begin with recalling some elementary properties on dimension, codimension, and on depth.

\begin{ppty}[dimension, codimension]\label{ppty:dim-codim}
Let \( X \) be a scheme and let \( \SF \) be
a quasi-coherent \( \SO_{X} \)-module \emph{of finite type} (cf.\ \cite[$0_{\text{I}}, (5.2.1)$]{EGA}),
i.e., \( \SF \) is quasi-coherent and locally finitely generated as an \( \SO_{X} \)-module.
Then, \( \Supp \SF \) is a closed subset (cf.\ \cite[$0_{\text{I}}$, (5.2.2)]{EGA}).
\begin{enumerate}
\item  \label{ppty:dim-codim:1}
If \( Y \) is a closed subscheme of \( X \)
such that \( Y = \Supp \SF \) as a set, then
\[ \dim \SF_{y} = \dim \SO_{Y, y} = \Codim(\overline{\{y\}}, Y) \]
for any point \( y \in Y \),
where \( \dim \SF_{y} \) is considered as the dimension of the closed subset
\( \Supp \SF_{y} \) of \( \Spec \SO_{X, y} \) (cf.\ \cite[IV, (5.1.2), (5.1.12)]{EGA}).

\item  \label{ppty:dim-codim:15}
The dimension of \( \SF \), denoted by \( \dim \SF \),
is defined as \( \dim \Supp \SF \) (cf.\ \cite[IV, (5.1.12)]{EGA}).
Then,
\[ \dim \SF = \sup\{ \dim \SF_{x} \mid x \in X\}\]
(cf.\ \cite[IV, (5.1.12.3)]{EGA}).
If \( X \) is locally Noetherian, then
\[ \dim \SF = \sup\{\dim \SF_{x} \mid x \text{ is a closed point of } X \} \]
by \cite[IV, (5.1.4.2), (5.1.12.1), and Cor.~(5.1.11)]{EGA}.
Note that the local dimension of \( \SF \) at a point \( x \), denoted by \( \dim_{x} \SF \), 
is just the infimum of \( \dim \SF|_{U} \) for all the open neighborhoods \( U \) of \( x \).

\item  \label{ppty:dim-codim:2}
For a closed subset \( Z \subset X \), the equality
\[ \Codim(Z, X) = \inf \{ \dim \SO_{X, z} \mid z \in Z\} \]
holds, and moreover, if \( X \) is locally Noetherian, then
\[ \Codim_{x}(Z, X) = \inf \{ \dim \SO_{X, z} \mid z \in Z, \, x \in \overline{\{z\}} \} \]
for any point \( x \in X \) (cf.\ \cite[IV, Cor.~(5.1.3)]{EGA}).
Note that \( \Codim(\emptyset, X) = +\infty \) and
that \( \Codim_{x}(Z, X) = +\infty \) if \( x \not\in Z \).
Furthermore, if \( Z \) is locally Noetherian, then the function
\( x \mapsto \Codim_{x}(Z, X) \)
is lower semi-continuous on \( X \)
(cf.\ \cite[$0_{\text{IV}}$, Cor.~(14.2.6)(ii)]{EGA}).
\end{enumerate}
\end{ppty}

\begin{dfn}[equi-dimensional]\label{dfn:equi-dim}
Let \( X \) be a scheme and \( \SF \) a quasi-coherent \( \SO_{X} \)-module of finite type.
Let \( A \) be a ring and \( M \) a finitely generated \( A \)-module.
\begin{enumerate}
\item 
We call \( X \) (resp.\ \( \SF \)) \emph{equi-dimensional} if all the irreducible components of
\( X \) (resp.\ \( \Supp \SF \))
have the same dimension. 

\item We call \( A \) (resp.\ \( M \)) \emph{equi-dimensional} if
all the irreducible components of \( \Spec A \) (resp.\ \( \Supp M \))
have the same dimension, where \( \Supp M \) is the closed subset of \( \Spec A \)
defined by the annihilator ideal \( \Ann(M) \).
Note that \( \Supp M \) equals \( \Supp M\sptilde \) for
the associated quasi-coherent sheaf \( M\sptilde \) on \( \Spec A \).

\item  \label{dfn:equi-dim:locally} We call \( X \) (resp.\ \( \SF \)) \emph{locally equi-dimensional}
if the local ring \( \SO_{X, x} \) (resp.\ the stalk \( \SF_{x} \) as an \( \SO_{X, x} \)-module)
is equi-dimensional for any point \( x \in X \).
\end{enumerate}
\end{dfn}

\begin{ppty}[catenary]\label{pprt:catenary}
A scheme \( X \) is said to be \emph{catenary} if
\[ \Codim(Y, Z) + \Codim(Z, T) = \Codim(Y, T) \]
for any irreducible closed subsets
\( Y \subset Z \subset T \) of \( X \) (cf.\ \cite[$0_{\text{IV}}$, Prop.~(14.3.2)]{EGA}).
A ring \( A \) is said to be catenary if \( \Spec A \) is so.
Then, for a scheme \( X \), it is catenary if and only if every local ring \( \SO_{X, x} \) is catenary
(cf.\ \cite[IV, Cor.~(5.1.5)]{EGA}).
If \( X \) is a locally Noetherian scheme and if \( \SO_{X, x} \) is catenary for a point \( x \in X \),
then
\[ \Codim_{x}(Y, X) = \dim \SO_{X, x} - \dim \SO_{Y, x} \]
for any closed subscheme \( Y \) of \( X \) containing \( x \) (cf.\ \cite[IV, Prop.~(5.1.9)]{EGA}).
\end{ppty}

\begin{ppty}[depth]\label{ppty:depthdfn}
Let \( A \) be a Noetherian ring, \( I \) an ideal of \( A \), and let
\( M \) be a finitely generated \( A \)-module.
The \( I \)-\emph{depth} of \( M \), denoted by \( \depth_{I} M\),
is defined as
the length of any maximal \( M \)-regular sequence contained in \( I \) when \( M \ne IM \), and
as \( +\infty \) when \( M = IM \).
Here, an element \( a \in I \) is said to be \( M \)-\emph{regular} if
\( a \) is not a zero divisor of \( M \), i.e.,
the multiplication map \( x \mapsto ax\) induces an injection \( M \to M \), and
a sequence \( a_{1}, a_{2}, \ldots, a_{n} \) of elements of \( I \) is
said to be \( M \)-\emph{regular} if \( a_{i} \) is \( M_{i} \)-regular for any \( i \),
where \( M_{i} = M/ (a_{1}, \ldots, a_{i-1})M \).
The following equality is well known
(cf.\ \cite[Prop.~3.3]{LC}, \cite[III, Prop.\ 2.4]{SGA2},
\cite[Th.~16.6, 16.7]{Matsumura}):
\[ \depth_{I} M = \inf\{ i \in \BZZ_{\geq 0} \mid \Ext^{i}_{A}(A/I, M) \ne 0\}. \]
If \( A \) is a local ring and if \( I \) is the maximal ideal \( \GM_{A} \),
then \( \depth_{I} M \) is denoted simply by \( \depth M \);
In this case, we have \( \depth M \leq \dim M \) when \( M \ne 0 \)
(cf.\ \cite[$0_{\text{IV}}$, (16.4.5.1)]{EGA},
\cite[Exer.\ 16.1, Th.~17.2]{Matsumura}).
\end{ppty}

\begin{dfn}[$Z$-depth]\label{dfn:Z-depth}
Let \( X \) be a locally Noetherian scheme and \( \SF \) a coherent \( \SO_{X} \)-module.
For a closed subset \( Z \) of \( X \),
the \emph{\( Z \)-depth} of \( \SF \) is defined as
\[ \depth_{Z} \SF = \inf\{ \depth \SF_{z} \mid z \in Z\}
\]
(cf.\ \cite[p.~43, Def.]{LC},
\cite[IV, (5.10.1.1)]{EGA}, \cite[III, Def.~(3.12)]{AltmanKleiman}),
where the stalk \( \SF_{z} \) of \( \SF \) at \( z \)
is regarded as an \( \SO_{X, z} \)-module. Note that \( \depth_{Z} 0 = +\infty \).
\end{dfn}

\begin{ppty}[{cf.\ {\cite[Th.~3.8]{LC}}}]\label{ppty:depth<=2}
In the situation above, for a given integer \( k \geq 1\),
one has the equivalence:
\[ \depth_{Z} \SF \geq k \quad
\Longleftrightarrow \quad \SH^{i}_{Z}(\SF) = 0 \quad \text{for any }  i < k .\]
Here, \( \SH^{i}_{Z}(\SF) \) stands for the \( i \)-th
local cohomology sheaf of \( \SF \) with support in \( Z \) (cf.\ \cite{LC}, \cite{SGA2}).
In particular, the condition: \( \depth_{Z} \SF \geq 1 \) (resp.\ \( \geq
2 \)) is equivalent to that the restriction homomorphism
\( \SF \to j_{*}(\SF|_{X \setminus Z}) \) is an injection (resp.\ isomorphism)
for the open immersion \( j \colon X \setminus Z \injmap X \).
Furthermore, the condition: \(\depth_{Z} \SF \geq 3 \) is equivalent to:
\( \SF \isom j_{*}(\SF|_{X \setminus Z}) \) and
\( R^{1}j_{*}(\SF|_{X\setminus Z}) = 0 \).
\end{ppty}

\begin{remn}[{cf.\ \cite[Cor.~3.6]{LC}, \cite[III, Cor.~3.14]{AltmanKleiman}}]
Let \( A \) be a Noetherian ring with an ideal \( I \)
and let \( M \) be a finitely generated \( A \)-module.
Then,
\[ \depth_{I} M = \depth_{Z} M\sptilde \]
for the closed subscheme \( Z = \Spec A/I \) of \( X = \Spec A \)
and for the coherent \( \SO_{X} \)-module \( M\sptilde \) associated with \( M \).
\end{remn}

\begin{rem}[associated prime]\label{rem:assocprime}
Let \( \SF \) be a coherent \( \SO_{X} \)-module on a locally Noetherian scheme \( X \). 
A point \( x \in X \) 
is called an \emph{associated point} of \( \SF \) if the maximal ideal
\( \GM_{x} \) is an associated prime of the stalk \( \SF_{x} \) (cf.\ \cite[IV, D\'ef.~(3.1.1)]{EGA}).
This condition is equivalent to: \( \depth \SF_{x} = 0 \).
We denote by \( \Ass(\SF) \) the set of associated points.
This is a discrete subset of \( \Supp \SF \).
If an associated point \( x \) of \( \SF \) is not a generic point of \( \SF \), i.e., \( \depth \SF_{x} = 0 \)
and \( \dim \SF_{x} > 0 \), then \( x \) is called the \emph{embedded point} of \( \SF \).
If \( X = \Spec A \) and \( \SF = M\sptilde \) for a Noetherian ring \( A \) and
for a finitely generated \( A \)-module \( M \), then \( \Ass(\SF) \) is just the set of
associated primes of \( M \), and the embedded points of \( \SF \) are the embedded primes of \( M \).
\end{rem}

\begin{rem}\label{rem:depth:local isom}
Let \( \phi \colon \SF \to j_{*}(\SF|_{X \setminus Z}) \) be the homomorphism
in Property~\ref{ppty:depth<=2}
and set \( U = X \setminus Z \).
Then, \( \phi \) is an injection (resp.\ isomorphism) at a point \( x \in Z \),
i.e., the homomorphism
\[ \phi_{x} \colon \SF_{x} \to (j_{*}(\SF|_{U}))_{x}\]
of stalks is an injection (resp.\ isomorphism), if and only if
\[ \depth \SF_{x'} \geq 1 \qquad (\text{resp. } \geq 2) \]
for any \( x' \in Z \) such that \( x \in \overline{\{x'\}} \).
In fact, \( \phi_{x} \) is identical to the inverse image \( p_{x}^{*}(\phi) \)
by a canonical morphism \( p_{x} \colon \Spec \SO_{X, x} \to X \), and it is regarded as
the restriction homomorphism of \( p_{x}^{*}(\SF) \) to the open subset \( U_{x} = p_{x}^{-1}(U) \)
via the base change isomorphism
\[ p_{x}^{*}(j_{*}(\SF|_{U})) \isom j_{x*}((p_{x}^{*}\SF)|_{U_{x}}) \]
(cf.\ Lemma~\ref{lem:flatbc} below),
where \( j_{x} \) stands for the open immersion \( U_{x} \injmap \Spec \SO_{X, x} \).
For the complement \( Z_{x} = p_{x}^{-1}(Z)\) of \( U_{x} \) in \( \Spec \SO_{X, x} \), by
Property~\ref{ppty:depth<=2}, we know that
\( p_{x}^{*}(\phi) \) is an injection (resp.\ isomorphism) if and only if
\[ \depth_{Z_{x}} p_{x}^{*}(\SF) \geq 1 \qquad (\text{resp. } \geq 2). \]
This implies the assertion, since \( Z_{x} \) is identical to
the set of points \( x' \in Z\) such that \( x \in \overline{\{x'\}} \).
\end{rem}

We recall Serre's condition \( \bfS_{k} \)
(cf.\ \cite[IV, D\'ef.~(5.7.2)]{EGA},
\cite[VII, Def.~(2.1)]{AltmanKleiman}, \cite[p.~183]{Matsumura}):

\begin{dfn}\label{dfn:SerreCond}
Let \( X \) be a locally Noetherian scheme, \( \SF \) a coherent \( \SO_{X} \)-module,
and \( k \) a positive integer.
We say that \( \SF \) satisfies \( \bfS_{k} \)
if the inequality
\[ \depth \SF_{x} \geq \inf \{k, \dim \SF_{x}\}\]
holds for any point \( x \in X \), where
the stalk \( \SF_{x} \) at \( x \) is considered as an \( \SO_{X, x} \)-module.
We say that \( \SF \) \emph{satisfies} \( \bfS_{k} \) \emph{at a point} \( x \in X\) if
\[ \depth \SF_{y} \geq \inf \{k, \dim \SF_{y}\}\]
for any point \( y \in X \) such that \( x \in \overline{\{y\}} \).
We say that \( X \) satisfies \( \bfS_{k} \), if \( \SO_{X} \) does so. 
\end{dfn}

\begin{remn}
In the situation of Definition~\ref{dfn:SerreCond}, assume that \( \SF = i_{*}(\SF')\)
for a closed immersion \( i \colon X' \injmap X\) and
for a coherent \( \SO_{X'} \)-module \( \SF' \).
Then, \( \SF \) satisfies \( \bfS_{k} \) if and only if \( \SF'\) does so.
In fact,
\[ \depth \SF_{x} = +\infty
\quad \text{ and } \quad
\dim \SF_{x} = -\infty \]
for any \( x \not\in X' \), and
\[ \depth \SF_{x} = \depth \SF'_{x}
\quad \text{ and } \quad
\dim \SF_{x} = \dim \SF'_{x} \]
for any \( x \in X' \)
(cf.\ \cite[$0_{\text{IV}}$, Prop.~(16.4.8)]{EGA}).
\end{remn}

\begin{rem}\label{rem:dfn:SerreCond:atx}
Let \( A \) be a Noetherian ring and \( M \) a finitely generated \( A \)-module.
For a positive integer \( k \),
we say that \( M \) \emph{satisfies} \( \bfS_{k} \) if the associated coherent sheaf
\( M\sptilde \) on \( \Spec A \) satisfies \( \bfS_{k} \).
Then, for \( X \), \( \SF \), and \( x \) in Definition~\ref{dfn:SerreCond},
\( \SF \) satisfies \( \bfS_{k} \) at \( x \) if and only if the \( \SO_{X, x} \)-module
\( \SF_{x} \) satisfies \( \bfS_{k} \).
In fact, by considering
\( \Supp \SF_{x} \) as a closed subset of \( \Spec \SO_{X, x} \) and
by the canonical morphism \( \Spec \SO_{X, x} \to X \),
we can identify \( \Supp \SF_{x} \)
with the set of points \( y \in \Supp \SF \) such that \( x \in \overline{\{y\}} \).
\end{rem}

\begin{dfn}[Cohen--Macaulay]\label{dfn:CM}
Let \( A \) be a Noetherian local ring and \( M \) a finitely generated \( A \)-module.
Then, \( M \) is said to be \emph{Cohen--Macaulay}
if \( \depth M = \dim M \) unless \( M = 0 \)
(cf.\ \cite[$0_{\text{IV}}$, D\'ef.\ (16.5.1)]{EGA}, \cite[\S17]{Matsumura}).
In particular, if \( \dim A = \depth A \), then \( A \) is called a Cohen--Macaulay local ring.
Let \( X \) be a locally Noetherian scheme and \( \SF \) a coherent \( \SO_{X} \)-module.
If the \( \SO_{X, x} \)-module \( \SF_{x} \) is Cohen--Macaulay for any \( x \in X \),
then \( \SF \) is said to be Cohen--Macaulay (cf.\ \cite[IV, D\'ef.\ (5.7.1)]{EGA}.
If \( \SO_{X} \) is Cohen--Macaulay, then \( X \) is called
a Cohen--Macaulay scheme.
\end{dfn}

\begin{rem}\label{rem:dfn:CM}
For \( A \) and \( M \) above,
it is known that if \( M \) is Cohen--Macaulay, then the localization \( M_{\Gp} \)
is also Cohen--Macaulay for any prime ideal \( \Gp \) of \( A \)
(cf.\ \cite[$0_{\text{IV}}$, Cor.~(16.5.10)]{EGA}, \cite[Th.~17.3]{Matsumura}).
Hence, \( M \) is Cohen--Macaulay if and only if \( M \) satisfies \( \bfS_{k} \)
for any \( k \geq 1 \).
\end{rem}

\begin{dfn}[$\bfS_{k}(\SF)$, $\CM(\SF)$]\label{dfn:SkCMlocus}
Let \( X \) be a locally Noetherian scheme  and
let \( \SF \) be a coherent \( \SO_{X} \)-module.
For an integer \( k \geq 1 \), the \emph{\( \bfS_{k} \)-locus}
\( \bfS_{k}(\SF) \) of \( \SF \) is
defined to be the set of points \( x \in X \) at which \( \SF \) satisfies \( \bfS_{k} \)
(cf.\ Definition~\ref{dfn:SerreCond}).
The \emph{Cohen--Macaulay locus} \( \CM(\SF) \)
of \( \SF \)
is defined to be the set of points \( x \in \SF \)
such that \( \SF_{x} \) is a Cohen--Macaulay \( \SO_{X, x} \)-module.
By definition and by Remark~\ref{rem:dfn:CM},
one has: \( \CM(\SF) = \bigcap_{k \geq 1} \bfS_{k}(\SF) \).
We define \( \bfS_{k}(X) := \bfS_{k}(\SO_{X}) \) and \( \CM(X) = \CM(\SO_{X}) \), and
call them the \( \bfS_{k} \)-locus and the Cohen--Macaulay locus of \( X \), respectively.
\end{dfn}

\begin{remn}
It is known that \( \bfS_{k}(\SF) \) and \( \CM(\SF) \) are
open subsets when \( X \) is locally a subscheme of a regular scheme
(cf.\ \cite[IV, Prop.~(6.11.2)(ii)]{EGA}).
In Proposition~\ref{prop:CMlocus} below, we shall prove the openness when \( X \)
admits a dualizing complex.
\end{remn}

\begin{remn}
For a locally Noetherian scheme \( X \), every generic point of \( X \) is contained in
the Cohen--Macaulay locus \( \CM(X) \).
For, \( \dim A = \depth A = 0 \) for any Artinian local ring \( A \).
\end{remn}

Lemmas~\ref{lem:basicSk} and \ref{lem:depth+codim+Sk} below are 
basic properties on the condition \( \bfS_{k} \).

\begin{lem}\label{lem:basicSk}
Let \( X \) be a locally Noetherian scheme and let \( \SG \) be a coherent \( \SO_{X} \)-module.
For a positive integer \( k \), the following conditions are equivalent to each other\emph{:}
\begin{enumerate}
\renewcommand{\theenumi}{\roman{enumi}}
\renewcommand{\labelenumi}{(\theenumi)}
\item  \label{lem:basicSk:condA}
The sheaf \( \SG \) satisfies \( \bfS_{k} \).
\item  \label{lem:basicSk:condB}
The inequality
\[ \depth_{Z} \SG \geq \inf\{k, \Codim(Z, \Supp \SG)\} \]
holds for any closed \emph{(}resp.\ irreducible and closed\emph{)} subset \( Z \subset \Supp \SG\).
\item  \label{lem:basicSk:condC}
The sheaf \( \SG \) satisfies \( \bfS_{k-1} \) when \( k \geq 2 \), and
\( \depth_{Z} \SG \geq k \) for any closed \emph{(}resp.\ irreducible and closed\emph{)}
subset \( Z \subset \Supp \SG \) such that
\( \Codim(Z, \Supp \SG)
\linebreak 
\geq k \).
\item  \label{lem:basicSk:condD}
There is a closed subset \( Z \subset \Supp \SG \) such that \( \depth_{Z} \SG \geq k \) and
\( \SG|_{X \setminus Z} \) satisfies \( \bfS_{k} \).
\end{enumerate}
\end{lem}

\begin{proof}
We may assume that \( \SG \) is not zero.
The equivalence \eqref{lem:basicSk:condA} \( \Leftrightarrow \)
\eqref{lem:basicSk:condB} follows from
Definitions~\ref{dfn:Z-depth} and \ref{dfn:SerreCond}
and from the equality: \( \dim \SG_{x} = \Codim(\overline{\{x\}}, \Supp \SG) \)
for \( x \in \Supp \SG \) in Property~\ref{ppty:dim-codim}\eqref{ppty:dim-codim:1}.
The equivalence \eqref{lem:basicSk:condA} \( \Leftrightarrow \) \eqref{lem:basicSk:condB}
implies the equivalence: \eqref{lem:basicSk:condB} \( \Leftrightarrow \) \eqref{lem:basicSk:condC}.
We have \eqref{lem:basicSk:condA} \( \Rightarrow \) \eqref{lem:basicSk:condD}
by taking a closed subset \( Z \) with \( \Codim(Z, \Supp \SG) \geq k \) using
the inequality in \eqref{lem:basicSk:condB}.
It is enough to show: \eqref{lem:basicSk:condD} \( \Rightarrow \) \eqref{lem:basicSk:condA}.
More precisely, it is enough to prove that, in the situation of \eqref{lem:basicSk:condD},
the inequality
\[ \depth \SG_{x} \geq \inf\{k, \dim \SG_{x}\} \]
holds for any point \( x \in X \).
If \( x \not\in Z \), then this holds, since \( \SG|_{X \setminus Z} \) satisfies \( \bfS_{k} \).
If \( x \in Z \), then \( \dim \SG_{x} \geq \depth \SG_{x} \geq \depth_{Z} \SG \geq k \)
(cf.\ Property~\ref{ppty:depthdfn} and Definition~\ref{dfn:Z-depth}), and
it induces the inequality above. Thus, we are done.
\end{proof}

\begin{lem}\label{lem:depth+codim+Sk}
Let \( X \) be a locally Noetherian scheme and \( \SG \) a coherent \( \SO_{X} \)-module. 
Then, for any closed subset \( Z \) of \( X \), the following hold\emph{:}
\begin{enumerate}
\item  \label{lem:depth+codim+Sk:1} One has the inequality 
\[ \depth_{Z} \SG \leq \Codim(Z \cap \Supp \SG, \Supp \SG). \]
\item  \label{lem:depth+codim+Sk:2} 
For an integer \( k > 0\),  if \( \SG \) satisfies \( \bfS_{k} \) and if
\( \Codim(Z \cap \Supp \SG, \Supp \SG) \geq k \), then \( \depth_{Z} \SG \geq k \). 
\end{enumerate}
\end{lem}

\begin{proof}
The inequality in \eqref{lem:depth+codim+Sk:1} follows from the inequality
\( \depth \SG_{x} \leq \dim \SG_{x} \) for any \( x \in \Supp \SG \),
since
\begin{align*}
\Codim(Z \cap \Supp \SG, \Supp \SG) &= \inf \{ \dim \SG_{x} \mid x \in Z \cap \Supp \SG\}
\quad \text{and} \\
\depth_{Z} \SG &=  \inf\{ \depth \SG_{x} \mid x \in Z \cap \Supp \SG\}
\end{align*}
when \( Z \cap \Supp \SG \ne \emptyset \),
by Property~\ref{ppty:dim-codim} and Definition~\ref{dfn:Z-depth}.
The assertion \eqref{lem:depth+codim+Sk:2} is derived from the equivalence
\eqref{lem:basicSk:condA} \( \Leftrightarrow \) \eqref{lem:basicSk:condB}
of Lemma~\ref{lem:basicSk}.
\end{proof}

For the conditions \( \bfS_{1} \) and \( \bfS_{2} \), we have immediately
the following corollary of Lemma~\ref{lem:basicSk} by considering
Property~\ref{ppty:depth<=2}.

\begin{cor}\label{cor:basicS1S2}
Let \( X \) be a locally Noetherian scheme and let \( \SG \) be a coherent \( \SO_{X} \)-module.
The following three conditions are equivalent to each other,
where \( j \) denotes the open immersion \( X \setminus Z \injmap X \)\emph{:}

\begin{enumerate}
\renewcommand{\theenumi}{\roman{enumi}}
\item \label{cor:basicS1S2:condA}
The sheaf \( \SG \) satisfies \( \bfS_{1} \) \emph{(}resp.\ \( \bfS_{2} \)\emph{)}.

\item \label{cor:basicS1S2:condB}
For any closed subset \( Z \subset \Supp \SG\) with
\( \Codim(Z, \Supp \SG) \geq 1 \) \emph{(}resp.\ \( \geq 2 \)\emph{)},
the restriction homomorphism
\( \SG \to j_{*}(\SG|_{X \setminus Z}) \) is injective 
\emph{(}resp.\ an isomorphism, and \( \SG \) satisfies \( \bfS_{1} \)\emph{)}.

\item \label{cor:basicS1S2:condC}
There is a closed subset \( Z \subset \Supp \SG \) such that \( \SG|_{X \setminus Z} \) 
satisfies \( \bfS_{1} \) \emph{(}resp.\ \( \bfS_{2} \)\emph{)}
and the restriction homomorphism
\( \SG \to j_{*}(\SG|_{X \setminus Z}) \)
is injective \emph{(}resp.\ an isomorphism\emph{)}.
\end{enumerate}
\end{cor}

\begin{remn}
Let \( X \) be a locally Noetherian scheme and \( \SG \) a coherent \( \SO_{X} \)-module.
Then, by definition, \( \SG \) satisfies \( \bfS_{1} \) if and only if \( \SG \) has no embedded points 
(cf.\ Remark~\ref{rem:assocprime}).
In particular, the following hold when \( \SG \) satisfies \( \bfS_{1} \):
\begin{enumerate}
\item \label{rem:basicS1:1} Every coherent \( \SO_{X} \)-submodule of \( \SG \) satisfies \( \bfS_{1} \)
(cf.\ Lemma~\ref{lem:depth<=2}\eqref{lem:depth<=2:1} below).

\item \label{rem:basicS1:2} The sheaf \( \SHom_{\SO_{X}}(\SF, \SG) \) satisfies \( \bfS_{1} \)
for any coherent \( \SO_{X} \)-module \( \SF \).

\item \label{rem:basicS1:3} Let \( T \) be the closed subscheme defined
by the annihilator of \( \SG \), i.e., \( \SO_{T} \) is the image
of the natural homomorphism \( \SO_{X} \to \SHom_{\SO_{X}}(\SG, \SG) \).
Then, \( T \) also satisfies \( \bfS_{1} \).
\end{enumerate}
\end{remn}

\begin{lem}\label{lem:depth<=2}
Let \( X \) be a locally Noetherian scheme and
let \( \SG \) be the kernel of a homomorphism \( \SE^{0} \to \SE^{1} \)
of coherent \( \SO_{X} \)-modules.
\begin{enumerate}
\item \label{lem:depth<=2:0}
Let \( Z \) be a closed subset of \( X \).
If \( \depth_{Z} \SE^{0} \geq 1 \), then \( \depth_{Z} \SG \geq 1 \).
If \( \depth_{Z} \SE^{0} \geq 2 \) and \( \depth_{Z} \SE^{1} \geq 1 \),
then \( \depth_{Z} \SG \geq 2 \).

\item \label{lem:depth<=2:1}
If \( \SE^{0} \) satisfies \( \bfS_{1} \), then \( \SG \) satisfies \( \bfS_{1} \).

\item \label{lem:depth<=2:2}
Assume that \( \Supp \SG \subset \Supp \SE^{1} \).
If \( \SE^{1} \) satisfies \( \bfS_{1} \) and
\( \SE^{0} \) satisfies \( \bfS_{2} \), then
\( \SG \) satisfies \( \bfS_{2} \).
\end{enumerate}
\end{lem}

\begin{proof}
Let \( \SB \) be the image of \( \SE^{0} \to \SE^{1} \). Then, we have an exact sequence
\[ 0 \to \SH^{0}_{Z}(\SG) \to \SH^{0}_{Z}(\SE^{0}) \to \SH^{0}_{Z}(\SB)
\to \SH^{1}_{Z}(\SG) \to \SH^{1}_{Z}(\SE^{0})\]
and an injection \( \SH^{0}_{Z}(\SB) \to \SH^{0}_{Z}(\SE^{1}) \) of local cohomology sheaves
with support in \( Z \) (cf.\ \cite[Prop.\ 1.1]{LC}).
Thus, \eqref{lem:depth<=2:0} is derived from Property~\ref{ppty:depth<=2}.
The remaining assertions \eqref{lem:depth<=2:1} and \eqref{lem:depth<=2:2} are
consequences of \eqref{lem:depth<=2:0} above
and the equivalence: \eqref{lem:basicSk:condA} \( \Leftrightarrow \)
\eqref{lem:basicSk:condB} in Lemma~\ref{lem:basicSk}.
\end{proof}

\begin{lem}\label{lem:proj+S1}
Let \( P \) be the \( n \)-dimensional
projective space \( \BPP^{n}_{\Bbbk}\) over a field \( \Bbbk \) and
let \( \SG \) be a coherent \( \SO_{P} \)-module such that
\( \SG \) satisfies \( \bfS_{1} \) and that every irreducible component of \( \Supp \SG \) has
positive dimension.
Then, \( \OH^{0}(P, \SG(m)) = 0 \) for any \( m \ll 0 \),
where we write \( \SG(m) = \SG \otimes_{\SO_{P}} \SO_{P}(m) \).
\end{lem}

\begin{proof}
We shall prove by contradiction.
Assume that \( \OH^{0}(P, \SG(-m)) \ne 0 \) for infinitely many \( m > 0 \).
There is a member \( D \) of \( |\SO_{P}(k)| \) for some \( k > 0 \) such that
\( D \cap \Ass(\SG) = \emptyset\) (cf.\ Remark~\ref{rem:assocprime}).
Thus, the inclusion \( \SO_{P}(-D) \subset \SO_{P} \)
induces an injection \( \SG(-D) := \SG \otimes_{\SO_{P}} \SO_{P}(-D) \to \SG \).
Hence, we have an injection \( \SG(-k) \isom \SG(-D) \to \SG \),
and we may assume that \( \OH^{0}(P, \SG(-m)) = \OH^{0}(P, \SG) \ne 0\) for any \( m > 0 \)
by replacing \( \SG \) with \( \SG(-l) \) for some \( l > 0\).
Let \( \xi \) be a non-zero element of \( \OH^{0}(P, \SG) \),
which corresponds to a non-zero homomorphism \( \SO_{P} \to \SG \).
Let \( T \) be the closed subscheme of \( P \)
such that \( \SO_{T} \) is the image of \( \SO_{P} \to \SG \).
Then, \( T \) is non-empty and is contained in the affine open subset \( P \setminus D \),
since \( \xi \in \OH^{0}(P, \SG(-D))\).
Therefore, \( T \) is a finite set, and \( T \subset \Ass(\SG) \).
Since \( \SG \) satisfies \( \bfS_{1} \), 
every point of \( T \) is an irreducible component of \( \Supp \SG \).
This contradicts the assumption.
\end{proof}

\begin{dfn}[reflexive sheaf]\label{dfn:reflexive}
For a scheme \( X \) and an \( \SO_{X} \)-module \( \SF \), we write
\( \SF^{\vee} \) for the dual \( \SO_{X} \)-module \( \SHom_{\SO_{X}}(\SF, \SO_{X}) \).
The double-dual \( \SF^{\vee\vee} \) of \( \SF \) is defined as \( (\SF^{\vee})^{\vee} \).
The natural composition homomorphism \( \SF \otimes \SF^{\vee} \to \SO_{X} \) defines
a canonical homomorphism \( c_{\SF} \colon \SF \to \SF^{\vee\vee} \).
Note that \( c_{\SF^{\vee}} \) is always an isomorphism.
If \( \SF \) is a quasi-coherent \( \SO_{X} \)-module of finite type and if \( c_{\SF} \)
is an isomorphism, then \( \SF \) is said to be reflexive.
\end{dfn}

\begin{rem}\label{rem:dfn:reflexive}
Let \( \pi \colon Y \to X \) be a flat morphism of locally Noetherian schemes.
Then, the dual operation \( {}^{\vee} \) commutes with \( \pi^{*} \), i.e.,
there is a canonical isomorphism
\[ \pi^{*}\SHom_{\SO_{X}}(\SF, \SO_{X}) \isom \SHom_{\SO_{Y}}(\pi^{*}\SF, \SO_{Y})\]
for any coherent \( \SO_{X} \)-module \( \SF \).
In particular, if \( \SF \) is reflexive, then so is \( \pi^{*}\SF \).
This isomorphism is derived from \cite[$0_{\mathrm{I}}, (6.7.6)$]{EGA}, since
every coherent \( \SO_{X} \)-module has a finite presentation locally on \( X \).
\end{rem}

\begin{lem}\label{lem:j*reflexive}
Let \( X \) be a locally Noetherian scheme, \( Z \) a closed subset, and
\( \SG \) a coherent \( \SO_{X} \)-module.
\begin{enumerate}
\item \label{lem:j*reflexive:1}
For an integer \( k = 1 \) or \( 2 \),
assume that \( \depth_{Z} \SO_{X} \geq k \)
and that \( \SG \) is reflexive.
Then, \( \depth_{Z} \SG \geq k \).
\item \label{lem:j*reflexive:1a}
For an integer \( k = 1 \) or \( 2 \), assume that \( X \) satisfies \( \bfS_{k} \)
and that \( \SG \) is reflexive.
Then, \( \SG \) satisfies \( \bfS_{k} \).
\item \label{lem:j*reflexive:2}
Assume that \( \depth_{Z} \SO_{X} \geq 1 \) and that
\( \SG|_{X \setminus Z} \) is reflexive.
If \( \depth_{Z} \SG \geq 2 \), then \( \SG \) is reflexive.
\end{enumerate}
\end{lem}

\begin{proof}
For the proof of \eqref{lem:j*reflexive:1},
by localizing \( X \), we may assume that there is an exact sequence
\( \SE_{1} \to \SE_{0} \to \SG^{\vee} \to 0 \)
for some free \( \SO_{X} \)-modules \( \SE_{0} \) and \( \SE_{1} \) of finite rank.
Taking the dual, we have an exact sequence
\( 0 \to \SG \isom \SG^{\vee\vee} \to \SE_{0}^{\vee} \to \SE_{1}^{\vee} \)
(cf.\ the proof of \cite[Proposition~1.1]{HaRef}).
The condition: \( \depth_{Z} \SO_{X} \geq k \)
implies that
\( \depth_{Z} \SE_{i}^{\vee} \geq k\)
for \( i = 0 \), \( 1 \).
Thus,  \( \depth_{Z} \SG \geq k\)
by Lemma~\ref{lem:depth<=2}\eqref{lem:depth<=2:0}.
This proves \eqref{lem:j*reflexive:1}.
The assertion \eqref{lem:j*reflexive:1a} is a consequence of \eqref{lem:j*reflexive:1}
(cf.\ Definition~\ref{dfn:SerreCond}).
We shall show \eqref{lem:j*reflexive:2}.
Let \( j \colon X \setminus Z \injmap X \) be the open immersion.
Then, \( \SG \isom j_{*}(\SG_{X \setminus Z}) \) by Property~\ref{ppty:depth<=2},
since \( \depth_{Z} \SG \geq 2 \) by assumption.
Hence, we have a splitting of the canonical homomorphism \( \SG \to \SG^{\vee\vee} \)
into the double-dual by the commutative diagram
\[ \begin{CD}
\SG @>>> \SG^{\vee\vee} \\
@V{\isom}VV @VVV \\
j_{*}(\SG|_{X \setminus Z}) @>{\isom}>> \phantom{.}j_{*}(\SG^{\vee\vee}|_{X \setminus Z}).
\end{CD}\]
Thus, we have an injection \( \SC \injmap \SG^{\vee\vee} \)
from \( \SC := \SG^{\vee\vee}/\SG \), where \( \Supp \SC \subset Z \).
The injection corresponds to
a homomorphism \( \SC \otimes \SG^{\vee} \to \SO_{X} \),
but this is zero, since \( \depth_{Z} \SO_{X} \geq 1 \).
Therefore, \( \SC = 0 \) and \( \SG \) is reflexive. This proves \eqref{lem:j*reflexive:2},
and we are done.
\end{proof}

\begin{cor}\label{cor:prop:S1S2:reflexive}
Let \( X \) be a locally Noetherian scheme, \( Z \) a closed subset,
and \( \SG \) a coherent \( \SO_{X} \)-module.
Assume that \( \SG|_{X \setminus Z} \) is reflexive 
and \( \Codim(Z, X) \geq 1 \).
Let us consider the following three conditions\emph{:}
\begin{enumerate}
    \renewcommand{\theenumi}{\roman{enumi}}
    \renewcommand{\labelenumi}{(\theenumi)}
\item \label{cor:prop:S1S2:reflexive:cond1}
\( \SG \) satisfies \( \bfS_{2} \) and \( \Codim(Z \cap \Supp \SG, \Supp \SG) \geq 2 \)\emph{;}

\item \label{cor:prop:S1S2:reflexive:cond2}
\( \depth_{Z} \SG \geq 2 \)\emph{;}

\item \label{cor:prop:S1S2:reflexive:cond3}
\( \SG \) is reflexive.
\end{enumerate}
Then,
\eqref{cor:prop:S1S2:reflexive:cond1} \( \Rightarrow \) \eqref{cor:prop:S1S2:reflexive:cond2}
holds always.
If \( \depth_{Z} \SO_{X} \geq 1 \), then
\eqref{cor:prop:S1S2:reflexive:cond2} \( \Rightarrow \)
\eqref{cor:prop:S1S2:reflexive:cond3} holds, and if  \( \depth_{Z} \SO_{X} \geq 2 \),
then \eqref{cor:prop:S1S2:reflexive:cond2} \( \Leftrightarrow \)
\eqref{cor:prop:S1S2:reflexive:cond3} holds.
If \( X \) satisfies \( \bfS_{2} \) and \( \Codim(Z, X) \geq 2 \), then these three conditions
are equivalent to each other.
\end{cor}

\begin{proof}
The implication \eqref{cor:prop:S1S2:reflexive:cond1} \( \Rightarrow \)
\eqref{cor:prop:S1S2:reflexive:cond2} is shown in Lemma~\ref{lem:depth+codim+Sk}\eqref{lem:depth+codim+Sk:2}.
The next implication
\eqref{cor:prop:S1S2:reflexive:cond2} \( \Rightarrow \)
\eqref{cor:prop:S1S2:reflexive:cond3}
in case \( \depth_{Z} \SO_{X} \geq 1 \) follows from
Lemma~\ref{lem:j*reflexive}\eqref{lem:j*reflexive:2},
and the converse implication \eqref{cor:prop:S1S2:reflexive:cond3} \( \Rightarrow \)
\eqref{cor:prop:S1S2:reflexive:cond2} in case \( \depth_{Z} \SO_{X} \geq 2 \)
follows from Lemma~\ref{lem:j*reflexive}\eqref{lem:j*reflexive:1}.
Assume that \( X \) satisfies \( \bfS_{2} \) and \( \Codim(Z, X) \geq 2 \).
Then, \( \depth_{Z} \SO_{X} \geq 2 \) by Lemma~\ref{lem:depth+codim+Sk}\eqref{lem:depth+codim+Sk:2}, 
and we have
\eqref{cor:prop:S1S2:reflexive:cond2} \( \Leftrightarrow \)
\eqref{cor:prop:S1S2:reflexive:cond3} in this case. It remains to prove:
\eqref{cor:prop:S1S2:reflexive:cond2} \( \Rightarrow \)
\eqref{cor:prop:S1S2:reflexive:cond1}.
Assume that \( \depth_{Z} \SG \geq 2\). Then,
\( \Codim(Z \cap \Supp \SG, \Supp \SG) \geq 2 \) by Lemma~\ref{lem:depth+codim+Sk}\eqref{lem:depth+codim+Sk:1}.
On the other hand, the reflexive sheaf \( \SG|_{X \setminus Z} \) satisfies \( \bfS_{2} \)
by Lemma~\ref{lem:j*reflexive}\eqref{lem:j*reflexive:1a},
since \( X \setminus Z \) satisfies \( \bfS_{2} \). Thus,
\( \SG \) satisfies \( \bfS_{2} \) by the equivalence
\eqref{lem:basicSk:condA} \( \Leftrightarrow \) \eqref{lem:basicSk:condD}
of Lemma~\ref{lem:basicSk}, and we are done.
\end{proof}

\begin{remn}
If \( X \) is a locally Noetherian scheme satisfying \( \bfS_{1} \),
then the support of a reflexive \( \SO_{X} \)-module is a union of irreducible components of
\( X \).
In fact, if \( \SG \) is reflexive, then \( \depth_{Z} \SG \geq 1 \) for any closed subset \( Z \)
with \( \Codim(Z, X) \geq 1 \), by Lemma~\ref{lem:j*reflexive}\eqref{lem:j*reflexive:1},
and we have \( \Codim(Z \cap \Supp \SG, \Supp \SG) \geq 1 \)
by Lemma~\ref{lem:depth+codim+Sk}\eqref{lem:depth+codim+Sk:1}:
This means that
\( \Supp \SG \) is a union of irreducible components of \( X \).
In particular, if \( X \) is irreducible and satisfies \( \bfS_{1} \), 
and if \( \SG \ne 0 \), then \( \Supp \SG = X \).
However, \( \Supp \SG \ne X \) in general when \( X \) is reducible.
For example, let \( R \) be a Noetherian
ring with two \( R \)-regular elements \( u \) and \( v \),
and set \( X := \Spec R/uvR \) and \( \SG := (R/uR)\sptilde \).
Then, we have an isomorphism \( \SHom_{\SO_{X}}(\SG, \SO_{X}) \isom \SG  \) by
the natural exact sequence
\[ 0 \to R/uR \to R/uvR \xrightarrow{u \times} R/uvR \to R/uR \to 0. \]
Thus, \( \SG \) is a reflexive \( \SO_{X} \)-module, but \( \Supp \SG \ne X \)
when \( u \not\in \sqrt{vR} \).
\end{remn}

We have discussed properties \( \bfS_{1} \) and \( \bfS_{2} \) for general coherent sheaves.
Finally in Section~\ref{subsect:SerreBasics},
we note the following well-known
facts on locally Noetherian schemes satisfying \( \bfS_{2} \).

\begin{fact}\label{fact:S2}
Let \( X \) be a locally Noetherian scheme satisfying \( \bfS_{2} \).

\begin{enumerate}
\item  \label{fact:S2:1}
If \( X \) is catenary (cf.\ Property~\ref{pprt:catenary}),
then \( X \) is locally equi-dimensional (cf.\ Definition~\ref{dfn:equi-dim}\eqref{dfn:equi-dim:locally})
(cf.\ \cite[IV, Cor.~(5.1.5), (5.10.9)]{EGA}).

\item  \label{fact:S2:2}
For any open subset \( X^{\circ} \) with
\( \Codim(X \setminus X^{\circ}, X) \geq 2 \) and
for any connected component \( X_{\alpha} \) of \( X \),
the intersection \( X_{\alpha} \cap X^{\circ} \) is connected.
This is a consequence of a result of Hartshorne
(cf.\ \cite[IV, Th.~(5.10.7)]{EGA}, \cite[III, Th.~3.6]{SGA2}).
\end{enumerate}
\end{fact}


\subsection{Relative \texorpdfstring{$\bfS_{k}$}{Sk}-conditions}
\label{subsect:Rel}

Here, we shall consider the relative \( \bfS_{k} \)-condition for morphisms 
of locally Noetherian schemes.

\begin{nota}\label{nota:F_(t)}
Let \( f \colon Y \to T \) be a morphism of schemes.
For a point \( t \in T \), the fiber \( f^{-1}(t) \)
of \( f \) over \( t \) is defined as \( Y \times_{T} \Spec \Bbbk(t) \), and it is
denoted by \( Y_{t} \).
For an \( \SO_{Y} \)-module \( \SF \),
the restriction \( \SF \otimes_{\SO_{Y}} \SO_{Y_{t}} \isom \SF \otimes_{\SO_{T}} \Bbbk(t) \)
to the fiber \( Y_{t} \) is denoted by \( \SF_{(t)} \).
\end{nota}

\begin{remn}
The restriction \( \SF_{(t)} \) is identified with the inverse image \( p_{t}^{*}(\SF) \)
for the projection \( p_{t} \colon Y_{t} \to Y \), and \( \Supp \SF_{(t)}  \)
is identified with \( Y_{t} \cap \Supp \SF = p_{t}^{-1}(\Supp \SF)\).
If \( f \) is the identity morphism \( Y \to Y \),
then \( \SF_{(y)} \) is a sheaf on \( \Spec \Bbbk(y) \) corresponding
to the vector space \(\SF_{y} \otimes \Bbbk(y)\) for \( y \in Y \).
\end{remn}

\begin{dfn}
For a morphism \( f \colon Y \to T \) of schemes and for an \( \SO_{Y} \)-module \( \SF \),
let \( \Fl(\SF/T) \) be the set of points \( y \in Y \) such that
\( \SF_{y} \) is a flat \( \SO_{T, f(y)} \)-module.
If \( Y = \Fl(\SF/T) \), then \( \SF \) is said to be \emph{flat over} \( T \), or \( f \)-\emph{flat}.
If \( S \) is a subset of \( \Fl(\SF/T) \), then
\( \SF \) is said to be \emph{flat over} \( T \) \emph{along} \( S \), or \( f \)-\emph{flat along} \( S \).
\end{dfn}

\begin{fact}\label{fact:elem-flat}
Let \( f \colon Y \to T \) be a morphism of locally Noetherian schemes and
\( k \) a positive integer.
For a coherent \( \SO_{Y} \)-module \( \SF \) and a coherent \( \SO_{T} \)-module \( \SG \),
the following results are known, where in \eqref{fact:elem-flat:2}, \eqref{fact:elem-flat:3}, and
\eqref{fact:elem-flat:4}, we fix
an arbitrary point \( y \in Y \), and set \( t = f(y) \):
\begin{enumerate}
\item  \label{fact:elem-flat:1}
If \( f \) is locally of finite type, then \( \Fl(\SF/T)  \) is open.

\item  \label{fact:elem-flat:2}
If \( \SF_{y} \) is flat over \( \SO_{T, t} \) and if
\( (\SF_{(t)})_{y} \) is a free \( \SO_{Y_{t}, y} \)-module,
then \( \SF_{y} \) is a free \( \SO_{Y, y} \)-module.
In particular, if \( \SF \) is flat over \( T \) and if \( \SF_{(t)} \) is locally free
for any \( t \in T \),
then \( \SF \) is locally free.

\item \label{fact:elem-flat:3}
If \( \SF_{y} \) is non-zero and flat over \( \SO_{T, t} \),
then the following equalities hold:
\begin{align}
\dim (\SF \otimes_{\SO_{Y}} f^{*}\SG)_{y}
&= \dim (\SF_{(t)})_{y} + \dim \SG_{t},
\label{eq:fact:elem-flat:1}\\
\depth (\SF \otimes_{\SO_{Y}} f^{*}\SG)_{y}
&= \depth (\SF_{(t)})_{y} + \depth \SG_{t}.
\label{eq:fact:elem-flat:2}
\end{align}

\item \label{fact:elem-flat:4}
If \( \SF_{y} \) is non-zero and flat over \( \SO_{T, t} \)
and if \( \SF \otimes_{\SO_{Y}} f^{*}\SG \) satisfies \( \bfS_{k} \) at \( y \),
then \( \SG \) satisfies \( \bfS_{k} \) at \( t \).

\item \label{fact:elem-flat:5}
Assume that \( \SF \) is flat over \( T \) along the fiber
\( Y_{t} \) over a point \( t \in f(\Supp \SF) \).
If \( \SF_{(t)} \) satisfies \( \bfS_{k} \)
and if \( \SG \) satisfies \( \bfS_{k} \) at \( t \),
then \( \SF \otimes_{\SO_{Y}} f^{*}\SG \)
also satisfies \( \bfS_{k} \) at any point of \( Y_{t} \).

\item \label{fact:elem-flat:6}
Assume that \( f \) is flat and that every fiber \( Y_{t} \) satisfies \( \bfS_{k} \).
Then, \( f^{*}\SG \) satisfies \( \bfS_{k} \) at \( y \)
if and only if \( \SG \) satisfies \( \bfS_{k} \) at \( f(y) \).
\end{enumerate}
The assertion \eqref{fact:elem-flat:1} is just \cite[IV, Th.~(11.1.1)]{EGA}.
The assertion \eqref{fact:elem-flat:2} is a consequence of
Proposition~\ref{prop:LCflat} and Lemma~\ref{lem:LCfree},
since \( \SO_{Y_{t}, y} = \SO_{Y, y}/I \)
for the ideal \( I = \GM_{T, t}\SO_{Y, y} \) and we have
\[ \Tor^{\SO_{Y, y}}_{1}(\SF_{y}, \SO_{Y, y}/I) = 0 \quad \text{and} \quad
(\SF_{(t)})_{y} \isom \SF_{y}/I\SF_{y}\]
under the assumption of \eqref{fact:elem-flat:2}.
Two equalities \eqref{eq:fact:elem-flat:1} and \eqref{eq:fact:elem-flat:2}
in \eqref{fact:elem-flat:3} follow from \cite[IV, Cor.~(6.1.2), Prop.~(6.3.1)]{EGA},
since
\[ (\SF \otimes_{\SO_{Y}} f^{*}\SG)_{y} \isom \SF_{y} \otimes_{\SO_{T, t}} \SG_{t}
\quad \text{and} \quad (\SF_{(t)})_{y}
\isom \SF_{y} \otimes_{\SO_{T, t}} \Bbbk(t). \]
The assertions \eqref{fact:elem-flat:4} and \eqref{fact:elem-flat:5} are shown in
\cite[IV, Prop.~(6.4.1)]{EGA} by the equalities
\eqref{eq:fact:elem-flat:1} and \eqref{eq:fact:elem-flat:2}, and the assertion
\eqref{fact:elem-flat:6} is a consequence
of \eqref{fact:elem-flat:4} and \eqref{fact:elem-flat:5}
(cf.\ \cite[IV, Cor.~(6.4.2)]{EGA}).
\end{fact}

\begin{cor}\label{cor:fact:elem-flat}
Let \( f \colon Y \to T \) be a flat morphism of locally Noetherian schemes.
Let \( W \) be a closed subset of \( T \) contained in \( f(Y) \). Then,
\[
\Codim(f^{-1}(W), Y) = \Codim(W, T) \quad \text{and} \quad
\depth_{f^{-1}(W)} f^{*}\SG = \depth_{W} \SG
\]
for any coherent \( \SO_{T} \)-module \( \SG \).
\end{cor}

\begin{proof}
We may assume that \( \SG \ne 0 \).
Then,
\begin{align*}
\Codim(f^{-1}(W), Y) &= \inf\{ \dim \SO_{Y, y} \mid y \in f^{-1}(W) \}, \\
\depth_{f^{-1}(W)} f^{*}\SG &= \inf\{ \depth (f^{*}\SG)_{y} \mid y \in f^{-1}(W) \},
\end{align*}
by Property~\ref{ppty:dim-codim} and Definition~\ref{dfn:Z-depth}.
Thus, we can prove the assertion
by applying \eqref{eq:fact:elem-flat:1} to \( (\SF, \SG) = (\SO_{Y}, \SO_{T}) \)
and \eqref{eq:fact:elem-flat:2} to \( (\SF, \SG) = (\SO_{Y}, \SG) \),
since
\[ \dim \SO_{T, t} = \Codim(W, T) \quad \text{and} \quad \dim \SO_{Y_{t}, y} = 0\]
for a certain generic point \( t \) of \( W \) and a generic point \( y \) of \( Y_{t} \),
and since
\[  \depth \SG_{t} = \depth_{W} \SG \quad \text{and} \quad
\depth ((f^{*}\SG)_{(t)})_{y} = \depth \SO_{Y_{t}, y} = 0 \]
for a certain point \( t \in W \cap \Supp \SG \) and
for a generic point \( y \) of \( Y_{t} \).
\end{proof}

\begin{dfn}\label{dfn:RelSkCMlocus}
Let \( f \colon Y \to T\) be a morphism of locally Noetherian schemes
and \( \SF \) a coherent \( \SO_{Y} \)-module.
As a relative version of Definition~\ref{dfn:SkCMlocus},
for a positive integer \( k \),
we define
\begin{align*}
\bfS_{k}(\SF/T) &:= \Fl(\SF/T) \cap \bigcup\nolimits_{t \in T} \bfS_{k}(\SF_{(t)})
\quad \text{and} \\
\CM(\SF/T) &:= \Fl(\SF/T) \cap \bigcup\nolimits_{t \in T} \CM(\SF_{(t)}),
\end{align*}
and call them the \emph{relative \( \bfS_{k} \)-locus}
and
the \emph{relative Cohen--Macaulay locus}
of \( \SF \)  \emph{over} \( T \),
respectively.
We also write
\[ \bfS_{k}(Y/T) = \bfS_{k}(\SO_{Y}/T) \quad \text{and} \quad \CM(Y/T) = \CM(\SO_{Y}/T),  \]
and call them the \emph{relative \( \bfS_{k} \)-locus} and
the \emph{relative Cohen--Macaulay locus for} \( f \), respectively.
The relative \( \bfS_{k} \)-condition
and the relative Cohen--Macaulay condition are defined as follows:

\begin{itemize}
\item  For a point \( y \in Y \) (resp.\ a subset \( S \subset Y \)),
we say that \( \SF \) satisfies
\emph{relative \( \bfS_{k} \) over} \( T \) \emph{at} \( y \)
(resp.\ \emph{along} \( S \))
if \( y \in \bfS_{k}(\SF/T) \) (resp.\ \( S \subset \bfS_{k}(\SF/T)) \).
We also say that \( \SF \) is
\emph{relatively Cohen--Macaulay over} \( T \) \emph{at} \( y \)
(resp.\ \emph{along} \( S \))
if \( y \in \CM(\SF/T) \) (resp.\ \( S \subset \CM(\SF/T) \)).

\item  We say that \( \SF \) satisfies
\emph{relative \( \bfS_{k} \) over} \( T \)
if \( Y = \bfS_{k}(\SF/T) \).
We also say that \( \SF \) is \emph{relatively Cohen--Macaulay over} \( T \)
if \( Y = \CM(\SF/T) \).
\end{itemize}
\end{dfn}

\begin{fact}\label{fact:dfn:RelSkCMlocus}
For \( f \colon Y \to T \) and \( \SF \) in
Definition~\ref{dfn:RelSkCMlocus}, assume that \( f \) is \emph{locally of finite type}
and \( \SF \) is flat over \( T \).
Then, the following properties are known:
\begin{enumerate}
\item \label{fact:dfn:RelSkCMlocus:1}
The subset \( \CM(\SF/T) \) is open (cf.\ \cite[IV, Th.~(12.1.1)(vi)]{EGA}).

\item \label{fact:dfn:RelSkCMlocus:2}
If \( \SF_{(t)} \) is locally equi-dimensional
(cf.\ Definition~\ref{dfn:equi-dim}\eqref{dfn:equi-dim:locally}) for any \( t \in T \),
then \( \bfS_{k}(\SF/T) \) is open for any \( k \geq 1 \)
(cf.\ \cite[IV, Th.~(12.1.1)(iv)]{EGA}).

\item \label{fact:dfn:RelSkCMlocus:3}
If \( Y \to T \) is flat, then \( \bfS_{k}(Y/T) \) is open for any \( k \geq 1 \)
(cf.\ \cite[IV, Th.~(12.1.6)(i)]{EGA}).
\end{enumerate}
\end{fact}

\begin{dfn}[$\bfS_{k}$-morphism and Cohen--Macaulay morphism]
\label{dfn:SkCMmorphism}
Let \( f \colon Y \to T \) be a morphism of locally Noetherian schemes
and \( k \) a positive integer.
The \( f \) is called
an \emph{\( \bfS_{k} \)-morphism}
(resp.\ a \emph{Cohen--Macaulay morphism})
if \( f \) is a flat morphism \emph{locally of finite type}
and \( Y = \bfS_{k}(Y/T)\) (resp.\ \( Y = \CM(Y/T) \)).
For a subset \( S \) of \( Y \), \( f \) is called
an \emph{\( \bfS_{k} \)-morphism}
(resp.\ a \emph{Cohen--Macaulay morphism}) \emph{along} \( S \)
if \( f|_{V} \colon V \to T \) is
so for an open neighborhood \( V \) of \( S \)
(cf.\ Fact~\ref{fact:dfn:RelSkCMlocus}\eqref{fact:dfn:RelSkCMlocus:3}).
\end{dfn}

\begin{remn}
The \( \bfS_{k} \)-morphisms and the Cohen--Macaulay morphisms defined
in \cite[IV, D\'ef.~(6.8.1)]{EGA} are not necessarily locally of finite type.
The definition of Cohen--Macaulay morphism in \cite[V, Ex.~9.7]{ResDual}
coincides with ours.
The notion of ``CM map'' in \cite[p.~7]{Conrad} is the same
as that of Cohen--Macaulay morphism
in our sense for morphisms of locally Noetherian schemes.
\end{remn}

\begin{lem}\label{lem:bc basic} 
Suppose that we are given a Cartesian diagram
\[
\begin{CD}
Y' @>{p}>> Y \\ @V{f'}VV @VV{f}V \\ T' @>{q}>> T
\end{CD}
\]
of schemes consisting of locally Noetherian schemes.
Let \( \SF \) be a coherent \( \SO_{Y} \)-module,
\( Z \) a closed subset of \( Y \),
\( k \) a positive integer,
and let \( t' \in T' \) and \( t \in T \) be points such that \( t = q(t') \).
\begin{enumerate}

\item \label{lem:bc basic:1}
If \( f \) is flat, then, for the fibers \( Y'_{t'} = f^{\prime -1}(t') \) and \( Y_{t} = f^{-1}(t) \), one has
\begin{align*}
\Codim(p^{-1}(Z) \cap Y'_{t'}, Y'_{t'}) &= \Codim(Z \cap Y_{t}, Y_{t}), \quad \text{and}\\
\depth_{p^{-1}(Z) \cap Y'_{t'}} \SO_{Y'_{t'}} &= \depth_{Z \cap Y_{t}} \SO_{Y_{t}}.
\end{align*}

\item  \label{lem:bc basic:2}
If \( \SF \) is flat over \( T \), then
\[ \depth_{p^{-1}(Z) \cap Y'_{t'}} (p^{*}\SF)_{(t')} = \depth_{Z \cap Y_{t}} \SF_{(t)}. \]

\item \label{lem:bc basic:3}
If \( \SF \) is flat over \( T \), then
\( \bfS_{k}(p^{*}\SF/T') \subset p^{-1}\bfS_{k}(\SF/T) \).
If \( f \) is locally of finite type in addition, then \(  \bfS_{k}(p^{*}\SF/T') = p^{-1}\bfS_{k}(\SF/T)\).

\item \label{lem:bc basic:4}
If \( f \) is locally of finite type and 
if \( \SF \) satisfies relative \( \bfS_{k} \) over \( T \), then
\( p^{*}\SF \) does so over \( T' \).

\item \label{lem:bc basic:5}
If \( f \) is an \( \bfS_{k} \)-morphism \emph{(}resp.\ Cohen--Macaulay morphism\emph{)},
then so is \( f' \).
\end{enumerate}
\end{lem}

\begin{proof}
The assertions \eqref{lem:bc basic:1} and \eqref{lem:bc basic:2} follow from
Corollary~\ref{cor:fact:elem-flat} applied to
the flat morphism \( Y'_{t'} \to Y_{t} \) and to
\( \SG = \SO_{Y_{t}} \) or \( \SG  = \SF_{(t)} \).
The first half of \eqref{lem:bc basic:3} follows from
Definition~\ref{dfn:RelSkCMlocus} and Fact~\ref{fact:elem-flat}\eqref{fact:elem-flat:4}
applied to \( Y'_{t'} \to Y_{t} \) and to \( (\SF, \SG) = (\SO_{Y'_{t'}}, \SF_{(t)}) \).
The latter half of \eqref{lem:bc basic:3}
follows from
Fact~\ref{fact:elem-flat}\eqref{fact:elem-flat:6}, since the fiber \( p^{-1}(y) \)
over a point \( y \in Y_{t} \) is isomorphic to \( \Spec \Bbbk(y) \otimes_{\Bbbk(t)} \Bbbk(t') \)
and since \( \Bbbk(y) \otimes_{\Bbbk(t)} \Bbbk(t') \) is Cohen--Macaulay
(cf.\ \cite[IV, Lem.\ (6.7.1.1)]{EGA}).
The assertion \eqref{lem:bc basic:4} is a consequence of \eqref{lem:bc basic:3}, and
the assertion \eqref{lem:bc basic:5} follows from \eqref{lem:bc basic:3} in the case: \( \SF = \SO_{Y} \),
by Definition~\ref{dfn:SkCMmorphism}.
\end{proof}

\begin{lem}\label{lem:relSkCodimDepth}
Let \( Y \to T \) be a morphism of locally Noetherian schemes
and let \( Z \) be a closed subset of \( Y \).
Let \( \SF \) be a coherent \( \SO_{Y} \)-module and \( k \) a positive integer.

\begin{enumerate}
\item \label{lem:relSkCodimDepth:1}
If \( \SF \) is flat over \( T \), then
\[ \depth_{Z} \SF \geq \inf\{ \depth_{Z \cap Y_{t}} \SF_{(t)} \mid t \in f(Z)\}. \]

\item  \label{lem:relSkCodimDepth:2}
If \( \SF \) satisfies relative \( \bfS_{k} \) over \( T \) and if
\[
\Codim(Z \cap \Supp \SF_{(t)}, \Supp \SF_{(t)}) \geq k
\]
for any \( t \in T \), then \( \depth_{Z} \SF \geq k\).

\item  \label{lem:relSkCodimDepth:3}
If \( Y \to T \) is flat and if one of the two conditions below
is satisfied, then \( \depth_{Z} \SO_{Y} \geq k \)\emph{:}
\begin{enumerate}
\item  \( \depth_{Y_{t} \cap Z} \SO_{Y_{t}} \geq k\) for any \( t \in T \)\emph{;}

\item  \( Y_{t} \) satisfies \( \bfS_{k} \) and \( \Codim(Y_{t} \cap Z, Y_{t}) \geq k  \)
for any \( t \in T \).
\end{enumerate}
\end{enumerate}
\end{lem}

\begin{proof}
For the first assertion \eqref{lem:relSkCodimDepth:1}, we may assume that \( Z \cap \Supp \SF \ne \emptyset\).
Then, by Definition~\ref{dfn:Z-depth}, we have the inequality in \eqref{lem:relSkCodimDepth:1}
from the equality \eqref{eq:fact:elem-flat:2} in Fact~\ref{fact:elem-flat}\eqref{fact:elem-flat:3}
in the case where \( \SG = \SO_{T} \),
since \( \depth \SO_{T, t} \geq 0 \) for any \( t \in T \).
The assertion \eqref{lem:relSkCodimDepth:2} is a consequence of \eqref{lem:relSkCodimDepth:1}
and Lemma~\ref{lem:depth+codim+Sk}\eqref{lem:depth+codim+Sk:2} 
applied to \( (Y_{t}, Z \cap Y_{t}, \SF_{(t)}) \).
The last assertion \eqref{lem:relSkCodimDepth:3} is derived from
\eqref{lem:relSkCodimDepth:1} and \eqref{lem:relSkCodimDepth:2}
in the case where \( \SF = \SO_{Y} \).
\end{proof}

The following result gives some relations between the reflexive modules and
the relative \( \bfS_{2} \)-condition.
Similar results can be found in \cite[\S 3]{HaKo}.

\begin{lem}\label{lem:US2add}
Let \( f \colon Y \to T \) be a flat morphism of locally Noetherian schemes,
\( \SF \) a coherent \( \SO_{Y} \)-module, and \( Z \) a closed subset of \( Y \).
Assume that
\[ \depth_{Y_{t} \cap Z} \SO_{Y_{t}} \geq 1 \]
for any fiber \( Y_{t} = f^{-1}(t) \).
Then, the following hold for the open immersion \( j \colon Y \setminus Z \injmap Y \)
and for the restriction homomorphism \( \SF \to j_{*}(\SF|_{Y \setminus Z}) \)\emph{:} 
\begin{enumerate}
\item \label{lem:US2add:1} If \( \SF|_{Y \setminus Z} \) is reflexive and
if \( \SF \isom j_{*}(\SF|_{Y \setminus Z}) \), then \( \SF \) is reflexive.

\item \label{lem:US2add:2-}
If \( \SF \) is reflexive and if
\( \depth_{Y_{t} \cap Z} \SO_{Y_{t}} \geq 2 \)
for any \( t \in T \), then \( \SF \isom j_{*}(\SF|_{Y \setminus Z}) \).

\item \label{lem:US2add:3a-} If \( \SF \) is flat over \( T \) and if
\( \depth_{Y_{t} \cap Z} \SF_{(t)} \geq 2 \)
for any \( t \in T \), then \( \SF \isom j_{*}(\SF|_{Y \setminus Z}) \).

\item  \label{lem:US2add:2}
If \( Y_{t} \) satisfies \( \bfS_{2} \) and \( \Codim(Y_{t} \cap Z, Y_{t}) \geq 2 \) for any \( t \in T \),
and if \( \SF \) is reflexive, then \( \SF \isom j_{*}(\SF|_{Y \setminus Z}) \).

\item \label{lem:US2add:3a} If \( \SF \) satisfies relative \( \bfS_{2} \) over \( T \) and if
\( \Codim(Z \cap \Supp \SF_{(t)},  \Supp \SF_{(t)} ) \geq 2  \)
for any \( t \in T \), then \( \SF \isom j_{*}(\SF|_{Y \setminus Z}) \).

\item \label{lem:US2add:3} In the situation of \eqref{lem:US2add:3a-} or \eqref{lem:US2add:3a},
if \( \SF_{(t)}|_{Y_{t} \setminus Z} \)
is reflexive, then \( \SF_{(t)} \) is reflexive\emph{;}
if \( \SF|_{Y \setminus Z} \) is reflexive, then \( \SF \) is reflexive.

\end{enumerate}
\end{lem}

\begin{proof}
Note that \( \SF \isom j_{*}(\SF|_{Y \setminus Z}) \) if and only if \( \depth_{Z} \SF \geq 2 \)
(cf.\ Property~\ref{ppty:depth<=2}).
We have \( \depth_{Z} \SO_{Y} \geq 1 \)
by Lemma~\ref{lem:relSkCodimDepth}\eqref{lem:relSkCodimDepth:3}.
Hence, \eqref{lem:US2add:1} is a consequence of Lemma~\ref{lem:j*reflexive}\eqref{lem:j*reflexive:2}.
In case \eqref{lem:US2add:2-}, we have \( \depth_{Z} \SO_{Y} \geq 2 \)
by Lemma~\ref{lem:relSkCodimDepth}\eqref{lem:relSkCodimDepth:3}, and \eqref{lem:US2add:2-}
is a consequence of Lemma~\ref{lem:j*reflexive}\eqref{lem:j*reflexive:1}.
The assertion \eqref{lem:US2add:3a-} follows from
Lemma~\ref{lem:relSkCodimDepth}\eqref{lem:relSkCodimDepth:1} with \( k = 2 \).
The assertions \eqref{lem:US2add:2} and \eqref{lem:US2add:3a}
are special cases of \eqref{lem:US2add:2-} and \eqref{lem:US2add:3a-}, respectively.
The first assertion of \eqref{lem:US2add:3} follows from Corollary~\ref{cor:prop:S1S2:reflexive}.
The second assertion of \eqref{lem:US2add:3} is derived
from \eqref{lem:US2add:1} and \eqref{lem:US2add:3a-}.
\end{proof}

\begin{remn}
The assumption of Lemma~\ref{lem:US2add} holds when \( Y_{t} \) satisfies \( \bfS_{1} \)
and \( \Codim(Y_{t} \cap Z, Y_{t}) \geq 1 \) for any \( t \in T \) 
(cf.\ Lemma~\ref{lem:depth+codim+Sk}\eqref{lem:depth+codim+Sk:2}).
\end{remn}

\begin{lem}\label{lem:bc reflexive}
In the situation of Lemma~\emph{\ref{lem:bc basic}}, assume that \( f \) is flat,
\( \SF|_{Y \setminus Z} \) is locally free, and
\[\depth_{Y_{t} \cap Z} \SO_{Y_{t}} \geq 2  \]
for any \( t \in T \).
Then, \( \SF^{\vee\vee} \isom j_{*}(\SF|_{Y \setminus Z}) \) for the open immersion
\( j \colon Y \setminus Z \injmap Y \), and
\[ (p^{*}\SF)^{\vee\vee} \isom (p^{*}(\SF^{\vee\vee}))^{\vee\vee}. \]
Moreover, \( \SF \) and \( p^{*}\SF \) are reflexive if \( \SF \) is flat over \( T \) and
\[ \depth_{Y_{t} \cap Z} \SF_{(t)} \geq 2 \]
for any \( t \in T \).
\end{lem}

\begin{proof}
Now, \( \depth_{Z} \SO_{Y} \geq 2 \)
by Lemma~\ref{lem:relSkCodimDepth}\eqref{lem:relSkCodimDepth:3}.
Hence, \( \depth_{Z} \SF^{\vee\vee} \geq 2\) by
Lemma~\ref{lem:j*reflexive}\eqref{lem:j*reflexive:1}, and this implies the first isomorphism
for \( \SF^{\vee\vee} \). We have
\[ \depth_{Y'_{t'} \cap p^{-1}(Z)} \SO_{Y'_{t'}} \geq 2  \]
by Lemma~\ref{lem:bc basic}\eqref{lem:bc basic:1}.
Hence, by the previous argument applied to \( p^{*}\SF \) and \( p^{*}(\SF^{\vee\vee}) \),
we have isomorphisms
\[ (p^{*}\SF)^{\vee\vee} \isom j'_{*}(p^{*}\SF|_{Y' \setminus p^{-1}(Z)}) 
\isom (p^{*}(\SF^{\vee\vee}))^{\vee\vee}\]
for the open immersion \( j' \colon Y' \setminus p^{-1}(Z) \injmap Y' \).
It remains to prove the last assertion. In this case, \( \SF \) is reflexive by
\eqref{lem:US2add:1} and \eqref{lem:US2add:3a-} of Lemma~\ref{lem:US2add}.
Moreover, by Lemma~\ref{lem:bc basic}\eqref{lem:bc basic:2}, we have
\[ \depth_{Y'_{t'} \cap p^{-1}(Z)} (p^{*}\SF)_{(t')} \geq 2\]
for any point \( t' \in T' \). Thus, \( p^{*}\SF \) is reflexive  by the same argument as above.
\end{proof}

\begin{lem}\label{lem:quasi-flat}
Let \( f \colon Y \to T \) be a morphism of locally Noetherian schemes,
and let \( Z \) be a closed subset of \( Y \).
Assume that \( f \) is \emph{quasi-flat} \emph{(}cf.\ \cite[IV, (2.3.3)]{EGA}\emph{)},
i.e., there is a coherent \( \SO_{Y} \)-module \( \SF \)
such that \( \SF \) is flat over \( T \) and \( \Supp \SF = Y \).
Then,
\begin{equation}\label{lem:quasi-flat|eq}
\Codim_{y}(Z, Y) \geq \Codim_{y}(Z \cap Y_{f(y)}, Y_{f(y)})
\end{equation}
for any point \( y \in Z \).
If \( \Codim(Z \cap Y_{t}, Y_{t}) \geq k \)
for a point \( t \in T \) and for an integer \( k \),
then there is an open neighborhood \( V \) of \( Y_{t} \) in \( Y \)
such that \( \Codim(Z \cap V, V) \geq k \).
\end{lem}

\begin{proof}
For the sheaf \( \SF \) above, we have
\( \Supp \SF_{(t)} = Y_{t} \) for any \( t \in T \).
If \( z \in Z \cap Y_{t}\), then
\begin{equation}\label{lem:quasi-flat|eq1}
\dim \SF_{z} = \dim \SO_{Y, z} \quad \text{and} \quad
\dim (\SF_{(t)})_{z} = \dim \SO_{Y_{t}, z}
\end{equation}
by Property~\ref{ppty:dim-codim}\eqref{ppty:dim-codim:1}, and moreover,
\begin{equation}\label{lem:quasi-flat|eq2}
\dim \SF_{z} = \dim (\SF_{(t)})_{z} + \dim \SO_{T, t} \geq \dim (\SF_{(t)})_{z}
\end{equation}
by \eqref{eq:fact:elem-flat:1}, since \( \SF \) is flat over \( T \).
Thus, we have \eqref{lem:quasi-flat|eq}
from \eqref{lem:quasi-flat|eq1} and \eqref{lem:quasi-flat|eq2}
by Property~\ref{ppty:dim-codim}\eqref{ppty:dim-codim:2}.
The last assertion follows from \eqref{lem:quasi-flat|eq} and the lower-semicontinuity
of the function \( x \mapsto \Codim_{x}(Z, Y) \)
(cf.\ \cite[$0_{\text{IV}}$, Cor.\ (14.2.6)]{EGA}).
In fact, the set of points \( y \in Y \) with \( \Codim_{y}(Z, Y) \geq k \)
is an open subset containing \( Y_{t} \).
\end{proof}

We introduce the following notion
(cf.\ \cite[IV, D\'ef.~(17.10.1)]{EGA} and \cite[p.~6]{Conrad}).

\begin{dfn}[pure relative dimension]\label{dfn:PureRelDim}
Let \( f \colon Y \to T \) be a morphism locally of finite type.
The \emph{relative dimension of} \( f \) at \( y \) is defined as
\( \dim_{y} Y_{f(y)} \), and it is denoted by \( \dim_{y} f \).
We say that \( f \) \emph{has pure relative dimension}
\( d \)
if \( d = \dim_{y} f \) for any \( y \in Y \).
The condition is equivalent to that
every non-empty fiber is equi-dimensional and
has dimension equal to \( d \).
\end{dfn}

\begin{rem}\label{rem:dfn:PureRelDim}
If a flat morphism \( f \colon Y \to T \) is locally of finite type and
it has pure relative dimension, then it is an \emph{equi-dimensional morphism}
in the sense of \cite[IV, D\'ef.~(13.3.2), ($\text{Err}_{\text{IV}}$, 35)]{EGA}.
Because, a generic point of \( Y \) is mapped a generic point of \( T \) by
\eqref{eq:fact:elem-flat:1} applied to
\( \SF = \SO_{Y} \) and \( \SG = \SO_{T} \),
and the condition \( \text{a}'') \) of \cite[IV, Prop.~13.3.1]{EGA} is satisfied.
\end{rem}

\begin{lem}\label{lem:S2Codim2}
Let \( f \colon Y \to T \) be a flat morphism locally of finite type
between locally Noetherian schemes.
For a point \( y \in Y \) and its image \( t = f(y) \),
assume that the fiber \( Y_{t} \) satisfies \( \bfS_{k} \) at \( y \)
for some \( k \geq 2 \). Let \( Y^{\circ} \) be an open subset of \( Y \)
with \( y \not\in Y^{\circ} \).
Then, there exists an open neighborhood \( U \) of \( y \) in \( Y \) such that
\begin{enumerate}
\item  \label{lem:S2Codim2:1} \( f|_{U} \colon U \to T \)
is an \( \bfS_{k} \)-morphism
having pure relative dimension, and

\item \label{lem:S2Codim2:2}
the inequality
\[ \Codim(U_{t'} \setminus Y^{\circ}, U_{t'})
\geq \Codim_{y}(Y_{t} \setminus Y^{\circ}, Y_{t})\]
holds for any \( t' \in f(U) \), where \( U_{t'} = U \cap Y_{t'} \).
\end{enumerate}
\end{lem}

\begin{proof}
By Fact~\ref{fact:dfn:RelSkCMlocus}\eqref{fact:dfn:RelSkCMlocus:3},
replacing \( Y \) with an open neighborhood of \( y \),
we may assume that \( f \) is an \( \bfS_{k} \)-morphism.
For any point \( y' \in Y\) and for the fiber \( Y_{t'} \) over \( t' = f(y') \),
the local ring \( \SO_{Y_{t'}, y'} \) is equi-dimensional
by Fact~\ref{fact:S2}\eqref{fact:S2:1},
since \( Y_{t'} \) is catenary satisfying \( \bfS_{2} \).
Moreover, the local ring has no embedded primes by the condition \( \bfS_{1} \).
If an \emph{associated prime cycle} \( \Gamma \) of \( Y_{t'} \) 
(cf.\ \cite[IV, D\'ef.~(3.1.1)]{EGA}) contains \( y' \), 
then \( \Gamma \) corresponds to an associated prime ideal \( \Gp \) of 
\( \SO_{Y', t'} \), and 
we have 
\[ \dim \Gamma = \dim_{y'} \Gamma 
= \dim \SO_{Y_{t'}, y'}/\Gp + \transdeg \Bbbk(y')/\Bbbk(t') \]
by  \cite[IV, Prop.\ (5.2.1), Cor.\ (5.2.3)]{EGA}. 
Since \( \Gp \) is minimal and \( \SO_{Y_{t'}, y'} \) is equi-dimensional, 
it follows that all the associated prime cycles of \( Y_{t'} \) containing \( y' \) 
have the same dimension. Thus, by \cite[IV, Th.~(12.1.1)(ii)]{EGA}, we may assume that
\( f \) has pure relative dimension, by replacing \( Y \)
with an open neighborhood of \( y \).
Consequently, \( Y \to T \) is an equi-dimensional morphism
(cf.\ Remark~\ref{rem:dfn:PureRelDim}). Then, the function
\[ Y \ni y' \mapsto \Codim_{y'}(Y_{f(y')} \setminus Y^{\circ}, Y_{f(y')}) \]
is lower semi-continuous by \cite[IV, Prop.~(13.3.7)]{EGA}.
Hence, we can take an open neighborhood \( U \) of \( y \) satisfying
the inequality in \eqref{lem:S2Codim2:2}. Thus,
we are done.
\end{proof}

\begin{cor}\label{cor:lem:S2Codim2}
Let \( f \colon Y \to T \) be an \( \bfS_{2} \)-morphism of
locally Noetherian schemes. Assume that every fiber \( Y_{t}  \)
is connected.
\begin{enumerate}
\item \label{cor:lem:S2Codim2:1} If \( T \) are connected,
then \( f \) has pure relative dimension. In particular,
\( f \) is an equi-dimensional morphism.

\item \label{cor:lem:S2Codim2:2}
If \( f \) is proper, then
the function \( t \mapsto \Codim(Y_{t} \cap Z, Y_{t})  \)
is lower semi-continuous on \( T \) for any closed subset \( Z \) of \( Y \).
\end{enumerate}
\end{cor}

\begin{proof}
We may assume that \( T \) is connected.
We know that every fiber \( Y_{t} \) is equi-dimensional
by the proof of Lemma~\ref{lem:S2Codim2}, since \( Y_{t} \) is connected.
Moreover, \( \dim Y_{t} \) is independent of the choice of \( t \in T\)
by Lemma~\ref{lem:S2Codim2}\eqref{lem:S2Codim2:1}, since \( T \) is connected.
Hence, \( f \) has pure relative dimension, and \eqref{cor:lem:S2Codim2:1}
has been proved.
In the case \eqref{cor:lem:S2Codim2:2}, \( f(Y) = T \), since \( f(Y) \) is open and closed.
Let us consider the set  \( F_{k} \) of points
\( y \in Y \) such that
\[ \Codim_{y} (Z \cap Y_{f(y)}, Y_{f(y)}) \leq k\]
for an integer \( k \).
Then, \( f(F_{k}) \) is the set of points \( t \in T \)
with \( \Codim(Y_{t} \cap Z, Y_{t}) \leq k \).
Now, \( F_{k} \) is closed by \eqref{cor:lem:S2Codim2:1}
and by \cite[IV, Prop.~(13.3.7)]{EGA}.
Since \( f \) is proper, \( f(F_{k}) \) is closed.
This proves \eqref{cor:lem:S2Codim2:2}, and we are done.
\end{proof}


\section{Relative \texorpdfstring{$\bfS_{2}$}{S2}-condition and flatness}
\label{sect:flat}

We shall study  \emph{restriction homomorphisms}
(cf.\ Definition~\ref{dfn:restmor} below) of coherent sheaves to open subsets
by applying the local criterion of flatness (cf.\ Section~\ref{subsect:LCflat}), and give 
several criteria for the restriction homomorphism on a fiber to be an isomorphism. 
In Section~\ref{subsect:Resthom}, we prove
the key proposition (Proposition~\ref{prop:key}) and discuss related properties.
Some applications of Proposition~\ref{prop:key} are given in Section~\ref{subsect:AppResthom}:  
Theorem~\ref{thm:invExt} is a criterion for a sheaf to be invertible, 
which is used in the proof of Theorem~\ref{thm:S2S3crit} below. 
Proposition~\ref{prop:inf+val} gives infinitesimal and valuative criteria 
for a reflexive sheaf to satisfy relative \( \bfS_{2} \)-condition. 
Theorem~\ref{thm:dd dec} on the relative \( \bfS_{2}\) refinement 
is analogous to the flattening stratification theorem by Mumford in \cite[Lect.\ 8]{Mumford} 
and to the representability theorem of unramified functors by Murre \cite{Murre}. 
Its local version is given as Theorem~\ref{thm:enhancement}.


\subsection{Restriction homomorphisms}
\label{subsect:Resthom}
In Section~\ref{subsect:Resthom},
we work under Situation~\ref{sit:YTZF} below unless otherwise stated:

\begin{sit}\label{sit:YTZF}
We fix a morphism \( f \colon Y \to T \) of locally Noetherian schemes,
a closed subset \( Z \) of \( Y \), and a coherent \( \SO_{Y} \)-module \( \SF \).
The complement of \( Z \) in \( Y \) is written as \( U \),
and \( j \colon U \injmap Y \) stands for the open immersion.
\end{sit}

\begin{dfn}\label{dfn:restmor}
The \emph{restriction morphism} of \( \SF \) to \( U \) is defined
as the canonical homomorphism
\[ \phi = \phi_{U}(\SF) \colon \SF \to j_{*}(\SF|_{U}). \]
Similarly, for a point \( t \in T \), the restriction homomorphism
of \( \SF_{(t)} \)  to \( U \) (or to \( U \cap Y_{t} \))
is defined as the canonical homomorphism
\[ \phi_{t} = \phi_{U}(\SF_{(t)}) \colon \SF_{(t)} \to j_{*}(\SF_{(t)}|_{U \cap Y_{t}}). \]
Here, \( U \cap Y_{t} \) is identical to \( U \times_{Y} Y_{t} \),
and \( j \) stands also for the open immersion \( U \cap Y_{t} \injmap Y_{t} \).
\end{dfn}

\begin{remn}
The homomorphism \( \phi_{t} \) is an isomorphism along \( U \cap Y_{t} \).
In particular, \( \phi_{t} \) is an isomorphism
if \( t \not\in f(Z) \).
\end{remn}

\begin{remn}
By Remark~\ref{rem:depth:local isom}, we see that
\( \phi \) is an injection (resp.\ isomorphism) along \( Y_{t} \) if and only if
\[ \depth \SF_{y} \geq 1 \qquad (\text{resp. }  \geq 2) \]
for any point \( y \in Z \) such that \( Y_{t} \cap \overline{\{y\}} \ne \emptyset\).
\end{remn}

We use the following notation only in Section~\ref{subsect:Resthom}.
\begin{nota}\label{nota:Resthom2}
For simplicity, we write
\[ \SF_{*} := j_{*}(\SF|_{U}) \quad \text{and} \quad \SF_{(t)*} := j_{*}(\SF_{(t)}|_{U \cap Y_{t}}).\]
When we fix a point \( t \) of \( f(Z) \), we write
\( A \) for the local ring \( \SO_{T, t} \) and \( \GM \) for the maximal ideal \( \GM_{T, t} \),
and for an integer \( n \geq 0 \), we set
\begin{gather*}
A_{n} := A/\GM^{n+1}, \quad T_{n} := \Spec A_{n}, \quad
Y_{n} := Y \times_{T} T_{n}, \\
U_{n} = Y_{n} \cap U, \quad \SF_{n} := \SF \otimes_{\SO_{Y}} \SO_{Y_{n}}, \quad
\SF_{n*} := j_{*}(\SF_{n}|_{U_{n}}).
\end{gather*}
In particular, \( Y_{t} = Y_{0} \), \( \SF_{(t)} = \SF_{0} \),
\( \SF_{(t)*} = \SF_{0*} \), and \( Y_{n} \) is a closed subscheme of \( Y_{m} \)
for any \( m \geq n \). Furthermore,
the restriction homomorphisms of \( \SF_{n} \) and \( (\SF_{n*})_{(t)} \), respectively,
are written by
\[ \phi_{n} \colon \SF_{n} \to \SF_{n*} = j_{*}(\SF_{n}|_{U_{n}}) \quad \text{and} \quad
\varphi_{n} \colon (\SF_{n*})_{(t)} = \SF_{n*} \otimes_{\SO_{Y_{n}}} \SO_{Y_{0}} \to \SF_{0*}. \]
\end{nota}

\begin{rem}\label{rem:nota:Resthom2}
The homomorphism \( \phi_{t} \) in Definition~\ref{dfn:restmor}
equals \( \phi_{0} \), and the diagram
\[ \begin{CD}
\SF_{n} \otimes_{\SO_{Y}} \SO_{Y_{0}} @>{\phi_{n} \otimes \SO_{Y_{0}}}>>
(\SF_{n*}) \otimes_{\SO_{Y}} \SO_{Y_{0}} \\
@V{\isom}VV @VV{\varphi_{n}}V \\
\SF_{0} @>{\phi_{0}}>> \SF_{0*}
\end{CD}\]
is commutative for any \( n \geq 0 \).
\end{rem}

\begin{lem}\label{lem:S2flat(new)}
Assume that \( \SF|_{U} \) is flat over \( T \).
\begin{enumerate}
\item \label{lem:S2flat(new):1}
For a point \( y \in Z \) and \( t = f(y) \),
if \( \phi_{t} \) is injective at \( y \), then
\( \SF_{y} \) is flat over \( \SO_{T, t} \).

\item \label{lem:S2flat(new):2}
For a point \( y \in Z \) and \( t = f(y) \),
if \( \phi_{t} \) is an isomorphism at \( y \), then
the restriction homomorphism \( \phi_{n} \colon \SF_{n} \to \SF_{n*} \)
is an isomorphism at \( y \) for any \( n \geq 0 \).

\item \label{lem:S2flat(new):3}
If \( \phi_{t} \) is an isomorphism for any \( t \in f(Z) \),
then \( \phi \) is also an isomorphism.

\end{enumerate}
\end{lem}

\begin{proof}
First, we shall prove \eqref{lem:S2flat(new):3}
assuming \eqref{lem:S2flat(new):1} and \eqref{lem:S2flat(new):2}.
Since \( \phi_{t} \) is an isomorphism for any \( t \in f(Z) \),
\( \SF \) is flat over \( T \) by \eqref{lem:S2flat(new):1},
and we have
\[ \depth_{Y_{t} \cap Z} \SF_{(t)} \geq 2 \]
by \eqref{lem:S2flat(new):2} (cf.\ Property~\ref{ppty:depth<=2}).
Then, \( \depth_{Z} \SF \geq 2 \)
by Lemma~\ref{lem:relSkCodimDepth}\eqref{lem:relSkCodimDepth:1},
and \( \phi \) is an isomorphism (cf.\ Property~\ref{ppty:depth<=2}).

Next, we shall prove \eqref{lem:S2flat(new):1} and \eqref{lem:S2flat(new):2}.
We may assume that \( T = \Spec A \) for a local Noetherian ring \( A \)
in which \( t = f(y)\) corresponds to the maximal ideal \( \GM \) of \( A \)
and that \( Y = \Spec \SO_{Y, y}\) for the given point \( y \)
(cf.\ Remark~\ref{rem:depth:local isom}).
We write \( \Bbbk = A/\GM = \Bbbk(t)\) and use Notation~\ref{nota:Resthom2}.
From the standard exact sequence
\[ 0 \to \GM^{n}/\GM^{n+1} \to A_{n} \to A_{n-1} \to 0 \]
of \( A \)-modules, by taking tensor products with \( \SF \) over \( A \),
we have an exact sequence
\begin{equation}\label{lem:S2flat(new):eq1}
\GM^{n}/\GM^{n+1} \otimes_{\Bbbk} \SF_{0} \xrightarrow{u_{n}} \SF_{n} \to \SF_{n-1} \to 0
\end{equation}
of \( \SO_{Y} \)-modules.
Here, the left homomorphism \( u_{n} \) is injective at \( y \) for any \( n \geq 0 \)
if and only if \( \SF_{y} \) is flat over \( \SO_{T, t} \)
by the local criterion of flatness (cf.\ Proposition~\ref{prop:LCflat}).
Now, \( u_{n} \) is injective on the open subset \( U_{n} \),
since \( \SF|_{U} \) is flat over \( T \), and \( u_{0} \) is the identity morphism.
For each \( n > 0 \), there is a natural commutative diagram
\[\begin{CD}
@. \GM^{n}/\GM^{n+1} \otimes_{\Bbbk} \SF_{0} @>{u_{n}}>> \SF_{n} @>>> \SF_{n-1} \\
@. @V{\id \otimes \phi_{0}}VV @V{\phi_{n}}VV @V{\phi_{n-1}}VV \\
0 @>>>\GM^{n}/\GM^{n+1} \otimes_{\Bbbk} j_{*}(\SF_{0}|_{U_{0}}) 
@>{j_{*}(u_{n}|_{U_{n}})}>> j_{*}(\SF_{n}|_{U_{n}})
@>>> j_{*}(\SF_{n-1}|_{U_{n-1}})
\end{CD}\]
of exact sequences.
By assumption, \( \phi_{0} = \phi_{t} \) is an injection (resp.\ isomorphism) at \( y \)
in case \eqref{lem:S2flat(new):1} (resp.\ \eqref{lem:S2flat(new):2})
(cf.\ Remark~\ref{rem:depth:local isom}).
Thus, \( u_{n} \) is injective at \( y \) for any \( n \) by the diagram.
This shows \eqref{lem:S2flat(new):1}. In case \eqref{lem:S2flat(new):2},
by induction on \( n \),
we see that \( \phi_{n} \) is an isomorphism at \( y \) for any \( n \), by the diagram.
Thus, \eqref{lem:S2flat(new):2} also holds, and we are done.
\end{proof}

Applying Lemma~\ref{lem:S2flat(new)} to \( \SF = \SO_{Y} \), we have:

\begin{cor}\label{cor:lem:S2flat(new)}
Suppose that \( U \) is flat over \( T \).  If the restriction homomorphism
\( \phi_{t}(\SO_{Y}) \colon \SO_{Y_{t}} \to j_{*}(\SO_{Y_{t} \cap U}) \) is injective
for a point \( t \in f(Z) \), then \( f \) is flat along \( Y_{t} \).
If \( \phi_{t}(\SO_{Y}) \) is an isomorphism for any \( t \in f(Z) \), then
\( \SO_{Y} \isom j_{*}(\SO_{U}) \).
\end{cor}

\begin{prop}[key proposition]\label{prop:key}
Suppose that there is an exact sequence
\[ 0 \to \SF \to \SE^{0} \to \SE^{1} \to \SG \to 0 \]
of coherent \( \SO_{Y} \)-modules such that
\begin{enumerate}
\renewcommand{\theenumi}{\roman{enumi}}
\renewcommand{\labelenumi}{(\theenumi)}
\item  \label{prop:key:ass1}
\( \SE^{0} \), \( \SE^{1} \), and \( \SG|_{U} \) are flat over \( T \), and

\item \label{prop:key:ass2}
the inequalities
\[\depth_{Z \cap Y_{t}} \SE^{0}_{(t)} \geq 2
\quad \text{and} \quad
\depth_{Z \cap Y_{t}} \SE^{1}_{(t)}  \geq 1 \]
hold for any \( t \in f(Z) \).
\end{enumerate}
Then, the following hold\emph{:}
\begin{enumerate}
\item  \label{prop:key:1}
The restriction homomorphism
\( \phi \colon \SF \to \SF_{*} = j_{*}(\SF|_{U})\) is an isomorphism.

\item  \label{prop:key:1a}
For a fixed point \( t \in f(Z) \) and for any integer \( n \geq 0 \),
\( \SF_{n*} = j_{*}(\SF_{n}|_{U_{n}}) \) \emph{(}cf.\ Notation~\emph{\ref{nota:Resthom2})}
is isomorphic to
the kernel \( \SF'_{n} \) of the homomorphism
\[ \SE^{0}_{n} = \SE^{0} \otimes_{\SO_{Y}} \SO_{Y_{n}}
\to \SE^{1}_{n} = \SE^{1} \otimes_{\SO_{Y}} \SO_{Y_{n}} \]
induced by \( \SE^{0} \to \SE^{1} \).
In particular, \( \SF_{n*} \) is coherent for any \( n \geq 0 \),
and \( \SF_{(t)*} = j_{*}(\SF_{(t)}|_{U \cap Y_{t}}) \)
is coherent for any \( t \in f(Z) \).

\item \label{prop:key:2}
For any point \( y \in Y \) and \( t = f(y) \),
the following conditions are equivalent to each other, where we use
Notation~\emph{\ref{nota:Resthom2}}
in \eqref{prop:key:condAA}, \eqref{prop:key:condBB},
and \eqref{prop:key:condBBB}\emph{:}
\begin{enumerate}
    \makeatletter
    \renewcommand{\p@enumii}{}
    \makeatother
\item \label{prop:key:condA}
\( \phi_{t} \colon \SF_{(t)} \to \SF_{(t)*} \)
is surjective at \( y \)\emph{;}

\item \label{prop:key:condB}
\( \phi_{t} \) is an isomorphism at \( y \)\emph{;}

\item \label{prop:key:condC}
\( \SG_{y} \) is flat over \( \SO_{T, t} \)\emph{;}

\addtocounter{enumii}{-3}
\renewcommand{\theenumii}{$\text{\alph{enumii}}'$}
\renewcommand{\labelenumii}{(\theenumii)}

\item \label{prop:key:condAA}
\( \varphi_{n} \colon (\SF_{n*})_{(t)} \to \SF_{0*} \) is surjective at \( y \)
for any \( n \geq 0 \)\emph{;}

\item \label{prop:key:condBB}
\( \varphi_{n} \) is an isomorphism at \( y \) for any \( n \geq 0 \)\emph{;}

\addtocounter{enumii}{-1}
\renewcommand{\theenumii}{$\text{\alph{enumii}}''$}
\renewcommand{\labelenumii}{(\theenumii)}
\item \label{prop:key:condBBB}
\( \phi_{n} \colon \SF_{n} \to \SF_{n*} \) is an isomorphism at \( y \)
for any \( n \geq 0 \).
\end{enumerate}
Note that if \eqref{prop:key:condC} is satisfied, then
\( \SF_{y} \) is also flat over \( \SO_{T, t} \).
\end{enumerate}
\end{prop}

\begin{proof}
By \eqref{prop:key:ass1}, \eqref{prop:key:ass2}
and by Lemma~\ref{lem:relSkCodimDepth}\eqref{lem:relSkCodimDepth:1},
we have \( \depth_{Z} \SE^{0} \geq 2 \) and \( \depth_{Z} \SE^{1} \geq 1\).
Thus, \( \depth_{Z} \SF \geq 2 \) by  Lemma~\ref{lem:depth<=2}\eqref{lem:depth<=2:0},
and we have \eqref{prop:key:1} (cf.\ Property~\ref{ppty:depth<=2}).
For each \( n \geq 0 \), the  exact sequence
\[ 0 \to \SF'_{n} \to \SE^{0}_{n} \to \SE^{1}_{n} \to \SG_{n} \to 0\]
on \( Y_{n} \) satisfies the conditions \eqref{prop:key:ass1}
and \eqref{prop:key:ass2} for the induced morphism \( Y_{n} \to T_{n} \),
where \( \SG_{n} = \SG \otimes \SO_{Y_{n}} \).
Thus, by \eqref{prop:key:1}, the restriction homomorphism
\[ \phi(\SF'_{n}) \colon \SF'_{n} \to (\SF'_{n})_{*} = j_{*}(\SF'_{n}|_{U_{n}}) \]
of \( \SF'_{n} \) is an isomorphism.
On the other hand,
there is a canonical homomorphism \( \psi_{n} \colon \SF_{n} = \SF \otimes \SO_{Y_{n}} \to \SF'_{n} \).
Note that \( \psi_{n} \)
is an isomorphism at a point \( y \in Y_{t} = Y_{0}\) if \( \SG_{y} \) is flat over \( \SO_{T, t} \).
In particular, \( \psi_{n} \) is an isomorphism on \( U_{n} \) by the condition \eqref{prop:key:ass1}.
Hence, \( (\SF'_{n})_{*} \isom \SF_{n*} \), 
and we have an isomorphism \( \SF'_{n} \isom \SF_{n*} \), 
by which \( \phi_{n} \) is isomorphic to \( \psi_{n} \).
This proves \eqref{prop:key:1a}.
For the proof of \eqref{prop:key:2}, we may assume that \( y \in Z \).
We shall show that there is an exact sequence
\begin{equation}\label{eq:prop:key:exact}
\Tor_{2}^{A}(\SG_{y}, \Bbbk) \to
(\SF_{(t)})_{y} =
\SF_{y} \otimes_{\SO_{Y, y}} \SO_{Y_{t}, y}
\xrightarrow{(\psi_{0})_{y}}
(\SF'_{0})_{y} \to \Tor_{1}^{A}(\SG_{y}, \Bbbk) \to 0
\end{equation}
of \( \SO_{Y, y} \)-modules,
where \( A = \SO_{T, t} \) and \( \Bbbk = \Bbbk(t) \):
For the image \( \SB \) of \( \SE^{0} \to \SE^{1} \), we have two short exact sequences
\( 0 \to \SF \to \SE^{0} \to \SB \to 0\) and \( 0 \to \SB \to \SE^{1} \to \SG \to 0 \) on \( Y \).
Then, the kernel of \( \SB_{0} = \SB \otimes \SO_{Y_{0}} \to \SE^{1}_{0} = \SE^{1} \otimes \SO_{Y_{0}} \)
is isomorphic to \( \STor^{\SO_{T}}_{1}(\SG, \Bbbk) \), and
the kernel of \( \SF_{0} \to \SE^{0}_{0} \) is isomorphic
to \( \STor^{\SO_{T}}_{1}(\SB, \Bbbk) \isom \STor^{\SO_{T}}_{2}(\SG, \Bbbk) \).
Then, we have the exact sequence \eqref{eq:prop:key:exact}
by applying the snake lemma to the commutative diagram
\[ \begin{CD}
@. \SF_{0} @>>> \SE^{0}_{0} @>>> \SB_{0} @>>> 0 \\
@. @V{\psi_{0}}VV @V{=}VV @VVV \\
0 @>>> \SF'_{0} @>>> \SE^{0}_{0} @>>> \SE^{1}_{0}
\end{CD}\]
of exact sequences. Note that \( \psi_{0} \isom \phi_{0} \) by the argument above.
We shall prove \eqref{prop:key:2} using \eqref{eq:prop:key:exact}.
If \eqref{prop:key:condA} holds, then \( \Tor_{1}^{A}(\SG_{y}, \Bbbk) = 0 \) by
\eqref{eq:prop:key:exact}, and it implies \eqref{prop:key:condC} by
the local criterion of flatness (cf.\ Proposition~\ref{prop:LCflat}),
since \( \SG_{y} \otimes \SO_{Y_{t}, y} \isom \SG_{y} \otimes \Bbbk \) is flat over \( \Bbbk \).
If \eqref{prop:key:condC} holds, then \( \Tor^{A}_{j}(\SG_{y}, \Bbbk) = 0 \)
for \( j = 1 \) and \( 2 \),
and it implies \eqref{prop:key:condB} by \eqref{eq:prop:key:exact}.
Thus, we have shown the equivalence of the three conditions
\eqref{prop:key:condA},
\eqref{prop:key:condB}, and  \eqref{prop:key:condC}.
By applying the equivalence of three conditions to
\( \SF'_{n} \isom \SF_{n*}\) and \( Y_{n} \to T_{n} \)
instead of \( \SF \) and \( Y \to T \), we see that \eqref{prop:key:condAA} and
\eqref{prop:key:condBB} are both equivalent to that
\( (\SG_{n})_{y} \) is flat over \( \SO_{T_{n}, t} \) for any \( n \geq 0 \);
This is also equivalent to \eqref{prop:key:condC} by
the local criterion of flatness
(cf.\ \eqref{prop:LCflat:1} \( \Leftrightarrow \) \eqref{prop:LCflat:4} in Proposition~\ref{prop:LCflat}).
If \eqref{prop:key:condC} holds, then \( \psi_{n} \colon \SF_{n} \to \SF'_{n} \) is an isomorphism
as we have noted before,
and the isomorphism \( \SF'_{n} \isom \SF_{n*} \)
in \eqref{prop:key:1a} implies \eqref{prop:key:condBBB}.
Conversely, if \eqref{prop:key:condBBB} holds, then \( \varphi_{n} \) is isomorphic to
the canonical isomorphism \( (\SF_{n})_{(t)} \isom \SF_{(t)} \) for any \( n \)
(cf.\ Remark~\ref{rem:nota:Resthom2}), and it implies
\eqref{prop:key:condBB}. Thus, we are done.
\end{proof}

\begin{remn}
The exact sequence \eqref{eq:prop:key:exact}
is obtained as the ``edge sequence'' of the spectral sequence
\[ E_{2}^{p, q} = \STor^{\SO_{T}}_{-p}(\SH^{q}(\SE^{\bullet}), \, \Bbbk(t))
\Rightarrow E^{p+q} = \SH^{p+q}(\SE^{\bullet}_{(t)})\]
of \( \SO_{Y_{t}} \)-modules (cf.\ \cite[III, (6.3.2.2)]{EGA})
arising from the quasi-isomorphism
\[ \SE^{\bullet}_{(t)} \isom_{\qis} \SE^{\bullet} \otimes^{\bfL}_{\SO_{T}} \Bbbk(t),\]
where \( \SE^{\bullet} \) and \( \SE^{\bullet}_{(t)} \) denote
the complexes \( [0 \to \SE^{0} \to \SE^{1} \to 0] \) and \( [0\to \SE^{0}_{(t)} \to \SE^{1}_{(t)} \to 0] \),
respectively.
\end{remn}

\begin{rem}\label{rem:phi_infty}
In the situation of Proposition~\ref{prop:key}\eqref{prop:key:1a},
the canonical homomorphism
\[ \phi_{\infty} = \varprojlim\nolimits_{n} \phi_{n} \colon \varprojlim\nolimits_{n} \SF_{n}
\to \varprojlim\nolimits_{n} \SF_{n*}\]
is an isomorphism, where the projective limit \( \varprojlim_{n} \) is taken
in the category of \( \SO_{Y} \)-modules.
This is shown as follows. Since \( \SF'_{n} \isom \SF_{n*} \), it is enough to show that
the homomorphism
\[ \psi_{\infty}(V) := \varprojlim\nolimits_{n} \OH^{0}(V, \psi_{n}) 
\colon \varprojlim\nolimits_{n} \OH^{0}(V, \SF_{n})
\to \varprojlim\nolimits_{n} \OH^{0}(V, \SF'_{n}) \]
is an isomorphism for any open affine subset \( V \) of \( Y \), where we note that
the global section functor \( \OH^{0}(V, \bullet) \) commutes with
\( \varprojlim \).
For \( R = \OH^{0}(V, \SO_{V}) \) and \( R_{n} = R/\GM^{n+1}R \isom \OH^{0}(V, \SO_{Y_{n}})  \),
we have two exact sequences:
\begin{gather*}
0 \to \OH^{0}(V, \SF) \to \OH^{0}(V, \SE^{0}) \to \OH^{0}(V, \SE^{1}), \\
0 \to \OH^{0}(V, \SF'_{n}) \to \OH^{0}(V, \SE^{0}) \otimes_{R} R_{n}
\to \OH^{0}(V, \SE^{1}) \otimes_{R} R_{n}.
\end{gather*}
Since the \( \GM R \)-adic completion \( \widehat{R} = \varprojlim\nolimits R_{n} \)
is flat over \( R \)
and since \( \varprojlim \) is left exact,
we have an isomorphism
\[ \OH^{0}(V, \SF) \otimes_{R} \widehat{R} \isom
\Ker(\OH^{0}(V, \SE^{0}) \otimes_{R} \widehat{R} \to \OH^{0}(V, \SE^{1}) \otimes_{R} \widehat{R})
\isom \varprojlim\nolimits_{n} \OH^{0}(V, \SF'_{n}). \]
Then, \( \psi_{\infty}(V) \) is an isomorphism, since
\[ \varprojlim\nolimits_{n} \OH^{0}(V, \SF_{n})
\isom \varprojlim\nolimits_{n} (\OH^{0}(V, \SF) \otimes_{R} R_{n})
\isom \OH^{0}(V, \SF) \otimes_{R} \widehat{R}. \]
\end{rem}

\begin{cor}\label{cor0:prop:key}
In the situation of Proposition~\emph{\ref{prop:key}},
assume that \( f \) is locally of finite type. Then,
the condition \eqref{prop:key:condC}
of Proposition~\emph{\ref{prop:key}}\eqref{prop:key:2} for a point \( y \in Y \)
is equivalent to\emph{:}

\begin{enumerate}
\renewcommand{\theenumi}{\alph{enumi}}
\renewcommand{\labelenumi}{(\theenumi)}
\addtocounter{enumi}{3}
\item  \label{cor0:prop:key:condD}
there is an open neighborhood \( V \) of \( y \) in \( Y \) such that
\( \SF|_{V} \) is flat over \( T \),
and \( \phi_{t} \) is an isomorphism on \( V \cap Y_{t} \)
for any \( t \in f(V)  \).
\end{enumerate}

Furthermore, if \( \SF_{(t)}|_{U \cap Y_{t}} \) satisfies \( \bfS_{2} \) for the point \( t = f(y)\)
and if \( \SF_{(t')} \) is equi-dimensional and
\begin{equation}\label{eq:cor0:prop:key}
\Codim(Z \cap Y_{t'} \cap \Supp \SF, Y_{t'} \cap \Supp \SF) \geq 2
\end{equation}
for any \( t' \in T \),
then \eqref{cor0:prop:key:condD} is equivalent to\emph{:}
\begin{enumerate}
\renewcommand{\theenumi}{\alph{enumi}}
\renewcommand{\labelenumi}{(\theenumi)}
\addtocounter{enumi}{4}
\item \label{cor0:prop:key:condE}
there is an open neighborhood \( V \) of \( y \) in \( Y \) such that
\( \SF|_{V} \) satisfies relative \( \bfS_{2} \) over \( T \),
i.e., \( V = \bfS_{2}(\SF|_{V}/T) \).
\end{enumerate}
\end{cor}

\begin{proof}
For the first assertion, by Proposition~\ref{prop:key}\eqref{prop:key:2},
it is enough to show
\eqref{prop:key:condC} \( \Rightarrow \) \eqref{cor0:prop:key:condD}
assuming that \( f \) is locally of finite type and \( y \in Z \).
When \eqref{prop:key:condC} holds, \( \SG|_{V} \) is flat over \( T \)
for an open neighborhood \( V \) of \( y \) in \( Y \),
by Fact~\ref{fact:elem-flat}\eqref{fact:elem-flat:1}.
Thus, \( \SF|_{V} \) is flat over \( T \)
by Proposition~\ref{prop:key}\eqref{prop:key:ass1},
and moreover, by Proposition~\ref{prop:key}\eqref{prop:key:2} applied to any point in \( V \),
we see that \( \phi_{t}\) is an isomorphism on \( Y_{t} \cap V \) for any \( t \in f(V \cap Z)\).
Since \( \phi_{t} \) is an isomorphism for any \( t \not\in f(Z) \), we have proved:
\eqref{prop:key:condC} \( \Rightarrow \) \eqref{cor0:prop:key:condD}.

We shall show
\eqref{cor0:prop:key:condD}  \( \Leftrightarrow \)  \eqref{cor0:prop:key:condE}
in the situation of the second assertion.
In this case, if \( \phi_{t} \) is an isomorphism, then \( \SF_{(t)} \) satisfies \( \bfS_{2} \) by
Corollary~\ref{cor:basicS1S2}. Hence, we have \eqref{cor0:prop:key:condD} \( \Rightarrow \)
\eqref{cor0:prop:key:condE} by
Fact~\ref{fact:dfn:RelSkCMlocus}\eqref{fact:dfn:RelSkCMlocus:2}.
Conversely, if \eqref{cor0:prop:key:condE} holds with \( V = Y \), then
\[ \depth_{Y_{t'} \cap Z} \SF_{(t')} \geq 2  \]
for any \( t' \in f(Z) \) by Lemma~\ref{lem:depth+codim+Sk}\eqref{lem:depth+codim+Sk:2},
since \( \SF_{(t')} \) satisfies \( \bfS_{2} \)
and the inequality \eqref{eq:cor0:prop:key} holds.
Hence, \( \phi_{t'} \) is an isomorphism for any \( t' \in f(Z) \),
and \eqref{cor0:prop:key:condD} holds. Thus, we are done.
\end{proof}

\begin{cor}\label{cor:SurjFlat(new)|Kollar}
In the situation of Proposition~\emph{\ref{prop:key}},
for a point \( t \in f(Z) \),
assume that the coherent \( \SO_{Y_{t}} \)-module
\( \SF_{(t)*} = j_{*}(\SF_{(t)}|_{Y_{t} \cap U})\) satisfies
\begin{equation}\label{cor:SurjFlat(new)|Kollar|eq}
\depth_{Y_{t} \cap Z} \SF_{(t)*} \geq 3.
\end{equation}
Then, the sheaves \( \SF \) and \( \SG \) are flat over \( T \) along \( Y_{t} \),
and the restriction homomorphism
\( \phi_{t} \colon \SF_{(t)} \to \SF_{(t)*} \) is an isomorphism.
\end{cor}

\begin{proof}
By Proposition~\ref{prop:key}\eqref{prop:key:2},
it is enough to prove that \( \phi_{t} \) is an isomorphism.
By \eqref{cor:SurjFlat(new)|Kollar|eq}, we have
\[ R^{1}j_{*}(\SF_{(t)}|_{U \cap Y_{t}}) = R^{1} j_{*} (\SF_{0}|_{U_{0}}) = 0 \]
(cf.\ Property~\ref{ppty:depth<=2}).
Hence, the exact sequence \eqref{lem:S2flat(new):eq1} in the proof
of Lemma~\ref{lem:S2flat(new)} induces an exact sequence
\[ 0 \to \GM^{n}/\GM^{n+1} \otimes_{\Bbbk} j_{*}(\SF_{0}|_{U_{0}}) \to j_{*}(\SF_{n}|_{U_{n}})
\to j_{*}(\SF_{n-1}|_{U_{n-1}}) \to 0. \]
Since \( \SF_{n*} = j_{*}(\SF_{n}|_{U_{n}}) \),
the homomorphism \( \varphi_{n} \)
is surjective for any \( n \geq 0 \).
Therefore, \( \phi_{t} \) is an isomorphism
by \eqref{prop:key:condBB} \( \Rightarrow \) \eqref{prop:key:condB}
of Proposition~\ref{prop:key}\eqref{prop:key:2}.
\end{proof}

\begin{rem}\label{rem:KollarFlatness}
Corollary~\ref{cor:SurjFlat(new)|Kollar} is similar to a special case of \cite[Th.\ 12]{KollarFlat},
where the sheaf corresponding to \( \SF \) above may not have an exact sequence
of Proposition~\ref{prop:key}.
However, \cite[Th.\ 12]{KollarFlat} is not true. Example~\ref{exam:8-3} below provides a counterexample.
\end{rem}

\begin{exam}\label{exam:8-3}
Let \( Y \) be an affine space \( \BAA_{\Bbbk}^{8} \)
of dimension \( 8 \) over a field \( \Bbbk \)
with a coordinate system \( (\ytt_{1}, \ytt_{2}, \ldots, \ytt_{8}) \).
Let \( T \) be a \( 3 \)-dimensional affine space \( \BAA^{3}_{\Bbbk} \)
and let \( f \colon Y \to T \) be the projection
defined by \( (\ytt_{1}, \ldots, \ytt_{8}) \mapsto (\ytt_{1}, \ytt_{2}, \ytt_{3}) \).
The fiber \( Y_{0} = f^{-1}(0)\) over the origin \( 0 = (0, 0, 0) \) of \( T \)
is of dimension \( 5 \). We define closed subschemes \( Z \) and \( V \) of \( Y \) by
\begin{align*}
Z &:= \{\ytt_{4} = \ytt_{5} = \ytt_{6} = 0\} \quad \text{and}\\
V &:= \{\ytt_{1} + \ytt_{2} \ytt_{7} + \ytt_{3} \ytt_{8}
= \ytt_{4} - \ytt_{1} = \ytt_{5} - \ytt_{2} = \ytt_{6} - \ytt_{3} = 0\}.
\end{align*}
Then, we can show the following properties:
\begin{enumerate}
\item \label{exam:8-3:1} \( V \isom \BAA^{4}_{\Bbbk} \),
and \( V \cap Y_{0} = V \cap Z = Y_{0} \cap Z \isom \BAA^{2}\);

\item \label{exam:8-3:2} \( \Codim(Z, Y) = \Codim(Z \cap Y_{0}, Y_{0}) = 3 \), and
\( \Codim(Z \cap V, V) = 2\);

\item \label{exam:8-3:3} \( V \setminus Y_{0} \to T \) is a smooth morphism of relative dimension one, but
the fiber \( V \cap Y_{0} \) of \( V \to T \) over \( 0 \) is two-dimensional.
\end{enumerate}
Let \( j \colon U \injmap Y \) be the open immersion from the complement \( U := Y \setminus Z \),
and we set \( \SF := \SO_{Y} \oplus \SO_{V} \) and \( \SF_{0} := \SF \otimes_{\SO_{Y}} \SO_{Y_{0}} \).
By \eqref{exam:8-3:1} and \eqref{exam:8-3:2}, we have isomorphisms
\begin{align}
j_{*}(\SF|_{U}) &\isom j_{*}\SO_{U} \oplus j_{*}\SO_{U \cap V} \isom \SO_{Y} \oplus \SO_{V}
\quad \text{and} \quad  \label{eq:1|exam:8-3}\\
j_{*}(\SF_{0}|_{U \cap Y_{0}}) &\isom j_{*}\SO_{U \cap Y_{0}} \isom \SO_{Y_{0}}, \label{eq:2|exam:8-3}
\end{align}
since \( U \cap V \cap Y_{0} = \emptyset \),
\( \depth_{Z} \SO_{Y} \geq 2\), \( \depth_{Z \cap V} \SO_{V} \geq 2\),
and \( \depth_{Z \cap Y_{0}} \SO_{Y_{0}} \geq 2 \).
Thus, we have:
\begin{enumerate}
\addtocounter{enumi}{3}
\item \label{exam:8-3:4}
\( \SF|_{U} = \SO_{U} \oplus \SO_{U \cap V}\) is flat over \( T \) by \eqref{exam:8-3:3};
\item \label{exam:8-3:5}
\( j_{*}(\SF|_{U}) \) is not flat over \( T \)
by \eqref{exam:8-3:3} and \eqref{eq:1|exam:8-3};
\item \label{exam:8-3:6}
\( j_{*}(\SF_{0}|_{U \cap Y_{0}}) \) satisfies \( \bfS_{3} \)
by \eqref{eq:2|exam:8-3};

\item \label{exam:8-3:7} the canonical homomorphism
\[ j_{*}(\SF|_{U}) \otimes_{\SO_{Y}} \SO_{Y_{0}} \to j_{*}(\SF_{0}|_{U \cap Y_{0}}) \]
is not an isomorphism by \eqref{eq:1|exam:8-3} and \eqref{eq:2|exam:8-3}.
\end{enumerate}
Thus, \( f \colon Y \to T \), \( \SF \), and \( U \) give a counterexample
to \cite[Th.\ 12]{KollarFlat}: The required assumptions are
satisfied by \eqref{exam:8-3:2},
\eqref{exam:8-3:4}, and \eqref{exam:8-3:6}, but the conclusion is denied by
\eqref{exam:8-3:5} and \eqref{exam:8-3:7}.

The kernel \( \SJ \) of \( \SO_{Y} \to \SO_{V} \)
has also an interesting infinitesimal property.
Let \( A = \Bbbk[\ytt_{1}, \ytt_{2}, \ytt_{3}] \) be the coordinate ring
of \( T \),  \( \GM = (\ytt_{1}, \ytt_{2}, \ytt_{3}) \) the maximal ideal
at the origin \( 0 \in T\), and set
\[ A_{n} = A/\GM^{n+1}, \,\, T_{n} = \Spec A_{n}, \,\, Y_{n} = Y \times_{T} T_{n}, \,\,
V_{n} = V \times_{T} T_{n}, \,\, \SJ_{n} = \SJ \otimes_{\SO_{Y}} \SO_{Y_{n}}\]
for each \( n \geq 0 \) as in Notation~\ref{nota:Resthom2}. Then,
we can prove:
\begin{equation}\label{eq:3|exam:8-3}
\SJ \not\isom \SO_{Y}, \quad \SJ|_{U} \not\isom \SO_{U}, \quad \SJ \isom \SJ_{*} = j_{*}(\SJ|_{U}), 
\quad \text{and} \quad \SJ_{n}|_{U \cap Y_{n}} \isom \SO_{U \cap Y_{n}}
\end{equation}
for any \( n \geq 0 \).
In fact, the first two of \eqref{eq:3|exam:8-3} are consequences of that the ideal sheaf
\( \SJ|_{U} \) of \( V \cap U \) is not an invertible \( \SO_{U} \)-module, and it is derived from
\( \Codim(V \cap U, U) = 4 > 1  \).
The third isomorphism of \eqref{eq:3|exam:8-3} follows
from \( \depth_{Z} \SO_{Y} \geq 2 \) and \( \depth_{Z} \SO_{V} \geq 2 \) (cf.\ \eqref{exam:8-3:2}), 
and the last one from that the kernel of \( \SJ_{n} \to \SO_{Y_{n}} \)
is isomorphic to \( \STor^{\SO_{Y}}_{1}(\SO_{V}, \SO_{Y_{n}}) \), 
which is supported on \( V \cap Y_{0} \subset Y \setminus U\). 
\end{exam}

\begin{rem}\label{rem:exam:8-3}
In the situation of Notation~\ref{nota:Resthom2},
not a few people may fail to believe the following wrong assertion:
\begin{enumerate}
\renewcommand{\theenumi}{$\ast$}
\renewcommand{\labelenumi}{(\theenumi)}
\item \label{rem:exam:8-3:0}
\emph{
If \( \phi \colon \SF \to j_{*}(\SF|_{U}) \) is an isomorphism, then the morphism 
\[ \phi^{\qcoh}_{\infty} \colon \varprojlim^{\qcoh}\nolimits_{n} \SF_{n} \to
\varprojlim\nolimits_{n}^{\qcoh} j_{*}(j^{*}\SF_{n}) \]
induced by \( \phi_{n} \colon \SF_{n} \to \SF_{n*} = j_{*}(j^{*}\SF_{n}) \) 
is also an isomorphism. 
}
\end{enumerate}
Here, \( \varprojlim^{\qcoh} \) stands for the projective limit 
in the category \( \QCoh(\SO_{Y}) \) of quasi-coherent \( \SO_{Y} \)-modules. 
We shall show that the ideal sheaf \( \SJ  \) in Example~\ref{exam:8-3} provides a counterexample 
of \eqref{rem:exam:8-3:0}, 
and explain why a usual projective limit argument does not work for the ``proof'' of \eqref{rem:exam:8-3:0}.

For simplicity, assume that \( Y \) is an affine Noetherian scheme \( \Spec R\),
and set \( M = \OH^{0}(Y, \SF)\), which is a finitely generated \( R \)-module.
We set \( R_{n} = R/\GM^{n+1}R \) and \( M_{n} = M \otimes_{R} R_{n} \) for integers \( n \geq 0 \).
Then \( \SF \isom M\sptilde\) and \(\SF_{n} \isom M_{n}\sptilde\).
Let \( \widehat{R} \) be the \( \GM R \)-adic completion \( \varprojlim_{n} R_{n} \), and
let \( \pi \colon \Spec \widehat{R} \to \Spec R = Y \) be the associated morphism of schemes. 
Then, we have isomorphisms 
\[ \varprojlim^{\qcoh}\nolimits_{n} \SF_{n} \isom 
(M \otimes_{R} \widehat{R})\sptilde \isom \pi_{*}(\pi^{*}\SF). \]
Note that the projective limit \( \varprojlim\nolimits_{n} \SF_{n} \) 
in the category \( \Mod(\SO_{Y}) \) of \( \SO_{Y} \)-modules 
is not quasi-coherent in general. 

We shall show that 
the ideal sheaf \( \SJ  \) in Example~\ref{exam:8-3} provides a counterexample of \eqref{rem:exam:8-3:0}. 
In this case, the left hand side of \( \phi_{\infty}^{\qcoh} \) is isomorphic to \( \pi_{*}(\pi^{*}\SJ) \) and 
the right hand side is to \( \pi_{*}\pi^{*}\SO_{Y} \) by the last isomorphisms of \eqref{eq:3|exam:8-3}. 
Now, \( \pi \) is faithfully flat if we replace \( Y = \BAA^{8}\) 
with \( \Spec \SO_{Y, 0} \) for the origin \( 0 = (0, 0, \ldots, 0) \in Y\). 
Then, \( \phi^{\qcoh}_{\infty} \) is never an isomorphism, since \( \SJ \not\subset \SO_{Y} \). 
On the other hand, \( \phi \colon \SJ \to  j_{*}(\SJ|_{U}) \) 
is an isomorphism as the third isomorphism of \eqref{eq:3|exam:8-3}. 

We remark here on the commutativity of \( \varprojlim \) with functors \( j_{*} \) and \( j^{*} \).
The direct image functor \( j_{*} \colon \QCoh(\SO_{U}) \to \QCoh(\SO_{Y})  \) is right adjoint to
the restriction functor \( j^{*} \colon \QCoh(\SO_{Y}) \to \QCoh(\SO_{U}) \). Thus, \( \varprojlim \)
commutes with \( j_{*} \), and we have an isomorphism 
\[ \alpha \colon j_{*}\!\left( \varprojlim\nolimits_{n}^{\qcoh} j^{*}\SF_{n} \right) 
\xrightarrow{\isom} \varprojlim\nolimits_{n}^{\qcoh} j_{*}(j^{*}\SF_{n}).\]
On the other hand, \( j^{*} \) does not have a left adjoint functor.
Because, the left adjoint functor \( j_{!} \colon \Mod(\SO_{U}) \to \Mod(\SO_{Y})\) 
of \( j^{*} \colon \Mod(\SO_{Y}) \to \Mod(\SO_{U}) \)
does not preserve quasi-coherent sheaves. 
Thus, \( \varprojlim \) does not commute with \( j^{*} \) in general, and
hence, the canonical morphism
\[ \beta \colon j^{*}\!\left(\varprojlim\nolimits_{n}^{\qcoh} \SF_{n}\right) 
\to \varprojlim\nolimits_{n}^{\qcoh} j^{*}\SF_{n} \]
in \( \QCoh(\SO_{U}) \) is not necessarily an isomorphism.

It is necessary to check the morphism \( \beta \) to be an isomorphism,
for the ``proof'' of \eqref{rem:exam:8-3:0} by the projective limit argument. 
In fact, we can prove: 
\begin{enumerate}
\renewcommand{\theenumi}{$\ast\ast$}
\renewcommand{\labelenumi}{(\theenumi)}
\item \label{rem:exam:8-3:1}
\emph{When \( \phi \) is an isomorphism,
\( \phi^{\qcoh}_{\infty}  \) is an isomorphism if and only if \( \beta  \) is so.}
\end{enumerate}
This is a point to which many people probably do not pay attention.
\begin{proof}[Proof of \eqref{rem:exam:8-3:1}]
Now, we have a commutative diagram
\[ \begin{CD}
\pi_{*}\pi^{*}\SF @>{\isom}>> \varprojlim^{\qcoh}\nolimits_{n} \SF_{n}
@>{\phi^{\qcoh}_{\infty}}>> \varprojlim\nolimits_{n}^{\qcoh} j_{*}j^{*}(\SF_{n})  \\
@V{\hat{\phi}}VV @VVV @A{\isom}A{\alpha}A \\
j_{*}j^{*}(\pi_{*}\pi^{*}\SF) @>{\isom}>>
j_{*}j^{*}\!\left( \varprojlim^{\qcoh}\nolimits_{n} \SF_{n} \right)
@>{j_{*}(\beta)}>> j_{*}\!\left( \varprojlim\nolimits_{n}^{\qcoh} j^{*}\SF_{n} \right)
\end{CD}\]
in \( \QCoh(\SO_{Y}) \),
where \( \hat{\phi} \) is the restriction homomorphism
of \( \pi_{*}\pi^{*}\SF \). Thus, it suffices to show that if \( \phi \) is an isomorphism,
then \( \hat{\phi} \) is so.
Let us consider an isomorphism
\[ \gamma \colon \pi_{*}\pi^{*}(j_{*}j^{*}\SF) \xrightarrow{\isom} j_{*}j^{*}(\pi_{*}\pi^{*}\SF)\]
defined as the composite
\[ \pi_{*}\pi^{*}(j_{*}j^{*}(\SF)) \xrightarrow[\isom]{\pi_{*}(\delta')} 
\pi_{*}\hat{j}_{*} (\pi_{U}^{*}(j^{*}\SF))
\isom j_{*}\pi_{U*} \hat{j}^{*}(\pi^{*}\SF) \\
\xrightarrow[\isom]{j_{*}(\delta'')^{-1}} j_{*} j^{*} \pi_{*} \pi^{*}\SF \]
of canonical isomorphisms; Here \( \hat{j} \) is the open immersion \( \pi^{-1}(U) \injmap \Spec \widehat{R} \)
and \( \pi_{U} \) is the restriction of \( \pi \) to \( \pi^{-1}(U) \),
and the morphisms
\[ \delta' \colon \pi^{*}j_{*}(j^{*}\SF) \xrightarrow{\isom} \hat{j}_{*} \pi_{U}^{*}(j^{*}\SF) 
\quad \text{and} \quad
\delta'' \colon j^{*} \pi_{*} (\pi^{*}\SF) \xrightarrow{\isom} \pi_{U*} \hat{j}^{*}(\pi^{*}\SF) \]
are flat base change isomorphisms (cf.\ Lemma~\ref{lem:flatbc}).
Then, we can write
\( \hat{\phi} = \gamma \circ (\pi_{*}\pi^{*}(\phi)) \) for the induced morphism
\[ \pi_{*}\pi^{*}(\phi) \colon \pi_{*}\pi^{*}\SF \to \pi_{*}\pi^{*}(j_{*}j^{*}\SF), \]
and this shows that if \( \phi \) is an isomorphism, then \( \hat{\phi} \) is so. Thus, we are done.
\end{proof}
\end{rem}

In the rest of Section~\ref{subsect:Resthom},
in Lemmas~\ref{lem:SurjFlat(reflexive)} and \ref{lem:SurjFlat(complex)} below,
we shall give sufficient conditions for \( \SF \) to admit
an exact sequence of Proposition~\ref{prop:key}.

\begin{lem}\label{lem:SurjFlat(reflexive)}
Suppose that  \( f \circ j \colon U \to T \) is flat and
\[ \depth_{Y_{t} \cap Z} \SO_{Y_{t}} \geq 2\]
for any \( t \in f(Z) \).
If \( \SF \) is a reflexive \( \SO_{Y} \)-module and if \( \SF|_{U} \) is locally free,
then there exists an exact sequence
\( 0 \to \SF \to \SE^{0} \to \SE^{1} \to \SG \to 0 \) \emph{locally on \( Y \)}
which satisfies the conditions \eqref{prop:key:ass1}
and \eqref{prop:key:ass2} of Proposition~\emph{\ref{prop:key}}.
\end{lem}

\begin{proof}
The morphism \( f \) is flat by Corollary~\ref{cor:lem:S2flat(new)}.
Since \( \SF \) is coherent, locally on \( Y \), we have a finite presentation
\[ \SO_{Y}^{\oplus m} \xrightarrow{h} \SO_{Y}^{\oplus n} \to \SF^{\vee} \to 0 \]
of the dual \( \SO_{Y} \)-module
\( \SF^{\vee} = \SHom_{\SO_{Y}}(\SF, \SO_{Y}) \).
Let \( \SK \) be the kernel of the left homomorphism \( h \).
Then, \( \SK|_{U} \) is locally free, since so is \( \SF|_{U} \).
We have an exact sequence
\[ 0 \to \SF \isom \SF^{\vee\vee} \to \SO_{Y}^{\oplus n} \xrightarrow{h^{\vee}} \SO_{Y}^{\oplus m} \]
by taking the dual.
Let \( \SG \) be the cokernel of \( h^{\vee} \).
Then, \( \SG|_{U} \) is isomorphic to the locally free sheaf
\( \SK^{\vee}|_{U}\).
Thus, the exact sequence
\[ 0 \to \SF \to \SO_{Y}^{\oplus n} \xrightarrow{h^{\vee}} \SO_{Y}^{\oplus m} \to \SG \to 0 \]
satisfies the conditions \eqref{prop:key:ass1} and \eqref{prop:key:ass2}
of Proposition~\ref{prop:key}.
\end{proof}

\begin{lem}\label{lem:SurjFlat(complex)}
Suppose that \( f \colon Y \to T \) is a flat morphism and
\begin{equation}\label{lem:SurjFlat(complex)|eq0}
\depth_{Y_{t} \cap Z} \SO_{Y_{t}} \geq 2
\end{equation}
for any \( t \in f(Z) \).
Moreover, suppose that there is a bounded complex
\[ \SE^{\bullet} = [ \cdots \to \SE^{i} \to \SE^{i+1} \to \cdots] \]
of locally free \( \SO_{Y} \)-modules of finite rank satisfying the following four conditions\emph{:}
\begin{enumerate}
\renewcommand{\theenumi}{\roman{enumi}}
\renewcommand{\labelenumi}{(\theenumi)}
\item \label{lem:SurjFlat(complex):0a}
\( \SH^{i}(\SE^{\bullet})|_{Y \setminus Z} = 0 \) for any \( i > 0\)\emph{;}

\item \label{lem:SurjFlat(complex):0}
\( \SF \isom \SH^{0}(\SE^{\bullet}) \)\emph{;}

\item \label{lem:SurjFlat(complex):1}
\( \SH^{i}(\SE^{\bullet}_{(t)}) = 0 \) for any \( i < 0 \)
and any \( t \in T \), where \( \SE^{\bullet}_{(t)}  \) stands for the complex
\[ [ \cdots \to \SE^{i}_{(t)} \to \SE^{i+1}_{(t)} \to \cdots]
\isom_{\qis} \SE^{\bullet} \otimes^{\bfL}_{\SO_{Y}} \SO_{Y_{t}}; \]

\item \label{lem:SurjFlat(complex):3}
the local cohomology group \( \BHH^{i}_{y}(M^{\bullet}) \) at the maximal ideal \( \GM_{Y, y} \)
for the complex
\[ M^{\bullet} =  \bigl(\tau^{\leq 1} \SE^{\bullet}_{(t)}\bigr)_{y} \]
of \( \SO_{Y, y} \)-modules is zero for any \( i \leq 1 \) and any \( y \in Z \),
where \( t = f(y) \).
\end{enumerate}
Then, \( \SH^{i}(\SE^{\bullet}) = 0 \) for any \( i < 0 \),
and \( \SF \) admits an exact sequence satisfying the conditions \eqref{prop:key:ass1}
and \eqref{prop:key:ass2} of Proposition~\emph{\ref{prop:key}}.
\end{lem}

\begin{proof}
For an integer \( k \), the truncated complex \( \tau^{\geq k}(\SE^{\bullet}) \) is expressed as
\[ [\cdots \to 0 \to \SC^{k} \to \SE^{k+1} \to \SE^{k+2} \to \cdots], \]
where \( \SC^{k} \) is the cokernel of \( \SE^{k-1} \to \SE^{k} \).
First, we shall show that \( \SE^{\bullet} \isom_{\qis} \tau^{\geq 0}(\SE^{\bullet}) \)
and \( \SC^{0} \) is flat over \( T \).
Note that it implies that
\( \SH^{i}(\SE^{\bullet}) = 0 \) for any \( i < 0 \).
Since \( \SE^{\bullet} \) is bounded, we have an integer \( k < 0\)
such that \( \SE^{\bullet} \isom_{\qis} \tau^{\geq k}(\SE^{\bullet}) \)
and \( \SC^{k} \) is flat over \( T \).
Then, by \eqref{lem:SurjFlat(complex):1}, one has
\[ \SH^{k}(\SE^{\bullet}_{(t)}) \isom \Ker(\SC^{k}_{(t)} \to \SE^{k+1}_{(t)}) = 0 \]
for any \( t \in T \).
Hence, \( \SC^{k} \to \SE^{k+1} \) is injective and
\( \SC^{k+1} \isom \SE^{k+1}/\SC^{k} \) is flat over \( T \) by
a version of local criterion of flatness (cf.\ Corollary~\ref{cor:LCflat2}).
Thus, \( \SE^{\bullet} \isom_{\qis} \tau^{\geq k+1}(\SE^{\bullet}) \),
and we can increase \( k \) by one.
Therefore, we can take \( k = 0 \), and consequently,
\( \SE^{\bullet} \isom_{\qis} \tau^{\geq 0}(\SE^{\bullet}) \),
and \( \SC^{0}  \) is flat over \( T \).
We write \( \SC := \SC^{0} \). Then,
\[ \SE^{\bullet}_{(t)} \isom_{\qis}
[\cdots \to 0 \to \SC_{(t)} \to \SE_{(t)}^{1} \to \SE_{(t)}^{2} \to \cdots] \]
for any \( t \in T \), since \( \SC \) and \( \SE^{i} \) are all flat over \( T \).

Second, we shall prove that
\begin{equation}\label{eq1:pf:lem:SurjFlat(complex)}
\depth_{Y_{t} \cap Z}\, \SC_{(t)} \geq 2
\end{equation}
for any \( t \in f(Z) \). We define \( \SK_{t} \) to be the kernel
of \( \SE^{1}_{(t)} \to \SE^{2}_{(t)} \).
Then,
\[ \depth_{Y_{t} \cap Z} \SE^{i}_{(t)} \geq 2 \quad \text{and} \quad
\depth_{Y_{t} \cap Z} \SK_{t} \geq 2 \]
for \( i = 0 \), \( 1 \), and for any \( t \in f(Z) \),
by \eqref{lem:SurjFlat(complex)|eq0} and by Lemma~\ref{lem:depth<=2}\eqref{lem:depth<=2:0}.
In particular, for any \( y \in Z \cap Y_{t}\), we have the vanishing
\begin{equation}\label{eq2:pf:lem:SurjFlat(complex)}
\BHH^{i}_{y}((\SK_{t})_{y}) = 0
\end{equation}
of the local cohomology group at \( y \) for any \( i \leq 1 \)
(cf.\ Property~\ref{ppty:depth<=2}).
By construction, we have a quasi-isomorphism
\[ \tau^{\leq 1}(\SE^{\bullet}_{(t)})
\isom_{\qis} [\cdots \to 0 \to \SC_{(t)} \to \SK_{t} \to 0 \to \cdots].\]
In view of the induced exact sequence
\[ \cdots \to \BHH^{i}_{y}(M^{\bullet}) \to
\BHH^{i}_{y}((\SC_{(t)})_{y}) \to \BHH^{i}_{y}((\SK_{t})_{y}) \to \cdots \]
of local cohomology groups, we have
\[ \BHH^{i}_{y}((\SC_{(t)})_{y}) = 0  \]
for any \( i \leq 1 \) by \eqref{lem:SurjFlat(complex):3} and \eqref{eq2:pf:lem:SurjFlat(complex)}.
Thus, we have \eqref{eq1:pf:lem:SurjFlat(complex)} (cf.\ Property~\ref{ppty:depth<=2}).

Finally, we consider the cokernel \( \SG \) of \( \SC \to \SE^{1} \).
Then, \( \SG|_{U} \) is flat over \( T \) by \eqref{lem:SurjFlat(complex):0a}.
Therefore, the exact sequence
\( 0 \to \SF \to \SC \to \SE^{1} \to \SG \to 0 \) satisfies the conditions
\eqref{prop:key:ass1} and \eqref{prop:key:ass2} of Proposition~\ref{prop:key}.
\end{proof}


\subsection{Applications of the key proposition}
\label{subsect:AppResthom}

First, we shall prove the following criterion for a sheaf to be invertible.

\begin{thm}\label{thm:invExt}
Let \( f \colon Y \to T \) be a morphism of locally Noetherian schemes,
\( Z \) a closed subset of \( Y \), \( \SF \) a coherent \( \SO_{Y} \)-module,
and \( t \) a point of \( f(Z) \).
We set \( U = Y \setminus Z \), and write \( j \colon U \injmap Y\) for the open immersion.
Assume that\emph{:}
\begin{enumerate}
\renewcommand{\theenumi}{\roman{enumi}}
\renewcommand{\labelenumi}{(\theenumi)}
\item \label{thm:invExt:cond0}
\( \depth_{Z} \SO_{Y} \geq 1 \),
\item \label{thm:invExt:cond1}
\( \SF|_{U} \) is flat over \( T \), \( \SF|_{U} \) is invertible, \( \depth_{Z} \SF \geq 2 \), and
\item \label{thm:invExt:cond2}
the direct image sheaf
\[ \SF_{(t)*} = j_{*}((\SF \otimes_{\SO_{Y}} \SO_{Y_{t}})|_{U \cap Y_{t}}) \]
\emph{(}cf.\ Definition~\emph{\ref{dfn:restmor})} is an invertible \( \SO_{Y_{t}} \)-module.
\end{enumerate}
Assume furthermore that one of the conditions \eqref{thm:invExt:3a} and \eqref{thm:invExt:3b}
below is satisfied\emph{:}
\begin{enumerate}
\renewcommand{\theenumi}{\alph{enumi}}
\renewcommand{\labelenumi}{(\theenumi)}
\item \label{thm:invExt:3a} \( \depth_{Z \cap Y_{t}} \SO_{Y_{t}} \geq 3 \)\emph{;}

\item  \label{thm:invExt:3b}
the double-dual \( \SF^{[r]}  \) of \( \SF^{\otimes r} \) is invertible along \( Y_{t} \)
for a positive integer \( r \) coprime to the characteristic of the residue field \( \Bbbk(t) \).
\end{enumerate}
Then, \( f \) is flat along \( Y_{t} \), and
\( \SF \) is invertible along \( Y_{t} \).
\end{thm}

\begin{proof}
We may replace \( Y \) with its open subset, since the assertions are local on \( Y \).
By \eqref{thm:invExt:cond1}, \( U \) is flat over \( T \).
Moreover,
\begin{equation}\label{thm:invExt|eq1}
\depth_{Z \cap Y_{t}} \SO_{Y_{t}} \geq 2
\end{equation}
by \eqref{thm:invExt:cond2}, since the isomorphism 
\( \SF_{(t)*} \isom j_{*}(\SF_{(t)*}|_{U \cap Y_{t}}) \)
implies that
\( \depth (\SF_{(t)*})_{y} = \depth \SO_{Y_{t}, y} \geq 2 \) for any \( y \in Z \cap Y_{t}  \).
Hence, \( f \colon Y \to T \) is flat along \( Y_{t} \) by Corollary~\ref{cor:lem:S2flat(new)}
(cf.\ Property~\ref{ppty:depth<=2}).
Then, \( \SF \) is a reflexive \( \SO_{Y} \)-module
by Lemma~\ref{lem:j*reflexive}\eqref{lem:j*reflexive:2}, since we have assumed
\( \depth_{Z} \SO_{Y} \geq 1 \) and \( \depth_{Z} \SF \geq 2 \)
in \eqref{thm:invExt:cond0} and \eqref{thm:invExt:cond1}.
Therefore, by \eqref{thm:invExt|eq1} and Lemma~\ref{lem:SurjFlat(reflexive)},
we may assume that
\( \SF \) admits an exact sequence of Proposition~\ref{prop:key}.

By Fact~\ref{fact:elem-flat}\eqref{fact:elem-flat:2},
we see that \( \SF \) is invertible along \( Y_{t} \)
if the two conditions below are both satisfied:
\begin{enumerate}
\item \label{thm:invExt:01} \( \SF \) is flat over \( T \) along \( Y_{t} \);
\item \label{thm:invExt:02} \( \phi_{t} \colon \SF_{(t)} \to \SF_{(t)*} \) is an isomorphism.
\end{enumerate}
Here, \eqref{thm:invExt:01} is a consequence of \eqref{thm:invExt:02}
by Proposition~\ref{prop:key}\eqref{prop:key:2}.
When \eqref{thm:invExt:3a} holds, we have
\[ \depth_{Y_{t} \cap Z} \SF_{(t)*} \geq 3 \]
by \eqref{thm:invExt:cond2}, and hence,
the condition  \eqref{thm:invExt:02} is satisfied by Corollary~\ref{cor:SurjFlat(new)|Kollar}.
Thus, it remains to prove \eqref{thm:invExt:02} assuming the condition \eqref{thm:invExt:3b}.

We use Notation~\ref{nota:Resthom2} for \( t \). By replacing \( Y \) with its open subset,
we may assume that \( Y \) is affine, and  there exist isomorphisms
\[ \SO_{Y_{t}} = \SO_{Y_{0}} \isom \SF_{(t)*} = \SF_{0*}
\quad \text{and} \quad \SF^{[r]} \isom \SO_{Y}\]
in \eqref{thm:invExt:cond2} and in \eqref{thm:invExt:3b}, respectively.
Note that we have
\[ \depth_{Y_{n} \cap Z} \SO_{Y_{n}} \geq 2 \]
for any \( n \geq 0 \):
This follows from \eqref{thm:invExt|eq1} by Lemma~\ref{lem:relSkCodimDepth}\eqref{lem:relSkCodimDepth:3}
applied to the flat morphism \( Y_{n} \to T_{n} \).
As a consequence,
\[ \OH^{0}(Y_{n}, \SO_{Y_{n}}) \isom \OH^{0}(U_{n}, \SO_{U_{n}}) \]
and for any \( n \), and the restriction homomorphism
\begin{equation}\label{thm:invExt|eq2}
\OH^{0}(U_{n}, \SO_{U_{n}}) \to \OH^{0}(U_{n-1}, \SO_{U_{n-1}})
\end{equation}
is surjective for any \( n > 0 \), since we have assumed that \( Y \) is affine.

We set \( \SN_{n} := \SF_{n}|_{U_{n}} \).
It is enough to show that \( \SN_{n} \isom \SO_{U_{n}} \) for all \( n \).
In fact, if this is true, then we have an isomorphism
\[ \SF_{n*} = j_{*}(\SN_{n}) \isom j_{*}(\SO_{U_{n}}) \isom \SO_{Y_{n}}, \]
and, as a consequence, the restriction homomorphism
\( \varphi_{n} \colon (\SF_{n*})_{(t)} \to \SF_{0*}\) is an isomorphism for any \( n \geq 0 \).
Hence, in this case, \( \phi_{t} \colon \SF_{(t)} \to \SF_{(t)*}\) is an isomorphism
by \eqref{prop:key:condBB} \( \Rightarrow \) \eqref{prop:key:condB}
of Proposition~\ref{prop:key}\eqref{prop:key:2}.

We shall prove \( \SN_{n} \isom \SO_{U_{n}} \) by induction on \( n \).
When \( n = 0 \), we have the isomorphism from
the isomorphism \( \SF_{0*} \isom \SO_{Y_{0}} \) above.
Assume that \( \SN_{n-1} \isom \SO_{U_{n-1}} \) for an integer \( n > 0\).
Let \( \SJ \) be the kernel of \( \SO_{Y_{n}} \to \SO_{Y_{n-1}} \).
Then, \( \SJ^{2} = 0 \) as an ideal of \( \SO_{Y_{n}} \), and
\[ \SJ \isom \GM^{n}/\GM^{n+1} \otimes_{\Bbbk} \SO_{Y_{0}}. \]
We have an exact sequence
\[ 0 \to \SJ \to \SO_{Y_{n}}^{\star} \to \SO_{Y_{n-1}}^{\star} \to 1 \]
of sheaves on \( |Y_{n}| = |Y_{0}|\) with respect to the Zariski topology,
where \( {}^{\star} \) stands for the subsheaf of invertible sections of a sheaf of rings,
and where a local section \( \zeta \) of \( \SJ \) is mapped to the invertible section
\( 1 + \zeta \) of \( \SO_{Y_{n}} \).
It induces a long exact sequence:
\[ \OH^{0}(U_{n}, \SO_{U_{n}}^{\star}) \xrightarrow{\text{res}^{0}}
\OH^{0}(U_{n-1}, \SO_{U_{n-1}}^{\star}) \\
\to \OH^{1}(U_{0}, \SJ) \to \Pic(U_{n}) \xrightarrow{\text{res}^{1}}  \Pic(U_{n-1}),\]
where \( \text{res}^{0} \) and \( \text{res}^{1} \) are restriction homomorphisms to \( U_{n-1} \).
Note that \( \text{res}^{0} \) is surjective, since so is \eqref{thm:invExt|eq2}.
Hence, the kernel of \( \text{res}^{1} \) is a \( \Bbbk \)-vector space isomorphic
to \( \OH^{1}(U_{0}, \SJ) \).
Now, the isomorphism class of \( \SN_{n} \) in \( \Pic(U_{n}) \) belongs to the kernel
by \( \SN_{n-1} \isom \SO_{U_{n-1}} \),
and its multiple by \( r \) is zero by \eqref{thm:invExt:3b},
where \( r \) is coprime to \( \chara(\Bbbk) \).
Thus, \( \SN_{n} \isom \SO_{U_{n}} \), and we are done.
\end{proof}

We have the following by a direct application of Proposition~\ref{prop:key}.

\begin{lem}\label{lem:1|prop:key}
Let \( f \colon Y \to T \) be a flat morphism of locally Noetherian schemes and
let \( Z \) be a closed subset of \( Y \) such that
\[ \depth_{Y_{t} \cap Z} \SO_{Y_{t}} \geq 2 \]
for any fiber \( Y_{t} \). 
Let \( q \colon T' \to T\) be a morphism from another locally Noetherian scheme \( T'\) 
such that \( Y' = Y \times_{T} T' \) is also locally Noetherian. 
We write \( f' \colon Y' \to T' \) and \( p \colon Y' \to Y \) for the projections. 
Let
\( 0 \to \SF \to \SE^{0} \to \SE^{1} \to \SG \to 0 \)
be an exact sequence of coherent \( \SO_{Y} \)-modules such that \( \SF|_{U} \),
\( \SE^{0} \), \( \SE^{1} \), and \( \SG|_{U} \) are locally free, where \( U = Y \setminus Z \).
Then, \( \SF \) is a reflexive \( \SO_{Y} \)-module, and
\[ (p^{*}\SF)^{\vee\vee} \isom \Ker( p^{*}\SE^{0} \to p^{*}\SE^{1} )
\isom j'_{*}(p^{*}\SF|_{U'})\]
for the open immersion \( j' \colon U' = p^{-1}(U) \injmap Y' \).
Moreover, \( (p^{*}\SF)^{\vee\vee} \) satisfies relative \( \bfS_{2} \) over \( T' \)
if and only if \( p^{*}\SG \) is flat over \( T' \). 
\end{lem}

\begin{proof}
The exact sequence satisfies the assumptions of Proposition~\ref{prop:key}
for \( Y \to T \). Hence, \( \SF \isom j_{*}(\SF|_{U}) \), i.e., \( \depth_{Z} \SF \geq 2 \),
by Proposition~\ref{prop:key}\eqref{prop:key:1}.
Moreover, \( \SF \) is reflexive by Lemma~\ref{lem:j*reflexive}\eqref{lem:j*reflexive:2},
since we have \( \depth_{Z} \SO_{Y} \geq 2  \)
by Lemma~\ref{lem:relSkCodimDepth}\eqref{lem:relSkCodimDepth:3}.
Let \( \SF' \) be the kernel of \( p^{*}\SE^{0} \to p^{*}\SE^{1} \).
Then, the exact sequence
\[ 0 \to \SF' \to p^{*}\SE^{0} \to p^{*}\SE^{1} \to p^{*}\SG \to 0 \]
on \( Y' \) satisfies the assumptions of Proposition~\ref{prop:key}  for \( f' \colon Y' \to T' \),
since
\[ \depth_{Y'_{t'} \cap p^{-1}(Z)} \SO_{Y'_{t'}} = \depth_{Y_{t} \cap Z} \SO_{Y_{t}} \geq 2   \]
for any \( t' \in T' \) and \( t = q(t') \),
by Lemma~\ref{lem:bc basic}\eqref{lem:bc basic:1}.
Hence, \( \SF' \isom j'_{*}(\SF'|_{U'}) \)
by Proposition~\ref{prop:key}\eqref{prop:key:1}.
Since \( \SF'|_{U'} \isom p^{*}\SF|_{U'} \),
we have \( \SF' \isom (p^{*}\SF)^{\vee\vee}\) by Lemma~\ref{lem:bc reflexive}.
Furthermore, by Proposition~\ref{prop:key}\eqref{prop:key:2},  we see
that \( \SF' \) satisfies relative \( \bfS_{2} \) over \( T' \) 
if and only if \( p^{*}\SG \) is flat over \( T' \).
\end{proof}

Here, we introduce the following notion 
useful to state results in the rest of Section~\ref{subsect:AppResthom}. 

\begin{dfn}\label{dfn:locFreeCodim1Fiber}
Let \( f \colon Y \to T \) be a morphism of locally Noetherian schemes and let \( \SF \) be a coherent \( \SO_{Y} \)-module. 
We say that \( \SF \) is \emph{locally free in codimension one on each fiber} of \( f \) if 
there is an open subset \( U \subset Y \) such that 
\( \SF|_{U} \) is locally free and \( \Codim(Y_{t} \setminus U, Y_{t}) \geq 2 \) 
for any fiber \( Y_{t} = f^{-1}(t)\). 
\end{dfn}

\begin{prop}[infinitesimal and valuative criteria]\label{prop:inf+val}
Let \( f \colon Y \to T \) be a flat morphism locally of finite type 
between locally Noetherian schemes and 
let \( y \in Y\) be a point such that \( f \) satisfies relative \( \bfS_{2} \) over \( T \) 
at \( y \). 
Let \( \SF \) be a reflexive \( \SO_{Y} \)-module 
which is locally free in codimension one on each fiber of \( f \). 
Then, \( \SF \) satisfies relative \( \bfS_{2} \) over \( T \) at \( y \) 
if one of the following two conditions \eqref{prop:inf+val:inf} and \eqref{prop:inf+val:val} 
is satisfied, 
where \( Y_{A} = Y \times_{T} \Spec A \) and \( \SF_{A} =  p^{*}_{A}\SF \) 
for the projection \( p_{A} \colon Y_{A} \to Y \)\emph{:}
\begin{enumerate}
\renewcommand{\theenumi}{\Roman{enumi}}
\renewcommand{\labelenumi}{(\theenumi)}
\item  \label{prop:inf+val:inf} 
Let \( \Spec A \to T \) be a morphism defined by a surjective local ring homomorphism 
\( \SO_{T, f(y)} \to A \) to an Artinian local ring \( A \). 
Then, the double-dual \( (\SF_{A})^{\vee\vee} \) 
satisfies relative \( \bfS_{2} \) over \( \Spec A \) at the point \( y_{A} = p_{A}^{-1}(y)\). 

\item \label{prop:inf+val:val} 
The local ring \( \SO_{T, f(y)} \) is reduced. Let \( \Spec A \to T \) be a morphism 
defined by a local ring homomorphism \( \SO_{T, f(y)} \to A \) to a discrete valuation ring \( A \). 
Then, the double-dual 
\( (\SF_{A})^{\vee\vee} \) satisfies relative \( \bfS_{2} \) over \( \Spec A \) 
at any point \( z \in Y_{A}\) lying over \( y \in Y\) 
and the closed point \( \GM_{A} \) of \( \Spec A \). 
\end{enumerate}
\end{prop}

\begin{proof}
We may assume that \( T = \Spec B \) for the local ring \( B = \SO_{T, f(y)} \) and 
we can localize \( Y \) freely. Thus, we may assume that 
\( Y = \Spec C \) for a finitely generated \( B \)-algebra \( C \) and 
moreover that \( Y \to T \) is an \( \bfS_{2} \)-morphism by 
Fact~\ref{fact:dfn:RelSkCMlocus}\eqref{fact:dfn:RelSkCMlocus:3}. 
By assumption, there is a closed subset \( Z \subset Y \) such that 
\( \SF|_{Y \setminus Z} \) is locally free and 
\( \Codim(Y_{t} \cap Z, Y_{t}) \geq 2 \) for any \( t \in T \). 
In particular, 
\[ \depth_{Y_{t} \cap Z} \SO_{Y_{t}} \geq 2 \]
for any \( t \in T \) (cf.\ Lemma~\ref{lem:depth+codim+Sk}\eqref{lem:depth+codim+Sk:2}). 
As in the proof of Lemma~\ref{lem:SurjFlat(reflexive)}, 
we have an exact sequence \( 0 \to \SF \to \SE^{0} \to \SE^{1} \to \SG \to 0 \) 
of coherent \( \SO_{Y} \)-modules 
such that \( \SE^{0} \), \( \SE^{1} \), and \( \SG|_{U} \) are locally free. 
By Proposition~\ref{prop:key}\eqref{prop:key:2}, 
we see that 
\( \SF \) satisfies relative \( \bfS_{2} \) over \( T \) at \( y \) 
if and only if the stalk \( \SG_{y} \) is flat over \( B\).  
Let \( \Spec A \to T \) be a morphism in \eqref{prop:inf+val:inf} or \eqref{prop:inf+val:val}. 
Then, \( Y_{A} = \Spec C \otimes_{B} A\) is Noetherian, 
since \( C \otimes_{B} A \) is a finitely generated \( A \)-algebra. 
By Proposition~\ref{prop:key}\eqref{prop:key:2} and 
Lemma~\ref{lem:1|prop:key}, we see also that  
\( (\SF_{A})^{\vee\vee} \) satisfies relative \( \bfS_{2} \) over \( \Spec A \) 
at a point \( z \) lying over \( y \) and \( \GM_{A} \) 
if and only if the stalk \( \SG_{A, z} \) is flat over \( A \) for 
the pullback \( \SG_{A} = p_{A}^{*}\SG\), where \( \SG_{A, z} \isom (\SG_{y} \otimes_{B} A)_{z} \).  
Therefore, the assertions in the cases \eqref{prop:inf+val:inf} and \eqref{prop:inf+val:val}, respectively,   
follow from the local criterion of flatness (cf.\ Proposition~\ref{prop:LCflat}\eqref{prop:LCflat:4}) 
for \( \SG_{y} \) over \( B \) and 
from the valuative criterion of flatness (cf.\ \cite[IV, Th.\ (11.8.1)]{EGA}) 
for \( \SG \) over \( T \) at \( y \).  
\end{proof}

\begin{dfn}[relative $\bfS_{2}$ refinement]\label{dfn:relS2refinement}
Let \( Y \to T \) be an \( \bfS_{2} \)-morphism of 
locally Noetherian schemes and let \( \SF \) be a reflexive \( \SO_{Y} \)-module 
which is locally free in codimension one on each fiber. 
A morphism \( S \to T \) from a locally Noetherian scheme \( S \) is called 
a \emph{relative} \( \bfS_{2} \) \emph{refinement} for \( \SF \) over \( T \) 
if the following conditions are satisfied: 
\begin{enumerate}
\renewcommand{\theenumi}{\roman{enumi}}
\renewcommand{\labelenumi}{(\theenumi)}
\item  \label{dfn:relS2refinement:cond1} 
\( S \to T \) is a monomorphism in the category of schemes (cf.\ Fact~\ref{fact:mono}); 

\item \label{dfn:relS2refinement:cond2} 
for any morphism \( T' \to T \) of locally Noetherian schemes, and for the pullback 
\( \SF' \) of \( \SF \) to the fiber product \( Y \times_{T} T' \), the double dual \( (\SF')^{\vee\vee} \) 
satisfies relative \( \bfS_{2} \) over \( T' \) if and only if \( T' \to T \) factors through \( S \to T\). 
\end{enumerate}
\end{dfn}

\begin{remn}
The fiber product \( Y \times_{T} T' \) in \eqref{dfn:relS2refinement:cond2} is locally Noetherian, 
since it is locally of finite type over \( T' \). 
Thus, we can consider the relative \( \bfS_{2} \)-condition for \( (\SF')^{\vee\vee} \) 
(cf.\ Definition~\ref{dfn:RelSkCMlocus}). 
By \eqref{dfn:relS2refinement:cond1} and \eqref{dfn:relS2refinement:cond2}, \( S \to T \) 
is unique up to unique isomorphism. 
\end{remn}

\begin{rem}\label{rem:dfn:relS2refinement}
In the situation of Definition~\ref{dfn:relS2refinement}\eqref{dfn:relS2refinement:cond2}, 
we write 
\( \SF \times_{T} T' \) for \( \SF' \), and we set 
\[ F(T'/T) = \begin{cases}
\star, &\text{ if } (\SF \times_{T} T')^{\vee\vee} \text{ satisfies relative } \bfS_{2} \text{ over } T', \\
\emptyset, &\text{ otherwise},
\end{cases}\]
where \( \star \) denotes a one-point set. 
For any morphism \( T'' \to T' \) from a locally Noetherian scheme \( T'' \), we can show that 
if \( F(T'/T) = \star \), then \( F(T''/T) = \star \). In fact, we have an isomorphism 
\[ (\SF \times_{T} T'')^{\vee\vee} \isom (\SF \times_{T} T')^{\vee\vee} \times_{T'} T'' \]
by Lemma~\ref{lem:bc reflexive}, and this sheaf satisfies relative \( \bfS_{2} \) over \( T'' \) by
Lemma~\ref{lem:bc basic}\eqref{lem:bc basic:3}. 
Therefore, \( F \) is regarded as a functor \( (\LNSch/T)^{\op} \to \Set \)
for the category \( \LNSch/T \) of locally Noetherian \( T \)-schemes, and 
the relative \( \bfS_{2} \) refinement is a \( T \)-scheme representing \( F \). 
\end{rem}

\begin{rem}
If \( T' = \Spec \Bbbk \) for a field \( \Bbbk \), then \( F(T'/T) = \star \), 
since \( (\SF')^{\vee\vee} \) satisfies \( \bfS_{2} \) by Corollary~\ref{cor:prop:S1S2:reflexive}. 
In particular, 
if the relative \( \bfS_{2} \) refinement \( S \to T \) exists, then it is bijective. 
\end{rem}

\begin{fact}\label{fact:mono}
Let \( h \colon S \to T \) be a morphism
locally of finite type between locally Noetherian schemes.
Then, \( h \) is a morphism \emph{locally of finite presentation} (cf.\ \cite[IV, \S1.4]{EGA}),
and we have the following properties:

\begin{enumerate}
\item  \label{fact:mono:1} 
The morphism \( h \) is a monomorphism in the category of schemes
if and only if \( h \) is radicial and unramified, by \cite[IV, Prop.\ (17.2.6)]{EGA}.

\item  If \( h \) is an unramified morphism, then it is \'etale locally a closed immersion, i.e.,
for any point \( s \in S \), there exists an open neighborhood \( V \) of \( s \) such that
the induced morphism \( V \to T \) is written as
the composite of a closed immersion \( V \to W \) and an \'etale morphism \( W \to T \)
(cf.\ \cite[IV, Cor.\ (18.4.7)]{EGA}, \cite[I, Cor.\ 7.8]{SGA1}).
\end{enumerate} 
\end{fact}

\begin{exam}
For a Noetherian scheme \( T \) and a finite number of locally closed subschemes 
\( S_{1} \), \( S_{2} \), \ldots, \( S_{k} \) of \( T \), 
assume that \( T \) is equal to the disjoint union \( \bigsqcup_{i = 1}^{k} S_{i} \) as a set;  
The collection \( \{S_{i}\} \) is called a \emph{stratification} in \cite[Lect.\ 8]{Mumford}.
Then, immersions \( S_{i} \subset T \) define a morphism \( h \colon S \to T \) from 
the scheme-theoretic disjoint union \( S = \bigsqcup_{i = 1}^{k} S_{i} \).
This \( h \) is a separated surjective monomorphism of finite type 
and is a \emph{local immersion} (cf.\ \cite[I, D\'ef.\ (4.5.1)]{EGA}), i.e., 
a closed immersion Zariski-locally. 
\end{exam}

\begin{exam} 
There is a separated monomorphism \( h \colon X \to Y \) of finite type 
of Noetherian schemes such that \( X \) is connected but \( h \) is not an immersion. 
An example is given as follows: 
For an algebraically closed field \( \Bbbk \), let \( D \) be a reduced effective divisor of degree three 
on the projective plane \( Y = \BPP^{2}_{\Bbbk} \) having a node \( P \). 
For the blowing up \( M \to Y \) at \( P \), let \( \overline{X} \) be the proper transform 
of \( D \) in \( M \) and let \( Q \in \overline{X} \) be one of the two points lying over \( P \). 
We set \( X := \overline{X} \setminus \{Q\} \). Then, \( X \) is connected, 
the induced morphism \( h \colon X \to  Y \) is a separated monomorphism of finite type inducing 
a bijection \( X \to D \), and \( h^{-1}(Y \setminus \{P\}) \isom D \setminus \{P\}\). 
However, \( h \) is not a closed immersion, since \( X \) is not isomorphic to \( D \).  
If \( D \) is irreducible, then \( h \) is not a local immersion. 
On the other hand, if \( D \) is reducible, then \( h \) is a local immersion. 
In fact, for another node \( P' \) of \( D \), the inverse image 
\( h^{-1}(Y \setminus \{P'\}) \) is isomorphic to a disjoint union of 
two locally closed subschemes of \( Y \). 
\end{exam}

The following is analogous to the flattening stratification theorem by Mumford in \cite[Lect.\ 8]{Mumford} 
or to the representability theorem of unramified functors by Murre \cite{Murre}: 
A similar result is stated by Koll\'ar in \cite[Th.~2]{KollarHusk} 
in the case where \( f\) is projective, but where the \( \bfS_{2} \)-condition for \( f \), etc.,   
are not assumed.

\begin{thm}\label{thm:dd dec}
Let \( f \colon Y \to T \) be an \( \bfS_{2} \)-morphism of locally Noetherian schemes.
Let \( \SF \) be a reflexive \( \SO_{Y} \)-module 
which is locally free in codimension one 
on each fiber \emph{(}cf.\ Definition~\emph{\ref{dfn:locFreeCodim1Fiber})}.
If the following condition \eqref{thm:dd dec:ass2} is satisfied, then 
there is a relative \( \bfS_{2} \) refinement for \( \SF \) over \( T \) 
as a separated morphism \( S \to T \) locally of finite type\emph{:}
\begin{enumerate}
\renewcommand{\theenumi}{\roman{enumi}}
\renewcommand{\labelenumi}{(\theenumi)}
\item  \label{thm:dd dec:ass2}
\( \SF|_{Y \setminus \Sigma} \) satisfies relative \( \bfS_{2} \) over \( T \)
for a closed subset \( \Sigma \subset Y\)
such that \( \Sigma \to T \) is proper.
\end{enumerate}
Furthermore, the morphism \( S \to T \) is a local immersion of finite type if 
\begin{enumerate}
\renewcommand{\theenumi}{\roman{enumi}}
\renewcommand{\labelenumi}{(\theenumi)}
\addtocounter{enumi}{1}
\item  \label{thm:dd dec:ass1} \( f \) is a projective morphism locally over \( T \).
\end{enumerate}
\end{thm}

\begin{proof}
For the first assertion, by Fact~\ref{fact:mono}\eqref{fact:mono:1}, 
it is enough to prove that the functor \( F \) in Remark~\ref{rem:dfn:relS2refinement} 
is representable by a separated morphism \( S \to T \) locally of finite type.  
We may replace \( T \) freely by an open subset, 
since \( S \to T\) is unique up to unique isomorphism and 
since the second assertion is also local on \( T \). Thus, 
we assume that \( T \) is an affine Noetherian scheme. 
We set \( U \) to be an open subset of \( Y \) such that \( \SF|_{U} \) is locally free and 
\( \Codim(Y_{t} \setminus U, Y_{t}) \geq 2 \) for any fiber \( Y_{t} \). 

We first consider the case \eqref{thm:dd dec:ass1}: 
We may assume that 
\( Y \) is a closed subscheme of \( \BPP^{N} \times T \) for some \( N > 0 \).
Let \( \SA \) be the \( f \)-ample invertible \( \SO_{Y} \)-module defined as the pullback 
of \( \SO(1) \) on \( \BPP^{N} \).
Then, we can construct an exact sequence
\[ (\SA^{\otimes -l'})^{\oplus m'} \to (\SA^{\otimes -l})^{\oplus m} \to \SF^{\vee} \to 0  \]
on \( Y \) for positive some integers \( m \), \( m' \), \( l \), and \( l' \),
where the kernel of the left homomorphism is
locally free on \( U \), since \( \SF^{\vee} \) is so.
Taking the dual, we have an exact sequence
\( 0 \to \SF \to \SE^{0} \to \SE^{1} \to \SG \to 0 \)
of coherent \( \SO_{Y} \)-modules
such that \( \SE^{0} \), \( \SE^{1} \), and \( \SG|_{U} \) are locally free
(cf.\ the proof of Lemma~\ref{lem:SurjFlat(reflexive)}). 
Let \( T' \to T \) be an arbitrary morphism from another locally Noetherian scheme \( T'\).
Then, \( F(T'/T) = \star \)
if and only if \( \SG \times_{T} T'\) is flat over \( T' \), by Lemma~\ref{lem:1|prop:key}.
Hence, the functor \( F \) is nothing but the ``universal flattening functor''
\( G \colon (\Sch/T)^{\op} \to \Set \)
for \( \SG \) (cf.\ Remark~\ref{rem:univFlatFct} below) restricted to the category \( \LNSch/T \).
Here,
\[ G(T'/T) = \begin{cases}
\star, &\text{if } \SG \times_{T} T' \text{ is flat over }  T', \\
\emptyset, &\text{otherwise,}
\end{cases}\]
for any \( T \)-scheme \( T' \).
By the Theorem of \cite[Lect.\ 8]{Mumford},
it is represented by a separated morphism \( S \to T \) of finite type
which is a local immersion.
Thus, we have proved the assertion in the case \eqref{thm:dd dec:ass1}.

In the case \eqref{thm:dd dec:ass2}, 
we can cover \( \Sigma \) by finitely many open affine subsets \( Y_{\lambda} \). 
We may assume that \( Y = \bigcup Y_{\lambda} \), since \( \SF|_{Y \setminus \Sigma} \) 
satisfies relative \( \bfS_{2} \) over \( T \). 
By Lemma~\ref{lem:SurjFlat(reflexive)}, 
we may also assume that 
there exists an exact sequence
\[ 0 \to \SF|_{Y_{\lambda}} \to \SE^{0}_{\lambda}
\to \SE^{1}_{\lambda} \to \SG_{\lambda} \to 0 \]
on each \( Y_{\lambda} \) such that
\( \SE^{0}_{\lambda} \) and \( \SE^{1}_{\lambda} \)
are free \( \SO_{Y_{\lambda}} \)-modules of finite rank,
and that \( \SG_{\lambda} \) is locally free on \( U_{\lambda} = U \cap Y_{\lambda}\).
Let \( T' \to T \) be an arbitrary morphism from a locally Noetherian scheme \( T' \). 
By Lemma~\ref{lem:1|prop:key}, we see that \( F(T'/T) = \star\) if and only if
\( \SG_{\lambda} \times_{T} T' \) is flat over \( T' \) for any \( \lambda \).
Let \( G_{\lambda} \colon (\Sch/T)^{\op} \to \Set \) be
the universal flattening functor for \( \SG_{\lambda} \), which is defined by
\[ G_{\lambda}(T'/T) = \begin{cases}
\star, &\text{if } \SG_{\lambda} \times_{T} T' \text{ is flat over } T'; \\
\emptyset, &\text{otherwise.}
\end{cases}\]
Let \( G \colon (\Sch/T)^{\op} \to \Set \) be the ``intersection'' functor of all \( G_{\lambda} \), i.e.,
\( G(T'/T) = \bigcap G_{\lambda}(T'/T) \) for any \( T/'T \).
By the argument above,
\( F \) is the restriction of \( G \) to \( \LNSch/T \). 
Every functor \( G_{\lambda} \) satisfies the conditions (\(\mathrm{F}_{1}\))--(\(\mathrm{F}_{8}\))
of \cite{Murre} except (\(\mathrm{F}_{3}\)), by the proof of \cite[Th.~2]{Murre}. 
Hence, the intersection functor \( G \) satisfies the same conditions except possibly
(\(\mathrm{F}_{3}\)) and (\(\mathrm{F}_{8}\)).
By \cite[Th.~1]{Murre}, we are reduced to check these two conditions for \( G \).
Since the two conditions concern only Noetherian schemes, we may take \( F = G \). 
We write \( F(T') = F(T'/T) \) for simplicity for a morphism \( T' \to T \).

We shall show that \( F \) satisfies (\(\mathrm{F}_{3}\)) (cf.\ \cite[($\mathrm{F}_{3}$), p.~244]{Murre}).
Let \( A \) be a Noetherian complete local ring
with maximal ideal \( \GM_{A} \) and let \( \Spec A \to T \) be a morphism. What we have to prove is
the bijectivity of the canonical map
\[ F(\Spec A) \to \varprojlim\nolimits_{n} F(\Spec A/\GM_{A}^{n}), \]
or equivalently that
\( F(\Spec A) = \star \) if  \( F(\Spec A/\GM^{n}_{A}) = \star \) for all \( n > 0\).
Assume the latter condition.
By Corollary~\ref{cor0:prop:key}
applied to \( Y_{\lambda} \times_{T} \Spec A \to \Spec A  \)
for each \( \lambda \),
we have an open neighborhood \( W_{\lambda} \) of the closed fiber \( Y_{\lambda} \times_{T} \Spec A/\GM_{A} \)
in \( Y_{\lambda} \times_{T} \Spec A \) such that \( (\SF \times_{T} \Spec A)^{\vee\vee}|_{W_{\lambda}} \)
satisfies relative \( \bfS_{2} \) over \( \Spec A \).
On the other hand, the restriction of \( (\SF \times_{T} \Spec A)^{\vee\vee} \)
to \( (Y \setminus \Sigma) \times_{T} \Spec A \)
also satisfies relative \( \bfS_{2} \) over \( \Spec A \). Then, the union
\( \bigcup W_{\lambda} \cup  ((Y \setminus \Sigma) \times_{T} \Spec A) \)
equals \( Y \times_{T} \Spec A \), since the complement of the union 
is proper over \( \Spec A \) but its image does not contain the closed point \( \GM_{A} \).
Therefore, \( F(\Spec A) = \star \).

Next, we shall show that \( F \) satisfies (\(\mathrm{F}_{8}\))
(cf.\ \cite[($\mathrm{F}_{8}$), p.~246]{Murre}). Let \( A \) be a Noetherian ring
containing a unique minimal prime ideal \( \Gp \) and let \( I \) be a nilpotent ideal of \( A \) such that
\( I\Gp = 0 \). Note that \( \Gp = \sqrt{0} \).
Let \( \Spec A \to T \) be a morphism and assume that
\( F(\Spec A/I) = \star \) but \( F(\Spec A_{\Gp}/I') = \emptyset \)
for any ideal \( I' \) of \( A_{\Gp} \) such that \( I' \subsetneq I_{\Gp} \).
What we have to prove is the existence of an element \( a \in A \setminus \Gp \)
having the following property:
\begin{enumerate}
\renewcommand{\theenumi}{$\diamond$}
\renewcommand{\labelenumi}{(\theenumi)}
\item \label{thm:dd dec:condF8}
For any element \( b \in A \setminus \Gp \) and for any ideal \( J \) of \( A_{ab} = A[(ab)^{-1}] \),
if \( J \subset IA_{ab} \) and if \( F(\Spec A_{ab}/J) = \star \),
then \( J = IA_{ab} \).
\end{enumerate}
For each \( \lambda \), we set \( B_{\lambda} \) to be an \( A \)-algebra such that
\( \Spec B_{\lambda} \isom Y_{\lambda} \times_{T} \Spec A \) over \( \Spec A \) and
let \( M_{\lambda} \) be a finitely generated \( B_{\lambda} \)-module such that the quasi-coherent sheaf
\( M_{\lambda}\sptilde \) on \( \Spec B_{\lambda} \)
is isomorphic to \( \SG_{\lambda} \times_{T} \Spec A  \). 
Note that
\[ M_{\lambda} \otimes_{A} A_{\Gp}/IA_{\Gp} \]
is a free \( A_{\Gp}/IA_{\Gp} \)-module, since it 
is flat over \( A_{\Gp}/IA_{\Gp} \) by \( G_{\lambda}(\Spec A_{\Gp}/IA_{\Gp}) = \star \)
and since \( A_{\Gp} \) is an Artinian local ring.
Hence,
\[ (M_{\lambda} \otimes_{A} A/I) \otimes_{A} A_{a} = M_{\lambda} \otimes_{A} A_{a}/IA_{a} \]
is a free \( A_{a}/IA_{a} \)-module for an element \( a \in A \setminus \Gp\).
For each \( \lambda \), let \( \ScS_{\lambda} \) be the set of ideals \( J \) of \( A_{a} \)
such that \( G_{\lambda}(\Spec A_{a}/J) = \star \), or equivalently, that
\( M_{\lambda} \otimes_{A} A_{a}/J \) is a flat \( A_{a}/J \)-module.
By \cite[IV, Cor.\ (11.4.4)]{EGA}, there exists
a unique minimal element \( I_{\lambda} = I_{\lambda, (a)} \) in \( \ScS_{\lambda} \), and
\begin{enumerate}
\renewcommand{\theenumi}{$\dag$}
\renewcommand{\labelenumi}{(\theenumi)}
\item  \label{thm:dd dec:I_lambda}
for any \( A_{a} \)-algebra \( A' \),
if \( M_{\lambda} \otimes_{A} A'\) is a flat \( A' \)-module, then
\( A' \) is an \( A_{a}/I_{\lambda} \)-algebra.
\end{enumerate}
Note that \( I_{\lambda} \) is nilpotent, since
the nilpotent ideal \( IA_{a} \) belongs to \( \ScS_{\lambda} \).
We define \( I_{(a)} := \sum I_{\lambda, (a)} \) as an ideal of \( A_{a} \). Then,
it has the following property:
\begin{enumerate}
\renewcommand{\theenumi}{$\ddag$}
\renewcommand{\labelenumi}{(\theenumi)}
\item  \label{thm:dd dec:I(a)}
For any \( A_{a} \)-algebra \( A' \), it is an \( A_{a}/I_{(a)} \)-algebra if and only if
\( M_{\lambda} \otimes_{A} A' \) is flat over \( A' \) for any \( \lambda \), i.e., \( F(\Spec A') = \star \).
\end{enumerate}
By the assumption of \( I_{\Gp} \), we have
\( (I_{(a)})_{\Gp} = I_{(a)}A_{\Gp} = IA_{\Gp} \).
Thus, there is an element \( a' \in A \setminus \Gp\) such that \( I_{(a)}A_{aa'} = IA_{aa'} \).
Here, \( I_{\lambda, (aa')} = I_{\lambda, (a)}A_{aa'} \) for any \( \lambda \)
by the property \eqref{thm:dd dec:I_lambda} of \( I_{\lambda} \).
Thus, \( I_{(aa')} = I_{(a)}A_{aa'} = IA_{aa'}\).
Therefore, \( aa' \) satisfies the condition \eqref{thm:dd dec:condF8} by
the property \eqref{thm:dd dec:I(a)}.
Thus, we have checked the conditions  (\(\mathrm{F}_{3}\)) and \ (\(\mathrm{F}_{8}\)), and 
the assertion in the case \eqref{thm:dd dec:ass2} has been proved.  
\end{proof}

\begin{rem}\label{rem:univFlatFct}
The (universal) flattening functor is introduced by Murre in \cite{Murre},
but its origin seems to go back to Grothendieck as the subtitle says.
Murre gives a criterion of the representability of the functor in \cite[\S3, (A)]{Murre},
whose prototype seems to be \cite[IV, Prop.~(11.4.5)]{EGA}.
Mumford considers the case of projective morphism in \cite[Lect.\ 8]{Mumford},
and proves the representability by using Hilbert polynomials,
where the representing scheme
is called the ``flattening stratification.'' He also mentioned that Grothendieck has proved
a weaker result by much deeper method.
Raynaud \cite[Ch.~3]{Raynaud} and Raynaud--Gruson \cite[Part 1, \S4]{RG}
give further criteria
of the representability of the universal flattening functor by another method. 
One of them is used in proving the following ``local version'' 
of the existence of relative \( \bfS_{2} \) refinement. 
\end{rem}

\begin{thm}\label{thm:enhancement}
Let \( f \colon Y \to T \) be an \( \bfS_{2} \)-morphism of locally Noetherian schemes 
and \( \SF \) a reflexive \( \SO_{Y} \)-module 
which is \emph{locally free in codimension one on each fiber} 
\emph{(}cf.\ Definition~\emph{\ref{dfn:locFreeCodim1Fiber})}.
Assume that \( T = \Spec R \) for a Henselian local ring \( R \) and let \( o \in T \) be the closed point.  
Then, for any point \( y \) of the closed fiber \( Y_{o} = f^{-1}(o) \), 
there is a closed subscheme \( S \subset T \) having the following universal property\emph{:}
Let \( T' = \Spec R' \to T = \Spec R\) be a morphism defined by a local ring homomorphism \( R \to R' \) 
for a Noetherian local ring \( R' \) and let \( o' \in T'\) be the closed point. 
Let \( f' \colon Y' = Y \times_{T} T' \to T' \) and \( p \colon Y' \to Y \) be the induced morphisms. 
Then, for the pullback \( \SF' := p^{*}\SF \), 
its double-dual \( (\SF')^{\vee\vee} \) on \( Y' \) satisfies relative \( \bfS_{2} \) over \( T' \) 
at any point \( y' \) of \( Y' \) with \( p(y') = y \) and \( f'(y') = o' \) 
if and only if \( T' \to T \) factors through \( S \). 
\end{thm}

\begin{proof}
Replacing \( Y \) with an open neighborhood of \( y \), 
we may assume that \( Y \) is an affine \( R \)-scheme of finite type. 
Then, by the same argument as in the proof of Lemma~\ref{lem:SurjFlat(reflexive)}, 
we have an exact sequence \( 0 \to \SF \to \SE^{0} \to \SE^{1} \to \SG \to 0 \) 
of coherent \( \SO_{Y} \)-modules such that \( \SE^{0} \) and \( \SE^{1} \) are locally free 
and \( \SG|_{U} \) is also locally free for the maximal open subset \( U \) 
such that \( \SF|_{U} \) is locally free. 
By Lemma~\ref{lem:1|prop:key}, \( (\SF')^{\vee\vee} \) satisfies relative \( \bfS_{2} \) over \( T' \) 
at a point \( y' \) lying over \( y \) if and only if \( (p^{*}\SG)_{y'} \) is flat over \( T' \). 
Therefore, the universal closed subscheme \( S \subset T \) exists by 
\cite[Part 1, Th.\ (4.1.2)]{RG} or \cite[Ch.\ 3, Th.\ 1]{Raynaud} applied to \( \SG \). 
\end{proof}


\section{Grothendieck duality}
\label{sect:GD}

We shall explain the theory of Grothendieck duality
with some base change theorems
referring to \cite{ResDual}, \cite{Conrad}, \cite{Lipman09}, etc.
We do not prove the main part of the duality theory but show
several consequences.
Some of them are useful for
studying \( \BQQ \)-Gorenstein schemes and \( \BQQ \)-Gorenstein
morphisms in Sections~\ref{sect:QGorSch} and \ref{sect:QGormor}.

Some well-known properties on the dualizing complex
are mentioned in Sections~\ref{subsect:dualizingcpx} and \ref{subsect:ordinaryDC}
based on arguments in \cite{ResDual} and \cite{Conrad}.
Section~\ref{subsect:dualizingcpx} explains some basic properties and results on
a locally Noetherian scheme admitting a dualizing complex, mainly on
the codimension function associated with the dualizing complex
and on interpretation of \( \bfS_{k} \)-conditions
for a coherent sheaf via the dualizing complex.
In Section~\ref{subsect:ordinaryDC}, we introduce a useful notion of \emph{ordinary dualizing complex}
for locally equi-dimensional locally Noetherian schemes, 
and study cohomology sheaves of ordinary dualizing complexes. 
Section~\ref{subsect:twisted inverse} explains the notion
of twisted inverse image
and the relative duality theory referring mainly to \cite{ResDual}, \cite{Conrad}, \cite{Lipman09}.
Our original base change result for the relative dualizing complex 
to the fiber is proved in Corollary~\ref{cor:BC}. 
In Section~\ref{subsect:CMmorGormor}, we explain the 
\emph{relative dualizing sheaf}
for a \emph{Cohen--Macaulay morphism} (cf.\ Definition~\ref{dfn:SkCMmorphism})
and its base change property referring to \cite{Conrad}, \cite{Sastry}, etc.


\subsection{Dualizing complex}
\label{subsect:dualizingcpx}

We shall begin with recalling the notion of dualizing complex, which is introduced in
\cite[V]{ResDual}.

\begin{dfn}\label{dfn:dualizingcomplex}
A dualizing complex \( \SR^{\bullet} \)
of a locally Noetherian scheme \( X \)
is defined to be a complex of \( \SO_{X} \)-modules bounded below such that
\begin{itemize}
\item  it has
coherent cohomology and has finite injective dimension, i.e.,
\( \SR^{\bullet} \in \bfD^{+}_{\coh}(X)_{\text{fid}} \)
in the sense of
\cite{ResDual}, and

\item  the natural morphism
\[ \SO_{X} \to \SRHom_{\SO_{X}}(\SR^{\bullet}, \SR^{\bullet}) \]
is a quasi-isomorphism (cf.\ \cite[V, Prop.~2.1]{ResDual}).
\end{itemize}
\end{dfn}

\begin{remn}
Every complex in \( \bfD^{+}_{\coh}(X)_{\text{fid}} \)
is quasi-isomorphic to a bounded complex
of quasi-coherent injective \( \SO_{X} \)-modules
when \( X \) is quasi-compact
(cf.\ \cite[II, Prop.\ 7.20]{ResDual}).
The derived functor \( \SRHom_{\SO_{X}} \) of the bi-functor \( \SHom_{\SO_{X}} \)
is considered as a functor
\[ \bfD(X)^{\op} \times \bfD(X) \ni (\SF^{\bullet}, \SG^{\bullet}) \mapsto
\SRHom_{\SO_{X}}(\SF^{\bullet}, \SG^{\bullet}) \in \bfD(X)\]
(cf.\ \cite[I, \S6]{ResDual}, \cite[Th.~A(ii)]{Spa}).
\end{remn}

\begin{examn}
A Noetherian local ring \( A \) is said to be \emph{Gorenstein}
if there is a finite injective resolution of \( A \).
In particular, \( \SO_{X} \) is a dualizing complex for \( X = \Spec A \).
There are known several conditions for a local ring \( A \) to be Gorenstein
(e.g.\ \cite[V, Th.~9.1]{ResDual}, \cite[Th.~18.1]{Matsumura}):
For example,
\( A \) is Gorenstein if and only if
\( A \) is Cohen--Macaulay and \( \Ext^{n}(A/\GM_{A}, A) \isom A/\GM_{A} \)
for the maximal ideal \( \GM_{A} \) and \( n = \dim A \).
A locally Noetherian scheme \( Y \) is said to be \emph{Gorenstein}
if every local ring \( \SO_{Y, y} \) is Gorenstein.
For a locally Noetherian scheme \( Y \),
it is Gorenstein of finite Krull dimension if and only if
\( \SO_{Y} \) is a dualizing complex (cf.\ \cite[II, Prop.\ 7.20]{ResDual}).
\end{examn}

\begin{examn}[{cf.\ \cite[V, Prop.~3.4]{ResDual}, \cite[Th.~18.6]{Matsumura}}]
For an Artinian local ring \( A \),
let \( I \) be an injective hull of the residue field \( A/\GM_{A} \).
Then, the associated
quasi-coherent sheaf \( I\sptilde \) on \( \Spec A \) is a dualizing complex.
\end{examn}

\begin{rem}[{\cite[V, \S10]{ResDual}}]\label{rem:ExistDC}
Let \( X \) be a locally Noetherian scheme.
If there is a morphism \( X \to Y \) of finite type to
a locally Noetherian scheme \( Y \) admitting a dualizing complex
in which the dimensions of fibers are bounded,
then \( X \) also admits a dualizing complex \cite[VI, Cor.~3.5]{ResDual}.
In particular, any scheme of finite type over a Noetherian Gorenstein scheme
of finite Krull dimension
admits a dualizing complex.
When \( X \) is connected, the dualizing complex is unique up to
quasi-isomorphism, shift, and up to tensor product with invertible sheaves
(cf.\ \cite[V, Th.~3.1]{ResDual}, \cite[(3.1.30)]{Conrad}).
\end{rem}

\begin{factn}
For a Noetherian ring \( A \), the affine scheme \( \Spec A \) admits a dualizing complex
if and only if there is a surjection \( B \to A \) from a Gorenstein ring \( B \)
of finite Krull dimension.
This is conjectured by Sharp \cite[Conj.~(4.4)]{Sharp} and has been proved
by Kawasaki \cite[Cor.~1.4]{Kawasaki}.
\end{factn}

We shall explain the notion of codimension function.

\begin{dfn}[{cf.\ \cite[V, p.~283]{ResDual}}]\label{dfn:codimfct}
Let \( X \) be a scheme such that every local ring \( \SO_{X, x} \) has finite Krull dimension.
A function \( d \colon X \to \BZZ \)
is called a \emph{codimension function} if
\[ d(x) = d(y) + \Codim(\overline{\{x\}}, \overline{\{y\}}) \]
for any points \( x \) and \( y \) such that \( x \in \overline{\{y\}} \).
\end{dfn}

\begin{remn}
Let \( X \) be a scheme whose local rings \( \SO_{X, x} \) all have finite Krull dimension.
If \( X \) admits a codimension function, then
\( X \) is catenary (cf.\ Property~\ref{pprt:catenary}). In fact,
\[ \Codim(\overline{\{x\}}, \overline{\{z\}}) =
\Codim(\overline{\{x\}}, \overline{\{y\}}) + \Codim(\overline{\{y\}}, \overline{\{z\}}) \]
holds for any \( x \), \( y \), \( z \in X\) satisfying \( x \in \overline{\{y\}}  \)
and \( y \in \overline{\{z\}} \).
Moreover, if the codimension function is bounded, then \( X \) has finite Krull dimension.
\end{remn}

\begin{lem}\label{lem:CodimFunction:constant}
Let \( X \) be a scheme such that every local ring \( \SO_{X, x} \) has finite Krull dimension,
and let \( d \colon X \to \BZZ \) be a codimension function. Then,
\[ d(y) - \dim \SO_{X, y} \geq d(x) - \dim \SO_{X, x} \]
holds for any points \( x \), \( y \in X\) with \( x \in \overline{\{y\}} \).
Moreover, the following three conditions are equivalent to each other\emph{:}
\begin{enumerate}
    \renewcommand{\theenumi}{\roman{enumi}}
    \renewcommand{\labelenumi}{(\theenumi)}
\item \label{lem:CodimFunction:constant:cond1}
the equality
\[ d(y) - \dim \SO_{X, y} = d(x) - \dim \SO_{X, x} \]
holds for any points \( x \), \( y \in X \)
with \( x \in \overline{\{y\}} \)\emph{;}

\item \label{lem:CodimFunction:constant:cond2}
the function \( X \ni x \mapsto d(x) - \dim \SO_{X, x} \in \BZZ\) is locally constant\emph{;}

\item \label{lem:CodimFunction:constant:cond3}
\( X \) is locally equi-dimensional
\emph{(}cf.\ Definition~\emph{\ref{dfn:equi-dim}}\eqref{dfn:equi-dim:locally}\emph{)}.
\end{enumerate}
\end{lem}

\begin{proof}
The first inequality is derived from the well-known inequality
\[ \dim \SO_{X, x} \geq \dim \SO_{Y, y} + \Codim(\overline{\{x\}}, \overline{\{y\}}) \]
(cf.\ Property~\ref{ppty:dim-codim}\eqref{ppty:dim-codim:1},
\cite[$0_{\text{IV}}$, Prop.~(14.2.2)]{EGA}).
To show the equivalence of three conditions
\eqref{lem:CodimFunction:constant:cond1}--\eqref{lem:CodimFunction:constant:cond3},
we may assume that \( X \) is connected.
Let \( \ScS \) be the set of generic points of irreducible components of \( X \) and,
for a point \( x \in X \), let \( \ScS(x) \) be the subset
consisting of \( y \in \ScS \) with \( x \in \overline{\{y\}} \).
Note that \( \SO_{X, x} \) is equi-dimensional if and only if
\begin{equation}\label{eq:lem:CodimFunction:constant}
\Codim(\overline{\{x\}}, \overline{\{y\}}) = \Codim(\overline{\{x\}}, X)
\end{equation}
for any \( y \in \ScS(x) \). In fact, a point \( y \in \ScS(x) \) corresponds
to a minimal prime ideal \( \Gp \) of
\( \SO_{X, x} \) via the natural morphism \( \Spec \SO_{X, x} \to X \), and
\eqref{eq:lem:CodimFunction:constant} is written as
\[ \dim \SO_{X, x}/\Gp = \dim \SO_{X, x} \]
(cf.\ Property~\ref{ppty:dim-codim}\eqref{ppty:dim-codim:1}).
The implication
\eqref{lem:CodimFunction:constant:cond2}
\( \Rightarrow \) \eqref{lem:CodimFunction:constant:cond1} is trivial,
and \eqref{lem:CodimFunction:constant:cond1} \( \Rightarrow \)
\eqref{lem:CodimFunction:constant:cond3} is shown by
the equality
\( \dim \SO_{X, x} = \Codim(\overline{\{x\}}, \overline{\{y\}})\) for any \( y \in \ScS(x) \),
which holds by \eqref{lem:CodimFunction:constant:cond1}.
It suffices to prove:
\eqref{lem:CodimFunction:constant:cond3}
\( \Rightarrow \)
\eqref{lem:CodimFunction:constant:cond2}.
In the situation of \eqref{lem:CodimFunction:constant:cond3},
by \eqref{eq:lem:CodimFunction:constant}, we have
\( d(x) - \dim \SO_{X, x} = d(y) = d(y) - \dim \SO_{X, y}  \)
for any \( x \in X\) and \( y \in \ScS(x) \).
This implies that
\( x \mapsto d(x) - \dim \SO_{X, x} \) is a constant function with value \( d(y) \) on
\( \overline{\{y\}} \) for any \( y \in \ScS \), and
\( d(y)  = d(y') \) for any points \( y \), \( y' \in \ScS \)
with \( \overline{\{y\}} \cap \overline{\{y'\}} \ne \emptyset \).
Consequently, \( x \mapsto d(x) - \dim \SO_{X, x} \) is constant on \( X \),
since \( X \) is connected. Thus, we are done.
\end{proof}

The importance of the codimension function comes from the following:

\begin{fact}\label{fact:CodimFunction}
Let \( X \) be a locally Noetherian scheme with a dualizing complex \( \SR^{\bullet} \).
Then, we can define a function \( d \colon X \to \BZZ \) by
\[ \BExt^{i}_{\SO_{X, x}}(\Bbbk(x), \SR^{\bullet}_{x})
= \OH^{i}(\RHom_{\SO_{X, x}}(\Bbbk(x), \SR^{\bullet}_{x})) =
\begin{cases}
0, & \text{ for } i \ne d(x); \\
\Bbbk(x), & \text{ for } i = d(x),
\end{cases}\]
where \( \Bbbk(x) \) denotes the residue field at \( x \) and
\( \SR^{\bullet}_{x} \) denotes the stalk at \( x \) (cf.\ \cite[V, Prop.~3.4]{ResDual}).
The function \( d \) is a bounded codimension function
(cf.\ \cite[V, Cor.~7.2]{ResDual}),
and we call \( d \)
the \emph{codimension function associated with} \( \SR^{\bullet} \).
In particular, \( X \) is catenary and has finite Krull dimension.
\end{fact}

The following result and Lemma~\ref{lem:DC-CM:supportSk} below are useful
for checking \( \bfS_{k} \)-conditions for coherent sheaves.

\begin{prop}\label{prop:DC-CM}
Let \( X \) be a locally Noetherian scheme admitting
a dualizing complex \( \SR^{\bullet} \) with codimension function
\( d \colon X \to \BZZ \). Let \( \SF \) be a coherent \( \SO_{X} \)-module.
For an integer \( j \), we set
\[ \SG^{(j)} := \SBExt^{j}_{\SO_{X}}(\SF, \SR^{\bullet})
:= \SH^{j}(\SRHom_{\SO_{X}}(\SF, \SR^{\bullet})). \]
Then, \( \SG^{(j)} \) is a coherent \( \SO_{X} \)-module
and
\[ \SG^{(j)}_{x} \isom \BExt^{j}_{\SO_{X, x}}(\SF_{x}, \SR^{\bullet}_{x}) \]
for the stalk \( \SG^{(j)}_{x} = (\SG^{(j)})_{x}\) at any point \( x \in X \).
Moreover, the following hold for a point \( x \in X \)\emph{:}
\begin{enumerate}
\item \label{prop:DC-CM:1}
If \( j - d(x) < -\dim \SF_{x} \) or \( j - d(x) > 0 \), then
\( \SG^{(j)}_{x} = 0 \).

\item \label{prop:DC-CM:2}
For an integer \( k \),
\( \depth \SF_{x} \geq k\) if and only if
\( \SG^{(j)}_{x} = 0 \) for any \( j > d(x) - k \).

\item  \label{prop:DC-CM:3}
For an integer \( k \),
\( \SF \) satisfies \( \bfS_{k} \) at \( x \) if and only if
\( \SG^{(j)}_{y} = 0 \)
for any point \( y \in X \) with \( x \in \overline{\{y\}} \) and
for any \( j > d(y) - \inf \{k, \dim \SF_{y}\} \).

\item  \label{prop:DC-CM:4}
\( \SF_{x}\) is a Cohen--Macaulay \( \SO_{X, x} \)-module if and only if
\( \SG^{(j)}_{x} = 0 \) for any \( j \ne d(x) - \dim \SF_{x} \).

\item \label{prop:DC-CM:5}
If \( x \in \Supp \SF \), then \( \SG^{(i)}_{x} \ne 0 \) for \( i = d(x) - \dim \SF_{x} \).
\end{enumerate}

\end{prop}

\begin{proof}
The first assertion is derived from: \( \SR^{\bullet} \in \bfD^{+}_{\coh}(X)_{\text{fid}} \).
The assertions \eqref{prop:DC-CM:1} and \eqref{prop:DC-CM:2} are essentially
proved in \cite[V]{ResDual};
\eqref{prop:DC-CM:1} is shown in the proof of
\cite[V, Prop.~3.4]{ResDual}, and \eqref{prop:DC-CM:2}
follows from the local duality theorem \cite[V, Cor.~6.3]{ResDual}.
The assertion \eqref{prop:DC-CM:3} follows from
\eqref{prop:DC-CM:2} and Definition~\ref{dfn:SerreCond}.
The assertion \eqref{prop:DC-CM:4} is a consequence of \eqref{prop:DC-CM:1} and \eqref{prop:DC-CM:2},
since \( \SF_{x} \) is Cohen--Macaulay if and only if
\( \depth \SF_{x} = \dim \SF_{x} \) unless \( \SF_{x} = 0 \).
The assertion \eqref{prop:DC-CM:5} is shown as follows.
For the given point \( x \in \Supp \SF \), we can find a point \( y \in \Supp \SF \)
such that \( \overline{\{y\}} \) is an irreducible component of \( \Supp \SF \) containing \( x \)
and \( \dim \SF_{x} = \Codim(\overline{\{x\}}, \overline{\{y\}}) \). Then,
\( d(x) - \dim \SF_{x} = d(y) \). If \( \SG^{(d(y))}_{x} = 0 \),
then \( \SG^{(d(y))}_{y} = 0 \), since \( x \in \overline{\{y\}} \).
But, in this case, \( \SG^{(j)}_{y} = 0 \) for any \( j \in \BZZ \) by
\eqref{prop:DC-CM:1}, i.e., \( \SRHom_{\SO_{X}}(\SF, \SR^{\bullet})_{y} \isom_{\qis} 0\).
This is a contradiction, since \( \SF_{y} \ne 0 \) and
\[ \SF \isom_{\qis} \SRHom_{\SO_{X}}(\SRHom_{\SO_{X}}(\SF, \SR^{\bullet}), \SR^{\bullet}) \]
by \cite[V, Prop.~2.1]{ResDual}.
Therefore, \( \SG^{(i)}_{x} \ne 0 \) for \( i = d(x) - \dim \SF_{x} = d(y) \).
\end{proof}

\begin{cor}\label{cor:lem:DC-CM:CM}
Let \( X \) be a locally Noetherian scheme admitting a dualizing complex
\( \SR^{\bullet} \) and let \( \SF \) be a coherent \( \SO_{X} \)-module.
\begin{enumerate}
\item \label{cor:lem:DC-CM:CM:1}
Assume that \( \Supp \SF \) is connected.
Then, \( \SF \) is a Cohen--Macaulay \( \SO_{X} \)-module
if and only if
\[ \SRHom_{\SO_{X}}(\SF, \SR^{\bullet}) \isom_{\qis} \SG[-c]\]
for a coherent \( \SO_{X} \)-module \( \SG \) and a constant \( c \in \BZZ \).
In this case, \( \SG \) is also a Cohen--Macaulay \( \SO_{X} \)-module and
\( \Supp \SG = \Supp \SF \).

\item  \label{cor:lem:DC-CM:CM:2}
Assume that \( X \) is connected. Then,
\( X \) is Cohen--Macaulay if and only if \( \SR^{\bullet} \isom_{\qis} \SL[-c]\)
for a coherent \( \SO_{X} \)-module \( \SL \) and a constant \( c \in \BZZ \).
In this case, \( \SL \) is also a Cohen--Macaulay \( \SO_{X} \)-module and
\( \Supp \SL = X\).

\item  \label{cor:lem:DC-CM:CM:3}
Assume that \( \SF \) be a Cohen--Macaulay \( \SO_{X} \)-module and
let \( S \) be a closed subscheme of \( X \) such that \( S = \Supp \SF \)
as a set. Then, \( S \) is locally equi-dimensional.
\end{enumerate}
\end{cor}

\begin{proof}
It suffices to prove \eqref{cor:lem:DC-CM:CM:1} and \eqref{cor:lem:DC-CM:CM:3},
since \eqref{cor:lem:DC-CM:CM:2} is a special case of \eqref{cor:lem:DC-CM:CM:1}.
Let \( d \colon X \to \BZZ \) be the codimension function associated with \( \SR^{\bullet} \).
First, we shall prove the ``if'' part of \eqref{cor:lem:DC-CM:CM:1}.
The quasi-isomorphism in \eqref{cor:lem:DC-CM:CM:1} implies that
\( \SG^{(j)} := \SBExt^{j}_{\SO_{X}}(\SF, \SR^{\bullet}) = 0 \)
for any \( j \ne c \) and \( \SG \isom \SG^{(c)} \). Then,
\( \SF \) is Cohen--Macaulay and \( c = d(x) - \dim \SF_{x} \) for any \( x \in X \)
by \eqref{prop:DC-CM:4} and \eqref{prop:DC-CM:5} of Proposition~\ref{prop:DC-CM}.
Second, we shall prove the remaining part of  \eqref{cor:lem:DC-CM:CM:1} and
\eqref{cor:lem:DC-CM:CM:3}. For the proof of \eqref{cor:lem:DC-CM:CM:3},
we may also assume that \( \Supp \SF \) is connected.
Suppose that \( \SF \) is Cohen--Macaulay.
Then, \( d(x) - \dim \SF_{x} = d(y) - \dim \SF_{y} \) holds
for any points \( x \), \( y \in S \) with \( x \in \overline{\{y\}} \) by
\eqref{prop:DC-CM:4} and \eqref{prop:DC-CM:5} of Proposition~\ref{prop:DC-CM}, where
we use the property that \( \SG^{(j)}_{x} = 0 \) implies \( \SG^{(j)}_{y} = 0 \).
As a consequence, \( c := d(x) - \dim \SF_{x} \) is constant on \( S = \Supp \SF \).
We have \( \dim \SF_{x} = \dim \SO_{S, x} \)
for any \( x \in S \) by Property~\ref{ppty:dim-codim}\eqref{ppty:dim-codim:1}.
Thus, \( S \) is locally equi-dimensional
by Lemma~\ref{lem:CodimFunction:constant},
and this proves \eqref{cor:lem:DC-CM:CM:3}.
Furthermore, \( \SRHom_{\SO_{X}}(\SF, \SR^{\bullet}) \isom \SG[-c]\) for
the cohomology sheaf \( \SG = \SG^{(c)} \). We have also
\[ \SF \isom \SRHom_{\SO_{X}}(\SG[-c], \SR^{\bullet}) \]
by \cite[V, Prop.~2.1]{ResDual}. Thus, \( \Supp \SG = \Supp \SF \),
and \( \SG \) is also a Cohen--Macaulay \( \SO_{X} \)-module by
the ``if'' part of \eqref{cor:lem:DC-CM:CM:1}.
Thus, we are done.
\end{proof}

\begin{lem}\label{lem:DC-CM:supportSk}
Let \( X \), \( \SR^{\bullet} \), \( \SF \), and \( \SG^{(j)} \)
be as in Proposition~\emph{\ref{prop:DC-CM}}.
Then, \( \SG^{(j)} = 0\) except for finitely many \( j \).
For a positive integer \( k \), the following hold\emph{:}
\begin{enumerate}
\item \label{lem:DC-CM:supportSk:1}
\( \SF \) satisfies \( \bfS_{k} \) at a point \( x \in \Supp \SF\) if and only if
\[ \Codim_{x}(\Supp \SG^{(i)} \cap \Supp \SG^{(j)},\, \Supp \SF) \geq k + i - j \]
for any \( i > j \)\emph{;}

\item \label{lem:DC-CM:supportSk:2}
\( \SF \) satisfies \( \bfS_{k} \) if and only if
\[ \Codim(\Supp \SG^{(i)} \cap \Supp \SG^{(j)},\, \Supp \SF) \geq k + i - j \]
for any \( i > j \).
\end{enumerate}
\end{lem}

\begin{proof}
The first assertion follows from Proposition~\ref{prop:DC-CM}\eqref{prop:DC-CM:1},
since \( d \colon X \to \BZZ \) is bounded and \( \dim X < \infty \) by
Fact~\ref{fact:CodimFunction}.
For integers \( i \), \(j \) with \( i > j \), we set
\[ Z^{(i, j)} := \Supp \SG^{(i)} \cap \Supp \SG^{(j)}. \]
Note that \( \Codim(Z^{(i, j)}, \Supp \SF) = +\infty \) if \( Z^{(i, j)} = \emptyset \).
The assertion \eqref{lem:DC-CM:supportSk:1}
is derived from  \eqref{lem:DC-CM:supportSk:2} applied to
the coherent sheaf \( (\SF_{x})\sptilde \) on \( \Spec \SO_{X, x} \)
associated with \( \SF_{x} \) (cf.\ Remark~\ref{rem:dfn:SerreCond:atx}), since
\[ \Codim_{x}(Z^{(i, j)},\, \Supp \SF) =
\Codim(\Supp \SG^{(i)}_{x} \cap \Supp \SG^{(j)}_{x},\, \Supp \SF_{x}) \]
(cf.\ Property~\ref{ppty:dim-codim}\eqref{ppty:dim-codim:2}).
Hence, it is enough to prove \eqref{lem:DC-CM:supportSk:2}.
Assume first that \( \SF \) satisfies \( \bfS_{k} \).
For integers \( i > j \) with \( Z^{(i, j)} \ne \emptyset \),
we can find a generic point \( x \) of \( Z^{(i, j)} \) such that
\[ \Codim(Z^{(i, j)},\, \Supp \SF) = \Codim(\overline{\{x\}}, \Supp \SF) = \dim \SF_{x}. \]
If \( \dim \SF_{x} \leq k  \), then
\( i = j = d(x) - \dim \SF_{x} \) by
\eqref{prop:DC-CM:1} and \eqref{prop:DC-CM:3} of Proposition~\ref{prop:DC-CM}.
This is a contradiction, since \( i > j \).
Thus, \( \dim \SF_{x} > k \), and
\[ d(x) - \dim \SF_{x} \leq j < i \leq d(x) - k \]
also by \eqref{prop:DC-CM:1} and \eqref{prop:DC-CM:3} of Proposition~\ref{prop:DC-CM}.
Hence, \( i - j \leq \dim \SF_{x} - k \), and this is equivalent to
the inequality in \eqref{lem:DC-CM:supportSk:2}.

Conversely, assume that the inequality
in \eqref{lem:DC-CM:supportSk:2} holds for any \( i > j \).
For a point \( x \in \Supp \SF \), we set \( c(x) := d(x) - \dim \SF_{x} \).
By \eqref{prop:DC-CM:1} and \eqref{prop:DC-CM:5} of Proposition~\ref{prop:DC-CM}, we know that
\( x \in \Supp \SG^{(c(x))} \) and \( x \not\in \Supp \SG^{(i)} \) for any \( i < c(x) \).
If \( \SG^{(i)}_{x} \ne 0\) for some \( i \ne c(x) \), then
\( i > c(x) \) and
\[ \dim \SF_{x} \geq \Codim(Z^{(i, c(x))},\, \Supp \SF)
\geq k + i - c(x) = k + i - d(x) + \dim \SF_{x}.\]
Hence, \( i \leq d(x) - k \) and \( \dim \SF_{x} > k \).
Thus, \( \SF \) satisfies \( \bfS_{k} \) by Proposition~\ref{prop:DC-CM}\eqref{prop:DC-CM:3}.
Therefore, \eqref{lem:DC-CM:supportSk:2} has been proved, and we are done.
\end{proof}

\begin{cor}\label{cor:lem:DC-CM:supportSk}
Let \( X \), \( \SR^{\bullet} \), \( \SF \), and \( \SG^{(j)} \)
be as in Proposition~\emph{\ref{prop:DC-CM}}.  
Let \( k \) be a positive integer.
\begin{enumerate}

\item \label{cor:lem:DC-CM:supportSk:1}
Assume that \( \Supp \SF \) is connected and locally equi-dimensional. 
Then, there is a positive integer \( c \) such that
\( c = d(x) - \dim \SF_{x} \) for any \( x \in X \).
For the integer \( c \), one has \( \Supp \SG^{(c)} = \Supp \SF\).
Moreover, \( \SF \) satisfies \( \bfS_{k} \) if and only if
\[ \Codim(\Supp \SG^{(j)},\, \Supp \SF) \geq k + j - c \]
for any \( j > c \).

\item \label{cor:lem:DC-CM:supportSk:2}
Assume that \( \Supp \SF \) is connected, equi-dimensional, and
\emph{equi-codimen\-sional} 
\emph{(}cf.\ \cite[$0_{\mathrm{IV}}$, D\'ef.~(14.2.1)]{EGA}\emph{)}.
Furthermore, assume that \( \Supp \SF \) is Noetherian.
Let \( c \) be the integer in \eqref{cor:lem:DC-CM:supportSk:1}.
Then, \( \SF \) satisfies \( \bfS_{k} \) if and only if
\[ \dim \Supp \SG^{(j)} \leq \dim \Supp \SF + c - j - k \]
for any \( j > c \).

\item \label{cor:lem:DC-CM:supportSk:3}
Assume that \( \SF_{x} \ne 0 \) and \( \SF_{x} \) is equi-dimensional
\emph{(}cf.\ Definition~\emph{\ref{dfn:equi-dim})}.
Then, \( \SF \) satisfies \( \bfS_{k} \) at \( x \) if and only if
\( \dim \SG^{(j)}_{x} \leq d(x) - j - k \) for any \( j \ne d(x) - \dim \SF_{x} \).

\end{enumerate}
\end{cor}

\begin{proof}
\eqref{cor:lem:DC-CM:supportSk:1}:
For a closed subscheme \( S \) with \( S = \Supp \SF \), we have the integer \( c \)
such that \( c = d(x) - \dim \SO_{S, x} = d(x) - \dim\SF_{x} \) for any \( x \in \Supp \SF \)
by Lemma~\ref{lem:CodimFunction:constant}.
Then, \( \Supp \SG^{(c)} = \Supp \SF \) by Proposition~\ref{prop:DC-CM}\eqref{prop:DC-CM:5}.
Assume that \( \SF \) satisfies \( \bfS_{k} \). Then,
\[ \Codim(\Supp \SG^{(j)}, \, \Supp \SF) =
\Codim(\Supp \SG^{(j)} \cap \Supp \SG^{(c)}, \, \Supp \SF)
\geq k + j - c\]
for any \( j > c \) by Lemma~\ref{lem:DC-CM:supportSk}. Conversely,
assume that the inequality in \eqref{cor:lem:DC-CM:supportSk:1} holds
for any \( j > c \). If \( \SG^{(j)}_{x} \ne 0 \) for some \( j > c \), then
\[ \dim \SF_{x} \geq \Codim(\Supp \SG^{(j)}, \Supp \SF)
\geq k + j - c = k + j - d(x) + \dim \SF_{x}\]
as in the proof of Lemma~\ref{lem:DC-CM:supportSk}\eqref{lem:DC-CM:supportSk:2}.
Hence, \( \SG^{(j)}_{x} \ne 0 \) implies that \( \dim \SF_{x} > k \) and
\( j \leq d(x) - k \). This means that \( \SF \) satisfies \( \bfS_{k} \)
by Proposition~\ref{prop:DC-CM}\eqref{prop:DC-CM:3}.

\eqref{cor:lem:DC-CM:supportSk:2}: 
The closed subset \( S = \Supp \SF \) is a bi-equi-dimensional Kolmogorov Noetherian space 
in the sense of EGA (cf.\ \cite[$0_{\text{IV}}$, Prop.\ (14.3.3)]{EGA}), 
since it is catenary and has finite Krull dimension (cf.\ Fact~\ref{fact:CodimFunction}). 
Then, 
\[ x \mapsto \dim \SO_{S, x} = \Codim(\overline{\{x\}}, S) \] 
is a codimension function on \( S \) 
by \cite[$0_{\text{IV}}$, (14.3.3.2)]{EGA}, and 
it implies that \( S \) is locally equi-dimensional by Lemma~\ref{lem:CodimFunction:constant}. 
Moreover, 
\( \dim Z + \Codim(Z, S) = \dim S \)
for any closed subset \( Z \subset S \)
by \cite[$0_{\text{IV}}$, Cor.\ (14.3.5)]{EGA}. 
Thus, \eqref{cor:lem:DC-CM:supportSk:2} follows from \eqref{cor:lem:DC-CM:supportSk:1}. 

\eqref{cor:lem:DC-CM:supportSk:3}:
The closed subset \( \Supp \SF_{x} \) of \( \Spec \SO_{X, x} \) is equi-codimensional, 
since \( x \) is the unique closed point of it. Hence, \( \Supp \SF_{x} \) is also 
a connected bi-equi-dimensional Kolmogorov Noetherian space and 
it is locally equi-dimensional by the same reason as above. 
Thus, we can apply
\eqref{cor:lem:DC-CM:supportSk:1}
to the coherent sheaf \( \SF_{x}\sptilde \) on \( \Spec \SO_{X, x} \) associated with
\( \SF_{x} \).  
Hence, \( \SF \) satisfies \( \bfS_{k} \) at \( x \) (cf.\ Remark~\ref{rem:dfn:SerreCond:atx}) 
if and only if
\[ \Codim(\Supp \SG^{(j)}_{x},\, \Supp \SF_{x}) \geq k + j - c(x) \] 
for any \( j > c(x) \), where \( c(x) := d(x) - \dim \SF_{x} \). 
Here, the left hand side equals \( \dim \SF_{x} - \dim \SG^{(j)}_{x} \) 
by \cite[$0_{\text{IV}}$, Cor.\ (14.3.5)]{EGA}.  
Therefore, the \( \bfS_{k} \)-condition at \( x \) is equivalent to that
\[ \dim \SG^{(j)}_{x} \leq c(x) - k - j + \dim \SF_{x} = d(x) - k - j\]
for any \( j > c(x) = d(x) - \dim \SF_{x} \). Thus, we have
\eqref{cor:lem:DC-CM:supportSk:3} by Proposition~\ref{prop:DC-CM}\eqref{prop:DC-CM:1},
and we are done.
\end{proof}

\begin{dfn}[$\Gor(X)$]\label{dfn:Gorlocus}
The \emph{Gorenstein locus} \( \Gor(X) \)
of a locally Noetherian scheme \( X \)
is defined to be the set of points \( x \in X \) such that \( \SO_{X, x} \) is Gorenstein.
\end{dfn}

\noindent Note that \( X \) is Gorenstein if and only if \( X = \Gor(X) \).
The following is a generalization of \cite[IV, Prop.~(6.11.2)(ii)]{EGA}
(cf.\ \cite[Prop.\ (3.2)]{Sharp77} for \( \Gor(X) \)).

\begin{prop}\label{prop:CMlocus}
Let \( X \) be a locally Noetherian scheme admitting a dualizing complex locally on \( X \)
and let \( \SF \) be a coherent \( \SO_{X} \)-module.
Then, \( \bfS_{k}(\SF) \) for all \( k \geq 1 \)
and \( \CM(\SF) \) are open subsets of \( X \).
In particular, \( \CM(X) \) is open.
Moreover, \( \Gor(X) \) is also open.
\end{prop}

\begin{proof}
Localizing \( X \), we may assume that \( X \) is an affine Noetherian scheme with
a dualizing complex \( \SR^{\bullet} \).
The openness of  \( \Gor(X) \) follows from that of \( \CM(X) \).
In fact, if \( X \) is Cohen--Macaulay, then
we may assume that \( \SR^{\bullet} \isom_{\qis} \SL \) for
a coherent \( \SO_{X} \)-module \( \SL \)
by Corollary~\ref{cor:lem:DC-CM:CM}\eqref{cor:lem:DC-CM:CM:2}, and
\( \Gor(X) \) is the maximal open subset on which \( \SL \) is invertible.
The openness of \( \CM(\SF) \) is derived from
Corollary~\ref{cor:lem:DC-CM:CM}\eqref{cor:lem:DC-CM:CM:1}.
This follows also from the openness
of \( \bfS_{k}(\SF) \) for all \( k \geq 1 \). In fact,
\( \CM(\SF) = \bfS_{k}(\SF)  \) for \( k \gg 0 \),
since \( \dim \SF \leq \dim X < \infty \)
(cf.\ Fact~\ref{fact:CodimFunction} and Remark~\ref{rem:dfn:CM}).
The openness of \( \bfS_{k}(\SF) \) is derived
from Lemma~\ref{lem:DC-CM:supportSk}\eqref{lem:DC-CM:supportSk:1},
since \( x \mapsto \Codim_{x}(Z, \Supp \SF) \) is lower semi-continuous
for any closed subset \( Z \subset \Supp \SF\)
(cf.\ Property~\ref{ppty:dim-codim}\eqref{ppty:dim-codim:2}).
\end{proof}

\begin{remn}
In the situation of Proposition~\ref{prop:CMlocus},
all \( \bfS_{k}(\SF) \) are open if and only if the map
\[ \Supp \SF \ni x \mapsto \codepth \SF_{x}
:= \dim \SF_{x} - \depth \SF_{x} \in \BZZ_{\geq 0}\]
is upper semi-continuous (cf.\ \cite[IV, Rem.~(6.11.4)]{EGA}).
\end{remn}

The following analogy of Fact~\ref{fact:elem-flat}\eqref{fact:elem-flat:6} for \( \SG = \SO_{Y} \)
is known:
\begin{fact}[{cf.\ \cite[Th.\ 23.4]{Matsumura}, \cite[V, Prop.~9.6]{ResDual}}]\label{fact:GorYTF}
Let \( Y \to T \) be a flat morphism of locally Noetherian schemes.
Then, \( Y \) is Gorenstein if and only if \( T \) and every fiber are Gorenstein.
\end{fact}


\subsection{Ordinary dualizing complex}
\label{subsect:ordinaryDC}

We introduce the notion of ordinary dualizing complex \( \SR^{\bullet} \)
and that of dualizing sheaf as the cohomology sheaf \( \SH^{0}(\SR^{\bullet}) \)
for locally Noetherian schemes which are \emph{locally equi-dimensional}
(cf.\ Definition~\ref{dfn:equi-dim}\eqref{dfn:equi-dim:locally}), especially
for locally Noetherian schemes satisfying \( \bfS_{2} \).
In many articles, the dualizing sheaf is usually defined
for a Cohen--Macaulay scheme, and it coincides with the dualizing sheaf in our sense
(cf.\ Remark~\ref{rem:lem:DC-CM:CM2} below).

\begin{dfn}\label{dfn:ordinaryDC+DS}
Let \( X \) be a locally Noetherian scheme.
\begin{enumerate}
\item  A dualizing complex \( \SR^{\bullet} \) of \( X \) is said to be \emph{ordinary} if
the codimension function \( d \) associated with \( \SR^{\bullet} \)
satisfies \( d(x) = \dim \SO_{X, x} \) for any \( x \in X \).

\item  A coherent sheaf \( \SL \) on \( X \) is called a \emph{dualizing sheaf}
of \( X \) if
\( \SL \isom \SH^{0}(\SR^{\bullet}) \) for an ordinary dualizing complex \( \SR^{\bullet} \)
of \( X \).
\end{enumerate}
\end{dfn}

As a corollary to Lemma~\ref{lem:CodimFunction:constant} above, we have:

\begin{lem}\label{lem:ordinaryDC}
Let \( X \) be a locally Noetherian scheme admitting a dualizing complex.
Then, \( X \) admits an ordinary dualizing complex if and only if
\( X \) is locally equi-dimensional.
In particular, \( X \) admits an ordinary dualizing complex if \( X \) satisfies \( \bfS_{2} \).
\end{lem}

\begin{proof}
We may assume that \( X \) is connected.
Let \( \SR^{\bullet} \) be a dualizing complex of \( X \)
with codimension function \( d \colon X \to \BZZ\).
If it is ordinary, then \( X \) is locally equi-dimensional
by Lemma~\ref{lem:CodimFunction:constant}.
Conversely, if \( X \) is locally equi-dimensional, then
\( d(x) - \dim \SO_{X, x} \) is a constant \( c \)
by Lemma~\ref{lem:CodimFunction:constant}, and hence,
the shift \( \SR^{\bullet}[c] \) is an ordinary dualizing complex.
The last assertion follows from Facts~\ref{fact:CodimFunction}
and \ref{fact:S2}\eqref{fact:S2:1}.
\end{proof}

\begin{remn}
For a locally Noetherian scheme,
the ordinary dualizing complex is unique up to quasi-isomorphism
and tensor product with an invertible sheaf
(cf.\ Remark~\ref{rem:ExistDC}).
Similarly, the dualizing sheaf is unique up to isomorphism
and tensor product with an invertible sheaf.
\end{remn}

\begin{rem}\label{rem:lem:DC-CM:CM2}
Let \( X \) be a locally Noetherian Cohen--Macaulay scheme admitting a dualizing complex.
Then, \( X \) has an ordinary dualizing complex \( \SR^{\bullet} \)
which is quasi-isomorphic to the dualizing sheaf \( \SL = \SH^{0}(\SR^{\bullet}) \).
Here, \( \SL \) is also a Cohen--Macaulay \( \SO_{X} \)-module.
These are derived from
Proposition~\ref{prop:DC-CM}\eqref{prop:DC-CM:4}
and Corollary~\ref{cor:lem:DC-CM:CM}\eqref{cor:lem:DC-CM:CM:2}.
In many articles, \( \SL \) is called a ``dualizing sheaf'' for
a locally Noetherian Cohen--Macaulay scheme.
\end{rem}

\begin{lem}\label{lem:DC-CM:ordinary}
Let \( X \) be a locally Noetherian
scheme admitting an ordinary dualizing complex \( \SR^{\bullet} \).
Let \( Z^{(i)} \) be the support of the cohomology sheaf \( \SH^{i}(\SR^{\bullet}) \)
for any \( i \in \BZZ \). Then, \( Z^{(i)} = \emptyset \) for any \( i < 0 \),
\( Z^{(0)} = X \), and
the following hold for any \( x \in X \)\emph{:}
\begin{enumerate}
\item \label{lem:DC-CM:ordinary:1}
\( x \not\in Z^{(i)} \) for any \( i > \dim \SO_{X, x} \)\emph{;}

\item \label{lem:DC-CM:ordinary:2}
\( \depth \SO_{X, x} = \dim \SO_{X, x} - \sup\{ j \mid x \in Z^{(j)}\} \)\emph{;}

\item \label{lem:DC-CM:ordinary:3}
for an integer \( k \geq 1 \), \( X \) satisfies \( \bfS_{k} \) at \( x \) if and only if
\[ \Codim_{x}(Z^{(j)}, X) \geq k + j\]
for any \( j > 0 \). This is also equivalent to\emph{:}
\[ \dim_{x} Z^{(j)} \leq \dim \SO_{X, x} - k - j  \]
for any \( j > 0 \)\emph{;}

\item \label{lem:DC-CM:ordinary:4}
\( \SO_{X, x} \) is Cohen--Macaulay if and only if
\( x \not\in Z^{(j)} \) for any \( j > 0 \).
\end{enumerate}
\end{lem}

\begin{proof}
Now,
\( d(x) = \dim \SO_{X, x} \) for the codimension function \( d \colon X \to \BZZ \)
associated with \( \SR^{\bullet} \), and
\( X \) is locally equi-dimensional by Lemma~\ref{lem:CodimFunction:constant}.
Thus, applying Proposition~\ref{prop:DC-CM} to \( \SF = \SO_{X} \), we have
the assertions except \eqref{lem:DC-CM:ordinary:3}.
The assertion \eqref{lem:DC-CM:ordinary:3} is obtained
by Corollary~\ref{cor:lem:DC-CM:supportSk}\eqref{cor:lem:DC-CM:supportSk:3}
(cf.\ Property~\ref{ppty:dim-codim}).
\end{proof}

\begin{remn}
Let \( X \) be a connected locally Noetherian scheme with a dualizing complex \( \SR^{\bullet} \)
such that \( \SH^{i}(\SR^{\bullet}) = 0 \) for any \( i < 0 \) and \( \SH^{0}(\SR^{\bullet}) \ne 0 \).
The sheaf \( \SH^{0}(\SR^{\bullet}) \) is called  the ``canonical module'' in many articles.
But as in Example~\ref{exam:PlaneLine} below, the support of the sheaf \( \SH^{0}(\SR^{\bullet}) \)
is not always \( X \). This is one of the reasons
why we do not consider \( \SH^{0}(\SR^{\bullet}) \) as the dualizing sheaf
for arbitrary locally Noetherian schemes.
\end{remn}

\begin{exam}\label{exam:PlaneLine}
Let \( P \) be a polynomial ring \( \Bbbk[\xtt, \ytt, \ztt] \)
of three variables over a field \( \Bbbk \). For the ideals \( I = (\xtt, \ytt) \) and \( J = (\ztt) \)
of \( P \), we set \( A := P/IJ \) and
\( R^{\bullet} := \RHom_{P}(A, P[1])\).
Then, we have a Noetherian affine scheme \( X = \Spec A\) and a dualizing complex
\( \SR^{\bullet} \) on \( X \) associated with \( R^{\bullet} \)
(cf.\ Example~\ref{exam:RDembeddable} below).
The \( X \) is a union of a plane \( \Spec P/J \) and a line \( \Spec P/I \)
in the three-dimensional affine
space \( \Spec P \isom \BAA^{3}_{\Bbbk}\),
where the plane and the line intersect at the origin \( O \)
corresponding to the maximal ideal \( (\xtt, \ytt, \ztt) \).
Note that the local ring \( \SO_{X, O}  \) is not equi-dimensional.
We can calculate the cohomology modules of \( R^{\bullet} \) as
\[ \OH^{i}(R^{\bullet}) \isom \Ext^{i+1}_{P}(A, P) \isom
\begin{cases}
0, & \text{ for any } i < 0 \text{ and } i > 1; \\
P/J, &\text{ for } i = 0; \\
P/I, &\text{ for } i = 1,
\end{cases}\]
by the free resolution
\[ 0 \to P \xrightarrow{g} P^{\oplus 2} \xrightarrow{f} P \to A \to 0, \]
where \( f \) and \( g \) are defined by
\[ f(a, b) = \xtt\ztt a + \ytt\ztt b
\quad \text{and} \quad
g(c) = (\ytt c, - \xtt c) \]
for any \( a \), \( b \), and \( c \in P \).
Consequently, \( \Supp \SH^{0}(\SR^{\bullet})  = \Spec R/J \) is a proper subset of \( X \).
\end{exam}

\begin{lem}\label{lem:DSSk}
Let \( X \) be a locally Noetherian scheme admitting an ordinary
dualizing complex \( \SR^{\bullet} \).
We set
\[
\SG^{(j)}_{\leq b} := \SBExt^{j}_{\SO_{X}}(\tau^{\leq b}(\SR^{\bullet}), \SR^{\bullet})
\quad  \text{ and } \quad
\SG^{(j)}_{\geq b} := \SBExt^{j}_{\SO_{X}}(\tau^{\geq b}(\SR^{\bullet}), \SR^{\bullet})
\]
for integers \( b \geq 0 \) and \( j \),
where \( \tau^{\leq b} \) and \( \tau^{\geq b} \) stand for the truncations
of a complex \emph{(}cf.\ Notation and conventions, \eqref{nc:truncshift}\emph{)}.
Then, the following hold\emph{:}
\begin{enumerate}

\item \label{lem:DSSk:1}
One has\emph{:}  \rule[-2ex]{0ex}{5ex}
\( \SG^{(0)}_{\geq 0} \isom \SO_{X} \) and
\( \SG^{(i)}_{\geq 0} = 0 \) for any  \( i \ne 0 \).

\item \label{lem:DSSk:2}
There exist an exact sequence
\[ 0 \to \SG^{(-1)}_{\leq b} \to \SG^{(0)}_{\geq b+1} \to \SO_{X} \to
\SG_{\leq b}^{(0)} \to \SG^{(1)}_{\geq b+1} \to 0\]
and an isomorphism
\[ \SG^{(j)}_{\leq b} \isom \SG^{(j+1)}_{\geq b+1} \]
for any \( j \ne \{0, -1\} \).
Moreover, \( \SG^{(j)}_{\leq 0}  = 0\) \rule[-2ex]{0ex}{2ex} for any \( j < 0 \).
\item  \label{lem:DSSk:3}
For any integers \( b \geq 0 \) and \( j \),
one has\emph{:}
\begin{itemize}
\item \rule[-2ex]{0ex}{5ex}
\(\Codim( \Supp \SG^{(j)}_{\geq b}, X)  \geq j + b \) for any \( j \in \BZZ \),

\item  \rule[-2ex]{0ex}{2ex}
\(  \Codim( \Supp \SG^{(j)}_{\leq b}, X) \geq j + b + 2 \) for any \( j \ne 0 \), and

\item  \rule[-2ex]{0ex}{2ex}
\( \Codim( \Supp \SG^{(0)}_{\leq b}, X) = 0 \).
\end{itemize}

\item \label{lem:DSSk:4}
If \( X \) satisfies \( \bfS_{k} \) for some \( k \geq 1 \), then
\begin{itemize}
\item
\(\SG^{(j)}_{\geq b} = 0\) for any \( b > 0 \)  and \( j < k\), \rule[-2ex]{0ex}{5ex} and
\item \( \SG^{(i)}_{\leq b} = 0\)  for any \(0 < i < k-1\).
\end{itemize}
\end{enumerate}
\end{lem}

\begin{proof}
We have a quasi-isomorphism \( \SR^{\bullet} \isom_{\qis} \tau^{\geq 0}(\SR^{\bullet}) \)
by Lemma~\ref{lem:DC-CM:ordinary}. Hence,
the first assertion \eqref{lem:DSSk:1} interprets the quasi-isomorphism
\[ \SRHom_{\SO_{X}}(\SR^{\bullet}, \SR^{\bullet}) \isom_{\qis} \SO_{X}.  \]
The exact sequence and the isomorphism in the second assertion \eqref{lem:DSSk:2}
are derived from the canonical distinguished triangle
\[ \cdots \to \tau^{\leq b}(\SR^{\bullet}) \to \SR^{\bullet}
\to \tau^{\geq b+1}(\SR^{\bullet})
\to \tau^{\leq b}(\SR^{\bullet})[1] \to \cdots. \]
The last vanishing in \eqref{lem:DSSk:2} is expressed as
\( \SBExt^{j}_{\SO_{X}}(\SL, \SR^{\bullet}) = 0 \)
for any \( j < 0 \), where \( \SL := \SH^{0}(\SR^{\bullet}) \), and
this is a consequence of Proposition~\ref{prop:DC-CM}\eqref{prop:DC-CM:1} applied to \( \SF = \SL \)
with the property \( X = \Supp \SL \) shown in Lemma~\ref{lem:DC-CM:ordinary}.
For the remaining assertions \eqref{lem:DSSk:3} and \eqref{lem:DSSk:4},
it is enough to consider only the sheaves \( \SG^{(j)}_{\geq b} \).
In fact, by \eqref{lem:DSSk:2}, we have an injection
\( \SG^{(i)}_{\leq b} \to \SG^{(i+1)}_{\geq b+1}\) for any \( i \ne 0 \),
and an exact sequence \( \SG^{(0)}_{\geq b+1} \to \SO_{X} \to \SG^{(0)}_{\leq b} \),
where \( \Codim(\Supp \SG^{(0)}_{\geq b+1}, X) > 0 \) by
the assertion for \( \SG^{(0)}_{\geq b+1} \).
Hence,
\[ \Codim(\Supp \SG^{(i)}_{\leq b}, X) \geq \Codim(\Supp \SG^{(i+1)}_{\geq b+1}, X) \]
for any \( i \ne 0 \) and \( \Codim(\Supp \SG^{(0)}_{\leq b}, X) = 0 \).
In order to prove \eqref{lem:DSSk:3} and \eqref{lem:DSSk:4} for \( \SG^{(j)}_{\geq b} \),
let us consider the spectral sequence
\begin{equation}\label{eq:1|lem:DSSk}
\SE_{2}^{p, q} = \SBExt^{p}_{\SO_{X}}(\SH^{-q}(\tau^{\geq b}(\SR^{\bullet})), \SR^{\bullet})
\Rightarrow \SE^{p+q} = \SG^{(p+q)}_{\geq b}
\end{equation}
of \( \SO_{X} \)-modules (cf.\ Remark~\ref{rem:SpSeq:lem:DSSk} below).
Assume that \( (\SE_{2}^{p, q})_{x} \ne 0  \) for a point \( x \in X \).
Then, \( -q \geq b \), and
\begin{equation}\label{eq:2|lem:DSSk}
\dim \SO_{X, x} \geq p \geq \dim \SO_{X, x} - \dim \SH^{-q}(\SR^{\bullet})_{x} \\
= \Codim_{x}(\Supp \SH^{-q}(\SR^{\bullet}), X)
\end{equation}
by Proposition~\ref{prop:DC-CM}\eqref{prop:DC-CM:1},
since \( d(x) = \dim \SO_{X, x} \) for the codimension function \( d \) of \( \SR^{\bullet} \).
In particular, \( p + q \leq \dim \SO_{X, x} - b \).
Therefore, if \( j + b > \dim \SO_{X, x} \), then \( x \not\in \Supp \SG^{(j)}_{\geq b}\),
since \( (\SE_{2}^{p, q})_{x} = 0  \)
for any integers \( p \), \( q \) with \( p+q = j \).
Thus, we have \eqref{lem:DSSk:3}.
Assume that \( X \) satisfies \( \bfS_{k} \).
If \( (\SE^{p, q}_{2})_{x} \ne 0\) and \( q < 0 \), then \( p + q \geq k \) by \eqref{eq:2|lem:DSSk},
since
\[ \Codim_{x}(\Supp \SH^{-q}(\SR^{\bullet}), X) \geq k - q\]
for any \( q < 0 \) by Lemma~\ref{lem:DC-CM:ordinary}\eqref{lem:DC-CM:ordinary:3}.
Hence, \( \SG^{(j)}_{\geq b} = 0 \)
for any \( b > 0 \) and \( j < k \), since \( \SE_{2}^{p, q} = 0  \)
for any integers \( p \), \( q \) with \( p+q = j \).
This proves \eqref{lem:DSSk:4},
and we are done.
\end{proof}

\begin{rem}\label{rem:SpSeq:lem:DSSk}
The spectral sequence \eqref{eq:1|lem:DSSk} is obtained by the same method as follows.
Let \( A \) be a commutative ring and let \( M^{\bullet} \) and \( N^{\bullet} \) be
complexes of \( A \)-modules such that \( N^{\bullet} \) is bounded below.
We shall construct a spectral sequence
\[ E_{2}^{p, q} = \BExt^{p}_{A}(\OH^{-q}(M^{\bullet}), N^{\bullet})
\Rightarrow E^{p+q} = \BExt^{p+q}_{A}(M^{\bullet}, N^{\bullet}), \]
where \( \BExt^{p}_{A} \) denotes the \( p \)-th hyper-ext group.
Since there is a quasi-isomorphism from \( N^{\bullet} \) into 
a complex of injective \( A \)-modules bounded below,
we may assume that \( N^{\bullet} \) itself is a complex of injective \( A \)-modules bounded below.
We consider a double complex \( K^{\bullet, \bullet}\) defined by \( K^{p, q} = \Hom_{A}(M^{-p}, N^{q}) \)
for \( p \), \( q \in \BZZ \) with the differentials \( d_{\mathrm{I}} \colon K^{p, q} \to K^{p+1, q}  \) 
and \( d_{\mathrm{II}} \colon K^{p, q} \to K^{p, q+1}  \), which are induced from the differentials
\( M^{-p-1} \to M^{-p} \) and \( N^{q} \to N^{q+1} \), respectively.
Then, \( \BExt^{k}(M^{\bullet}, N^{\bullet}) \) is isomorphic to the \( k \)-th cohomology group of
the total complex \( K^{\bullet} \) defined by \( K^{n} = \prod\nolimits_{p + q = n} K^{p, q} \)
(cf.\ \cite[I, Th.~6.4]{ResDual}).
Moreover, we have
\[ \OH^{q}_{\mathrm{I}}(K^{\bullet, p}) \isom \Hom_{A}(\OH^{-q}(M^{\bullet}), N^{p}) \]
for any \( p \) and \( q \), since \( N^{p} \) is now assumed to be injective.
Thus, we have the spectral sequence above as the well-known spectral sequence
\( \OH^{p}_{\mathrm{II}}\OH^{q}_{\mathrm{I}}(K^{\bullet, \bullet}) \Rightarrow \OH^{p+q}(K^{\bullet})\)
associated with the double complex \( K^{\bullet, \bullet} \).
\end{rem}

\begin{cor}\label{cor:lem:DSSk}
Let \( X \) be a locally Noetherian scheme admitting
an ordinary dualizing complex \( \SR^{\bullet} \).
For a point \( x \in X \) and for an integer \( b \geq 0 \),
the vanishing
\[ \BHH^{i}_{x}(\tau^{\leq b}(\SR^{\bullet})_{x}) = 0 \]
holds for any
\( i < b + 2 \) except \( i = \dim \SO_{X, x} \),
where \( \BHH^{i}_{x}(M^{\bullet}) \) stands for the local cohomology group
at the maximal ideal \( \GM_{x} \) for a complex \( M^{\bullet} \) of
\( \SO_{X, x} \)-modules bounded below.
\end{cor}

\begin{proof}
By the local duality theorem \cite[V, Th.~6.2]{ResDual}, we have
\[ \BHH^{i}_{x}(\tau^{\leq b}(\SR^{\bullet})_{x}) \isom
\Hom_{\SO_{X, x}}(\BExt^{-i}_{\SO_{X, x}}(\tau^{\leq b}(\SR^{\bullet})_{x},
\SR^{\bullet}_{x}[d(x)]), I_{x})\]
for the injective \( \SO_{X, x} \)-module \( I_{x} = \BHH^{d(x)}_{x}(\SR^{\bullet}_{x}) \),
where \( d(x) = \dim \SO_{X, x} \).
In particular,
\[ \BHH^{i}_{x}(\tau^{\leq b}(\SR^{\bullet})_{x}) \ne 0 \quad \text{if and only if} \quad
x \in \Supp \SG^{(d(x) - i)}_{\leq b}. \]
If \( d(x) - i \ne 0 \),
then the non-vanishing above implies that
\[ d(x) = \dim \SO_{X, x} \geq \Codim(\Supp \SG^{(d(x) - i)}_{\leq b}, X)
\geq d(x) - i + b + 2 \]
by Lemma~\ref{lem:DSSk}\eqref{lem:DSSk:3}. Thus, we have the vanishing for
\( i < b + 2 \) except \( i = d(x) \).
\end{proof}

\begin{prop}\label{prop:DSSk}
Let \( X \)  be a locally Noetherian scheme admitting an ordinary
dualizing complex \( \SR^{\bullet} \). Then, the dualizing sheaf
\( \SL = \SH^{0}(\SR^{\bullet}) \) satisfies \( \bfS_{2} \) and \( \Supp \SL = X \).
If \( X \) satisfies \( \bfS_{2} \), then
\( \SHom_{\SO_{X}}(\SL, \SL) \isom \SO_{X} \).
If \( X \) satisfies \( \bfS_{3} \), then
\( \SExt^{1}_{\SO_{X}}(\SL, \SL) = 0 \).
\end{prop}

\begin{proof}
We have \( \Supp \SL = X \) by Lemma~\ref{lem:DC-CM:ordinary}.
Hence,
\[ \dim \SL_{x} = \Codim(\overline{\{x\}}, \Supp \SL) = \dim \SO_{X, x} \]
for any \( x \in X \).
Applying Corollary~\ref{cor:lem:DC-CM:supportSk}\eqref{cor:lem:DC-CM:supportSk:1}
to \( \SL \), where \( c = d(x) - \dim \SL_{x} = 0\),
we see that \( \SL \) satisfies \( \bfS_{2} \)
by Lemma~\ref{lem:DSSk}\eqref{lem:DSSk:3},
since
\( \Codim(\Supp \SG^{(i)}_{\leq 0}, X) \geq i + 2 \) for any \( i > 0 \),
where \( \SG^{(i)}_{\leq 0} = \SExt^{i}_{\SO_{X}}(\SL, \SR^{\bullet}) \).
For the remaining assertions, we assume that \( X \) satisfies \( \bfS_{2} \) or \( \bfS_{3} \).
Note that we have an isomorphism
\[ \SHom_{\SO_{X}}(\SL, \SL) \isom \SBExt^{0}_{\SO_{X}}(\SL, \SR^{\bullet}) = \SG^{(0)}_{\leq 0} \]
and an injection
\[ \SExt^{1}_{\SO_{X}}(\SL, \SL) \to
\SBExt^{1}_{\SO_{X}}(\SL, \SR^{\bullet}) = \SG^{(1)}_{\leq 0} \]
by \( \SL \isom_{\qis} \tau^{\leq 0}(\SR^{\bullet}) \).
If \( X \) satisfies \( \bfS_{2} \), then \( \SG^{(0)}_{\geq 1} = \SG^{(1)}_{\geq 1} = 0 \)
by Lemma~\ref{lem:DSSk}\eqref{lem:DSSk:4},
and hence, \( \SO_{X} \isom \SG^{(0)}_{\leq 0}  \) by Lemma~\ref{lem:DSSk}\eqref{lem:DSSk:2};
thus, \( \SO_{X} \isom \SHom_{\SO_{X}}(\SL, \SL) \).
If \( X \) satisfies \( \bfS_{3} \), then \( \SG^{(1)}_{\leq 0} = 0\) by
Lemma~\ref{lem:DSSk}\eqref{lem:DSSk:4}, and consequently,
\( \SExt^{1}_{\SO_{X}}(\SL, \SL) = 0 \).
\end{proof}

\begin{remn}[$\bfS_{2}$-ification]
For a locally Noetherian scheme \( X \) admitting an ordinary
dualizing complex \( \SR^{\bullet} \) and for the dualizing sheaf
\( \SL = \SH^{0}(\SR^{\bullet}) \),
we consider the coherent \( \SO_{X} \)-module \( \SA := \SHom_{\SO_{X}}(\SL, \SL) \).
Then, we can show:
\begin{itemize}
\item \( \SA \) has a structure of \( \SO_{X} \)-algebra,

\item \( \SO_{X} \to \SA \) is an isomorphism on
the \( \bfS_{2} \)-locus \( \bfS_{2}(X) \) (cf.\ Definition~\ref{dfn:SkCMlocus}), and

\item  \( \SA \) satisfies \( \bfS_{2} \).
\end{itemize}
Therefore, the finite morphism
\( \SSpec_{X} \SA \to X \) is regarded as the so-called
``\( \bfS_{2} \)-ification'' of \( X \)
(cf.\ \cite[IV, (5.10.11), Prop.~(5.11.1)]{EGA}, \cite[Prop.~2]{Aoyama1},
\cite[Th.~3.2]{Aoyama2}, \cite[Prop.~2.7]{HH}).
Three properties above are shown as follows:
We know that \( \SL \) satisfies \( \bfS_{2} \),
\( U := \bfS_{2}(X) \) is an open subset by Proposition~\ref{prop:CMlocus},
and that \( \SO_{X} \to \SA\)
is an isomorphism on \( U \) by
Proposition~\ref{prop:DSSk}.
In particular, \( \SA \isom j_{*}(\SA|_{U}) \)
for the open immersion \( j \colon U \injmap X \),
since it is expressed as
\[ \SA = \SHom_{\SO_{X}}(\SL, \SL) \to j_{*}(\SA|_{U}) \isom
\SHom_{\SO_{X}}(\SL, j_{*}(\SL|_{U})). \]
Thus, \( \SA \) satisfies \( \bfS_{2} \) by
Corollary~\ref{cor:basicS1S2}, and consequently,
\( \SA \isom j_{*}\SO_{U}\) has an \( \SO_{X} \)-algebra structure.
\end{remn}

\begin{cor}\label{cor:CMCodOne}
Let \( X \) be a locally Noetherian scheme
admitting a dualizing complex \( \SR^{\bullet} \),
and set \( \SL :=  \SH^{0}(\SR^{\bullet}) \).
Let \( X^{\circ} \subset X \) be an open subset such that
\[ \Codim(X \setminus X^{\circ}, X) \geq 1
\quad \text{ and } \quad
\SR^{\bullet}|_{X^{\circ}} \isom_{\qis} \SL|_{X^{\circ}}. \]
Then, \( \SR^{\bullet} \) is ordinary and \( \SL \) satisfies \( \bfS_{2} \).
In particular, if \( \Codim(X \setminus X^{\circ}, X) \geq 2  \), then
\( \SL \isom j_{*}(\SL|_{X^{\circ}}) \)
for the open immersion \( j \colon X^{\circ} \injmap X \).
\end{cor}

\begin{proof}
It is enough to prove that \( \SR^{\bullet} \) is ordinary.
In fact, if so, then the dualizing sheaf \( \SL \) satisfies \( \bfS_{2} \)
by Proposition~\ref{prop:DSSk}, and we have the isomorphism
\( \SL \isom j_{*}(\SL|_{X^{\circ}}) \) by
Corollary~\ref{cor:basicS1S2}
when \( \Codim(X \setminus X^{\circ}, X) \geq 2 \).
Let \( d \colon X \to \BZZ \) be the codimension function
associated with \( \SR^{\bullet} \). Then, \( d(x) = \dim \SO_{X, x} \)
for any \( x \in X^{\circ} \) by Proposition~\ref{prop:DC-CM}\eqref{prop:DC-CM:5}
applied to \( \SF = \SO_{X^{\circ}} \).
For a point \( x \in X \setminus X^{\circ}\),
we have a generic point \( y \) of \( X \) such that \( x \in \overline{\{y\}} \) and
\( \Codim(\overline{\{x\}}, \overline{\{y\}}) = \dim \SO_{X, x} \).
Then, \( d(y) = 0 \), since \( y \in X^{\circ} \), and we have
\[ d(x) = d(y) + \Codim(\overline{\{x\}}, \overline{\{y\}}) = \dim \SO_{X, x}.\]
Thus, \( \SR^{\bullet} \) is ordinary.
\end{proof}


\subsection{Twisted inverse image}
\label{subsect:twisted inverse}

We shall explain the twisted inverse image functor, the relative duality theorem,
and some base change theorems referring to \cite{ResDual}, \cite{Conrad}, \cite{Lipman09}.
Let \( f \colon Y \to T \) be a morphism of locally Noetherian schemes
which is locally of finite type.
In the theory of Grothendieck duality,
the ``twisted inverse image functor'' \( f^{!} \)
plays an essential role, which
is unfortunately defined only when some suitable conditions are satisfied
(cf.\ \cite[III, Th.~8.7]{ResDual},
\cite[VII, Cor.~3.4]{ResDual}, \cite[Appendix, no.~4]{ResDual},
\cite{Neeman}, \cite{Conrad}, \cite{Lipman09}).
However, \( f^{!}\SO_{T} \) has a unique meaning at least
locally on \( Y \), where \( f^{!}\SO_{T} \) is expressed as
a complex of \( \SO_{Y} \)-modules with coherent cohomology
which vanish in sufficiently negative degree, i.e., \( f^{!}\SO_{T} \in \bfD^{+}_{\coh}(Y) \).
We write \( \omega^{\bullet}_{Y/T} := f^{!}\SO_{T} \) whenever
\( f^{!}\SO_{T} \) is defined, and call it
the \emph{relative dualizing complex for} \( Y/T \)
(or, with respect to \( f \)).
When \( T = \Spec A \), we write \( \omega^{\bullet}_{Y/A} \)
for \( \omega^{\bullet}_{Y/\Spec A} \).

\begin{exam}\label{exam:RDembeddable}
For a scheme \( S \), an \( S \)-morphism \( f \colon Y \to T \)
of locally Noetherian schemes over \( S \) is called an \( S \)-\emph{embeddable morphism}
if \( f = p \circ i \) for a finite morphism \( i \colon Y \to P \times_{S} T \)
and the second projection \( p \colon P \times_{S} T \to T \) for
a locally Noetherian \( S \)-scheme \( P \) such that
\( P \to S \) is a smooth separated morphism of pure relative dimension
(cf.\ \cite[(2.8.1)]{Conrad}, \cite[III, p.~189]{ResDual}).
When \( S = T \), an \( S \)-embeddable morphism is called simply an \emph{embeddable morphism}.
There is a theory of \( f^{!} \colon \bfD^{+}_{\qcoh}(T) \to \bfD^{+}_{\qcoh}(Y) \)
(resp.\ \( f^{!} \colon \bfD^{+}_{\coh}(T) \to \bfD^{+}_{\coh}(Y) \))
for the \( S \)-embeddable morphisms \( f \colon Y \to T \)
of locally Noetherian \( S \)-schemes as in \cite[III, Th.~8.7]{ResDual}
(cf.\ \cite[Th.~2.8.1]{Conrad}).
For a complex \( \SG^{\bullet} \in \bfD^{+}_{\qcoh}(T) \),
if \( f \) is separated and smooth of pure relative dimension \( d \)
(cf.\ Definition~\ref{dfn:PureRelDim}),
then
\[ f^{!}(\SG^{\bullet}) = \varOmega_{Y/T}^{d}[d] \otimes^{\bfL}_{\SO_{Y}} \bfL f^{*}(\SG^{\bullet}), \]
and if \( f \) is a finite morphism, then \( f^{!}(\SG^{\bullet}) \) is defined by
\[ \bfR f_{*} (f^{!}(\SG^{\bullet})) = \SRHom_{\SO_{T}}(f_{*}\SO_{Y}, \SG^{\bullet}).
\]
In the both cases of \( f \) above, \( f^{!}(\SG^{\bullet}) \in \bfD_{\coh}^{+}(Y) \)
if \( \SG^{\bullet} \in \bfD_{\coh}^{+}(T) \).
If \( f = g \circ h\) for two \( S \)-embeddable morphisms \( h \colon Y \to Z \)
and \( g \colon Z \to T \), then \( f^{!} \isom h^{!} \circ g^{!} \)
as functors \( \bfD^{+}_{\qcoh}(T) \to \bfD^{+}_{\qcoh}(Y) \)
(resp.\ \( \bfD^{+}_{\coh}(T) \to \bfD^{+}_{\coh}(Y) \)).
\end{exam}

\begin{exam}\label{exam:RDConrad}
Let \( f \colon Y \to T \) be a morphism of finite type
between Noetherian schemes.
Then, the dimensions of fibers are bounded.
Assume that \( T \) admits a dualizing complex \( \SR^{\bullet}_{T} \).
In this situation, we have the twisted inverse image functor
\( f^{!} \colon \bfD_{\coh}^{+}(T) \to \bfD_{\coh}^{+}(Y) \)
as follows
(cf.\ \cite[VI]{ResDual}, \cite[\S 3]{Conrad}).
For the dualizing complex \( \SR_{T}^{\bullet} \) of \( T \),
we have the corresponding \emph{residual complex}
\( E(\SR_{T}^{\bullet}) \) on \( T \)
(cf.\ \cite[VI, Prop.~1.1]{ResDual}, \cite[Lem.~3.2.1]{Conrad})
and the ``twisted inverse image''
\( f^{\triangle}(E(\SR^{\bullet}_{T})) \) on \( Y \) as
a residual complex on \( Y \) (cf.\ \cite[VI, Th.~3.1, Cor.\ 3.5]{ResDual}, \cite[\S 3.2]{Conrad}),
which corresponds to a dualizing complex
\[ \SR^{\bullet}_{Y} := f^{!}(\SR^{\bullet}_{T}) :=
Q(f^{\triangle}(E(\SR^{\bullet}_{T}))) \]
of \( Y \) (cf.\
\cite[VI, Prop.~1.1, Remarks in p.~306]{ResDual},  \cite[\S 3.3]{Conrad}).
Then, one can define \( f^{!} \colon \bfD_{\coh}^{+}(T) \to \bfD_{\coh}^{+}(Y)\) by
\[ f^{!}(\SG^{\bullet}) = \GD_{Y}(\bfL f^{*}(\GD_{T}(\SG^{\bullet}))), \]
where \( \GD_{Y} \) and \( \GD_{T} \) are the dualizing functors
defined by:
\[ \GD_{Y}(\SF^{\bullet}) := \SRHom_{\SO_{Y}}(\SF^{\bullet}, \SR_{Y}^{\bullet})
\quad \text{and} \quad
\GD_{T}(\SG^{\bullet}) := \SRHom_{\SO_{T}}(\SG^{\bullet}, \SR_{T}^{\bullet}).\]
The definition of \( f^{!} \) does not depend on the choice of
\( \SR^{\bullet}_{T} \) (cf.\ \cite[\S 3.3]{Conrad}), and
\( f^{!} \) satisfies expected compatible properties
in \cite[VII, Cor.~3.4]{ResDual} (cf.\ \cite[Th.~3.3.1]{Conrad}).
Moreover, when \( f \) is an embeddable morphism, then this \( f^{!} \) is isomorphic to
the functor \( f^{!} \) defined in Example~\ref{exam:RDembeddable}
(cf.\ \cite[VI, Th.~3.1, VII, Cor.~3.4]{ResDual}, \cite[\S 3.3]{Conrad}).
\end{exam}

The following is shown in \cite[V, Cor.~8.4, VI, Prop.~3.4]{ResDual} but with an error
concerning \( \pm \) (cf.\ \cite[(3.1.25), (3.2.4)]{Conrad}).

\begin{lem}
\label{lem:codimfunctionDiff}
Let \( f \colon Y \to T\) be a morphism of finite type between Noetherian schemes
such that \( T \) admits a dualizing complex \( \SR^{\bullet}_{T} \).
Let \( \SR^{\bullet}_{Y} \) be the induced dualizing complex \( f^{!}(\SR_{T}^{\bullet}) \) of \( Y \).
Let \( d_{T} \colon T \to \BZZ \) and \( d_{Y} \colon Y \to \BZZ \) be the codimension
functions associated with \( \SR^{\bullet}_{T} \) and \( \SR^{\bullet}_{Y} \), respectively.
Then,
\[ d_{Y}(y) = d_{T}(t) - \transdeg \Bbbk(y)/\Bbbk(t) \]
for any \( y \in Y \) with \( t = f(y) \), where \( \Bbbk(t) \) and \( \Bbbk(y) \)
denote the residue fields of \( \SO_{T, t} \) and \( \SO_{Y, y} \), respectively.
\end{lem}

\begin{proof}
Since the assertion is local on \( Y \), we may assume that \( Y \to T \) is an embeddable morphism.
Hence, it is enough to prove assuming that \( f \) is a finite morphism
or a smooth and separated morphism.
Assume first that \( f \) is finite.
Then,
\[ \bfR f_{*}\SRHom_{\SO_{Y}}(\SF, \SR^{\bullet}_{Y}) \isom
\SRHom_{\SO_{T}}(f_{*}\SF, \SR^{\bullet}_{T}) \]
for any coherent \( \SO_{Y} \)-module \( \SF \)
by \cite[III, Th.~6.7]{ResDual} (cf. Theorem~\ref{thm:DualityProper} below).
Applying this to \( \SF = \SO_{Z} \) for the closed subscheme
\( Z = \overline{\{y\}}\) with reduced structure, and localizing \( Y \),
we have
\[ \RHom_{\SO_{Y, y}}(\Bbbk(y), (\SR_{Y}^{\bullet})_{y}) \isom_{\qis}
\RHom_{\SO_{T, t}}(\Bbbk(y), (\SR_{T}^{\bullet})_{t}). \]
Since \( \Bbbk(y) \) is a finite-dimensional \( \Bbbk(t) \)-vector space,
we have \( \transdeg \Bbbk(y)/\Bbbk(t) = 0\) and \( d_{Y}(y) = d_{T}(t) \).
Thus, we are done in the case where \( f \) is finite.
Assume next that \( f \) is smooth and separated.
We may assume furthermore that \( f \) has pure relative dimension \( d \)
by localizing \( Y \). Then,
\[ \SR^{\bullet}_{Y} \isom_{\qis}
\varOmega^{d}_{Y/T}[d] \otimes^{\bfL}_{\SO_{Y}} \bfL f^{*}(\SR_{T}^{\bullet})
\]
as in Example~\ref{exam:RDembeddable}, and it implies that
\[ \SH^{i}(\SR^{\bullet}_{Y})_{y} \isom
\SH^{i+d}(\SR^{\bullet}_{T})_{t} \otimes_{\SO_{T, t}} \SO_{Y, y}, \]
since \( f \) is flat.
Here, \( \SH^{i}(\SR^{\bullet}_{Y})_{y} \ne 0 \) if and only if
\(\SH^{i+d}(\SR^{\bullet}_{T})_{t} \ne 0 \), since \( f \) is faithfully flat.
We know that
\[ d_{T}(t) - \dim \SO_{T, t} = \inf\{i \mid \SH^{i}(\SR^{\bullet}_{T})_{t} \ne 0\} \]
by \eqref{prop:DC-CM:1} and \eqref{prop:DC-CM:5} of Proposition~\ref{prop:DC-CM}.
The similar formula holds also for \( (Y, y) \) and \( \SR^{\bullet}_{Y} \).
Thus,
\[ d_{Y}(y) - \dim \SO_{Y, y} = d_{T}(t) - \dim \SO_{T, t} - d.\]
Since \( f \) is flat, we have
\[ \dim \SO_{Y, y} = \dim \SO_{T, t} + \dim \SO_{Y_{t}, y} \]
for the fiber \( Y_{t} = f^{-1}(t) \)
by \eqref{eq:fact:elem-flat:1}. Furthermore, we have
\[ d = \dim_{y} Y_{t} = \dim \SO_{Y_{t}, y} + \transdeg \Bbbk(y)/\Bbbk(t), \]
since \( Y_{t} \) is algebraic over \( \Bbbk(t) \) (cf.\ \cite[IV, Cor.~(5.2.3)]{EGA}).
Therefore,
\[ d_{Y}(y) = d_{T}(t) - d + \dim \SO_{Y, y} - \dim \SO_{T, t} =
d_{T}(t) - \transdeg \Bbbk(y)/\Bbbk(t). \qedhere\]
\end{proof}

\begin{dfn}[canonical dualizing complex]\label{dfn:DCalg}
Let \( X \) be an algebraic scheme over a field \( \Bbbk \), i.e., a \( \Bbbk \)-scheme
of finite type.
We define the \emph{canonical dualizing complex}
\( \omega_{X/\Bbbk}^{\bullet} \)
of \( X \)
to be the twisted inverse image \( f^{!}(\Bbbk) \)
for the structure morphism \( f \colon X \to \Spec \Bbbk \).
\end{dfn}

The dualizing complex \( \omega_{X/\Bbbk}^{\bullet} \) has the following property related to
Serre's conditions \( \bfS_{k} \).

\begin{lem}\label{lem:algschemeSk}
Let \( X \) be an \( n \)-dimensional algebraic scheme over a field \( \Bbbk \).
For an integer \( i \), let \( Z_{i} \) be the support of
\( \SH^{-i}(\omega_{X/\Bbbk}^{\bullet}) \). Then,
\( Z_{i} = \emptyset \) for any \( i > n \), and \( Z_{n}\) is the union of
irreducible components of \( X \) of dimension \( n \).
If \( X \) is equi-dimensional, then
\( \omega_{X/\Bbbk}^{\bullet}[-n] \) is an ordinary dualizing complex,
and the following hold for integers \( k \geq 1 \)\emph{:}
\( X \) satisfies \( \bfS_{k} \) if and only if \( \dim Z_{i} \leq i -k \)
for any \( i \ne n \).
\end{lem}

\begin{proof}
By Lemma~\ref{lem:codimfunctionDiff}, \( d(x) = -\transdeg \Bbbk(x)/\Bbbk \)
for the codimension function \( d \colon X
\linebreak 
\to \BZZ \)
associated with the dualizing complex \( \omega_{X/\Bbbk}^{\bullet} \)
(cf.\ Example~\ref{exam:RDConrad}). Moreover,
\begin{equation}\label{eq:lem:algschemeSk}
n \geq \dim_{x} X = \dim \SO_{X, x} + \transdeg \Bbbk(x)/\Bbbk
\end{equation}
by \cite[IV, Cor.~(5.2.3)]{EGA}.
Thus, \( d(x) - \dim \SO_{X, x} = - \dim_{x} X \geq -n \), and
\( \SH^{-i}(\omega_{X/\Bbbk}^{\bullet})
\linebreak 
= 0\) for any \( i > n  \) by
Proposition~\ref{prop:DC-CM}\eqref{prop:DC-CM:1} applied to the case where
\( (\SR^{\bullet}, \SF) = (\omega_{X/\Bbbk}^{\bullet}, \SO_{X}) \).
Thus, \( Z_{i} = \emptyset \) for any \( i > n \).
If \( \dim_{x} X = n \), then
\( \SH^{-n}(\omega_{X/\Bbbk}^{\bullet})_{x} \ne 0  \) by
Proposition~\ref{prop:DC-CM}\eqref{prop:DC-CM:5}.
If \( \dim_{x} X < n \), then
\( \SH^{-n}(\omega_{X/\Bbbk}^{\bullet})_{x} = 0  \) by
Proposition~\ref{prop:DC-CM}\eqref{prop:DC-CM:1}.
Therefore, \( Z_{n} \) is just
the union of irreducible components of \( X \) of dimension \( n \).

Assume that \( X \) is equi-dimensional, i.e., \( \dim_{x} X = n \) for any \( x \in X \).
Then, \( \omega_{X/\Bbbk}^{\bullet}[-n] \)
is an ordinary dualizing complex, since the associated codimension function is
\( x \mapsto d(x) + n = \dim \SO_{X, x} \).
Moreover, \( X \) is equi-codimensional,
since \( n = \dim_{z} X = \dim \SO_{X, z} \) for any closed point \( z \) of \( X \)
by \eqref{eq:lem:algschemeSk}.
Thus, the assertion on \( \bfS_{k} \) is a consequence of
Corollary~\ref{cor:lem:DC-CM:supportSk}\eqref{cor:lem:DC-CM:supportSk:2},
since \( d(x) - \dim \SO_{X, x} = -n \) for any \( x \in X \).
\end{proof}

\begin{dfn}[canonical sheaf]\label{dfn:canosheaf}
Let \( X \) be an algebraic scheme over a field \( \Bbbk \).
Assume that \( X \) is locally equi-dimensional.
This is satisfied for example when \( X \) satisfies \( \bfS_{2} \)
(cf.\ Fact~\ref{fact:S2}\eqref{fact:S2:1}).
Then, we define the canonical sheaf
\( \omega_{X/\Bbbk} \) by
\[  \omega_{X/\Bbbk}|_{X_{\alpha}} :=
\SH^{-\dim X_{\alpha}}(\omega_{X/\Bbbk}^{\bullet})|_{X_{\alpha}}\]
for any connected component \( X_{\alpha} \) of \( X \).
\end{dfn}

\begin{remn}
The canonical sheaf \( \omega_{X/\Bbbk} \)
is a dualizing sheaf of \( X \) in the sense of
Definition~\ref{dfn:ordinaryDC+DS}. In fact,
\[ \omega^{\bullet}_{X_{\alpha}/\Bbbk}[-\dim X_{\alpha}] =
\omega_{X/\Bbbk}^{\bullet}[-\dim X_{\alpha}]|_{X_{\alpha}}\]
is an ordinary dualizing complex of the connected component \( X_{\alpha} \)
by Lemma~\ref{lem:algschemeSk}.
In particular, if \( X \) is connected and Cohen--Macaulay, then
\( \omega^{\bullet}_{X/\Bbbk} \isom_{\qis} \omega_{X/\Bbbk}[\dim X] \).
\end{remn}

By Corollary~\ref{cor:CMCodOne}, we have:

\begin{cor}\label{cor:S2S2}
For an algebraic scheme \( X \) over \( \Bbbk \), if it is locally equi-dimensional,
then \( \omega_{X/\Bbbk} \) satisfies \( \bfS_{2} \).
\end{cor}

For a proper morphism \( f \colon Y \to T \) of Noetherian schemes,
we have the following general result
on the twisted inverse image functor \( f^{!} \),
which is derived from \cite[Th.~4.1.1]{Lipman09}:

\begin{thm}[Grothendieck duality for a proper morphism]\label{thm:DualityProper}
Let \( f \colon Y \to T \) be a proper morphism of Noetherian schemes.
Then, there is a triangulated functor
\( f^{!} \colon \bfD_{\qcoh}(T) \to \bfD_{\qcoh}(Y) \) which induces
\( \bfD_{\coh}^{+}(T) \to \bfD_{\coh}^{+}(Y)  \)
and which is right adjoint to the derived functor
\( \bfR f_{*} \colon \bfD_{\qcoh}(Y) \to \bfD_{\qcoh}(T) \) in the sense
that there is a functorial isomorphism
\[ \RHom_{\SO_{T}}(\bfR f_{*}(\SF^{\bullet}), \SG^{\bullet}) \isom_{\qis}
\RHom_{\SO_{Y}}(\SF^{\bullet}, f^{!}(\SG^{\bullet}))\]
for \( \SF^{\bullet} \in \bfD_{\qcoh}(Y) \) and \( \SG^{\bullet} \in \bfD_{\qcoh}(T) \).
\end{thm}

\begin{remn}
In \cite[Th.~4.1.1]{Lipman09}, the existence of a similar right adjoint \( f^{\times} \)
is proved for a quasi-compact and quasi-separated morphism
\( f \colon Y \to T \)
of quasi-compact and quasi-separated schemes \( Y \) and \( T \). When \( f \) is proper, it is written as
\( f^{!} \)
(cf.\ the paragraph just before \cite[Cor.\ 4.2.2]{Lipman09}).
By \cite[Th.~A]{Spa}, the total derived functor \( \RHom_{\SO_{X}} \) of \( \Hom_{\SO_{X}} \)
exists for any scheme \( X \) as a bi-functor \( \bfD(X)^{\op} \times \bfD(X) \to \bfD(\BZZ)  \),
and there exists also the total right derived functor \( \bfR f_{*} \colon \bfD(Y) \to \bfD(T) \) of
the direct image functor \( f_{*}\).
When \( f \colon Y \to T \) is a proper morphism of Noetherian schemes, we have:

\begin{itemize}
\item  \( \bfR f_{*} (\bfD_{\qcoh}(Y)) \subset \bfD_{\qcoh}(T) \) by
\cite[Prop.~3.9.1]{Lipman09},

\item  \( \bfR f_{*}(\bfD_{\coh}^{+}(Y)) \subset \bfD_{\coh}^{+}(T) \) by \cite[II, Prop.~2.2]{ResDual},
and
\item  \( \bfR f_{*}(\bfD_{\coh}^{-}(Y)) \subset \bfD_{\coh}^{-}(T) \)
by the explanation just before \cite[Cor.\ 4.2.2]{Lipman09}.
\end{itemize}
The functor \( f^{!} \) is bounded below (cf.\ \cite[Def.\ 11.1.1]{Lipman09}).
Thus, \( f^{!}(\bfD_{\qcoh}^{+}(T)) \subset \bfD_{\qcoh}^{+}(Y) \).
The inclusion \( f^{!}(\bfD_{\coh}^{+}(T)) \subset \bfD_{\coh}^{+}(Y) \) is proved
firstly by reducing to the case where \( T \) is the spectrum of a Noetherian local ring by the base change
isomorphism (cf.\ \cite[Cor.\ 4.4.3]{Lipman09}), and secondly by applying \cite[Lem.\ 1]{Verdier}.
\end{remn}

\begin{remn}
When \( T \) admits a dualizing complex (or a residual complex),
Theorem~\ref{thm:DualityProper} for \( \SG^{\bullet} \in \bfD_{\coh}^{+}(T) \)
is a consequence of \cite[VII, Cor.~3.4]{ResDual}.
In \cite[Th.~2]{ResDualDel}, Deligne has proved Theorem~\ref{thm:DualityProper}
for \( \SF^{\bullet} \in \bfD_{\coh}^{b}(Y) \) without
assuming the existence of dualizing complex of \( T \).
These results are summarized by Verdier as \cite[Th.~1]{Verdier},
which is almost the same as
Theorem~\ref{thm:DualityProper} in the case where \( T \) has finite Krull dimension.
Neeman \cite{Neeman} gives a new idea toward the proof of Theorem~\ref{thm:DualityProper}
by using Brown representability.
He generalizes Theorem~\ref{thm:DualityProper}
to the case where \( Y \) and \( T \) are only quasi-compact and separated schemes
but \( \bfD_{\qcoh}(T) \) and \( \bfD_{\qcoh}(Y) \) are replaced
with \( \bfD(\QCoh(\SO_{T})) \) and \( \bfD(\QCoh(\SO_{Y})) \),
respectively (cf.\ \cite[Exam.~4.2]{Neeman}).
Neeman's idea is used in Lipman's article \cite{Lipman09},
which contains generalizations of Theorem~\ref{thm:DualityProper}
to non-proper and non-Noetherian case.
\end{remn}

The sheafified form of the duality theorem is as follows (cf.\ \cite[Th.~4.2]{Lipman09}):

\begin{cor}\label{cor:thm:DualityProper}
In the situation of Theorem~\emph{\ref{thm:DualityProper}},
there exists a canonical quasi-isomorphism
\[ \bfR f_{*}\SRHom_{\SO_{Y}}(\SF^{\bullet}, f^{!}\SG^{\bullet})
\isom_{\qis} \SRHom_{\SO_{T}}(\bfR f_{*}\SF^{\bullet}, \SG^{\bullet}) \]
for any \( \SF^{\bullet} \in \bfD_{\qcoh}(Y) \) and \( \SG^{\bullet} \in \bfD_{\qcoh}(T) \).
\end{cor}

As a special case of Theorem~\ref{thm:DualityProper}, we have the following,
which is called the Serre duality theorem for coherent sheaves.

\begin{cor}\label{cor:SerreDual}
Let \( X \) be a projective scheme over a field \( \Bbbk \).
Then, there is a canonical quasi-isomorphism
\[ \RHom_{\SO_{X}}(\SF^{\bullet}, \omega^{\bullet}_{X/\Bbbk}) \isom_{\qis}
\RHom_{\Bbbk}(\bfR\Gamma(X, \SF^{\bullet}), \Bbbk)\]
for any \( \SF^{\bullet} \in \bfD^{+}_{\coh}(X) \). In particular,
\[ \BExt^{i}_{\SO_{X}}(\SF^{\bullet}, \omega^{\bullet}_{X/\Bbbk}) \isom
\Hom_{\Bbbk}(\BHH^{i}(X, \SF^{\bullet}), \Bbbk)\]
for any \( i \), where \( \BExt^{i} \) and \( \BHH^{i} \)
stands for the \( i \)-th hyper-Ext group and \( i \)-th hyper cohomology group, respectively.
\end{cor}

\begin{exam}\label{exam:DeligneVerdier}
Let \( f \colon Y \to T \) be a separated morphism of finite type
between Noetherian schemes.
By the Nagata compactification theorem (cf.\ \cite{Nagata1}, \cite{Nagata2},
\cite{Lu}, \cite{Conrad-NagataCpt}, \cite{Deligne-NagataCpt}),
\( f \) is expressed as the composite \( \pi \circ j\)
of an open immersion \( j \colon Y \injmap Z\) and
a proper morphism \( \pi \colon Z \to T\).
Using the functor \( \pi^{!} \colon \bfD_{\qcoh}^{+}(T) \to \bfD_{\qcoh}^{+}(Z) \)
in Theorem~\ref{thm:DualityProper},
we define the twisted inverse image functor
\( f^{!} \colon \bfD_{\qcoh}^{+}(T) \to \bfD_{\qcoh}^{+}(Y) \)
as \( \bfL j^{*} \circ \pi^{!} \).
This is well-defined up to functorial isomorphism, i.e., it is independent of
the choice of factorization \( f = \pi \circ j \), by
\cite[Th.~2]{ResDualDel}, \cite[Cor.~1]{Verdier}.
Deligne \cite{ResDualDel} defines a functor \( \bfR f_{!} \colon
\pro\bfD^{b}_{\coh}(Y) \to \pro\bfD^{b}_{\coh}(T)\) and
shows in \cite[Th.~2]{ResDualDel} that \( f^{!} \) above is a right adjoint of \( \bfR f_{!} \).
\end{exam}

\begin{fact}\label{fact:DeligneVerdierLipman}
The twisted inverse image functors in Example~\ref{exam:DeligneVerdier}
have the following properties. Let \( f \colon Y \to T \) be a separated morphism
of finite type between Noetherian schemes.
\begin{enumerate}
\item \label{fact:DeligneVerdierLipman:compo}
\emph{Let \( h \colon X \to Y \) be a separated morphism of
finite type from another Noetherian scheme \( X \). Then,
there is a functorial isomorphism
\( (f \circ h)^{!} \isom h^{!} \circ f^{!} \). }
\item \label{fact:DeligneVerdierLipman:smooth}
\emph{If \( f \colon Y \to T \) is a smooth morphism of pure relative dimension \( d \),
then \( f^{!}(\SG^{\bullet}) \isom_{\qis} \varOmega^{d}_{Y/T}[d] \otimes^{\bfL}_{\SO_{Y}}
\bfL f^{*} (\SG^{\bullet})\).}

\item \label{fact:DeligneVerdierLipman:DualizingCpx}
\emph{If \( T \) admits a dualizing complex, then
\( f^{!} \) is functorially isomorphic to the twisted inverse image functor
\( \GD_{Y} \circ \bfL f^{*} \circ \GD_{T} \) in Example}~\ref{exam:RDConrad}.
\item \label{fact:DeligneVerdierLipman:flatbc}
\emph{For a flat morphism \( g \colon T' \to T \) from a Noetherian scheme \( T' \),
let \( Y' \) be the fiber product \( Y \times_{T} T' \) and let
\(f' \colon  Y' \to T' \) and \( g' \colon Y' \to Y \) be the induced morphisms.
Then, \( g^{\prime *} \circ f^{!}  \isom f^{\prime !} \circ g^{*}\).}
\end{enumerate}
The property \eqref{fact:DeligneVerdierLipman:compo} is derived from the isomorphism
\( \bfR (f \circ h)_{!} \isom \bfR f_{!} \circ \bfR h_{!} \) shown in \cite[no.~3]{ResDualDel}.
This is also proved in \cite[Th.~4.8.1]{Lipman09}.
The properties  \eqref{fact:DeligneVerdierLipman:smooth} and
\eqref{fact:DeligneVerdierLipman:DualizingCpx} are proved in \cite[Th.~3, Cor.~3]{Verdier}
and \cite[(4.9.4.2), Prop.~4.10.1]{Lipman09}.
In order to prove the property \eqref{fact:DeligneVerdierLipman:flatbc},
we may assume that \( f \) is proper, and in this case, this is shown in
\cite[Cor.~4.4.3]{Lipman09} (cf.\ \cite[Th.~2]{Verdier}).
As a refinement of the property \eqref{fact:DeligneVerdierLipman:compo} above,
\( f \mapsto f^{!} \) can be regarded as a pseudo-functor, and
Lipman proves in \cite[Th.4.8.1]{Lipman09} the uniqueness of \( f \mapsto f^{!} \)
under three conditions corresponding to:
\begin{itemize}
\item  \( f^{!} \) is a right adjoint of \( \bfR f_{*} \)
when \( f \) is proper (Theorem~\ref{thm:DualityProper});

\item  The property \eqref{fact:DeligneVerdierLipman:smooth} above for \'etale \( f \);

\item  The property \eqref{fact:DeligneVerdierLipman:flatbc}
above for proper \( f \) and \'etale \( g \).
\end{itemize}
\end{fact}

\begin{fact}\label{fact:Lipman}
The following are also known for a flat separated morphism
\( f \colon Y \to T\) of finite type between Noetherian schemes:
\begin{enumerate}
\item \label{fact:Lipman:1}
The twisted inverse image
\( f^{!}\SO_{T} \) is an \emph{\( f \)-perfect} complex in \( \bfD_{\coh}(Y) \)
(cf.\ \cite[III, Prop.~4.9]{IllusieSGA6}, \cite[Th.~4.9.4]{Lipman09}).
For the definition of ``\( f \)-perfect,'' see \cite[III, D\'ef.~4.1]{IllusieSGA6}
(cf.\ Remark~\ref{rem:f-perf} below).
Note that a coherent \( \SO_{Y} \)-module flat over \( T \) is \( f \)-perfect.

\item \label{fact:Lipman:2}
For an \( f \)-perfect complex \( \SE^{\bullet} \),
\[ \GD_{Y/T}(\SE^{\bullet}) := \SRHom_{\SO_{Y}}(\SE^{\bullet}, f^{!}\SO_{T}) \]
is also \( f \)-perfect and the canonical morphism
\[ \SE^{\bullet} \to \GD_{Y/T}(\GD_{Y/T}(\SE^{\bullet})) \]
is a quasi-isomorphism (cf.\ \cite[III, Cor.~4.9.2]{IllusieSGA6}).
In particular,
\begin{equation}\label{eq:IllusieQIS}
\SO_{Y} \to \SRHom_{\SO_{Y}}(f^{!}\SO_{T}, f^{!}\SO_{T})
\end{equation}
is a quasi-isomorphism (cf.\ \cite[p.~234]{Lipman09}).

\item \label{fact:Lipman:3}
There is a quasi-isomorphism
\[ f^{!}(\SF^{\bullet}) \otimes_{\SO_{Y}}^{\bfL} \bfL f^{*}(\SG^{\bullet})
\isom_{\qis} f^{!}(\SF^{\bullet} \otimes_{\SO_{T}}^{\bfL} \SG^{\bullet})\]
for any \( \SF^{\bullet} \), \( \SG^{\bullet} \in \bfD_{\qcoh}^{+}(T) \)
with \( \SF^{\bullet} \otimes_{\SO_{T}}^{\bfL} \SG^{\bullet} \in \bfD_{\qcoh}^{+}(T)  \)
(cf.\ \cite[Th.~4.9.4]{Lipman09}).
In particular,
\begin{equation}\label{eq:Lipman}
f^{!}\SO_{T} \otimes^{\bfL}_{\SO_{Y}} \bfL f^{*}(\SG^{\bullet})
\isom_{\qis} f^{!}(\SG^{\bullet})
\end{equation}
for any \( \SG^{\bullet} \in \bfD_{\qcoh}^{+}(T) \).
Similar results are proved in
\cite[V, Cor.~8.6]{ResDual}, \cite[Cor.~2]{Verdier}, and \cite[Th.~5.4]{Neeman}.
\end{enumerate}
\end{fact}

\begin{rem}[{cf.\ \cite[III, Prop.~4.4]{IllusieSGA6}}]\label{rem:f-perf}
Let \( f \colon Y \to T \) be a morphism of finite type
between Noetherian schemes
and let \( \SF^{\bullet} \) be an object of \( \bfD_{\qcoh}(Y) \).
Assume that \( f \) is the composite \( g \circ i \) of
a closed immersion \( i \colon Y \to P \) and
a smooth separated morphism \( g \colon P \to T \).
Then, \( \SF^{\bullet} \) is \( f \)-perfect if and only if
\( \bfR i_{*}(\SF^{\bullet})\)
is perfect (cf.\ \cite[I, D\'ef.~4.7]{IllusieSGA6}), i.e.,
locally on \( P \),
it is quasi-isomorphic to a bounded complex of free \( \SO_{P} \)-modules.
\end{rem}

\begin{lem}\label{lem:BC0}
Let \( f \colon Y \to T \) be a flat separated morphism of finite type
between Noetherian schemes in which \( T \) admits a dualizing complex.
Let \( g \colon T' \to T \) be a finite morphism from another Noetherian scheme \( T' \).
For the fiber product \( Y' = Y \times_{T} T' \), let \( f \colon Y' \to T' \)
and \( g' \colon Y' \to Y \) be the projections. Thus, we have a Cartesian diagram\emph{:}
\[ \begin{CD}
Y' @>{g'}>> Y \\
@V{f'}VV @VV{f}V \\
T' @>{g}>> \phantom{.}T.
\end{CD}\]
In this situation, there is a natural quasi-isomorphism
\[ \bfL g^{\prime *} (f^{!}\SO_{T}) \isom_{\qis}
f^{\prime\, !}\SO_{T'}.  \]
\end{lem}

\begin{proof}
Let \( \GD_{T} \), \( \GD_{Y} \), \( \GD_{T'} \), and \( \GD_{Y'} \), respectively,
be the dualizing functors on \( T \), \( Y \), \( T' \), and \( Y' \)
defined by a dualizing complex on \( T \)
and their transforms by \( f^{!} \), \( g^{!} \),
and \( (f \circ g')^{!} \isom (g \circ f')^{!}\)
(cf.\ Fact~\ref{fact:DeligneVerdierLipman}\eqref{fact:DeligneVerdierLipman:compo})
as in Example~\ref{exam:RDConrad}.
For any \( \SG^{\bullet} \in \bfD_{\coh}^{+}(T') \),
we have
\begin{align*}
f^{!}(\bfR g_{*} (\SG^{\bullet})) &\isom_{\qis}
\GD_{Y} \circ \bfL f^{*} \circ \GD_{T} (\bfR g_{*} (\SG^{\bullet}))
\isom_{\qis} \GD_{Y} \circ \bfL f^{*} \circ \bfR g_{*}(\GD_{T'}(\SG^{\bullet})) \\
&\isom_{\qis} \GD_{Y} \circ \bfR g'_{*} \circ \bfL f^{\prime *}(\GD_{T'}(\SG^{\bullet}))
\isom_{\qis} \bfR g'_{*} \circ \GD_{Y'}(\bfL f^{\prime *}(\GD_{T'}(\SG^{\bullet}))) \\
&\isom_{\qis} \bfR g'_{*} (f^{\prime\, !}(\SG^{\bullet})),
\end{align*}
where we use the flat base change isomorphism:
\( \bfL f^{*} \circ \bfR g_{*} \isom_{\qis} \bfR g'_{*} \circ \bfL f^{\prime *} \)
(cf.\ Proposition~\ref{prop:flatbc}),
and the duality isomorphisms:
\(  \GD_{T} \circ \bfR g_{*} \isom_{\qis} \bfR g_{*} \circ \GD_{T'}\)
and \(  \GD_{Y} \circ \bfR g'_{*} \isom_{\qis} \bfR g'_{*} \circ \GD_{Y'} \)
for the finite morphisms \( g \) and \( g' \) (cf.\ Corollary~\ref{cor:thm:DualityProper}).
On the other hand, we have
\[
f^{!}(\bfR g_{*} (\SG^{\bullet})) \isom_{\qis}
f^{!}\SO_{T} \otimes^{\bfL}_{\SO_{Y}}
\bfL f^{*}(\bfR g_{*}\SG^{\bullet}) \isom_{\qis}
f^{!}\SO_{T} \otimes^{\bfL}_{\SO_{Y}}
\bfR g'_{*} (\bfL f^{\prime *}(\SG^{\bullet}))
\]
by the quasi-isomorphism \eqref{eq:Lipman} in Fact~\ref{fact:Lipman}
and by the flat base change isomorphism.
Substituting \( \SG^{\bullet} = \SO_{T'} \), we have a quasi-isomorphism
\[ f^{!}\SO_{T} \otimes^{\bfL}_{\SO_{Y}} \bfR g'_{*}\SO_{Y'}
\isom_{\qis} \bfR g'_{*}(f^{\prime\, !}\SO_{T'}). \]
It is associated with a morphism \( \bfL g^{\prime *} (f^{!}\SO_{T}) \to f^{\prime\, !}\SO_{T'} \)
in \( \bfD_{\coh}^{+}(Y') \) which induces a quasi-isomorphism by taking \( \bfR g'_{*} \).
Hence,
\( \bfL g^{\prime *} (f^{!}\SO_{T}) \isom_{\qis} f^{\prime\, !}\SO_{T'} \).
\end{proof}

\begin{cor}[{cf.\ \cite[Prop.\ 3.3(1)]{Pa}}]\label{cor:BC}
Let \( f \colon Y \to T \) be a flat separated morphism of finite type
between Noetherian schemes.
For a point \( t \in T \),
let \( \phi_{t} \colon \Spec \Bbbk(t) \to T\) be the
canonical morphism for the residue field \( \Bbbk(t) \), and
let \( \psi_{t} \colon Y_{t} = f^{-1}(t) \to Y \) be
the base change of \( \phi_{t} \)
by \( f \colon Y \to T \). Then, the canonical dualizing complex \( \omega^{\bullet}_{Y_{t}/\Bbbk(t)} \)
defined in Definition~\emph{\ref{dfn:DCalg}}
is quasi-isomorphic to \( \bfL \psi_{t}^{*} (f^{!}\SO_{T}) \).
\end{cor}

\begin{proof}
Let \( \Spec \SO_{T, t} \to T \) be the canonical morphism from the spectrum of
the local ring \( \SO_{T, t} \). Considering the completion \( \widehat{\SO}_{T, t} \)
of \( \SO_{T, t} \) and the surjection \( \widehat{\SO}_{T, t} \to \Bbbk(t) \) to the residue field,
we have a flat morphism
\[ \tau \colon T^{\flat} := \Spec \widehat{\SO}_{T, t} \to \Spec \SO_{T, t} \to T \]
and a closed immersion \( \iota \colon \Spec \Bbbk(t) \injmap T^{\flat} \).
Let \( Y^{\flat} \) be the fiber product \( Y \times_{T} T^{\flat} \) and let
\( f^{\flat} \colon Y^{\flat} \to T^{\flat} \) and \( \tau' \colon Y^{\flat} \to Y \)
be projections, which make a Cartesian diagram:
\[ \begin{CD}
Y^{\flat} @>{\tau'}>> Y \\
@V{f^{\flat}}VV @VV{f}V \\
T^{\flat} @>{\tau}>> \phantom{.}T.
\end{CD}\]
By Fact~\ref{fact:DeligneVerdierLipman}\eqref{fact:DeligneVerdierLipman:flatbc},
we have a quasi-isomorphism
\[ \bfL \tau^{\prime *} (f^{!}\SO_{T}) \isom_{\qis} f^{\flat !}\SO_{T^{\flat}}.\]
Hence, we may assume from the beginning that \( T = T^{\flat} \).
Then, \( \phi_{t} \) is the closed immersion \( \iota \).
Now, \( T \) admits a dualizing complex, since we have a surjection to \( \widehat{\SO}_{T, t} \)
from a complete regular local ring by Cohen's structure theorem.
Thus, we are done by Lemma~\ref{lem:BC0}.
\end{proof}


\subsection{Cohen--Macaulay morphisms and Gorenstein morphisms}
\label{subsect:CMmorGormor}

The notions of Cohen--Macaulay morphism and Gorenstein morphism are introduced in
\cite[IV, D\'ef.~(6.8.1)]{EGA} and  \cite[V, Ex.~9.7]{ResDual}.
By \cite[Sect.~3.5]{Conrad} or \cite[Th.~2.2.3]{Sastry},
one can define the relative dualizing sheaf for a Cohen--Macaulay morphism
(cf.\ Definition~\ref{dfn:reldualsheaf} below),
and prove a base change property (cf.\ Theorem~\ref{thm:basechange} below).
We shall explain these facts.

We have defined the notion of
Cohen--Macaulay morphism in Definition~\ref{dfn:SkCMmorphism}.
The notion of Gorenstein morphism is defined as follows.

\begin{dfn}[$\Gor(Y/T)$]\label{dfn:RelGorlocus}
Let \( Y \) and \( T \) be locally Noetherian schemes and
\( f \colon Y \to T \) a flat morphism locally of finite type.
We define
\[\Gor(Y/T) := \bigcup\nolimits_{t \in T} \Gor(Y_{t}),
\]
and call it the \emph{relative Gorenstein locus for}
\( f \). The flat morphism \( f \) is called  a \emph{Gorenstein morphism}
if \( \Gor(Y/T) = Y \).
\end{dfn}

\begin{remn}
The Gorenstein locus \( \Gor(Y/T) \) is open.
In fact, this is characterized as the maximal open
subset of the relative Cohen--Macaulay locus
\( Y^{\flat} = \CM(Y/T)\) on which
the relative dualizing sheaf \( \omega_{Y^{\flat}/T} \) is invertible
(cf.\ Lemma~\ref{lem:CMmorphism} below),
where \( Y^{\flat} \) is open by Fact~\ref{fact:dfn:RelSkCMlocus}\eqref{fact:dfn:RelSkCMlocus:1}.
\end{remn}

The following characterizations of Cohen--Macaulay morphism and
Gorenstein morphism are known:

\begin{lem}[{\cite[V, Exer.~9.7]{ResDual}, \cite[Th.~3.5.1]{Conrad}}]
\label{lem:CMmorphism}
Let \( f \colon Y \to T \) be a flat morphism locally of finite type
between locally Noetherian schemes.
Then, \( f \) is Cohen--Macaulay if and only if, locally on \( Y \),
the twisted inverse image \( f^{!}\SO_{T} \)
is quasi-isomorphic to an \( f \)-flat coherent \( \SO_{Y} \)-module \( \omega_{Y/T} \)
up to shift. Here, \( f \) is Gorenstein if and only if \( \omega_{Y/T} \) is
invertible.
\end{lem}

\begin{proof}
We may assume that \( f \) is a separated morphism of finite type between affine Noetherian schemes
by localizing \( Y \) and \( T \).
Assume first that \( f^{!}\SO_{T} \isom_{\qis} \omega_{Y/T}[d]\) for a coherent
\( \SO_{Y} \)-module \( \omega_{Y/T} \) flat over \( T \) and for an integer \( d \).
For an arbitrary fiber \( Y_{t} \),
the dualizing complex \( \omega_{Y_{t}/\Bbbk(t)}^{\bullet} \) is quasi-isomorphic to
\( \omega_{Y/T} \otimes_{\SO_{Y}} \SO_{Y_{t}}[d] \)
by Corollary~\ref{cor:BC}. Hence, \( Y_{t} \) is Cohen--Macaulay
by Corollary~\ref{cor:lem:DC-CM:CM}\eqref{cor:lem:DC-CM:CM:2}
or Lemma~\ref{lem:DC-CM:ordinary}\eqref{lem:DC-CM:ordinary:4}.

Conversely, assume that every fiber \( Y_{t} \) is Cohen--Macaulay.
Then, we may assume that \( f \) has pure relative dimension \( d \) by
Lemma~\ref{lem:S2Codim2}.
We shall show that
\[ f^{!}\SO_{T} \isom_{\qis} \omega_{Y/T}[d] \]
for the cohomology sheaf \( \omega_{Y/T} := \SH^{-d}(f^{!}\SO_{T}) \) and
that \( \omega_{Y/T} \) is flat over \( T \).
For a point \( t \in T \) and the inclusion morphism \( \psi_{t} \colon Y_{t} \to Y \), we have
a quasi-isomorphism
\begin{equation}\label{eq:lem:CMmorphism}
\bfL \psi_{t}^{*}(f^{!}\SO_{T} ) \isom_{\qis} \omega_{Y_{t}/\Bbbk(t)}[d]
\end{equation}
for the canonical sheaf \( \omega_{Y_{t}/\Bbbk(t)} \) by Corollary~\ref{cor:BC}.
Now, \( f^{!}\SO_{T} \) belongs to \( \bfD^{-}_{\coh}(\SO_{Y})\).
In fact, \( f^{!}\SO_{T} \) is \( f \)-perfect by Fact~\ref{fact:Lipman}\eqref{fact:Lipman:1}.
For the stalk \( (f^{!}\SO_{T})_{y} \) at a point \( y \in Y_{t}\), we have
\[ (f^{!}\SO_{T})_{y}[-d] \otimes^{\bfL}_{\SO_{T, t}} \Bbbk(t)
\isom_{\qis} (\omega_{Y_{t}/\Bbbk(t)})_{y} \]
by \eqref{eq:lem:CMmorphism}.
By applying Lemma~\ref{lem:CMmorphism:complex} below to \( (f^{!}\SO_{T})_{y}[-d] \)
and \( \SO_{T, t} \to \SO_{Y, y} \), we see that
\( \SH^{i}(f^{!}\SO_{T})_{y} = 0 \) for any \( i \ne -d \) and
\( \SH^{-d}(f^{!}\SO_{T})_{y} \) is a flat \( \SO_{T, t} \)-module with an isomorphism
\[ \SH^{-d}(f^{!}\SO_{T})_{y} \otimes_{\SO_{T, t}} \Bbbk(t) \isom (\omega_{Y_{t}/\Bbbk(t)})_{y}. \]
Since these hold for arbitrary point \( y \in Y\),
there is a quasi-isomorphism \( f^{!}\SO_{T} \isom_{\qis} \omega_{Y/T}[d] \)
and \( \omega_{Y/T} \) is flat over \( T \).
Therefore, we have proved the first assertion on a characterization of Cohen--Macaulay morphism.
For the second assertion, we assume that
\( f \) is a Cohen--Macaulay morphism.
Then, \( \omega_{Y/T} \) is flat over \( T \). Thus,
\( \omega_{Y/T} \) is invertible along a fiber \( Y_{t} \) if and only if
\( \omega_{Y/T} \otimes_{\SO_{Y}} \SO_{Y_{t}} \) is invertible
(cf.\ Fact~\ref{fact:elem-flat}\eqref{fact:elem-flat:2}).
By the isomorphism \( \omega_{Y/T} \otimes_{\SO_{Y}} \SO_{Y_{t}} \isom \omega_{Y_{t}/\Bbbk(t)} \),
we see that \( Y_{t} \) is Gorenstein if and only if \( \omega_{Y/T} \) is invertible along \( Y_{t} \).
Thus, the second assertion follows, and we are done.
\end{proof}

The following is used in the proof of Lemma~\ref{lem:CMmorphism} above:

\begin{lem}\label{lem:CMmorphism:complex}
Let \( A \) be a Noetherian local ring with residue field \( \Bbbk \) and
let \( A \to B \) be a local ring homomorphism to another Noetherian local ring \( B \).
Let \( L^{\bullet} \) be a complex of \( B \)-modules such that
\( \OH^{l}(L^{\bullet}) = 0\) for \( l \gg 0 \) and \( \OH^{i}(L^{\bullet}) \)
is a finitely generated \( B \)-modules for any \( i \in \BZZ \), i.e.,
\( L^{\bullet} \in \bfD_{\coh}^{-}(B) \).
Assume that
\[ \OH^{i}(L^{\bullet} \otimes^{\bfL}_{A} \Bbbk) = 0 \]
for any \( i > 0 \).
Then, \( \OH^{i}(L^{\bullet}) = 0 \)
for any \( i > 0 \),
and there exist an isomorphism
\[ \OH^{0}(L^{\bullet}) \otimes_{A} \Bbbk \isom \OH^{0}(L^{\bullet} \otimes^{\bfL}_{A} \Bbbk) \]
and an exact sequence
\[ \Tor^{A}_{2}(\OH^{0}(L^{\bullet}), \Bbbk) \to \OH^{-1}(L^{\bullet}) \otimes_{A} \Bbbk
\xrightarrow{h} \OH^{-1}(L^{\bullet} \otimes^{\bfL}_{A} \Bbbk)
\to \Tor^{A}_{1}(\OH^{0}(L^{\bullet}), \Bbbk) \to 0. \]
Consequently, the following hold\emph{:}
\begin{enumerate}
\item \label{lem:CMmorphism:complex:1}
\( \OH^{0}(L^{\bullet}) \) is flat over \( A \) if and only if
the homomorphism \( h \) above is surjective.

\item \label{lem:CMmorphism:complex:2}
If \( \OH^{i}(L^{\bullet} \otimes^{\bfL}_{A} \Bbbk) = 0 \) for any \( i \ne 0 \),
then \( L^{\bullet} \) is quasi-isomorphic to a flat \( A \)-module.
\end{enumerate}
\end{lem}

\begin{proof}
There is a standard spectral sequence
\[ E_{2}^{p, q} = \Tor^{A}_{-p}(\OH^{q}(L^{\bullet}), \Bbbk) \Rightarrow
E^{p+q} = \OH^{p+q}(L^{\bullet} \otimes^{\bfL}_{A} \Bbbk)\]
(cf.\ \cite[III, (6.3.2.2)]{EGA}),
where \( E_{2}^{p, q} = 0 \) for any \( p > 0 \).
Let \( a \) be an integer such that \( \OH^{l}(L^{\bullet}) = 0 \) for any \( l > a \).
Then, \( E_{2}^{p, q} = 0 \) for any \( q > a \), and we have
\( E^{a} \isom E_{2}^{0, a} \) and an exact sequence
\[ E_{2}^{-2, a} \to E_{2}^{0, a-1} \to E^{a-1} \to E_{2}^{-1, a} \to 0. \]
Hence, if \( a > 0 \), then \( \OH^{a}(L^{\bullet}) = 0 \) by \( E_{2}^{0, a} = 0 \),
and we may decrease
\( a \) by \( 1 \). Thus, we can choose \( a = 0 \), and we have the required
isomorphism and exact sequence.
The assertion \eqref{lem:CMmorphism:complex:1}
is derived from the local criterion of flatness
(cf.\ Proposition~\ref{prop:LCflat}),
since \( \OH^{0}(L^{\bullet}) \) is flat over \( A \)
if and only if \( \Tor^{A}_{1}(\OH^{0}(L^{\bullet}), \Bbbk) = 0\).
The assertion \eqref{lem:CMmorphism:complex:2} follows from \eqref{lem:CMmorphism:complex:1}
and \( \tau^{\leq -1}(L^{\bullet}) \isom_{\qis} 0\),
the latter of which is obtained by applying the result above to
the complex \( \tau^{\leq -1}(L^{\bullet}) \) instead of \( L^{\bullet} \).
\end{proof}

\begin{fact}\label{fact:Conrad+Sastry}
Let \( f \colon Y \to T \) be a Cohen--Macaulay morphism having pure relative dimension \( d \)
(cf.\ Definition~\ref{dfn:PureRelDim}). In  \cite[Sect.~3.5]{Conrad}, Conrad defines
a sheaf \( \omega_{f} \), called the \emph{dualizing sheaf} for \( f \), on \( Y \) such that
\[ \omega_{f}|_{U} \isom \SH^{-d}((f|_{U})^{!}\SO_{T}) \]
for any open subset \( U \subset Y \) such that
the restriction \( f|_{U} \colon U \to T\) factors as a closed immersion \( U \injmap P \)
followed by a smooth separated morphism \( P \to T \) with pure relative dimension.
Here, the sheaf \( \omega_{f} \) is obtained by gluing the sheaves \( \SH^{-d}((f|_{U})^{!}\SO_{T}) \)
along natural isomorphisms, where the compatibility of gluing
is checked by explicit calculation of \( \Ext \) groups.
In \cite[Th.~2.3.3, 2.3.5]{Sastry}, Sastry defines the same sheaf \( \omega_{f} \) by another method:
This is obtained by gluing \( \SH^{-d}((f|_{V})^{!}\SO_{T})  \) for open subsets \( V \subset Y \)
such that \( f|_{V} \) factors as an open immersion \( V \injmap \overline{V}  \) followed by
a \( d \)-proper morphism \( \overline{V} \to T \) in the sense of \cite[Def.~2.2.1]{Sastry}.
\end{fact}

\begin{dfn}[relative dualizing sheaf]\label{dfn:reldualsheaf}
Let \( f \colon Y \to T \) be a Cohen--Macaulay morphism.
For any connected component \( Y_{\alpha} \) of \( Y \),
it is shown in Lemma~\ref{lem:S2Codim2} that the restriction morphism
\( f_{\alpha} = f|_{Y_{\alpha}} \colon Y_{\alpha} \to T \) has pure relative dimension.
Thus, one can consider the dualizing sheaf \( \omega_{f_{\alpha}} \)
in Fact~\ref{fact:Conrad+Sastry} for \( f_{\alpha} \).
We define the \emph{relative dualizing sheaf} \( \omega_{Y/T} \)
of \( Y \) over \( T \) by
\[ \omega_{Y/T}|_{Y_{\alpha}} = \omega_{f_{\alpha}} \]
for any connected component \( Y_{\alpha} \).
The \( \omega_{Y/T} \) is also called
the \emph{relative dualizing sheaf} for \( f \) or
the \emph{relative canonical sheaf}
of \( Y \) over \( T \)
(cf.\ Definition~\ref{dfn:relcanosheaf} below).
We sometimes write \( \omega_{f} \) for \( \omega_{Y/T} \).
\end{dfn}

By Corollary~\ref{cor:lem:DC-CM:CM}\eqref{cor:lem:DC-CM:CM:2}
and Lemma~\ref{lem:CMmorphism}, we have:

\begin{cor}\label{cor:omegaRelCM}
For a Cohen--Macaulay morphism \( f \colon Y \to T \),
the relative dualizing sheaf \( \omega_{Y/T} \) is relatively
Cohen--Macaulay over \( T \) \emph{(}cf.\ Definition~\emph{\ref{dfn:RelSkCMlocus})}
and \( \Supp \omega_{Y/T} = Y \).
\end{cor}

By Lemma~\ref{lem:US2add}\eqref{lem:US2add:3a}, we have also:

\begin{cor}\label{cor:pushomegaCM}
For a Cohen--Macaulay morphism \( f \colon Y \to T \),
let \( Y^{\circ} \) be an open subset of \( Y \) such that
\( \Codim(Y_{t} \setminus Y^{\circ}, Y_{t}) \geq 2 \)
for any fiber \( Y_{t} = f^{-1}(t)\).
Then, \( \omega_{Y/T} \isom j_{*}(\omega_{Y^{\circ}/T}) \)
for the open immersion \( j \colon Y^{\circ} \injmap Y \).
\end{cor}

The following base change property is known for the relative dualizing sheaves
(cf.\ \cite[Th.~3.6.1]{Conrad}, \cite[Prop.~(9)]{Kleiman},
\cite[Th.~2.3.5]{Sastry}):

\begin{thm}\label{thm:basechange}
Let \( f \colon Y \to T \) be a Cohen-Macaulay morphism. For an arbitrary
morphism \( T' \to T \) from a locally Noetherian scheme \( T' \),
let \( Y' \) be the fiber product \( Y \times_{T} T' \) and let
\( p \colon Y' \to Y \) be the projection.
Then,
\( p^{*}(\omega_{Y/T})  \isom \omega_{Y'/T'} \).
\end{thm}

\begin{remn}
Conrad \cite[Th.~3.6.1]{Conrad} and Sastry \cite[Th.~2.3.5]{Sastry} prove
Theorem~\ref{thm:basechange} assuming that \( f \) has pure relative dimension,
but it is enough for the proof,
since the restriction of \( f \) to any connected component of \( Y \) has pure relative dimension
(cf.\ Lemma~\ref{lem:S2Codim2}).
When \( f \) is proper, Theorem~\ref{thm:basechange} is shown by Kleiman
\cite[Prop.~(9)(iii)]{Kleiman}, whose proof uses another version of twisted inverse image functor
\( f^{!} \).
The proof of \cite[Th.~3.6.1]{Conrad} is
based on arguments in \cite[V]{ResDual},
while the proof of \cite[Th.~2.3.5]{Sastry} is
based on arguments in
\cite{ResDualDel}, \cite{Verdier}, \cite{Kleiman}, and \cite{Lipman09}.
\end{remn}


\section{Relative canonical sheaves}
\label{sect:Relcan}

As a generalization of the relative dualizing sheaf
for a Cohen--Macaulay morphism, we introduce the notion of
\emph{relative canonical sheaf} for an arbitrary
\emph{\( \bfS_{2} \)-morphism} (cf.\ Definition~\ref{dfn:SkCMmorphism}).
We give some base change properties of the relative canonical sheaf and its ``multiple.''
These are used for studying \( \BQQ \)-Gorenstein morphisms in Section~\ref{sect:QGormor}.
In Section~\ref{subsect:RelomegaS2}, we shall study the relative canonical sheaf and the  
conditions for it to satisfy relative \( \bfS_{2}\). 
Section~\ref{subsect:bcS3} is devoted to prove Theorem~\ref{thm:S2S3crit} 
on a criterion for a certain sheaf related to the relative canonical sheaf to be invertible. 
This provides sufficient conditions for the base change homomorphism of 
the relative canonical sheaf to the fiber to be an isomorphism.


\subsection{Relative canonical sheaf for an \texorpdfstring{$\bfS_{2}$}{S2}-morphism}
\label{subsect:RelomegaS2}

First of all, we shall give a partial generalization of
the notion of canonical sheaf in Definition~\ref{dfn:canosheaf} as follows.

\begin{dfn}\label{dfn:canosheaf2}
Let \( X \) be a \( \Bbbk \)-scheme locally of finite type for a field \( \Bbbk \).
Assume that
\begin{itemize}
\item  \( X \) is locally equi-dimensional, and

\item  \( \Codim(X \setminus X^{\flat}, X) \geq 2 \) for
the Cohen--Macaulay locus \( X^{\flat} = \CM(X) \).
\end{itemize}
Note that this assumption is verified when \( X \) satisfies \( \bfS_{2} \).
For the relative dualizing sheaf \( \omega_{X^{\flat}/\Bbbk} \) over \( \Spec \Bbbk \)
in Definition~\ref{dfn:reldualsheaf}
and for the open immersion \( j^{\flat} \colon X^{\flat} \injmap X \), we set
\[ \omega_{X/\Bbbk} := j^{\flat}_{*}(\omega_{X^{\flat}/\Bbbk}) \]
and call it the \emph{canonical sheaf} of \( X \).
\end{dfn}

\begin{remn}
By Corollaries~\ref{cor:CMCodOne} and \ref{cor:S2S2}, we have the following properties in
the situation of Definition~\ref{dfn:canosheaf2}:
\begin{enumerate}
\item  Let \( U \) be an arbitrary open subset of \( X \)
which is of finite type over \( \Spec \Bbbk \).
Then, \( \omega_{X/\Bbbk}|_{U} \) is isomorphic to the canonical sheaf \( \omega_{U/\Bbbk} \)
defined in Definition~\ref{dfn:canosheaf}.
Thus, the use of the same symbol \( \omega_{X/\Bbbk} \) for the canonical sheaf causes no confusion.

\item \label{rem:canosheaf2:S2} The canonical sheaf \( \omega_{X/\Bbbk} \) is coherent
and satisfies \( \bfS_{2} \).

\end{enumerate}
\end{remn}

\begin{lem}\label{lem:dfn:canosheaf2}
Let \( X \) be a scheme locally of finite type over a field \( \Bbbk \).
Assume that \( X \) is Gorenstein in codimension one and satisfies \( \bfS_{2} \).
Then, \( \omega_{X/\Bbbk} \) is reflexive, and every reflexive \( \SO_{X} \)-module
satisfies \( \bfS_{2} \). In particular,
the double-dual \( \omega^{[m]}_{X/\Bbbk} \) of \( \omega_{X/\Bbbk}^{\otimes m} \)
satisfies \( \bfS_{2} \) for any \( m \in \BZZ\).
\end{lem}

\begin{proof}
Let \( Z \) be the complement of the Gorenstein locus of \( X \) (cf.\ Definition~\ref{dfn:Gorlocus}).
Then, \( \Codim(Z, X) \geq 2 \) and \( \omega_{X/\Bbbk}|_{X \setminus Z} \) is invertible.
Hence, \( \omega_{X/\Bbbk} \) is reflexive by Corollary~\ref{cor:prop:S1S2:reflexive},
since \( \omega_{X/\Bbbk} \) satisfies \( \bfS_{2} \) and \( \Supp \omega_{X/\Bbbk} = X \).
Every reflexive \( \SO_{X} \)-module satisfies \( \bfS_{2} \)
by Lemma~\ref{lem:j*reflexive}\eqref{lem:j*reflexive:1a}.
\end{proof}

The definition of the canonical sheaf above
is partially extended to the relative situation as follows.

\begin{dfn}[relative canonical sheaf]\label{dfn:relcanosheaf}
Let \( f \colon Y \to T \) be an \( \bfS_{2} \)-morphism of locally Noetherian schemes.
Let \( j \colon Y^{\flat} \injmap Y \) be the open immersion from
the relative Cohen--Macaulay locus \( Y^{\flat} = \CM(Y/T) \).
Note that \( \Codim(Y_{t} \setminus Y^{\flat}, Y_{t}) \geq 3 \) for any fiber \( Y_{t} = f^{-1}(t) \),
since \( Y_{t} \) satisfies \( \bfS_{2} \). In this situation, we define
\[ \omega_{Y/T} := j_{*}(\omega_{Y^{\flat}/T}) \]
for the relative dualizing sheaf \( \omega_{Y^{\flat}/T} \) for \( f|_{Y^{\flat}} \)
in the sense of Definition~\ref{dfn:reldualsheaf}.
We call \( \omega_{Y/T} \) also the \emph{relative canonical sheaf}
of \( Y\) over \( T \).
\end{dfn}

\begin{lem}\label{lem:omegaFlatbc}
Let \( f \colon Y \to T \) be an \( \bfS_{2} \)-morphism
of locally Noetherian schemes and let
\[ \begin{CD}
Y' @>{p}>> Y \\ @V{f'}VV @VVV \\ T' @>>> T
\end{CD}\]
be a Cartesian diagram such that
\( T' \) is a locally Noetherian scheme flat over \( T \).
Then, \( \omega_{Y'/T'} \isom p^{*}(\omega_{Y/T}) \).
\end{lem}

\begin{proof}
Let \( Y^{\flat} \) (resp.\ \( Y^{\prime \flat} \)) be the relative Cohen--Macaulay locus
for \( f \) (resp.\ \( f' \)) and let \( j \colon Y^{\flat} \injmap Y \) (resp.\
\( j' \colon Y^{\prime \flat} \injmap Y' \)) be the open immersion.
Then, \( Y^{\prime \flat} = p^{-1}(Y^{\flat}) \) by Lemma~\ref{lem:bc basic}\eqref{lem:bc basic:3}
for \( \SF=\SO_Y\), and \( j' \) is induced from \( j \).
Let \( p^{\flat} \colon Y^{\prime \flat} \to Y^{\flat} \) be the restriction of \( p \).
Then, \( \omega_{Y^{\prime \flat}/T'} \isom p^{\flat *}(\omega_{Y^{\flat}/T}) \) by
Theorem~\ref{thm:basechange}.
Thus, we have
\[ \omega_{Y'/T'} \isom j'_{*}(p^{\flat *}(\omega_{Y^{\flat}/T}))
\isom p^{*}(j_{*}(\omega_{Y^{\flat}/T})) \isom p^{*}\omega_{Y/T}\]
by the flat base change isomorphism
(cf.\ Lemma~\ref{lem:flatbc})
for the Cartesian diagram composed of \( p \), \( p^{\flat} \), \( j \), and \( j' \).
\end{proof}

\begin{prop}\label{prop:BC-S2CM}
Let \( f \colon Y \to T \) be an \( \bfS_{2} \)-morphism
of locally Noetherian schemes.
Then, the relative canonical sheaf \( \omega_{Y/T} \) defined in
Definition~\emph{\ref{dfn:relcanosheaf}} is coherent, and moreover,
if \( f \) is a separated morphism of pure relative dimension \( d \), then
\[ \SH^{i}(f^{!}\SO_{T}) \isom
\begin{cases}
0, &\text{ if  } i < -d; \\
\omega_{Y/T}, &\text{ if  } i = -d,
\end{cases}\]
for the twisted inverse image \( f^{!}\SO_{T} \).
Let \( Y^{\circ} \) be an open subset of \( \CM(Y/T) \) such that
\( \Codim(Y_{t} \setminus Y^{\circ}, Y_{t}) \geq 2 \)
for any fiber \( Y_{t} = f^{-1}(t)\).
For a point \( t \in T \), let
\[ \phi_{t} \colon \omega_{Y/T} \otimes_{\SO_{Y}} \SO_{Y_{t}}
\to \omega_{Y_{t}/\Bbbk(t)} = j_{t*}(\omega_{Y^{\circ} \cap Y_{t}/\Bbbk(t)}) \]
be the homomorphism induced from the base change isomorphism
\begin{equation}\label{eq:prop:BC-S2CM}
\omega_{Y^{\circ}/T} \otimes_{\SO_{Y^{\circ}}} \SO_{Y^{\circ} \cap Y_{t}}
\isom \omega_{Y^{\circ} \cap Y_{t}/\Bbbk(t)}
\end{equation}
\emph{(}cf.\ Theorem~\emph{\ref{thm:basechange})}, where
\( j_{t} \colon Y^{\circ} \cap Y_{t} \injmap Y_{t} \)
denotes the open immersion. Then, for any point \( y \in Y \),
the following three conditions are equivalent to each other\emph{:}
\begin{enumerate}
    \renewcommand{\theenumi}{\alph{enumi}}
     \renewcommand{\labelenumi}{(\theenumi)}
\item \label{prop:BC-S2CM:condA} The homomorphism \( \phi_{f(y)} \) is surjective at \( y \).

\item \label{prop:BC-S2CM:condB} The homomorphism \( \phi_{f(y)} \) is an isomorphism at \( y \).

\item \label{prop:BC-S2CM:condC}
There is an open neighborhood \( U \) of \( y \) in \( Y \) such that
\( \omega_{Y/T}|_{U} \)
satisfies \emph{relative} \( \bfS_{2} \) \emph{over} \( T \)
\emph{(}cf.\ Definition~\emph{\ref{dfn:RelSkCMlocus})}.
\end{enumerate}
\end{prop}

\begin{proof}
The coherence of \( \omega_{Y/T} \) and the conditions \eqref{prop:BC-S2CM:condA}--\eqref{prop:BC-S2CM:condC}
are local on \( Y \). 
Hence, we may assume that \( f \) is a separated morphism of pure relative dimension \( d \)
by Lemma~\ref{lem:S2Codim2}\eqref{lem:S2Codim2:1}.
Then, we have the twisted inverse image \( f^{!}\SO_{T} \) with a quasi-isomorphism
\[ (f^{!}\SO_{T})|_{Y^{\flat}} \isom_{\qis} \omega_{Y^{\flat}/T}[d] \]
for \( Y^{\flat} = \CM(Y/T) \) by Lemma~\ref{lem:CMmorphism}, and we have
a canonical homomorphism
\[ \phi \colon \SH^{-d}(f^{!}\SO_{T}) \to j^{\flat}_{*}(\omega_{Y^{\flat}/T})
= \omega_{Y/T} \]
for the open immersion \( j^{\flat} \colon Y^{\flat} \injmap Y \).
In order to prove that \( \phi \) is an isomorphism, since it is a local condition,
we may replace \( Y \) with an open subset freely.
Thus, we may assume that

\begin{itemize}
\item  \( f \) is the composite \( p \circ \iota  \)
of a closed immersion \( \iota \colon Y \injmap P \)
and a smooth affine morphism \( p \colon P \to T \).
\end{itemize}

By Fact~\ref{fact:Lipman}\eqref{fact:Lipman:1} and Remark~\ref{rem:f-perf}, we know that
\( \bfR \iota_{*}(f^{!}\SO_{T}) \) is perfect.
Hence, by localizing \( Y \), we may assume that

\begin{itemize}
\item  \( \bfR \iota_{*}(f^{!}\SO_{T}) \) is quasi-isomorphic to
a bounded complex \( \SE^{\bullet} = [\cdots \to \SE^{i} \to \SE^{i+1} \to \cdots]  \)
of free \( \SO_{P} \)-modules of finite rank.
\end{itemize}

Then, we have an isomorphism
\( \SH^{i}(\SE^{\bullet}) \isom \iota_{*}\SH^{i}(f^{!}\SO_{T}) \)
for any \( i \in \BZZ \). Moreover, there exist quasi-isomorphisms
\begin{align*}
\SE^{\bullet} \otimes^{\bfL}_{\SO_{P}} \SO_{P_{t}} &\isom_{\qis}
\bfR \iota_{*}(f^{!}\SO_{T} \otimes^{\bfL}_{\SO_{Y}} \bfL \iota^{*}\SO_{P_{t}})
\isom_{\qis}  \bfR \iota_{*}(f^{!}\SO_{T} \otimes^{\bfL}_{\SO_{Y}} \bfL f^{*}\Bbbk(t)) \\
&\isom_{\qis}
\bfR \iota_{*}(f^{!}\SO_{T} \otimes^{\bfL}_{\SO_{Y}} \SO_{Y_{t}})
\isom_{\qis} \bfR \iota_{t*}(\omega_{Y_{t}/\Bbbk(t)}^{\bullet})
\end{align*}
for any \( t \in T \) and
for the induced closed immersion \( \iota_{t} \colon Y_{t} \injmap P_{t} = p^{-1}(t) \).
In fact, the first quasi-isomorphism is known
as the projection formula (cf.\ \cite[II, Prop.\ 5.6]{ResDual}), the quasi-isomorphisms
\[ \SO_{P_{t}} \isom_{\qis} \bfL p^{*}\Bbbk(t)  \quad \text{and} \quad
\bfL f^{*}\Bbbk(t) \isom_{\qis} \SO_{Y_{t}} \]
are derived from the flatness of \( p \) and \( f \),
and the quasi-isomorphism
\[ f^{!}\SO_{T} \otimes^{\bfL}_{\SO_{Y}} \SO_{Y_{t}} \isom_{\qis} \omega_{Y_{t}/\Bbbk(t)}^{\bullet}\]
is obtained by Corollary~\ref{cor:BC}.
We shall show that the three data:
\[  \SE^{\bullet}[-d], \quad Z := \iota(Y \setminus Y^{\circ}),
\quad  \SF := \SH^{0}(\SE^{\bullet}[-d]) \isom \iota_{*}\SH^{-d}(f^{!}\SO_{T}),\]
satisfy the conditions of Lemma~\ref{lem:SurjFlat(complex)} for the morphism \( P \to T \).
The required inequality \eqref{lem:SurjFlat(complex)|eq0}
of Lemma~\ref{lem:SurjFlat(complex)} is derived from
\[ \depth_{P_{t} \cap Z} \SO_{P_{t}} = \Codim(P_{t} \cap Z, P_{t}) = \Codim(Y_{t} \cap Z, P_{t})
\geq \Codim(Y_{t} \cap Z, Y_{t}) \geq 2\]
(cf.\ Lemma~\ref{lem:basicSk}).
The condition \eqref{lem:SurjFlat(complex):0a} of Lemma~\ref{lem:SurjFlat(complex)}
is derived from (cf.\ Lemma~\ref{lem:CMmorphism}):
\[ \SH^{i}(\SE^{\bullet})|_{P \setminus Z} \isom \iota_{*}(\SH^{i}(f^{!}\SO_{T}))|_{P \setminus Z}
\isom \begin{cases}
0, & \text{ if } i \ne -d; \\
\iota_{*}\omega_{Y^{\circ}/T}, & \text{ if } i = -d,
\end{cases}\]
and the next condition \eqref{lem:SurjFlat(complex):0} has no meaning now.
The condition \eqref{lem:SurjFlat(complex):1} follows from
\[ \SH^{i}(\SE \otimes^{\bfL}_{\SO_{P}} \SO_{P_{t}})
\isom \iota_{t*}(\SH^{i}(\omega_{Y_{t}/\Bbbk(t)}^{\bullet})) = 0 \]
for any \( i < -d \) (cf.\ Lemma~\ref{lem:algschemeSk}).
The last condition \eqref{lem:SurjFlat(complex):3} of Lemma~\ref{lem:SurjFlat(complex)}
is a consequence of Corollary~\ref{cor:lem:DSSk} applied to
the ordinary dualizing complex \( \omega_{Y_{t}/\Bbbk(t)}^{\bullet}[-d] \)
(cf.\ Lemma~\ref{lem:algschemeSk}) and to \( b = 1 \), since
\begin{itemize}
\item the complex \( M^{\bullet} \) in Lemma~\ref{lem:SurjFlat(complex)}\eqref{lem:SurjFlat(complex):3}
is quasi-isomorphic to the stalk of
\[ \tau^{\leq 1}(\bfR \iota_{*}\omega_{Y_{t}/\Bbbk(t)}^{\bullet}[-d]) \isom_{\qis}
\bfR \iota_{*}( \tau^{\leq 1}(\omega_{Y_{t}/\Bbbk(t)}^{\bullet}[-d])), \quad \text{and}\]
\item  \( \dim \SO_{P_{t}, z} \geq \Codim(Z \cap Y_{t}, Y_{t}) \geq 2 \) for any
\( z \in Z \) with \( t = f(z) \).
\end{itemize}
Therefore, all the conditions of Lemma~\ref{lem:SurjFlat(complex)} are satisfied, and consequently,
\[ \SH^{i}(\SE^{\bullet}) \isom \iota_{*} \SH^{i}(f^{!}\SO_{T}) =  0 \]
for any \( i < -d \), and we can apply Proposition~\ref{prop:key} to \( \SF \)
via Lemma~\ref{lem:SurjFlat(complex)}.
Then,
\( \SF \isom j_{*}(\SF|_{P \setminus Z}) \) for the open immersion \( j \colon P \setminus Z \injmap P \)
by Proposition~\ref{prop:key}\eqref{prop:key:1}, and it implies that the morphism
\( \phi \) above is an isomorphism. Moreover, the three conditions
\eqref{prop:BC-S2CM:condA}--\eqref{prop:BC-S2CM:condC} are equivalent to each other
by  Proposition~\ref{prop:key}\eqref{prop:key:2}
and Corollary~\ref{cor0:prop:key}. Thus, we are done.
\end{proof}

\begin{prop}\label{prop:BCGor}
Let \( f \colon Y \to T \) be an \( \bfS_{2} \)-morphism of locally Noetherian schemes
and let \( j \colon Y^{\circ} \injmap Y\) be the open immersion from an open subset \( Y^{\circ}  \)
of the relative Gorenstein locus \( \Gor(Y/T)\) for \( f \).
Assume that
\begin{itemize}
\item
\( \Codim(Y_{t} \setminus Y^{\circ}, Y_{t}) \geq 2 \) for any fiber \( Y_{t} = f^{-1}(t)  \).
\end{itemize}
For an integer \( m \) and for the relative canonical sheaf \( \omega_{Y/T} \),
let \( \omega_{Y/T}^{[m]} \)
denote the double-dual of \( \omega_{Y/T}^{\otimes m} \).
Then,
\[ \omega_{Y/T}^{[m]} \isom  j_{*}(\omega_{Y^{\circ}/T}^{\otimes m}) \]
for any \( m \). In particular, \( \omega_{Y/T} \) is reflexive.
For an integer \( m \) and a point \( t \in T \), let
\[ \phi_{t}^{[m]} \colon \omega_{Y/T}^{[m]} \otimes_{\SO_{Y}} \SO_{Y_{t}} \to
\omega_{Y_{t}/\Bbbk(t)}^{[m]} = j_{t*}(\omega_{Y^{\circ} \cap Y_{t}/\Bbbk(t)}^{\otimes m})\]
be the homomorphism induced from the base change isomorphism
\eqref{eq:prop:BC-S2CM},
where \( j_{t} \colon Y^{\circ} \cap Y_{t} \injmap Y_{t} \)
denotes the open immersion.
Then, for any integer \( m \) and any point \( y \in Y \),
the following three conditions are equivalent to each other\emph{:}
\begin{enumerate}
    \renewcommand{\theenumi}{\alph{enumi}}
     \renewcommand{\labelenumi}{(\theenumi)}
\item \label{prop:BCGor:condA}
The homomorphism \( \phi^{[m]}_{f(y)} \) is surjective at \( y \).

\item \label{prop:BCGor:condB}
The homomorphism \( \phi^{[m]}_{f(y)} \) is an isomorphism at \( y \).

\item \label{prop:BCGor:condC}
There is an open neighborhood \( V \) of \( y \) in \( Y \) such that
\( \omega_{Y/T}^{[m]}|_{V} \) satisfies relative \( \bfS_{2} \) over \( T \).
\end{enumerate}
\end{prop}

\begin{proof}
We apply some results in Section~\ref{subsect:Resthom} to the reflexive sheaf
\( \SF = \omega_{Y/T}^{[m]} \) and the closed subset \( Z := Y \setminus Y^{\circ} \).
By assumption, \( \SF|_{Y \setminus Z} \) is invertible and
\( \depth_{Y_{t} \cap Z} \SO_{Y_{t}} \geq 2 \) (cf.\ Lemma~\ref{lem:basicSk}\eqref{lem:basicSk:condB}).
Thus, we can apply
Lemma~\ref{lem:SurjFlat(reflexive)}, and consequently, we can assume that
\( \SF \) has an exact sequence of Proposition~\ref{prop:key},
by replacing \( Y \) with its open subset.
Then, \( \omega_{Y/T}^{[m]} \isom j_{*}(\omega_{Y^{\circ}/T}^{\otimes m}) \)
by Proposition~\ref{prop:key}\eqref{prop:key:1}.
In case \( m = 1 \), we have
\( \omega^{[1]}_{Y/T} \isom \omega_{Y/T} \isom j_{*}(\omega_{Y^{\circ}/T}) \)
by Corollary~\ref{cor:pushomegaCM} and Definition~\ref{dfn:relcanosheaf}, and as a consequence,
\( \omega_{Y/T} \) is reflexive.
The equivalence of three conditions
\eqref{prop:BCGor:condA}--\eqref{prop:BCGor:condC} is derived from
Proposition~\ref{prop:key}\eqref{prop:key:2} and Corollary~\ref{cor0:prop:key}.
\end{proof}

\begin{cor}\label{cor:BC-S2CMGor}
Let us consider a Cartesian diagram
\[ \begin{CD}
Y' @>{p}>> Y \\ @V{f'}VV @VV{f}V \\ T' @>{q}>> T
\end{CD}\]
of locally Noetherian schemes in which \( f \) is a flat morphism locally of finite type.
Then, \(  p^{-1}\CM(Y/T) = \CM(Y'/T') \) and \( p^{-1}\Gor(Y/T) = \Gor(Y'/T') \).
Assume that \( f \) is an \( \bfS_{2} \)-morphism. Then\emph{:}
\begin{enumerate}
\item  \label{cor:BC-S2CMGor:1}
If \( \omega_{Y/T} \) satisfies relative \( \bfS_{2} \) over \( T \), then
\(  p^{*}\omega_{Y/T} \isom \omega_{Y'/T'}\).

\item \label{cor:BC-S2CMGor:2}
If every fiber \( Y_{t} = f^{-1}(t)\) is Gorenstein in codimension one,
then for any \( m \in \BZZ \), there is a canonical isomorphism
\[ (p^{*}\omega^{[m]}_{Y/T})^{\vee\vee} \isom \omega^{[m]}_{Y'/T'}.\]
Here, if \( \omega^{[m]}_{Y/T} \) satisfies relative \( \bfS_{2} \) over \( T \),
then
\( p^{*}\omega^{[m]}_{Y/T} \isom \omega^{[m]}_{Y'/T'} \).
\end{enumerate}
\end{cor}

\begin{proof}
The equality for \( \CM \) is derived from Lemma~\ref{lem:bc basic}\eqref{lem:bc basic:3}
for \( \SF = \SO_{Y} \).
If \( f \) is a Cohen--Macaulay morphism, then
\( p^{*}\omega_{Y/T} \isom \omega_{Y'/T'} \) by Theorem~\ref{thm:basechange}.
This implies the equality for \( \Gor \)
by the Remark of Definition~\ref{dfn:RelGorlocus}.
Assume that \( f \) is an \( \bfS_{2} \)-morphism. Then, \( f' \) is so by
Lemma~\ref{lem:bc basic}\eqref{lem:bc basic:5}.
For open subsets \( Y^{\flat} := \CM(Y/T) \) and \( Y^{\prime \flat} := p^{-1}(Y^{\flat}) \),
we have
\[ \Codim(Y'_{t'} \setminus Y^{\prime \flat}, Y'_{t'}) = \Codim(Y_{t} \setminus Y^{\flat}, Y_{t}) \geq 3 \]
for any \( t' \in T' \) and \( t = q(t) \) by Lemma~\ref{lem:bc basic}\eqref{lem:bc basic:1}
and by the \( \bfS_{2} \)-condition of \( Y_{t} \).
If \( \omega_{Y/T} \) satisfies relative \( \bfS_{2} \) over \( T \), then
the canonical base change isomorphism
\begin{equation}\label{eq:cor:BC-S2CMGor}
p^{*}\omega_{Y^{\flat}/T} \isom \omega_{Y^{\prime \flat}/T'}
\end{equation}
in Theorem~\ref{thm:basechange} induces an isomorphism
\[ p^{*}\omega_{Y/T} \isom j'_{*} (p^{*}\omega_{Y/T}|_{Y^{\prime \flat}})
\isom j'_{*}\omega_{Y^{\prime \flat}/T'} = \omega_{Y'/T'}\]
for the open immersion \( j' \colon Y^{\prime \flat} \injmap Y' \),
by Lemma~\ref{lem:relSkCodimDepth}\eqref{lem:relSkCodimDepth:2} applied to
\( (\SF, Z) = (p^{*}\omega_{Y/T}, Y' \setminus Y^{\prime \flat}) \).
This proves \eqref{cor:BC-S2CMGor:1}.
In the situation of \eqref{cor:BC-S2CMGor:2}, \( \Codim(Y_{t} \setminus Y^{\circ}, Y_{t}) \geq 2 \)
for any \( t \in T \), where  \( Y^{\circ} = \Gor(Y/T) \).
In particular,
\[ \depth_{Y_{t} \setminus Y^{\circ}} \SF \otimes_{\SO_{Y}} \SO_{Y_{t}} \geq 2  \]
for any coherent \( \SO_{Y} \)-module \( \SF \) satisfying relative \( \bfS_{2} \) over \( T \),
by Lemma~\ref{lem:depth+codim+Sk}\eqref{lem:depth+codim+Sk:2}.
Thus, \eqref{cor:BC-S2CMGor:2} is a consequence of Lemma~\ref{lem:bc reflexive}
via the isomorphism \eqref{eq:cor:BC-S2CMGor}.
\end{proof}

\begin{prop}\label{prop:Extomega}
Let \( f \colon Y \to T \) be an \( \bfS_{2} \)-morphism of locally Noetherian schemes.
Then,
\[ \SHom_{\SO_{Y}}(\omega_{Y/T}, \omega_{Y/T}) \isom \SO_{Y} \]
for the relative canonical sheaf \( \omega_{Y/T} \) in the sense of
Definition~\emph{\ref{dfn:relcanosheaf}}.
If every fiber satisfies \( \bfS_{3} \), then
\[ \SExt^{1}_{\SO_{Y}}(\omega_{Y/T}, \omega_{Y/T}) = 0. \]
\end{prop}

\begin{proof}
Let \( j \colon Y^{\flat} \injmap Y \) be the
open immersion from the relative Cohen--Macaulay locus \( Y^{\flat} = \CM(Y/T)\).
Now, we have a quasi-isomorphism
\[ \SO_{Y^{\flat}} \isom \SRHom_{\SO_{Y^{\flat}}}(\omega_{Y^{\flat}/T}, \omega_{Y^{\flat}/T}) \]
by \eqref{eq:IllusieQIS} in Fact~\ref{fact:Lipman}\eqref{fact:Lipman:2}.
This induces another quasi-isomorphism
\[ \SRHom_{\SO_{Y}}(\omega_{Y/T}, \bfR j_{*}(\omega_{Y^{\flat}/T})) \isom_{\qis}
\bfR j_{*}\SRHom_{\SO_{Y^{\flat}}}(\omega_{Y^{\flat}/T}, \omega_{Y^{\flat}/T})
\isom \bfR j_{*}\SO_{Y^{\flat}} \]
and the spectral sequence
\[ \SE_{2}^{p, q} = \SExt^{p}_{\SO_{Y}}(\omega_{Y/T}, R^{q}j_{*}(\omega_{Y^{\flat}/T}))
\Rightarrow \SE^{p+q} = R^{p+q}j_{*}\SO_{Y^{\flat}}.\]
Since \( \omega_{Y/T} = j_{*}(\omega_{Y^{\flat}/T}) \),
the isomorphism \( \SE_{2}^{0, 0} \isom \SE^{0} \)
and the injection \( \SE^{1, 0}_{2} \injmap \SE^{1} \), respectively, correspond to an isomorphism
\( \SHom_{\SO_{Y}}(\omega_{Y/T}, \omega_{Y/T}) \isom j_{*}\SO_{Y^{\flat}} \) and an injection
\( \SExt^{1}_{\SO_{Y}}(\omega_{Y/T}, \omega_{Y/T}) \injmap R^{1}j_{*}\SO_{Y^{\flat}}\).
Therefore, it suffices to prove that
\begin{enumerate}
\item \label{prop:Extomega:pf1} \( \SO_{Y} \isom j_{*}\SO_{Y^{\flat}} \), and

\item \label{prop:Extomega:pf2}
if every fiber satisfies \( \bfS_{3} \), then \( R^{1}j_{*}\SO_{Y^{\flat}} = 0 \).
\end{enumerate}
Here, \eqref{prop:Extomega:pf1} (resp.\ \eqref{prop:Extomega:pf1} 
with the conclusion of \eqref{prop:Extomega:pf2})
is equivalent to:
\( \depth_{Z} \SO_{Y} \geq 2 \) (resp.\ \( \geq 3 \)) for \( Z := Y \setminus Y^{\flat} \)
(cf.\ Property~\ref{ppty:depth<=2}).
If a fiber \( Y_{t} \) satisfies \( \bfS_{k} \), then
\( \Codim(Z \cap Y_{t}, Y_{t}) > k \), and \( \depth_{Z \cap Y_{t}} \SO_{Y_{t}} \geq k \)
by Lemma~\ref{lem:depth+codim+Sk}\eqref{lem:depth+codim+Sk:2}.
Hence, we have \( \depth_{Z} \SO_{Y} \geq 2 \) (resp.\ \( \geq 3 \))
by Lemma~\ref{lem:relSkCodimDepth}\eqref{lem:relSkCodimDepth:3} when every fiber \( Y_{t} \)
satisfies \( \bfS_{2} \) (resp.\ \( \bfS_{3} \)).
Thus, we are done.
\end{proof}


\subsection{Some base change theorems for the relative canonical sheaf}
\label{subsect:bcS3}

For an \( \bfS_{2} \)-morphism \( f \colon Y \to T \) of locally Noetherian schemes
and for a fiber \( Y_{t} = f^{-1}(t) \), let
\[ \phi_{t}(\omega_{Y/T}) \colon \omega_{Y/T} \otimes_{\SO_{Y}} \SO_{Y_{t}}
\to \omega_{Y_{t}/\Bbbk(t)} = j^{\flat}_{*}(\omega_{Y^{\flat}_{t}/\Bbbk(t)})\]
be the canonical homomorphism induced from the base change isomorphism
\[  \omega_{Y^{\flat}/T} \otimes_{\SO_{Y^{\flat}}} \SO_{Y^{\flat}_{t}}
\isom \omega_{Y^{\flat}_{t}/\Bbbk(t)}\]
(cf.\ Theorem~\ref{thm:basechange}),
where \( Y^{\flat} = \CM(Y/T) \), \( Y^{\flat}_{t} = Y^{\flat} \cap Y_{t} \),
and \( j^{\flat} \) is the open immersion \( Y^{\flat} \injmap Y \).
The homomorphism \( \phi_{t}(\omega_{Y/T}) \) is not necessarily an isomorphism
(e.g.\ Fact~\ref{fact:Pa2} below).
We shall give a sufficient condition for \( \phi_{t}(\omega_{Y/T}) \) to be an isomorphism
in Theorem~\ref{thm:S2S3crit} below.

\begin{lem}\label{lem:critCM}
Let \( f \colon Y \to T \) be a Cohen--Macaulay morphism
of locally Noetherian schemes.
Let \( \SL \) be a coherent \( \SO_{Y} \)-module flat over \( T \) with an isomorphism
\begin{equation}\label{eq:lem:critCM}
\SL \otimes_{\SO_{Y}} \SO_{Y_{t}} \isom \omega_{Y_{t}/\Bbbk(t)}
\end{equation}
for the fiber \( Y_{t} = f^{-1}(t) \) over a given point \( t \in T \).
Then, for the sheaf \( \SM := \SHom_{\SO_{Y}}(\SL, \omega_{Y/T})\),
the canonical homomorphism
\( \SL \otimes \SM \to \omega_{Y/T} \)
is an isomorphism along \( Y_{t} \), and \( \SM \) is an invertible sheaf along \( Y_{t} \)
with an isomorphism  \( \SM \otimes_{\SO_{Y}} \SO_{Y_{t}} \isom \SO_{Y_{t}} \).
\end{lem}

\begin{proof}
Since the assertions are local on \( Y_{t} \), we may assume that
\begin{enumerate}
\item \label{lem:critCM:1}
\( f \) has pure relative dimension \( d \) (cf.\ Lemma~\ref{lem:S2Codim2}), and
\item \label{lem:critCM:2}
\( f \) is the composite \( p \circ \iota \) of a closed immersion \( \iota \colon Y \injmap P \)
and a smooth affine morphism \( p \colon P \to T \) of pure relative dimension \( e \).
\end{enumerate}
Then, \( f^{!}\SO_{T} \isom \omega_{Y/T}[d]\) and \( \omega_{Y/T} \) is flat over \( T \)
by Lemma~\ref{lem:CMmorphism}.
The complex \( \SRHom_{\SO_{Y}}(\SL, f^{!}\SO_{T}) \) is \( f \)-perfect
by Fact~\ref{fact:Lipman}\eqref{fact:Lipman:2}, and there is a quasi-isomorphism
\[ \bfR \iota_{*}\SRHom_{\SO_{Y}}(\SL, f^{!}\SO_{T})
\isom_{\qis} \SRHom_{\SO_{P}}(\iota_{*}\SL, \omega_{P/T}[e]) \]
by Corollary~\ref{cor:thm:DualityProper},
where \( p^{!}\SO_{T} = \omega_{P/T}[e] \) by \eqref{lem:critCM:2} above.
Localizing \( Y \), by Remark~\ref{rem:f-perf}, we may assume furthermore that
\begin{enumerate}
    \addtocounter{enumi}{2}
\item  \label{lem:critCM:3}
\( \bfR \iota_{*}\SRHom_{\SO_{Y}}(\SL, f^{!}\SO_{T}) \) is
quasi-isomorphic to
a bounded complex \( \SE^{\bullet} = [\cdots \to \SE^{i} \to \SE^{i+1} \to \cdots] \)
of free \( \SO_{P} \)-modules of finite rank.
\end{enumerate}
Note that we have an isomorphism
\[ \SH^{- d}(\SE^{\bullet}) \isom \iota_{*}\SHom_{\SO_{Y}}(\SL, \omega_{Y/T}) \isom \iota_{*}\SM. \]
For the closed immersion \( \iota \colon Y \injmap P\) and
the induced closed immersion
\( \iota_{t} \colon Y_{t} \injmap P_{t} = p^{-1}(t) \),
we have quasi-isomorphisms
\begin{align*}
\SE^{\bullet} \otimes^{\bfL}_{\SO_{P}} \SO_{P_{t}}
&\isom_{\qis}
\SRHom_{\SO_{P_{t}}}((\iota_{*} \SL) \otimes^{\bfL}_{\SO_{P}} \SO_{P_{t}},
\omega_{P/T}[e] \otimes^{\bfL}_{\SO_{P}} \SO_{P_{t}}) \\
&\isom_{\qis}
\SRHom_{\SO_{P_{t}}}(\iota_{t*}(\SL \otimes_{\SO_{Y}} \SO_{Y_{t}}), \omega_{P_{t}/\Bbbk(t)}[e])
\end{align*}
by \cite[I, Prop.~7.1.2]{IllusieSGA6}, since \( \SL \) is flat over \( T \),
\( \iota_{*}\SL \) is perfect
(cf.\ Fact~\ref{fact:Lipman}\eqref{fact:Lipman:1} and Remark~\ref{rem:f-perf}), and
since \( P \to T \) is smooth.
From the isomorphism \eqref{eq:lem:critCM}
and the base change isomorphism
\[ \phi_{t}(\omega_{Y/T}) \colon \omega_{Y/T} \otimes_{\SO_{Y}} \SO_{Y_{t}}
\isom \omega_{Y_{t}/\Bbbk(t)} \]
(cf.\ Theorem~\ref{thm:basechange}), by duality for \( \iota_{t} \)
(cf.\ Corollary~\ref{cor:thm:DualityProper}),
we have quasi-isomorphisms
\begin{align*}
\SE^{\bullet} \otimes^{\bfL}_{\SO_{T}} \Bbbk(t)
&\isom_{\qis} \SE^{\bullet} \otimes^{\bfL}_{\SO_{P}} \SO_{P_{t}}
\isom_{\qis} \bfR\iota_{t*}\SRHom_{\SO_{Y_{t}}}(\SL \otimes_{\SO_{Y}} \SO_{Y_{t}},
\omega_{Y_{t}/\Bbbk(t)}[d])
\\
&\isom_{\qis} \bfR\iota_{t*}\SRHom_{\SO_{Y_{t}}}(\omega_{Y_{t}/\Bbbk(t)},
\omega_{Y_{t}/\Bbbk(t)}[d])
\isom_{\qis} \iota_{t*}\SO_{Y_{t}}[d],
\end{align*}
where the last quasi-isomorphism follows from that
\( \omega_{Y_{t}/\Bbbk(t)}[d] \) is a dualizing complex of \( Y_{t} \).
Then, by Lemma~\ref{lem:CMmorphism:complex}, we see that
\begin{enumerate}
    \addtocounter{enumi}{3}
\item  \label{lem:critCM:4}
\( \SE^{\bullet}[-d] \) is quasi-isomorphic to
\( \SH^{-d}(\SE^{\bullet}) \isom \iota_{*}\SM \)
along \( Y_{t} \),

\item \label{lem:critCM:5}
\( \iota_{*}\SM\) is flat over \( T \) along \( Y_{t} \), and

\item \label{lem:critCM:6}
there is an isomorphism
\[ \iota_{*}\SM \otimes_{\SO_{P}} \SO_{P_{t}} \isom
\SH^{-d}(\SE^{\bullet} \otimes^{\bfL}_{\SO_{T}} \Bbbk(t) )
\isom \iota_{t*}\SO_{Y_{t}}. \]
\end{enumerate}
Hence, \( \SM \) is flat over \( T \) along \( Y_{t} \) with an isomorphism
\( \SM \otimes_{\SO_{Y}} \SO_{Y_{t}} \isom \SO_{Y_{t}} \)
by \eqref{lem:critCM:5} and \eqref{lem:critCM:6}.
As a consequence, \( \SM \) is an invertible \( \SO_{Y} \)-module
along \( Y_{t} \)
by Fact~\ref{fact:elem-flat}\eqref{fact:elem-flat:2}.
Now, we have a quasi-isomorphism
\[ \SRHom_{\SO_{Y}}(\SL, \omega_{Y/T}) \isom_{\qis} \SM \]
along \( Y_{t} \) by \eqref{lem:critCM:3} and \eqref{lem:critCM:4}.
By the duality quasi-isomorphism
\[ \SL \isom_{\qis} \SRHom_{\SO_{Y}}(\SRHom_{\SO_{Y}}(\SL, \omega_{Y/T}), \omega_{Y/T}) \]
(cf.\ Fact~\ref{fact:Lipman}\eqref{fact:Lipman:2}),
we have an isomorphism
\[ \SL \isom \SHom_{\SO_{Y}}(\SM, \omega_{Y/T}) \isom \omega_{Y/T} \otimes_{\SO_{Y}} \SM^{-1}\]
along \( Y_{t} \),
since \( \SM \) is invertible
along \( Y_{t} \).
Thus, we are done.
\end{proof}

\begin{thm}\label{thm:S2S3crit}
For an \( \bfS_{2} \)-morphism \( f \colon Y \to T \) of locally Noetherian schemes,
let \( \SL \) be a coherent
\( \SO_{Y} \)-module and set \( \SM := \SHom_{\SO_{Y}}(\SL, \omega_{Y/T}) \).
For an open subset \( U \) of \( Y \) and for the fiber \( Y_{t} = f^{-1}(t) \)
over a given point \( t \in T \),
assume that
\begin{enumerate}
    \renewcommand{\theenumi}{\roman{enumi}}
    \renewcommand{\labelenumi}{(\theenumi)}
\item \label{thm:S2S3crit:cond1}
\( \Codim(Y_{t} \setminus U, Y_{t}) \geq 2 \),

\item  \label{thm:S2S3crit:cond2prime} \( \SL \) is flat over \( T \) with
an isomorphism
\( \SL \otimes_{\SO_{Y}} \SO_{Y_{t}} \isom \omega_{Y_{t}/\Bbbk(t)}\), and

\item \label{thm:S2S3crit:cond3} one of the following two conditions is satisfied\emph{:}
\begin{enumerate}
    \makeatletter
    \renewcommand{\p@enumii}{}
    \makeatother
\item  \label{thm:S2S3crit:cond3:a}
\( Y_{t} \) satisfies \( \bfS_{3} \) and \( \Codim(Y_{t} \setminus U, Y_{t}) \geq 3 \)\emph{;}

\item \label{thm:S2S3crit:cond3:b}
there is a positive integer \( r \) coprime to the characteristic of \( \Bbbk(t) \)
such that \( \SL^{[r]} = (\SL^{\otimes r})^{\vee\vee} \) and
\( \omega_{Y/T}^{[r]} = (\omega_{Y/T}^{\otimes r})^{\vee\vee} \) are invertible
\( \SO_{Y} \)-module along \( Y_{t} \).
\end{enumerate}
\end{enumerate}
Then, \( \SM \) is an invertible \( \SO_{Y} \)-module along \( Y_{t} \)
with an isomorphism \( \SM \otimes_{\SO_{Y}} \SO_{Y_{t}} \isom \SO_{Y_{t}} \), and
the canonical homomorphism
\( \SL \otimes_{\SO_{Y}} \SM \to \omega_{Y/T} \) is an isomorphism along \( Y_{t} \).
Moreover, the ``base change homomorphism''
\[ \phi_{t}(\omega_{Y/T}) \colon \omega_{Y/T} \otimes_{\SO_{Y}} \SO_{Y_{t}}
\to \omega_{Y_{t}/\Bbbk(t)} \]
is an isomorphism.
\end{thm}

\begin{proof}
Since the assertions are local on \( Y_{t} \), we may replace \( Y \)
with an open subset freely.
Let \( Y^{\flat} \) be the relative Cohen--Macaulay locus \( \CM(Y/T) \), which is an open subset
by Fact~\ref{fact:dfn:RelSkCMlocus}\eqref{fact:dfn:RelSkCMlocus:1}.
Then, \( \Codim(Y_{t} \setminus Y^{\flat}, Y_{t}) \geq 3 \) (resp.\ \( \geq 4 \)
in the case \eqref{thm:S2S3crit:cond3:a}),
since \( Y_{t} \) satisfies \( \bfS_{2} \) (resp.\ \( \bfS_{3} \)).
We set \( U^{\flat} := U \cap Y^{\flat} \). Then,
\begin{equation}\label{eq:thm:S2S3crit|0}
\Codim(Y_{t} \setminus U^{\flat}, Y_{t}) =
\Codim((Y_{t} \setminus U) \cup (Y_{t} \setminus Y^{\flat}), Y_{t}) \geq 2
\qquad (\text{resp.} \, \geq 3).
\end{equation}
By Lemma~\ref{lem:critCM} applied to the Cohen--Macaulay morphism \( U^{\flat} \to T \),
there is an isomorphism
\begin{enumerate}
\item \label{thm:S2S3crit:3}
\( \SM|_{U^{\flat}} \otimes_{\SO_{U^{\flat}}} \SO_{U^{\flat} \cap Y_{t}}
\isom \SO_{U^{\flat} \cap Y_{t}} \),
\end{enumerate}
and there is an open neighborhood \( U' \) of \( U^{\flat} \cap Y_{t} \) in \( U^{\flat} \) such that
\begin{enumerate}
\addtocounter{enumi}{1}
\item \label{thm:S2S3crit:1}
\( \SM|_{U'} \) is an invertible sheaf, and

\item \label{thm:S2S3crit:2}
the canonical homomorphism \( \SL \otimes_{\SO_{Y}} \SM \to \omega_{Y/T} \)
is an isomorphism on \( U' \).
\end{enumerate}
We set \( Z = Y \setminus U' \).
Then, \( \Codim(Y_{t} \cap Z, Y_{t}) = \Codim(Y_{t} \setminus U^{\flat}, Y_{t}) \geq 2 \)
by \eqref{eq:thm:S2S3crit|0}.
Since \( f \) is an \( \bfS_{2} \)-morphism, by Lemma~\ref{lem:S2Codim2},
we may assume that
\( \Codim(Y_{t'} \cap Z, Y_{t'}) \geq 2 \) for any \( t' \in T \)
by replacing \( Y \) with an open subset. Then,
\( \depth_{Z} \SO_{Y} \geq 2 \) by Lemma~\ref{lem:relSkCodimDepth}\eqref{lem:relSkCodimDepth:3}, and
\[ \omega_{Y/T} \isom j_{*}(\omega_{U/T}) \isom j'_{*}(\omega_{U'/T}) \]
for the open immersion \( j' \colon U' \injmap Y \)
by Corollary~\ref{cor:pushomegaCM}.
In particular, \( \depth_{Z} \SM \geq 2 \), i.e.,
\( \SM \isom j'_{*}(\SM|_{U'}) \), by the isomorphism
\[ \SHom_{\SO_{Y}}(\SL, \omega_{Y/T})
\isom \SHom_{\SO_{Y}}(\SL, j'_{*}(\omega_{U'/T}))
\isom j'_{*}\SHom_{\SO_{U'}}(\SL|_{U'}, \omega_{U'/T}). \]
By \eqref{thm:S2S3crit:cond2prime}, \( \SL \) satisfies relative \( \bfS_{2} \)
over \( T \) along \( Y_{t} \),
since \( \omega_{Y_{t}/\Bbbk(t)} \) satisfies \( \bfS_{2} \) by Corollary~\ref{cor:S2S2}.
Hence, we have also an isomorphism
\( \SL \isom j'_{*}(\SL|_{U'}) \) by Lemma~\ref{lem:US2add}\eqref{lem:US2add:3a}.

We shall show that \( \SM \) is invertible along \( Y_{t} \) by applying
Theorem~\ref{thm:invExt} to \( Y \to T \), the closed subset \( Z = Y \setminus U'\),
and to the sheaf \( \SM \) as \( \SF \).
By the previous argument, we have checked
the conditions \eqref{thm:invExt:cond0} and
\eqref{thm:invExt:cond1} of Theorem~\ref{thm:invExt}.
The condition \eqref{thm:invExt:cond2} is derived from \eqref{thm:S2S3crit:3}: In fact, we have
\begin{equation}\label{eq:thm:S2S3crit|1}
\SM_{(t)*} = j'_{*}((\SM \otimes_{\SO_{Y}} \SO_{Y_{t}})|_{U' \cap Y_{t}})
\isom j'_{*}(\SO_{U' \cap Y_{t}})
\isom \SO_{Y_{t}},
\end{equation}
since we have
\( \depth_{Y_{t} \cap Z} \SO_{Y_{t}} \geq 2 \) by the \( \bfS_{2} \)-condition on \( Y_{t} \)
and by \(\Codim(Y_{t} \cap Z, Y_{t}) \geq 2 \)
(cf.\ Lemma~\ref{lem:basicSk}).
Similarly, in the case of \eqref{thm:S2S3crit:cond3:a} above, we have
the condition \eqref{thm:invExt:3a} of
Theorem~\ref{thm:invExt} by the \( \bfS_{3} \)-condition on \( Y_{t} \)
and by \( \Codim(Y_{t} \cap Z, Y_{t}) = \Codim(Y_{t} \setminus U^{\flat}, Y_{t}) \geq 3 \)
(cf.\ \eqref{eq:thm:S2S3crit|0}).
In the case of \eqref{thm:S2S3crit:cond3:b} above,
\( \SM^{[r]} \) is an invertible \( \SO_{Y} \)-module along \( Y_{t} \). In fact,
the restriction homomorphisms
\[ \SM^{[r]} \to j'_{*}(\SM^{[r]}|_{U'}) \quad \text{and} \quad
\omega_{Y/T}^{[r]} \to j'_{*}(\omega_{U'/T}^{[r]})\]
are isomorphisms by Lemma~\ref{lem:US2add}\eqref{lem:US2add:2},
since \( \SM^{[r]}  \) and \( \omega_{Y/T}^{[r]} \) are reflexive, and the isomorphism
\[ \SL^{[r]}|_{U'} \otimes_{\SO_{U'}} \SM^{[r]}|_{U'} \isom \omega_{U'/T}^{[r]}\]
obtained by \eqref{thm:S2S3crit:1} and \eqref{thm:S2S3crit:2} induces an isomorphism
\begin{align*}
\SM^{[r]} \isom  j'_{*}(\SM^{[r]}|_{U'})
&\isom j'_{*}\SHom_{\SO_{U'}}(\SL^{[r]}|_{U'}, \omega_{U'/T}^{[r]}) \\
&\isom \SHom_{\SO_{Y}}(\SL^{[r]}, j'_{*}(\omega_{U'/T}^{[r]}))
\isom \SHom_{\SO_{Y}}(\SL^{[r]}, \omega_{Y/T}^{[r]}).
\end{align*}
Thus, the condition \eqref{thm:invExt:3b} of Theorem~\ref{thm:invExt}
is also satisfied in the case of \eqref{thm:S2S3crit:cond3:b}.
Hence, we can apply Theorem~\ref{thm:invExt}, and as a result,
we see that \( \SM \) is an invertible sheaf along \( Y_{t} \).

Then, we have an isomorphism
\( \SM \otimes_{\SO_{Y}} \SO_{Y_{t}} \isom \SO_{Y_{t}} \) by \eqref{eq:thm:S2S3crit|1},
and the canonical homomorphism \( \SL \otimes_{\SO_{Y}} \SM \to \omega_{Y/T} \)
is an isomorphism along \( Y_{t} \) by \eqref{thm:S2S3crit:2}:
In fact, it is expressed as the composite
\[ \SL \otimes_{\SO_{Y}} \SM \isom j'_{*}(\SL|_{U'}) \otimes_{\SO_{Y}} \SM
\to j'_{*}(\SL|_{U'} \otimes \SM|_{U'}) \isom j'_{*}(\omega_{U'/T}) \isom \omega_{Y/T}, \]
where the middle arrow is an isomorphism along \( Y_{t} \) by the projection formula,
since \( \SM \) is invertible along \( Y_{t} \).
In particular, \( \omega_{Y/T} \otimes_{\SO_{Y}} \SO_{Y_{t}}\) satisfies \( \bfS_{2} \),
and as a consequence, \( \phi_{t}(\omega_{Y/T}) \) is an isomorphism by \eqref{eq:thm:S2S3crit|0}.
Thus, we are done.
\end{proof}


\section{\texorpdfstring{$\BQQ$}{Q}-Gorenstein  schemes}
\label{sect:QGorSch}

A normal algebraic variety defined over a field is said to be \( \BQQ \)-Gorenstein
if some positive multiple of the canonical divisor is Cartier.
We shall generalize the notion of \( \BQQ \)-Gorenstein
to locally Noetherian schemes. 
In Section~\ref{subsect:QGorSch}, 
the notion of \( \BQQ \)-Gorenstein scheme is defined and its basic properties are given.
In Section~\ref{subsect:cone}, we consider the case of affine cones over 
polarized projective schemes over a field, and determine 
when it is a \( \BQQ \)-Gorenstein scheme.


\subsection{Basic properties of \texorpdfstring{$\BQQ$}{Q}-Gorenstein schemes}
\label{subsect:QGorSch}

\begin{dfn}[$\BQQ$-Gorenstein scheme]\label{dfn:QGorSch}
Let \( X \) be a locally Noetherian scheme
admitting a dualizing complex
locally on \( X \) and assume that \( X \) is
\emph{Gorenstein in codimension one}, i.e., \( \Codim(X \setminus X^{\circ}) \geq 2 \)
for the Gorenstein locus \( X^{\circ} = \Gor(X)\)
(cf.\ Definition~\ref{dfn:Gorlocus}).
\begin{enumerate}
\item \label{dfn:QGorSch:1}
The scheme \( X \) is said to be \emph{quasi-Gorenstein}
(or \emph{\( 1 \)-Gorenstein})
\emph{at} a point \( P \) if there exist an open neighborhood \( U \) of \( P \) and
a dualizing complex \( \SR^{\bullet} \) of \( U \)
such that \( \SH^{0}(\SR^{\bullet}) \) is invertible at \( P \).
If \( X \) is quasi-Gorenstein at every point, then \( X \) is said to be
quasi-Gorenstein (or \( 1 \)-Gorenstein).

\item \label{dfn:QGorSch:2}
The scheme \( X \) is said to be \emph{\( \BQQ \)-Gorenstein at}
\( P \) if
there exist an open neighborhood \( U \) of \( P \),
a dualizing complex \( \SR^{\bullet} \) of \( U \), and an integer \( r > 0 \)
such that \( \SL = \SH^{0}(\SR^{\bullet}) \) is invertible
on the Gorenstein locus \( U^{\circ} = U \cap X^{\circ} \)
and
\[ j_{*}\left(\SL^{\otimes r}|_{U^{\circ}}\right) \]
is invertible at \( P \), where \( j \colon U^{\circ} \injmap U \) denotes the open immersion.
If \( X \) is \( \BQQ \)-Gorenstein at every point,
then \( X \) is said to be \( \BQQ \)-Gorenstein.
\end{enumerate}
\end{dfn}

\begin{dfn}[Gorenstein index]\label{dfn:QGorSch:GorIndex}
For a \( \BQQ \)-Gorenstein scheme \( X \),
the \emph{Gorenstein index}
of \( X \) at \( P \in X \)
is defined to be the smallest positive integer \( r \)
satisfying the condition \eqref{dfn:QGorSch:2} of Definition~\ref{dfn:QGorSch}
for an open neighborhood of \( P \).
The least common multiple of Gorenstein indices of \( X \) at all the points is called
the \emph{Gorenstein index} of \( X \), which might be \( +\infty \).
\end{dfn}

\begin{remn}
The conditions \eqref{dfn:QGorSch:1} and \eqref{dfn:QGorSch:2} of
Definition~\ref{dfn:QGorSch} do not depend on the choice of \( \SR^{\bullet} \) by
the essential uniqueness of the dualizing complex
(cf.\ Remark~\ref{rem:ExistDC}).
\end{remn}

\begin{lem}\label{lem:QGorSch}
\begin{enumerate}
\item \label{lem:QGorSch:1}
A quasi-Gorenstein
\emph{(}\( 1 \)-Gorenstein\emph{)}
scheme is nothing but
a \( \BQQ \)-Gorenstein
scheme of Gorenstein index one.
\item \label{lem:QGorSch:2}
Every \( \BQQ \)-Gorenstein scheme satisfies \( \bfS_{2} \).
\end{enumerate}
\end{lem}

\begin{proof}
\eqref{lem:QGorSch:1}: Let \( X \) be a locally Noetherian scheme
admitting a dualizing complex \( \SR^{\bullet} \) such that
it is Gorenstein in codimension one and that
\( \SL := \SH^{0}(\SR^{\bullet}) \)
is invertible on the Gorenstein locus \( X^{\circ} \).
Then, \( \SL \) satisfies \( \bfS_{2} \) and
\( \SL \to j_{*}(\SL|_{X^{\circ}}) \)
is an isomorphism by Corollary~\ref{cor:CMCodOne}.
Hence, \( X \) is \( \BQQ \)-Gorenstein with
Gorenstein index one if and only if \( \SL \) is invertible,
equivalently, \( X \) is quasi-Gorenstein.

\eqref{lem:QGorSch:2}: We may assume that \( X \) admits
a dualizing complex \( \SR^{\bullet} \)
such that \( \SL = \SH^{0}(\SR^{\bullet}) \) is invertible on
the Gorenstein locus \( X^{\circ} \), since the condition \( \bfS_{2} \) is local.
Then, \( \SM_{r} := j_{*}(\SL^{\otimes r}|_{X^{\circ}}) \) is invertible
for some \( r \) by Definition~\ref{dfn:QGorSch}\eqref{dfn:QGorSch:2}.
Hence, \( \SM_{r} \) satisfies \( \bfS_{2} \) by
Corollary~\ref{cor:basicS1S2}.
Therefore, \( X \) satisfies \( \bfS_{2} \).
\end{proof}

\begin{lem}\label{lem:QGorSch3}
Let \( X \) be a locally Noetherian scheme admitting
a dualizing complex \( \SR^{\bullet} \).
For the cohomology sheaf \( \SL := \SH^{0}(\SR^{\bullet}) \)
and for an open subset \( U \) with \( \Codim(X \setminus U, X) \geq 2 \),
assume that \( \SL|_{U} \) is invertible and
\( \SR^{\bullet}|_{U} \isom_{\qis} \SL|_{U}\).
Then, the following hold\emph{:}
\begin{enumerate}
\item \label{lem:QGorSch3:1}
If \( X \) satisfies \( \bfS_{1} \), then
\( \SR^{\bullet} \) is an ordinary dualizing complex of \( X \)
and the dualizing sheaf \( \SL \) is a reflexive \( \SO_{X} \)-module
satisfying \( \bfS_{2} \).

\item \label{lem:QGorSch3:2}
If \( X \) satisfies \( \bfS_{2} \), then
the double-dual \( \SL^{[m]} \) of \( \SL^{\otimes m} \)
satisfies \( \bfS_{2} \) for any integer \( m \), and in particular,
\[\SL^{[m]} \isom j_{*}(\SL^{\otimes m}|_{U}) \]
for the open immersion \( j \colon U \injmap X \).

\item  \label{lem:QGorSch3:3}
The scheme \( X \) is \( \BQQ \)-Gorenstein if and only if
\( X \) satisfies \( \bfS_{2} \) and, locally on \( X \),
there is a positive integer \( r \) such that \( \SL^{[r]} \) is invertible.
\end{enumerate}
\end{lem}

\begin{proof} \eqref{lem:QGorSch3:1}: This follows from Corollary~\ref{cor:CMCodOne}
with Lemmas \ref{lem:basicSk} and
\ref{lem:j*reflexive}\eqref{lem:j*reflexive:2}.

\eqref{lem:QGorSch3:2}:
Since \( \depth_{X \setminus U} \SO_{X} \geq 2 \) by the \( \bfS_{2} \)-condition,
we have the isomorphism
\( \SL^{[m]} \isom j_{*}(\SL^{\otimes m}|_{U}) \)
by Lemma~\ref{lem:j*reflexive}\eqref{lem:j*reflexive:1}.
Hence, \( \SL^{[m]} \) satisfies \( \bfS_{2} \) by
Corollary~\ref{cor:basicS1S2}, since \( \SL|_{U} \) is invertible.

\eqref{lem:QGorSch3:3}:
This is a consequence of \eqref{lem:QGorSch3:2} above and Lemma~\ref{lem:QGorSch}\eqref{lem:QGorSch:2}
by the uniqueness of dualizing complex explained in Remark~\ref{rem:ExistDC}.
\end{proof}

\begin{exam}
Let \( X \) be a \( \Bbbk \)-scheme locally of finite type
for a field \( \Bbbk \).
Assume that \( X \) satisfies \( \bfS_{2} \) and
\( \Codim(X \setminus X^{\circ}, X) \geq 2 \)
for the Gorenstein locus \( X^{\circ} = \Gor(X) \).
Let \( \omega_{X/\Bbbk} \) be the canonical sheaf defined in
Definition~\ref{dfn:canosheaf2} and let
\( \omega_{X/\Bbbk}^{[m]} \) denote
the double-dual of \( \omega_{X/\Bbbk}^{\otimes m} \)
for any \( m \in \BZZ \) (cf.\ Proposition~\ref{prop:BCGor}).
Then, \( X \) is \( \BQQ \)-Gorenstein at a point \( x \) if and only if
\( \omega^{[r]}_{X/\Bbbk} \) is invertible at \( x \) for some \( r > 0 \).
\end{exam}

\begin{exam}\label{exam:dfn:QGorScheme:3}
Let \( X \) be a normal algebraic \( \Bbbk \)-variety for a field \( \Bbbk \),
i.e., a normal integral separated scheme of finite type over \( \Bbbk \).
Then, \( X \) is \( \BQQ \)-Gorenstein if and only if
the multiple \( rK_{X} \) of the canonical divisor \( K_{X} \)
is Cartier for some \( r > 0 \).
In fact, \( X \) satisfies \( \bfS_{2} \),
\( \omega_{X^{\circ}/\Bbbk} \isom \SO_{X^{\circ}}(K_{X}) \) for
the Gorenstein locus \( X^{\circ} = \Gor(X) \), where \( \Codim(X \setminus X^{\circ}, X) \geq 2 \),
and hence
\( \omega^{[m]}_{X/\Bbbk} \isom \SO_{X}(mK_{X}) \) for any \( m \in \BZZ \).
\end{exam}

\begin{lem}\label{lem:QGorScheme:smooth}
Let \( X \) be a locally Noetherian scheme and let \( \pi \colon Y \to X \) be
a smooth surjective morphism.
Then, for any integer \( k \ge 1 \),
\( Y \) satisfies \( \bfS_{k} \) if and only if \( X \) satisfies \( \bfS_{k} \).
In particular, \( Y \) is Cohen--Macaulay if and only if \( X \) is so.
Moreover, \( Y \) is Gorenstein if and only if \( X  \) is so.
Assume that \( X \) admits a dualizing complex locally on \( X \).
Then, \( Y \) is quasi-Gorenstein \emph{(}resp.\ \( \BQQ \)-Gorenstein of index \( r \)\emph{)}
if and only if \( X \) is so.
\end{lem}

\begin{proof}
The first assertion follows from Fact~\ref{fact:elem-flat}\eqref{fact:elem-flat:6}.
In particular, we have
the equivalence for the Cohen--Macaulay property (cf.\ Remark~\ref{rem:dfn:CM}).
The Gorenstein case follows from Fact~\ref{fact:GorYTF}.
It remains to prove the case of \( \BQQ \)-Gorenstein property, since ``quasi-Gorenstein'' is nothing but
``\( \BQQ \)-Gorenstein of index one'' (cf.\ Lemma~\ref{lem:QGorSch}\eqref{lem:QGorSch:1}).
Since the \( \BQQ \)-Gorenstein property is local and it implies \( \bfS_{2} \),
we may assume that
\begin{itemize}
\item \( X \) has a dualizing complex \( \SR_{X}^{\bullet} \),
\item  \( X \) and \( Y \) are affine schemes satisfying \( \bfS_{2} \), and
\item \( \pi = p \circ \lambda \) for an \'etale morphism \( \lambda \colon Y \to X \times \BAA^{d} \)
and the first projection \( p \colon X \times \BAA^{d} \to X \)
for the ``\( d \)-dimensional affine space''
\( \BAA^{d} = \Spec \BZZ[\xtt_{1}, \ldots, \xtt_{d}] \) for some integer \( d \geq 0 \)
(cf.\ \cite[IV, Cor.~(17.11.4)]{EGA}).
\end{itemize}
In particular, \( \pi \) has pure relative dimension \( d \).
We may assume also that \( \SR^{\bullet}_{X} \) is
an ordinary dualizing complex by Lemma~\ref{lem:ordinaryDC}.
We set \( \SL_{X} \) to be the dualizing sheaf \( \SH^{0}(\SR^{\bullet}_{X}) \).

By Examples~\ref{exam:RDembeddable} and \ref{exam:RDConrad}, we see that
\( \SR^{\bullet}_{Y} := \pi^{!}(\SR_{X}^{\bullet})\)
is a dualizing complex of \( Y \), and we have an isomorphism
\[ \omega_{Y/X} \isom  \varOmega_{Y/X}^{d}
\isom \lambda^{*} (\omega_{X \times \BAA^{d}/X}) \isom \SO_{Y}\]
for the relative dualizing sheaf \( \omega_{Y/X} \).
Thus, \( \pi^{!}(\SO_{X}) \isom_{\qis} \SO_{Y}[d] \), and
\[ \SR^{\bullet}_{Y}
\isom_{\qis} \pi^{!}(\SO_{X}) \otimes^{\bfL}_{\SO_{Y}} \bfL \pi^{*}(\SR^{\bullet}_{X})
\isom_{\qis} \bfL \pi^{*}(\SR^{\bullet}_{X})[d]\]
(cf.\ Example~\ref{exam:RDembeddable},
Fact~\ref{fact:DeligneVerdierLipman}\eqref{fact:DeligneVerdierLipman:smooth}).
Since \( Y \) satisfies \( \bfS_{2} \), the shift
\( \SR_{Y}^{\bullet}[-d] \) is an ordinary dualizing complex on \( Y \) by the proof of
Lemma~\ref{lem:ordinaryDC}.
Here, the associated dualizing sheaf \( \SL_{Y} := \SH^{0}(\SR^{\bullet}_{Y}[-d]) \)
is isomorphic to \( \pi^{*}(\SL_{X}) \).
Since \( \pi \) is faithfully flat, we see that
\( \SL_{Y} \) is invertible if and only if \( \SL_{X} \)
is so (cf.\ Lemma~\ref{lem:LocFreeDecsent}).
For an integer \( m \),
let \( \SL_{X}^{[m]} \) (resp.\ \( \SL_{Y}^{[m]} \)) be the double-dual of
\( \SL_{X}^{\otimes m} \) (resp.\ \( \SL_{Y}^{\otimes m} \)).
Then, \( \SL^{[m]}_{Y} \isom \pi^{*}(\SL^{[m]}_{X})  \)
for any \( m \in \BZZ \) by Remark~\ref{rem:dfn:reflexive}.
Hence, for a given integer \( r \), \( \SL_{Y}^{[r]} \) is invertible
if and only if \( \SL_{X}^{[r]} \) is invertible
by the same argument as above.
Therefore, by Lemma~\ref{lem:QGorSch3}\eqref{lem:QGorSch3:3},
\( Y \) is \( \BQQ \)-Gorenstein of index \( r \) if and only if \( X \) is so.
Thus, we are done.
\end{proof}

\begin{rem}\label{rem:QGorScheme:etale}
By Lemma~\ref{lem:QGorScheme:smooth},
we see that the \( \BQQ \)-Gorenstein property is local even in the \'etale topology.
More precisely, for an \'etale morphism \( X' \to X \), for a point \( P \in X \), and for
a point \( P' \in X' \) lying over \( P \),
\( X \) is \( \BQQ \)-Gorenstein of index \( r \) at \( P \)
if and only if \( X' \) is so at \( P' \).
\end{rem}


\subsection{Affine cones of polarized projective schemes over a field}
\label{subsect:cone}

For an affine cone over a projective scheme over a field \( \Bbbk \),
we shall determine
when it is Cohen--Macaulay, Gorenstein, \( \BQQ \)-Gorenstein, etc.,
under suitable conditions.
We fix a field \( \Bbbk \) which is not necessarily algebraically closed.

\begin{dfn}[affine cone]\label{dfn:affcone}
A \emph{polarized projective scheme}
over \( \Bbbk \) is a pair
\( (S, \SA) \) consisting of a projective scheme \( S \) over \( \Bbbk \) and
an ample invertible sheaf \( \SA \) on \( S \).
The polarized projective scheme \( (S, \SA) \) is said to be \emph{connected}
if \( S \) is connected.
For a connected polarized projective scheme \( (S, \SA) \),
the \emph{affine cone}
of \( (S, \SA) \) defined to be \( \Spec R \) for the graded \( \Bbbk \)-algebra
\[ R = R(S, \SA) := \bigoplus\nolimits_{m \geq 0} \OH^{0}(S, \SA^{\otimes m}).
\]
We denote the affine cone by \( \Cone(S, \SA) \).
Note that  the closed subscheme of \( \Cone(S, \SA) = \Spec R \) defined by the ideal
\[ R_{+} = \bigoplus\nolimits_{m > 0} \OH^{0}(S, \SA^{\otimes m})  \]
of \( R \) is isomorphic to \( \Spec \OH^{0}(S, \SO_{S}) \), and
the support of the closed subscheme is a point, since the finite-dimensional \( \Bbbk \)-algebra
\( \OH^{0}(S, \SO_{S}) \) is an Artinian local ring by the connectedness of \( S \).
The point is called the \emph{vertex} of \( \Cone(S, \SA) \).
\end{dfn}

\begin{remn}
The \( \Bbbk \)-algebra \( R(S, \SA) \) above is finitely generated, since \( S \) is projective
and \( \SA \) is ample. Moreover, \( S \isom \Proj R(S, \SA) \).
In some articles, the affine cone of \( (S, \SA) \) is defined to be
\( \Spec R' \) for the graded subring \( R' \) of \( R \) such that
\( R'_{n} = R_{n} \) for any \( n > 0 \) and \( R'_{0} = \Bbbk \).
\end{remn}

Similar results to the following are well-known
on the structure of affine cones (cf.\ \cite[II, Prop.~(8.6.2), (8.8.2)]{EGA}).

\begin{lem}\label{lem:cone1}
For a connected polarized projective scheme \( (S, \SA) \) over \( \Bbbk \),
let \( X \) be the affine cone \( \Cone(S, \SA) \).
Let \( \pi \colon Y \to S \)
be the geometric line bundle associated with \( \SA \),
i.e., \( Y = \BVV(\SA) = \SSpec_{S} \SR\), where
\( \SR = \bigoplus\nolimits_{m \geq 0} \SA^{\otimes m}  \).
Let \( E \) be the zero-section of \( \pi \) corresponding to the projection
\( \SR \to \SO_{S} \)
to the component of degree zero. Then,
\( E \) is a \emph{relative Cartier divisor over} \( S \)
\emph{(}cf.\ \cite[IV, D\'ef.~(21.15.2)]{EGA}\emph{)} with
an isomorphism
\( \SO_{Y}(-E) \isom \pi^{*}\SA \). Moreover,
there exists a projective \( \Bbbk \)-morphism \( \mu \colon Y \to X \) such that
\begin{enumerate}
\item \label{lem:cone1:1} \( \SO_{X} \to \mu_{*}\SO_{Y} \) is an isomorphism,
\item \label{lem:cone1:2} \( \pi^{*}\SA \) is \( \mu \)-ample,
\item \label{lem:cone1:3} \( \mu^{-1}(P) = E \) as a closed subset of \( Y \) for the vertex \( P \)
of \( X \), and
\item \label{lem:cone1:4}
\( \mu \) induces an isomorphism \( Y \setminus E \isom X \setminus P \).
\end{enumerate}
\end{lem}

\begin{proof}
For an open subset \( U = \Spec B\) of \( S \)
with an isomorphism \( \ep \colon \SA|_{U} \isom \SO_{U} \),
we have an isomorphism \( \varphi \colon \pi^{-1}(U) \isom \Spec B[\ttt] \)
for the polynomial \( B \)-algebra \( B[\ttt] \) of one variable
such that \( \varphi \) induces an isomorphism
\[ \OH^{0}(\pi^{-1}(U), \SO_{Y}) = \bigoplus\nolimits_{m \geq 0} \OH^{0}(U, \SA^{\otimes m})
\isom B[\ttt] = \bigoplus\nolimits_{m \geq 0} B\ttt^{m}\]
of graded \( B \)-algebras.
Then, \( E|_{\pi^{-1}(U)} \) is a Cartier divisor corresponding to
\( \Div(\ttt) \) on \( \Spec B[\ttt] \), which is relatively Cartier over \( \Spec B \)
(cf.\ \cite[IV, (21.15.3.3)]{EGA}).
Thus, \( E \) is a relative Cartier divisor over \( S \), since such open subsets \( U \)
cover \( S \).
The exact sequence \( 0 \to \SO_{Y}(-E) \to \SO_{Y} \to \SO_{E} \to 0 \)
induces an isomorphism
\[ \pi_{*}\SO_{Y}(-E) \isom \bigoplus\nolimits_{m \geq 1} \SA^{\otimes m} \isom
\SA \otimes_{\SO_{S}} \SR(-1) \]
of graded \( \SR\)-modules, where \( \SR(-1) \) denotes the twisted graded module.
In particular, \( \SO_{Y}(-E) \isom \pi^{*}\SA \).

The canonical homomorphisms
\( \OH^{0}(S, \SA^{\otimes m}) \otimes_{\Bbbk} \SO_{S} \to \SA^{\otimes m} \)
induce a graded homomorphism \( \Phi \colon R \otimes_{\Bbbk} \SO_{S} \to \SR \)
of graded \( \SO_{S} \)-algebras, where \( R := R(S, \SA) \).
The cokernel of \( \Phi \) is
a finitely generated \( \SO_{S} \)-module,
since \( \SA^{\otimes m} \) is generated by global sections for \( m \gg 0 \).
Hence, \( \SR \) is a finitely generated \( R \otimes_{\Bbbk} \SO_{S} \)-module.
Therefore, \( \Phi \) defines a finite morphism
\[ \nu \colon Y = \SSpec_{S} \SR \to \SSpec_{S} (R \otimes_{\Bbbk} \SO_{S}) \isom
X \times_{\Spec \Bbbk} S \]
over \( S \).
Let \( p_{1} \colon X \times_{\Spec \Bbbk} S \to X \) and \( p_{2} \colon X \times_{\Spec \Bbbk} S \to S \)
be the first and second projections.
Then, \( \mu := p_{1} \circ \nu \colon Y \to X \) is a projective morphism,
since \( S \) is projective over \( \Bbbk \).
Here, \( \SO_{X} \isom \mu_{*}\SO_{Y} \),
since \( \OH^{0}(Y, \SO_{Y}) \isom \OH^{0}(S, \SR) \isom R \).
Moreover \( \pi^{*}\SA \) is \( \mu \)-ample, since \( p_{2}^{*}\SA \) is relatively ample
over \( X \) and \( \pi^{*}\SA \) is the pullback by the finite morphism \( \nu \).
Thus, \( \mu \) satisfies the conditions
\eqref{lem:cone1:1} and \eqref{lem:cone1:2}.
Since the projection \( \SR \to \SO_{S} \) defining \( E \) induces the projection
\( R = \OH^{0}(S, \SR) \to \OH^{0}(S, \SO_{S}) \) to the component of degree zero,
the scheme-theoretic image
\( \mu(E) \) is the zero-dimensional closed subscheme
\( \Spec \OH^{0}(S, \SO_{S})\)
of \( X \) defined by the ideal
\( R_{+} = \bigoplus\nolimits_{m > 0} \OH^{0}(S, \SA^{\otimes m}) \)
of \( R \). Hence, the image \( \mu(E) \)
is set-theoretically the vertex \( P \).
We shall show that the morphism
\[ \mu' \colon Y' := Y \setminus \mu^{-1}(P) \to X' := X \setminus P \]
induced by \( \mu \) is an isomorphism.
Since \( \mu \) is proper, so is \( \mu' \). Moreover,
the structure sheaf \( \SO_{Y'} \) is \( \mu' \)-ample,
since \( \pi^{*}\SA \isom \SO_{Y}(-E)\) is \( \mu \)-ample by \eqref{lem:cone1:2}.
Hence, \( \mu' \) is a finite morphism.
Thus, \( \mu' \) is an isomorphism by \eqref{lem:cone1:1}, since
\( \SO_{X'} \isom \mu'_{*}\SO_{Y'} \).
As a consequence, \eqref{lem:cone1:4} is derived from \eqref{lem:cone1:3},
and it remains to prove \eqref{lem:cone1:3} for \( \mu \) and \( P \).

For a global section \( f \) of \( \SA^{\otimes m} \) for some \( m > 0 \),
we set \( V(f) \) to be the closed subscheme \( \Spec (R/fR) \) of \( X = \Spec R\)
by regarding \( f \) as a homogeneous element of \( R \) of degree \( m \).
We also set a closed subscheme \( W(f) \) of \( S \)
to be the ``zero-subscheme'' of \( f \), i.e., it is defined by the exact sequence
\[ \SA^{\otimes -m} \xrightarrow{\otimes f} \SO_{S} \to \SO_{W(f)} \to 0.\]
The condition \eqref{lem:cone1:3} is derived from the following \eqref{lem:cone1:3a}
for any \( f \) and
for any affine open subsets \( U = \Spec B \) with an isomorphism \( \ep \colon \SA|_{U} \isom \SO_{U} \):
\begin{enumerate}
    \renewcommand{\theenumi}{$\ast$}
    \renewcommand{\labelenumi}{(\theenumi)}
\item \label{lem:cone1:3a} \( \mu^{-1}V(f) \cap \pi^{-1}(U) = (\pi^{-1}W(f) \cup E)
\cap \pi^{-1}(U)\) as a subset of \( \pi^{-1}(U) \).
\end{enumerate}
In fact, if \eqref{lem:cone1:3a} holds for all \( U \) and \( f \), then
\( \mu^{-1}V(f) = \pi^{-1}W(f) \cup E \) for any \( f \),
and we have \( \mu^{-1}(P) = E  \) by \( \bigcap_{f} V(f) = P \)
and \( \bigcap_{f} W(f) = \emptyset\).
Here, \( \bigcap_{f} V(f) = P \) and \( \bigcap_{f} W(f) = \emptyset \) hold,
since all of such \( f \in R\) generate the ideal \( R_{+} \) and since \( \SA \) is ample.
We shall prove \eqref{lem:cone1:3a} as follows.
Let \( \varphi \colon \pi^{-1}(U) \isom \Spec B[\ttt] \) be the isomorphism above
defined by \( \ep \).
We set
\[ b = \ep^{\otimes m}(f|_{U}) \in \OH^{0}(U, \SO_{U}) = B \]
for the induced isomorphism \( \ep^{\otimes m} \colon \SA^{\otimes m}|_{U} \isom \SO_{U} \).
Then, \( W(f) \cap U = \Spec B/bB \), and \( \varphi \) induces isomorphisms
\( \mu^{-1}V(f) \cap \pi^{-1}(U)
\isom \Spec B[\ttt]/(b\ttt^{m}) \) and
\( E \cap \pi^{-1}(U)
\linebreak 
\isom \Spec B[\ttt]/(\ttt) \).
This implies \eqref{lem:cone1:3a}, and we are done.
\end{proof}

\begin{cor}\label{cor:cone1}
In the situation of Lemma~\emph{\ref{lem:cone1}},
for an integer \( k \geq 1 \),
\( S \) satisfies \( \bfS_{k} \)  if and only if
\( X \setminus P \) satisfies \( \bfS_{k} \).
Moreover, \( S \) is Cohen--Macaulay
\emph{(}resp.\ Gorenstein, resp.\ quasi-Gorenstein,
resp.\ \( \BQQ \)-Gorenstein of Gorenstein index \( r \)\emph{)}
if and only if \( X \setminus P \) is so.
\end{cor}

\begin{proof}
This is a consequence of Lemmas~\ref{lem:QGorScheme:smooth} and \ref{lem:cone1},
since \( X \setminus P \isom Y \setminus E \) is smooth and surjective over \( S \).
\end{proof}

The following result is essentially well-known
(cf.\ \cite[Prop.~1.7]{MoriAffCone}, \cite[Lem.~4.3]{Pa}).

\begin{prop}\label{prop:coneSk}
Let \( X \) be the affine cone of a connected polarized projective
scheme \( (S, \SA) \) over \( \Bbbk \) and
let \( P \) be the vertex of \( X \).
For a coherent \( \SO_{S} \)-module \( \SG \), we set \( \SF = \mu_{*}(\pi^{*}\SG) \)
for the morphisms \( \mu \colon Y \to X \) and \( \pi \colon Y \to S \)
in Lemma~\emph{\ref{lem:cone1}}
for the geometric line bundle \( Y = \BVV_{S}(\SA)\) over \( S \).
We define also \( \widetilde{\SF} := j_{*}(\SF|_{X \setminus P}) \)
for the open immersion \( j \colon X \setminus P \injmap X \), and
for simplicity, we define
\[ \OH^{i}(\SG(m)) := \OH^{i}(S, \SG \otimes_{\SO_{S}} \SA^{\otimes m}) \]
for \( m \in \BZZ \) and \( i \geq 0 \).
Then, the following hold\emph{:}
\begin{enumerate}
\addtocounter{enumi}{-1}
\item \label{prop:coneSk:0}
If \( \SG = \SO_{S} \), then \( \SF \isom \SO_{X} \).

\item \label{prop:coneSk:D1}
The inequality \( \depth \SF_{P} \geq 1 \) holds\emph{;}
Equivalently, \( \SF \injmap \widetilde{\SF} \) is injective.

\item \label{prop:coneSk:D2}
The inequality \( \depth \SF_{P} \geq 2 \) holds if and only if
\( \OH^{0}(\SG(m)) = 0 \) for any \( m < 0 \).
This condition is also equivalent to that \( \SF \isom \widetilde{\SF} \).

\item \label{prop:coneSk:coh}
The quasi-coherent \( \SO_{X} \)-module \( \widetilde{\SF} \) is coherent
if and only if \( \OH^{0}(\SG(m)) = 0 \) for \( m \ll 0 \).
In particular, \( \widetilde{\SF} \) is coherent if \( \SG \) satisfies \( \bfS_{1} \)
and every irreducible component of \( \Supp \SG \) has positive dimension.

\item \label{prop:coneSk:D3}
Assume that \( \widetilde{\SF} \) is coherent. Then, for an integer \( k \geq 3 \),
\( \depth \widetilde{\SF}_{P} \geq k \) holds if and only if \( \OH^{i}(\SG(m)) = 0 \)
for any \( m \in \BZZ\) and  \( 0 < i < k - 1\).

\item  \label{prop:coneSk:S1}
The \( \SF \) satisfies \( \bfS_{1} \) if and only if \( \SG \) satisfies \( \bfS_{1} \).

\item  \label{prop:coneSk:S2}
The \( \SF \) satisfies \( \bfS_{2} \) if and only if
\( \SG \) satisfies \( \bfS_{2} \) and \( \OH^{0}(\SG(m)) = 0 \)
for any \( m < 0 \).

\item \label{prop:coneSk:Sk}
Assume that \( \widetilde{\SF} \) is coherent. Then, for an integer \( k \geq 3 \),
\( \widetilde{\SF} \) satisfies \( \bfS_{k} \) if and only if
\( \SG \) satisfies \( \bfS_{k} \) and \( \OH^{i}(\SG(m)) = 0 \)
for any \( m \in \BZZ \) and \( 0 < i < k -1 \).

\item \label{prop:coneSk:CM}
Assume that \( \widetilde{\SF} \) is coherent.
Then, \( \widetilde{\SF} \) is a Cohen--Macaulay \( \SO_{X} \)-module
if and only \( \SG \) is a Cohen--Macaulay \( \SO_{S} \)-module and \( \OH^{i}(\SG(m)) = 0 \)
for any \( m \in \BZZ\) and \( 0 < i < \dim \Supp \SG\).
\end{enumerate}
\end{prop}

\begin{proof}
The assertion \eqref{prop:coneSk:0} is a consequence of Lemma~\ref{lem:cone1}\eqref{lem:cone1:1}.
We consider the local cohomology sheaves
\( \SH^{i}_{P}(\SF') \)
with support in \( P \) for \( \SF' = \SF \) or \( \SF' = \widetilde{\SF} \).
These are quasi-coherent sheaves on \( X \) supported on \( P \)
(cf.\ \cite[Prop.~2.1]{LC}).
Thus,
\[ \OH^{i}_{P}(X, \SF') \isom \OH^{0}(X, \SH^{i}_{P}(\SF')) \]
and it is also isomorphic to the stalk
\( (\SH^{i}_{P}(\SF'))_{P}\) at \( P \).
Note that, for a positive integer \( k \), when \( \SF' \) is coherent,
\( \depth \SF'_{P} \geq k \) if and only if
\( (\SH^{i}_{P}(\SF'))_{P} = 0 \) for any \( i < k \)
(cf.\ Property~\ref{ppty:depth<=2}).
There exist an exact sequence
\[ 0 \to \OH^{0}_{P}(X, \SF') \to \OH^{0}(X, \SF')
\to \OH^{0}(X \setminus P, \SF') \to
\OH^{1}_{P}(X, \SF') \to 0\]
and isomorphisms
\( \OH^{i}(X \setminus P, \SF') \isom \OH^{i+1}_{P}(X, \SF') \)
for all \( i \geq 1 \) (cf.\ \cite[Prop.~2.2]{LC}).
Hence, if \( \SF' \) is a coherent \( \SO_{X} \)-module, then
\( \depth \SF'_{P} \geq k \) if and only if
\begin{enumerate}
    \renewcommand{\theenumi}{\roman{enumi}}
    \renewcommand{\labelenumi}{(\theenumi)}
\item \label{prop:coneSk:condD1}
\( \OH^{0}(X, \SF') \to \OH^{0}(X \setminus P, \SF')  \) is injective,
when \( k = 1 \),

\item \label{prop:coneSk:condD2}
\( \OH^{0}(X, \SF') \to \OH^{0}(X \setminus P, \SF') \) is an isomorphism,
when \( k = 2 \), and

\item  \label{prop:coneSk:condD3}
\( \OH^{0}(X, \SF') \to \OH^{0}(X \setminus P, \SF') \) is an isomorphism and
\( \OH^{i}(X \setminus P, \SF') = 0 \) for any \( 0 <  i  < k -1\), when \( k \geq 3 \).
\end{enumerate}
By construction and by Lemma~\ref{lem:cone1}\eqref{lem:cone1:4},
we have isomorphisms
\begin{align*}
\OH^{0}(X, \SF) \isom \OH^{0}(Y, \pi^{*}\SG)
&\isom \bigoplus\nolimits_{m \geq 0} \OH^{0}(\SG(m)),
\quad \text{and}\\
\OH^{i}(X \setminus P, \SF) \isom \OH^{i}(Y \setminus E, \pi^{*}\SG)
&\isom \bigoplus\nolimits_{m \in \BZZ} \OH^{i}(\SG(m))
\end{align*}
for any \( i \geq 0 \),
where the homomorphism \( \OH^{0}(Y, \pi^{*}\SG) \to \OH^{0}(Y \setminus E, \pi^{*}\SG) \)
is an injection and is the identity on each component \( \OH^{0}(\SG(m)) \)
of degree \( m \geq 0\).
We have \eqref{prop:coneSk:D1}, \eqref{prop:coneSk:D2}, and \eqref{prop:coneSk:D3}
by considering the conditions \eqref{prop:coneSk:condD1}--\eqref{prop:coneSk:condD3} above.
Moreover, \eqref{prop:coneSk:coh} holds,
since \( \widetilde{\SF} \) is coherent if and only if \( \widetilde{\SF}_{P}/\SF_{P} \) 
is a finite-dimensional \( \Bbbk \)-vector space, and since we have an isomorphism
\[ \widetilde{\SF}_{P}/\SF_{P} \isom \bigoplus\nolimits_{m < 0} \OH^{0}(\SG(m)) \]
by the argument above: This implies the first half of  \eqref{prop:coneSk:coh}, and
the second half follows from Lemma~\ref{lem:proj+S1}.

For an integer \( k > 0 \), \( \SF|_{X \setminus P} \) satisfies \( \bfS_{k} \) if and only if
\( \SG \) satisfies \( \bfS_{k} \) by \cite[IV, Cor.~(6.4.2)]{EGA}, since
\( Y \setminus E \isom X \setminus P \).
Thus, the assertion \eqref{prop:coneSk:S1}
(resp.\ \eqref{prop:coneSk:S2}, resp.\ \eqref{prop:coneSk:Sk})
follows from \eqref{prop:coneSk:D1} (resp.\ \eqref{prop:coneSk:D2}, resp.\ \eqref{prop:coneSk:D3})
by the equivalence: \eqref{lem:basicSk:condA} \( \Leftrightarrow \) \eqref{lem:basicSk:condD}
in Lemma~\ref{lem:basicSk} applied to \( Z = P \).
The last assertion \eqref{prop:coneSk:CM} is a consequence of
\eqref{prop:coneSk:Sk}, since
\( \dim \widetilde{\SF}_{P} = \dim \Supp \SG + 1\).
\end{proof}

\begin{prop}\label{prop:coneGor}
Let \( (S, \SA) \) be a connected polarized projective scheme over \( \Bbbk \)
and let \( X \) be the affine cone \( \Cone(S, \SA) \).
Let \( \pi \colon Y \to S \) and \( \mu \colon Y \to X \) be the morphisms
in Lemma~\emph{\ref{lem:cone1}}.
Assume that \( X \) satisfies \( \bfS_{2} \) and \( n := \dim S > 0 \).
Then,
\begin{enumerate}
    \addtocounter{enumi}{-1}
\item \label{prop:coneGor:0}
\( S \) and \( Y \) also satisfy \( \bfS_{2} \), and the schemes
\( S \), \( Y \), and \( X \) are all equi-dimensional.
\end{enumerate}
Let \( \omega_{X/\Bbbk} \) \emph{(}resp.\ \( \omega_{Y/\Bbbk} \), resp.\ \( \omega_{S/\Bbbk} \)\emph{)}
be the canonical sheaf of \( X \) \emph{(}resp.\ \( Y \), resp.\ \( S \)\emph{)}
in the sense of Definition~\emph{\ref{dfn:canosheaf}},
and let \( \omega_{X/\Bbbk}^{[r]} \) \emph{(}resp.\ \( \omega_{S/\Bbbk}^{[r]} \)\emph{)}
denote the double-dual of \(  \omega_{X/\Bbbk}^{\otimes r}\)
\emph{(}resp.\ \(  \omega_{S/\Bbbk}^{\otimes r}\)\emph{)} for an integer \( r \).

\begin{enumerate}
\item \label{prop:coneGor:1}
There exist isomorphisms
\begin{align}
\omega_{Y/\Bbbk} &\isom
\pi^{*}(\omega_{S/\Bbbk} \otimes_{\SO_{S}} \SA) \quad \text{and}
\label{eq:prop:coneGor:isom1} \\
\omega_{Y/\Bbbk}^{[r]} &\isom
\pi^{*}(\omega_{S/\Bbbk}^{[r]} \otimes_{\SO_{S}} \SA^{\otimes r} )
\label{eq:prop:coneGor:isom2}
\end{align}
for any integer \( r \). Moreover, \( \omega_{X/\Bbbk}^{[r]} \) is isomorphic to
the double-dual of \( \mu_{*}(\omega_{Y/\Bbbk}^{[r]}) \) for any integer \( r \).

\item \label{prop:coneGor:2}
For any integer \( r \) and for any integer \( k \geq 3 \),
\[ \depth (\omega_{X/\Bbbk}^{[r]})_{P} \geq k \]
holds for the vertex \( P \) of \( X \) if and only if
\[  \OH^{i}(S, \omega_{S/\Bbbk}^{[r]} \otimes_{\SO_{S}} \SA^{\otimes m}) = 0 \]
for any \( m \in \BZZ \) and any \( 0 < i < k - 1 \).
Moreover, \( \omega_{X/\Bbbk}^{[r]} \) satisfies \( \bfS_{k} \)
for the same \( r \) and \( k \)
if and only if
\( \omega_{S/\Bbbk}^{[r]} \) satisfies \( \bfS_{k} \) and
\[  \OH^{i}(S, \omega_{S/\Bbbk}^{[r]} \otimes_{\SO_{S}} \SA^{\otimes m}) = 0 \]
for any \( m \in \BZZ \) and any \( 0 < i < k - 1 \).

\item \label{prop:coneGor:3}
For any positive integer \( r \),
the following three conditions are equivalent to each other\emph{:}
\begin{enumerate}
    \makeatletter
    \renewcommand{\p@enumii}{}
    \makeatother
    \renewcommand{\theenumii}{\roman{enumii}}
    \renewcommand{\labelenumii}{(\theenumii)}
\item \label{prop:coneGor:cond1} \( \omega_{X/\Bbbk}^{[r]} \isom \SO_{X} \)\emph{;}
\item \label{prop:coneGor:cond2} \( \omega_{X/\Bbbk}^{[r]} \) is invertible\emph{;}
\item \label{prop:coneGor:cond3}
\( \omega_{S/\Bbbk}^{[r]} \isom \SA^{\otimes l} \) for an integer \( l \).
\end{enumerate}
\end{enumerate}
\end{prop}

\begin{proof}
The assertion \eqref{prop:coneGor:0} is a consequence of
Proposition~\ref{prop:coneSk}\eqref{prop:coneSk:S2} for \( \SG = \SO_{S} \),
Lemma~\ref{lem:QGorScheme:smooth},
and Fact~\ref{fact:S2}\eqref{fact:S2:1}.
Let \( \omega_{S/\Bbbk}^{\bullet} \)
(resp.\ \( \omega_{Y/\Bbbk}^{\bullet} \), resp.\ \( \omega_{X/\Bbbk}^{\bullet} \))
be the canonical dualizing complex of \( S \) (resp.\ \( Y \), resp.\ \( X \))
in the sense of Definition~\ref{dfn:DCalg}.
Note that \( \omega^{\bullet}_{S/\Bbbk}[-n] \) (resp.\ \( \omega_{Y/\Bbbk}^{\bullet}[-n-1] \),
resp.\ \( \omega_{X/\Bbbk}^{\bullet}[-n-1] \))
is an ordinary dualizing complex by Lemma~\ref{lem:algschemeSk} for \( n = \dim S \).
Then,
\[ \omega_{Y/\Bbbk}^{\bullet}
\isom_{\qis} \varOmega^{1}_{Y/S}[1] \otimes^{\bfL}_{\SO_{Y}} \bfL \pi^{*}(\omega_{S/\Bbbk}^{\bullet})
\isom_{\qis} \bfL \pi^{*}( \SA \otimes^{\bfL}_{\SO_{S}} \omega_{S/\Bbbk}^{\bullet})[1],\]
since \( \pi \) is separated and smooth (cf.\ Example~\ref{exam:RDembeddable}) and
since there is an isomorphism \( \varOmega^{1}_{Y/S} \isom \pi^{*}\SA \)
(cf.\ \cite[IV, Cor.~16.4.9]{EGA}).
Thus, we have the isomorphism \eqref{eq:prop:coneGor:isom1}.
By taking double-dual of tensor powers of both sides of \eqref{eq:prop:coneGor:isom1},
we have the isomorphism \eqref{eq:prop:coneGor:isom2}
for any integer \( r \) by Remark~\ref{rem:dfn:reflexive}.
Since \( X \) satisfies \( \bfS_{2} \), any reflexive \( \SO_{X} \)-module \( \SF \)
satisfies \( \bfS_{2} \) by Corollary~\ref{cor:prop:S1S2:reflexive},
and moreover, \( \depth_{P} \SF \geq 2 \), since \( \Codim(P, X) = \dim X = n + 1 \geq 2  \).
Thus, we have isomorphisms
\[ \omega_{X/\Bbbk}^{[r]} \isom j_{*}(\omega^{[r]}_{X \setminus P/\Bbbk})
\isom j_{*}(\mu_{*}(\omega^{[r]}_{Y/\Bbbk})|_{X \setminus P}))
\isom (\mu_{*}(\omega^{[r]}_{Y/\Bbbk}))^{\vee\vee}\]
for any integer \( r \) and for the open immersion \( j \colon X \setminus P \injmap X \).
This proves \eqref{prop:coneGor:1}.

By \eqref{prop:coneGor:1}, we see that
\eqref{prop:coneGor:2} is a consequence of
\eqref{prop:coneSk:D3} and \eqref{prop:coneSk:Sk} of
Proposition~\ref{prop:coneSk} applied to the case:
\( \SG = \omega^{[r]}_{S/\Bbbk} \otimes \SA^{\otimes r}\),
where \( \widetilde{\SF} \isom \omega_{X/\Bbbk}^{[r]} \).
It remains to prove the equivalence of the conditions
\eqref{prop:coneGor:cond1}--\eqref{prop:coneGor:cond3}
of \eqref{prop:coneGor:3}.
Since \eqref{prop:coneGor:cond1} \( \Rightarrow \) \eqref{prop:coneGor:cond2}
is trivial, it is enough to prove
\eqref{prop:coneGor:cond2} \( \Rightarrow \) \eqref{prop:coneGor:cond3}
and \eqref{prop:coneGor:cond3} \( \Rightarrow \) \eqref{prop:coneGor:cond1}.

Proof of
\eqref{prop:coneGor:cond3} \( \Rightarrow  \) \eqref{prop:coneGor:cond1}:
Assume that \( \omega_{S/\Bbbk}^{[r]} \isom \SA^{\otimes l}  \)
for some \( r > 0 \) and \( l \in \BZZ \).
Since \( \SO_{Y}(-E) \isom \pi^{*}\SA \) for the zero-section \( E \)
of Lemma~\ref{lem:cone1}, we have
\[ \omega_{Y/\Bbbk}^{[r]} \otimes_{\SO_{Y}} \SO_{Y}((r + l)E) \isom
\pi^{*}(\omega^{[r]}_{S/\Bbbk} \otimes \SA^{\otimes r} \otimes_{\SO_{S}} \SA^{\otimes -(r+l)})
\isom \SO_{Y} \]
from the isomorphism in \eqref{prop:coneGor:1}.
By taking \( \mu_{*} \), we have:
\( \omega_{X/\Bbbk}^{[r]} \isom \pi_{*}\SO_{Y} \isom \SO_{X} \).

Proof of
\eqref{prop:coneGor:cond2} \( \Rightarrow  \) \eqref{prop:coneGor:cond3}:
Assume that \( \omega_{X/\Bbbk}^{[r]} \) is invertible. Then,
\( \omega_{Y/\Bbbk}^{[r]} \) is invertible on \( Y \setminus E \),
since \( Y \setminus E \isom X \setminus P \).
Moreover, \( \omega_{S/\Bbbk}^{[r]} \) is invertible
by \eqref{eq:prop:coneGor:isom2},
since \( Y \setminus E \to S \) is faithfully flat (cf.\ Lemma~\ref{lem:LocFreeDecsent}).
Thus, \( \omega_{Y/\Bbbk}^{[r]} \) is also invertible again
by \eqref{eq:prop:coneGor:isom2}.
There is an injection
\[ \phi \colon \omega_{Y/\Bbbk}^{[r]} \otimes_{\SO_{Y}} \SO_{Y}(-bE)
\injmap \mu^{*}(\omega_{X/\Bbbk}^{[r]}) \]
for some integer \( b \) such that the cokernel of \( \phi \) is supported on \( E \).
In fact, for any integer \( b \), we have a canonical homomorphism
\begin{multline*}
\mu_{*}(\omega_{Y/\Bbbk}^{[r]} \otimes_{\SO_{Y}} \SO_{Y}(-bE)) \injmap
j_{*}(\mu_{*}(\omega_{Y/\Bbbk}^{[r]} \otimes_{\SO_{Y}} \SO_{Y}(-bE))|_{X \setminus P})
\\ \isom j_{*}(\mu_{*}(\omega_{Y/\Bbbk}^{[r]})|_{X \setminus P})
\isom \omega^{[r]}_{X/\Bbbk}
\end{multline*}
whose cokernel is supported on \( P \), and
if \( b \) is sufficiently large, then
\[ \mu^{*}\mu_{*}(\omega_{Y/\Bbbk}^{[r]} \otimes_{\SO_{Y}} \SO_{Y}(-bE)) \to
\omega_{Y/\Bbbk}^{[r]} \otimes_{\SO_{Y}} \SO_{Y}(-bE)\]
is surjective, since \( \SO_{Y}(-E) \isom \pi^{*}\SA\) is relatively ample over \( X \).
Thus,
\[
\mu^{*}\mu_{*}(\omega_{Y/\Bbbk}^{[r]} \otimes_{\SO_{Y}} \SO_{Y}(-bE))
\to \mu^{*}(\omega^{[r]}_{X/\Bbbk})
\]
induces the injection \( \phi \), since the invertible sheaf \( \mu^{*}(\omega^{[r]}_{X/\Bbbk}) \) 
does not contain non-zero coherent \( \SO_{Y} \)-submodule whose support is contained in \( E \), 
by the \( \bfS_{1} \)-condition on \( Y \).
Let \( b \) be a minimal integer with an injection \( \phi \) above.
Then, \( \phi \) is an isomorphism.
This is shown as follows. The homomorphism
\[ \phi|_{E} \colon (\omega_{Y/\Bbbk}^{[r]} \otimes_{\SO_{Y}} \SO_{Y}(-bE)) \otimes_{\SO_{Y}} \SO_{E}
\to  \mu^{*}(\omega_{X/\Bbbk}^{[r]}) \otimes_{\SO_{Y}} \SO_{E}\]
is not zero by  the minimality of \( b \).
Here, \( \phi|_{E} \)
corresponds to a non-zero homomorphism
\[ \omega_{S/\Bbbk}^{[r]} \otimes_{\SO_{S}} \SA^{\otimes (r + b)} \to \SO_{S}\]
by the isomorphism \( \pi|_{E} \colon E \isom S \) and by
\eqref{eq:prop:coneGor:isom2}.
In particular, there is an non-empty open subset \( U \subset S\)
such that \( \phi \) is an isomorphism on \( \pi^{-1}(U) \).
On the other hand, since \( \phi \) is an injection between invertible sheaves,
there is an effective Cartier divisor
\( D \) on \( Y \) such that the cokernel of \( \phi \) is isomorphic to
\( \SO_{D} \otimes_{\SO_{Y}} \pi^{*}(\omega_{X/\Bbbk}^{[r]}) \) and that
\( \Supp D \subset E \). Then, \( D \) is a relative Cartier divisor over \( S \),
since every fiber of \( \pi \) is \( \BAA^{1} \) (cf.\ \cite[IV, (21.15.3.3)]{EGA}).
Thus, \( \pi|_{D} \colon D \to S \) is a flat and finite morphism.
If \( D \ne 0 \), then \( \pi(D) = S \) by the connectedness of \( S \), and
it contradicts \( \Supp D \cap \pi^{-1}(U) = \emptyset \).
Thus, \( D = 0 \), and consequently, \( \phi \) is an isomorphism.

Therefore, we have an isomorphism
\[ \omega_{S/\Bbbk}^{[r]} \otimes_{\SO_{S}} \SA^{\otimes (r + b)} \isom \SO_{S} \]
corresponding to the isomorphism \( \phi|_{E} \), and
the condition \eqref{prop:coneGor:cond3} is satisfied for \( l = -(r + b) \).
Thus, we have proved the equivalence of
\eqref{prop:coneGor:cond1}--\eqref{prop:coneGor:cond3}, and we are done.
\end{proof}

\begin{cor}\label{cor:coneGor}
Let \( X \) be the affine cone of a connected polarized scheme
\( (S, \SA) \) over \( \Bbbk \).
Assume that \( n = \dim S > 0 \) and
\( \OH^{0}(S, \SA^{\otimes m}) = 0 \) for any \( m < 0 \).
Then, the following hold\emph{:}

\begin{enumerate}
\item \label{cor:coneGor:Gor}
The scheme \( X \) is Gorenstein if and only if
\begin{itemize}
\item  \( S \) is Gorenstein,

\item
\( \OH^{i}(S, \SA^{\otimes m}) = 0 \) for any \( 0 < i < n \) and any \( m \in \BZZ \), and

\item  \( \omega_{S/\Bbbk} \isom \SA^{\otimes l} \) for some integer \( l \).
\end{itemize}

\item  \label{cor:coneGor:qG}
The scheme \( X \) is quasi-Gorenstein
if and only if \( S \) is quasi-Gorenstein and
\( \omega_{S/\Bbbk} \isom \SA^{\otimes l} \) for some integer \( l \).

\item \label{cor:coneGor:QGor}
The scheme \( X \) is \( \BQQ \)-Gorenstein
if and only if \( S \) is \( \BQQ \)-Gorenstein
and \( \omega_{S/\Bbbk}^{[r]} \isom \SA^{\otimes l} \)
for some integers \( r > 0 \) and \( l \).
\end{enumerate}
\end{cor}

\begin{proof}
The assertion \eqref{cor:coneGor:Gor} follows from \eqref{cor:coneGor:qG} and
Proposition~\ref{prop:coneSk}\eqref{prop:coneSk:CM}.
The ``only if'' parts of \eqref{cor:coneGor:qG} and \eqref{cor:coneGor:QGor}
are shown as follows.
Assume that \( X \) is \( \BQQ \)-Gorenstein of Gorenstein index \( r \).
Note that \( X \) is quasi-Gorenstein if and only if \( r = 1 \) by
Lemma~\ref{lem:QGorSch}\eqref{lem:QGorSch:1}.
Then, \( S \) is \( \BQQ \)-Gorenstein by Corollary~\ref{cor:cone1}.
Moreover,
\( \omega_{S/\Bbbk}^{[r]} \isom \SA^{\otimes l} \) for some \( l \in \BZZ \)
by the implication \eqref{prop:coneGor:cond2} \( \Rightarrow \)
\eqref{prop:coneGor:cond3} of Proposition~\ref{prop:coneGor}\eqref{prop:coneGor:3}.
Thus, the ``only if'' parts are proved.
The ``if'' parts of
\eqref{cor:coneGor:qG} and \eqref{cor:coneGor:QGor} are shown as follows.
Assume that \( S \) is \( \BQQ \)-Gorenstein.
Then, \( X \setminus P \) is \( \BQQ \)-Gorenstein by Corollary~\ref{cor:cone1}.
In particular, \( \Codim(X \setminus X^{\circ}, X) \geq 2\) for the Gorenstein locus
\( X^{\circ} = \Gor(X) \).
Moreover, \( X \) satisfies \( \bfS_{2} \)
by Proposition~\ref{prop:coneSk}\eqref{prop:coneSk:S2}, since \( S \) satisfies \( \bfS_{2} \)
and \( \OH^{0}(S, \SA^{\otimes m}) = 0 \) for any \( m < 0 \) by assumption.
If \( \omega_{S/\Bbbk}^{[r]} \isom \SA^{\otimes l} \)
for integers \( r > 0 \) and \( l \),
then \( \omega_{X/\Bbbk}^{[r]} \) is invertible by the implication
\eqref{prop:coneGor:cond3} \( \Rightarrow \)
\eqref{prop:coneGor:cond2} of Proposition~\ref{prop:coneGor}\eqref{prop:coneGor:3}.
Thus, \( X \) is \( \BQQ \)-Gorenstein. This proves the ``if''
part of \eqref{cor:coneGor:QGor}.
The ``if'' part of \eqref{cor:coneGor:qG} follows also from the argument above
by setting \( r = 1 \).
Thus, we are done.
\end{proof}

\begin{cor}\label{cor:coneGor2}
Let \( X \) be the affine cone of a connected polarized scheme
\( (S, \SA) \) over \( \Bbbk \).
Assume that \( S \) is Cohen--Macaulay, \( n := \dim S > 0 \),
and
\[ \OH^{i}(S, \SA^{\otimes m})
= \OH^{i}(S, \omega_{S/\Bbbk} \otimes \SA^{\otimes m}) = 0 \]
for any \( i > 0 \) and \( m > 0 \).
Then, the following hold\emph{:}
\begin{enumerate}
\item \label{cor:coneGor2:-1}
The affine cone \( X \) satisfies \( \bfS_{2} \).
In particular, \( S \) is reduced \emph{(}resp.\ normal\emph{)} if and only if \( X \) is so.

\item \label{cor:coneGor2:0}
The following conditions are equivalent to each other for an integer \( k \geq 3 \)\emph{:}
\begin{enumerate}
\item  \( \depth \SO_{X, P} \geq k \)\emph{;}

\item  \( X \) satisfies \( \bfS_{k} \)\emph{;}

\item  \( \OH^{i}(S, \SO_{S}) = 0 \) for any \( 0 < i < k-1 \).
\end{enumerate}

\item \label{cor:coneGor2:1}
The affine cone \( X \) is Cohen--Macaulay if and only if \( \OH^{i}(S, \SO_{S}) = 0 \)
for any \( 0 < i < n \).

\item \label{cor:coneGor2:12}
The following conditions are equivalent to each other for an integer \( k \geq 3 \)\emph{:}
\begin{enumerate}
\item  \( \depth (\omega_{X/\Bbbk})_{P} \geq k \)\emph{;}

\item  \( \omega_{X/\Bbbk} \) satisfies \( \bfS_{k} \)\emph{;}

\item  \( \OH^{i}(S, \SO_{S}) = 0 \) for any \( n - k + 1 < i < n \).
\end{enumerate}

\item \label{cor:coneGor2:2}
When \( S \) is Gorenstein, \( X \) is \( \BQQ \)-Gorenstein if and only if
\( \omega_{S/\Bbbk}^{\otimes r} \isom \SA^{\otimes l} \)
for some integers \( r > 0 \) and \( l \).

\item \label{cor:coneGor2:3}
When \( S \) is Gorenstein, \( X \) is Gorenstein if and only if
\( \omega_{S/\Bbbk} \isom \SA^{\otimes l} \)
for some \( l \in \BZZ \) and
if \( \OH^{i}(S, \SO_{S}) = 0 \)
for any \( 0 < i < n\).
\end{enumerate}
\end{cor}

\begin{proof}
By duality (cf.\ Corollary~\ref{cor:SerreDual}),
we have
\[ \OH^{i}(S, \SA^{\otimes m}) \isom
\OH^{n - i}(S, \omega_{S/\Bbbk} \otimes_{\SO_{S}} \SA^{\otimes -m})^{\vee} \]
for any integers \( m \) and \( i \), and by assumption,
this is zero either if \( m > 0 \) and \( i > 0 \) or if \( m < 0 \) and \( i < n \).
Thus, \( X \) satisfies \( \bfS_{2} \) by considering the case: \( m < 0  \) and \( i = 0 \)
and by Proposition~\ref{prop:coneSk}\eqref{prop:coneSk:S2} applied to \( \SG = \SO_{S} \).
This proves \eqref{cor:coneGor2:-1}.
The assertion \eqref{cor:coneGor2:0} (resp.\ \eqref{cor:coneGor2:12}) is a consequence of
\eqref{prop:coneSk:D3} and \eqref{prop:coneSk:Sk} of
Proposition~\ref{prop:coneSk} applied to \( \SG = \SO_{S} \)
(resp.\ \( \SG = \omega_{S/\Bbbk} \otimes \SA\)).
Similarly, the assertion \eqref{cor:coneGor2:1} is a consequence of
Proposition~\ref{prop:coneSk}\eqref{prop:coneSk:CM} applied to \( \SG = \SO_{S} \).
Moreover, the assertion \eqref{cor:coneGor2:2} (resp.\ \eqref{cor:coneGor2:3}) is
derived from  \eqref{cor:coneGor:QGor} (resp.\ \eqref{cor:coneGor:Gor}) of
Corollary~\ref{cor:coneGor}. Thus, we are done.
\end{proof}


\section{\texorpdfstring{$\BQQ$}{Q}-Gorenstein morphisms}
\label{sect:QGormor}

Section~\ref{sect:QGormor} introduces the notion of ``\( \BQQ \)-Gorenstein morphism''
and its weak forms:  ``naively \( \BQQ \)-Gorenstein morphism''
and ``virtually \( \BQQ \)-Gorenstein morphism.''  We inspect relations between these three
notions, and 
prove basic properties and several theorems on \( \BQQ \)-Gorenstein morphisms.

In Sections~\ref{subsect:QGormor} and \ref{subsect:virQGormor}, 
we define the notions of \( \BQQ \)-Gorenstein morphism, naively \( \BQQ \)-Gorenstein 
morphism, and virtually \( \BQQ \)-Gorenstein morphism, 
and we discuss their properties 
giving some criteria for a morphism to be \( \BQQ \)-Gorenstein. 
A \( \BQQ\)-Gorenstein morphism is always
naively and virtually \( \BQQ\)-Gorenstein. 
In Section~\ref{subsect:QGormor}, we provide a new example 
of naively \( \BQQ \)-Gorenstein morphisms which are not \( \BQQ \)-Gorenstein, 
by Lemma~\ref{lem:Lee} and Example~\ref{exam:KummerType}, 
and discuss the relative Gorenstein index for a naively \( \BQQ\)-Gorenstein morphism
in Proposition~\ref{prop:GorIndexQGorMor}.
Theorem~\ref{thm:wQGvsQG} in 
Section~\ref{subsect:virQGormor} shows that 
a virtually \( \BQQ \)-Gorenstein morphism is a \( \BQQ \)-Gorenstein morphism under some mild conditions. 
In Section~\ref{subsect:propQGormor}, 
several basic properties including base change of \( \BQQ\)-Gorenstein morphisms 
and of their variants are discussed. 

Finally, in Section~\ref{subsect:thmsQGormor}, we shall prove notable theorems. 
We prove three criteria for a morphism to be \( \BQQ\)-Gorenstein: 
an infinitesimal criterion (Theorem~\ref{thm:InfQGorCrit}), 
a valuative criterion (Theorem~\ref{thm:valcritQGor}), and 
a criterion by \( \bfS_3\)-conditions on fibers (Theorem~\ref{thm:S3Gor2}).
Moreover, we prove the existence theorem of 
\( \BQQ\)-Gorenstein refinement (Theorem~\ref{thm:QGorRef}) 
and its variants (Theorems~\ref{thm:QGorRefLocal} and \ref{thm:dec naive}).


\subsection{\texorpdfstring{$\BQQ$}{Q}-Gorenstein morphisms and naively \texorpdfstring{$\BQQ$}{Q}-Gorenstein morphisms}
\label{subsect:QGormor}

\begin{dfn}\label{dfn:QGorMor}
Let \( f \colon Y \to T \) be an \( \bfS_{2} \)-morphism of locally Noetherian schemes
such that every fiber is \( \BQQ \)-Gorenstein.
Let \( \omega_{Y/T} \) denote the relative canonical sheaf
in the sense of Definition~\ref{dfn:relcanosheaf} and let \( \omega_{Y/T}^{[m]} \)
denote the double-dual of \( \omega_{Y/T}^{\otimes m} \) for \( m \in \BZZ \).
\begin{enumerate}
\item \label{dfn:QGorMor:naive}
The morphism \( f \) is said to be \emph{naively} \( \BQQ \)-\emph{Gorenstein}
at a point \( y \in Y\)
if \( \omega_{Y/T}^{[r]} \) is invertible at \( y \) for some integer \( r > 0 \).
If \( f \) is naively \( \BQQ \)-Gorenstein at every point of \( Y \), then
it is called a \emph{naively} \( \BQQ \)-\emph{Gorenstein morphism}.

\item \label{dfn:QGorMor:QGor} If \( \omega_{Y/T}^{[m]} \) satisfies
relative \( \bfS_{2} \) over \( T \) (cf.\ Definition~\ref{dfn:RelSkCMlocus})
for any \( m \in \BZZ \),
then \( f \) is called  a \( \BQQ \)-\emph{Gorenstein morphism}.
If \( f|_{U} \colon U \to T \) is a \( \BQQ \)-Gorenstein morphism
for an open neighborhood \( U \) of a point \( y \in Y \),
then \( f \) is said to be \( \BQQ \)-Gorenstein at \( y \).
\end{enumerate}
\end{dfn}

\begin{rem}\label{rem:dfn:QGorMor}
For an \( \bfS_{2} \)-morphism \( f \colon Y \to T \) of locally Noetherian schemes,
if every fiber is Gorenstein in codimension one and if \( \omega_{Y/T} \)
is an invertible \( \SO_{Y} \)-module, then \( f \) is a \( \BQQ \)-Gorenstein morphism.
In fact, \( \omega_{Y/T}^{[m]} \isom \omega_{Y/T}^{\otimes m} \) satisfies relative
\( \bfS_{2} \) over \( T \)
for any \( m \in \BZZ \) and every fiber \( Y_{t} = f^{-1}(t)\) is \( \BQQ \)-Gorenstein
of Gorenstein index one,
since \( \omega_{Y/T} \otimes_{\SO_{Y}} \SO_{Y_{t}} \isom \omega_{Y_{t}/\Bbbk(t)} \)
(cf.\ Proposition~\ref{prop:BC-S2CM}).
\end{rem}

The \( \BQQ \)-Gorenstein morphisms and the naively \( \BQQ \)-Gorenstein morphisms
are characterized as follows.

\begin{lem}\label{lem:LocDefQGor}
Let \( Y \) and \( T \) be locally Noetherian schemes and \( f \colon Y \to T \)
a flat morphism locally of finite type.
Let \( j \colon Y^{\circ} \injmap Y \) be the open immersion
from an open subset \( Y^{\circ} \) of
the relative Gorenstein locus \( \Gor(Y/T) \).
For a point \( y \in Y \), the fibers \( Y_{t} = f^{-1}(t) \) and
\( Y^{\circ}_{t} = Y^{\circ} \cap Y_{t} \) over \( t = f(y) \), and for
a positive integer \( r \),
let us consider the following conditions\emph{:}
\begin{enumerate}
    \renewcommand{\theenumi}{\roman{enumi}}
    \renewcommand{\labelenumi}{(\theenumi)}
\item \label{lem:LocDefQGor:cond1}
The fiber \( Y_{t} \) satisfies \( \bfS_{2} \) at \( y \)
and \( \Codim_{y}(Y_{t} \setminus Y^{\circ}, Y_{t}) \geq 2 \).

\item \label{lem:LocDefQGor:cond2}
The direct image sheaf \( j_{*}(\omega_{Y^{\circ}/T}^{\otimes r}) \) is invertible at \( y \).

\item \label{lem:LocDefQGor:cond12}
The fiber \( Y_{t} \) is \( \BQQ \)-Gorenstein at \( y \), and
\( r \) is divisible by the Gorenstein index of \( Y_{t} \) at \( y \).

\item \label{lem:LocDefQGor:cond3}
For any \( 0 < k \leq r \), the base change homomorphism
\[ \phi_{t}^{[k]} \colon j_{*}(\omega_{Y^{\circ}/T}^{\otimes k}) \otimes_{\SO_{Y}} \SO_{Y_{t}}
\to \omega_{Y_{t}/\Bbbk(t)}^{[k]} = j_{*}(\omega^{\otimes k}_{Y^{\circ}_{t}/\Bbbk(t)})\]
induced from the base change isomorphism
\( \omega_{Y^{\circ}/T} \otimes \SO_{Y_{t}} \isom \omega_{Y^{\circ}_{t}/\Bbbk(t)} \)
\emph{(}cf.\ Proposition~\emph{\ref{prop:BCGor})} is surjective at \( y \).

\item \label{lem:LocDefQGor:cond5}
There is an open neighborhood \( U \) of \( y \) such that
\( f|_{U} \colon U \to T \) is a naively \( \BQQ \)-Gorenstein morphism
and \( \omega_{U/T}^{[r]} \) is invertible.

\item \label{lem:LocDefQGor:cond6}
There is an open neighborhood \( U \) of \( y \)
such that \( f|_{U} \colon U \to T \) is a \( \BQQ \)-Gorenstein morphism and
\( \omega_{U/T}^{[r]} \) is invertible.

\end{enumerate}
Then, one has the following equivalences and implication
on these conditions\emph{:}
\begin{itemize}
\item  \eqref{lem:LocDefQGor:cond1} \( + \) \eqref{lem:LocDefQGor:cond2}
\( \Leftrightarrow \) \eqref{lem:LocDefQGor:cond5}\emph{;}

\item  \eqref{lem:LocDefQGor:cond1} \( + \) \eqref{lem:LocDefQGor:cond2}
\( \Rightarrow \) \eqref{lem:LocDefQGor:cond12}\emph{;}

\item \eqref{lem:LocDefQGor:cond12} \( + \) \eqref{lem:LocDefQGor:cond3}
\( \Leftrightarrow \) \eqref{lem:LocDefQGor:cond6}.
\end{itemize}
\end{lem}

\begin{proof}
First, we shall prove: \eqref{lem:LocDefQGor:cond1} \( + \) \eqref{lem:LocDefQGor:cond2}
\( \Rightarrow \) \eqref{lem:LocDefQGor:cond12}.
We set \( \SM_{r} := j_{*}(\omega_{Y^{\circ}/T}^{\otimes r}) \).
Then, \( \SM_{r} \otimes_{\SO_{Y}} \SO_{Y_{t}}\) is invertible
at \( y \) by \eqref{lem:LocDefQGor:cond2}, and
\[ \SM_{r} \otimes_{\SO_{Y}} \SO_{Y_{t}} \to
j_{*}\left( (\SM_{r} \otimes_{\SO_{Y}} \SO_{Y_{t}})|_{Y^{\circ}} \right)
\isom j_{*}(\omega_{Y_{t}^{\circ}/\Bbbk(t)}^{\otimes r})
\isom \omega_{Y_{t}/\Bbbk(t)}^{[r]}\]
is an isomorphism at \( y \) by \eqref{lem:LocDefQGor:cond1}.
In particular, \( \omega_{Y_{t}/\Bbbk(t)}^{[r]} \) is invertible at \( y \).
Thus, \eqref{lem:LocDefQGor:cond12} holds
(cf.\ Definitions~\ref{dfn:QGorSch}\eqref{dfn:QGorSch:2} and \ref{dfn:QGorSch:GorIndex}).

Second, we shall prove \eqref{lem:LocDefQGor:cond5} \( \Rightarrow \)
\eqref{lem:LocDefQGor:cond1} \( + \) \eqref{lem:LocDefQGor:cond2} and
\eqref{lem:LocDefQGor:cond6} \( \Rightarrow \)
\eqref{lem:LocDefQGor:cond12} \( + \) \eqref{lem:LocDefQGor:cond3}.
We may assume that \( f \) is naively \( \BQQ \)-Gorenstein.
Since every fiber \( Y_{t} \) is a \( \BQQ \)-Gorenstein scheme,
we have \eqref{lem:LocDefQGor:cond1} (cf.\ Definition~\ref{dfn:QGorSch}).
Moreover,
\[ \omega_{Y/T}^{[k]} \isom j_{*}(\omega_{Y^{\circ}/T}^{\otimes k})\]
for any \( k \in \BZZ \)
by Proposition~\ref{prop:BCGor}.
Hence, \eqref{lem:LocDefQGor:cond2} is also satisfied, since \( \omega_{Y/T}^{[r]}  \)
is invertible for an integer \( r > 0 \).
If \( f \) is \( \BQQ \)-Gorenstein, then
\( \omega_{Y/T}^{[k]} \) is flat over \( T \) and
\( \omega_{Y/T}^{[k]} \otimes_{\SO_{Y}} \SO_{Y_{t}} \)
satisfies the \( \bfS_{2} \)-condition for any \( t \in T \)
(cf.\ Definition~\ref{dfn:QGorMor}\eqref{dfn:QGorMor:QGor}); thus,
\( \phi^{[k]}_{t} \) is an isomorphism for any \( t \in T \) and \( k \in \BZZ \), and
in particular, \eqref{lem:LocDefQGor:cond12} and \eqref{lem:LocDefQGor:cond3} are satisfied.

Finally, we shall prove
\eqref{lem:LocDefQGor:cond1} \( + \) \eqref{lem:LocDefQGor:cond2}
\( \Rightarrow \) \eqref{lem:LocDefQGor:cond5} and
\eqref{lem:LocDefQGor:cond12} \( + \) \eqref{lem:LocDefQGor:cond3}
\( \Rightarrow \) \eqref{lem:LocDefQGor:cond6}.
Assume that \eqref{lem:LocDefQGor:cond1} holds.
By Lemma~\ref{lem:S2Codim2}, there is an open neighborhood \( U \) of \( y \) such that
\( f|_{U} \colon U \to T \) is an \( \bfS_{2} \)-morphism having pure relative dimension
and \( \Codim(U_{t'} \setminus Y^{\circ}, U_{t'}) \geq 2\)
for any \( t' \in f(U) \), where \( U_{t'} = U \cap Y_{t'} \); Thus,
\[ \omega_{U/T}^{[k]} \isom j_{*}(\omega_{U \cap Y^{\circ}/T}^{\otimes k})\]
for any \( k \in \BZZ \) by
Lemma~\ref{lem:US2add}\eqref{lem:US2add:2}.
Therefore, if \eqref{lem:LocDefQGor:cond2} also holds, then
\( \omega_{U'/T}^{[r]} \) is invertible for an open neighborhood \( U' \) of \( y \) in \( U \),
and \( f|_{U'} \colon U' \to T \) is a naively \( \BQQ \)-Gorenstein morphism.
This proves \eqref{lem:LocDefQGor:cond1} \( + \) \eqref{lem:LocDefQGor:cond2}
\( \Rightarrow \) \eqref{lem:LocDefQGor:cond5}.
Next, assume that \eqref{lem:LocDefQGor:cond12} and \eqref{lem:LocDefQGor:cond3} hold.
Note that \eqref{lem:LocDefQGor:cond12} implies \eqref{lem:LocDefQGor:cond1}.
Thus, we have the same open neighborhood \( U \) of \( y \) as above.
For any integer \( 0 < k \leq r\), there is an open neighborhood \( U'_{k} \)
of \( y \) in \( U' \) above such that
\( \omega_{U'_{k}/T}^{[k]} \) satisfies relative \( \bfS_{2} \) over \( T \),
by \eqref{lem:LocDefQGor:cond3} and by Proposition~\ref{prop:BCGor}.
In particular, \( \omega_{Y/T}^{[r]} \) is invertible at \( y \)
by Fact~\ref{fact:elem-flat}\eqref{fact:elem-flat:2}.
In fact, it is flat over \( T \) at \( y \) and its restriction to the fiber \( Y_{t} \)
is invertible at \( y \).
Then, \( \omega_{Y/T}^{[r]} \) is invertible on an open neighborhood \( U''_{r} \) of \( y \)
in \( U'_{r} \).
We set \( U'' \) to be the intersection of \( U'_{k} \) for all \( 0 < k < r \) and \( U''_{r} \).
Then, \( \omega_{U''/T}^{[l]} \) satisfies relative \( \bfS_{2} \) over \( T \)
for any \( l \in \BZZ \), since
\[ \omega_{U''/T}^{[l]} \isom (\omega_{U''/T}^{[r]})^{\otimes m} \otimes \omega_{U''/T}^{[k]} \]
for integers \( m \) and \( k \) such that \( l = mr + k \) and \( 0 \leq k < r \).
This means that \( f|_{U''} \colon U'' \to T \) is a \( \BQQ \)-Gorenstein morphism,
and it proves \eqref{lem:LocDefQGor:cond12} \( + \) \eqref{lem:LocDefQGor:cond3}
\( \Rightarrow \) \eqref{lem:LocDefQGor:cond6}.
Thus, we are done.
\end{proof}

\begin{remn}
For \( f \colon Y \to T \) and
\( j \colon Y^{\circ} \injmap Y \) in
Lemma~\ref{lem:LocDefQGor}, we have:
\begin{enumerate}
\item \label{cor:lem:LocDefQGor:1}
The set of points \( y \in Y \) satisfying
the condition \eqref{lem:LocDefQGor:cond1}
of Lemma~\ref{lem:LocDefQGor} is open.

\item  \label{cor:lem:LocDefQGor:2}
If every fiber of \( f \) satisfies \( \bfS_{2} \) and is Gorenstein in codimension one,
then \( \SO_{Y} \isom j_{*}\SO_{Y^{\circ}} \) and \( \Codim(Y \setminus Y^{\circ}, Y) \geq 2 \).
Here, if \( Y \) is connected in addition, then \( f \) has pure relative dimension.

\item \label{cor:lem:LocDefQGor:3}
The set of points \( y \in Y\) at which  \( f \) is naively \( \BQQ \)-Gorenstein is open.

\item \label{cor:lem:LocDefQGor:4}
The set of points \( y \in Y\) at which  \( f \) is \( \BQQ \)-Gorenstein is open.
\end{enumerate}
In fact, the property \eqref{cor:lem:LocDefQGor:1} is mentioned
in the proof of Lemma~\ref{lem:LocDefQGor}, and
the property \eqref{cor:lem:LocDefQGor:2} is derived
from Lemmas~\ref{lem:US2add}\eqref{lem:US2add:2}, \ref{lem:quasi-flat}, and \ref{lem:S2Codim2}.
The properties \eqref{cor:lem:LocDefQGor:3} and \eqref{cor:lem:LocDefQGor:4} are deduced
from Definition~\ref{dfn:QGorMor}.
\end{remn}

An \( \bfS_{2} \)-morphism of locally Noetherian schemes is not necessarily 
naively \( \BQQ \)-Gorenstein even if every fiber is \( \BQQ \)-Gorenstein. 
The following example is well known.

\begin{exam}\label{exam:P114}
Let \( S = \BPP(\SO \oplus \SO(4)) \) be the Hirzebruch surface
of degree \( 4 \) over an algebraically closed field \( \Bbbk \). 
By contracting of the unique \( (-4) \)-curve \( \Gamma \), we have a birational morphism \( S \to X \)
to the weighted projective plane  \( X = \BPP(1, 1, 4) \). 
Note that \( X \) is \( \BQQ \)-Gorenstein and its Gorenstein index is two. 
Let \( \eta \) be the extension class in \( \Ext^{1}_{\BPP^{1}}(\SO(4), \SO_{\BPP^{1}}) \) of an exact sequence 
\begin{equation}\label{eq:1|exam:P114}
0 \to \SO_{\BPP^{1}} \to \SO(2) \oplus \SO(2) \to \SO(4) \to 0
\end{equation}
on \( \BPP^{1} \). We set \( T = \BAA^{1} = \Spec \Bbbk[\ttt]\) and \( P = \BPP^{1} \times_{\Bbbk} T \), 
and let \( p \colon P \to \BPP^{1} \) and \( q \colon P \to T \) be projections. 
Let us consider the element \( \eta_{P} \) of 
\( \Ext^{1}_{P}(p^{*}\SO(4), \SO_{P}) \) corresponding to \( \eta \otimes \ttt \) 
by the isomorphism 
\[ \Ext^{1}_{P}(p^{*}\SO(4), \SO_{P}) \isom 
\Ext^{1}_{\BPP^{1}}(\SO(4), \SO_{\BPP^{1}}) \otimes_{\Bbbk} \OH^{0}(T, \SO_{T}), \]
and let 
\begin{equation}\label{eq:2|exam:P114}
0 \to \SO_{P} \to \SE \to p^{*}\SO(4) \to 0 
\end{equation}
be an exact sequence on \( P \) whose extension class is \( \eta_{P} \). 
Let \( \pi \colon V \to P \) be the \( \BPP^{1} \)-bundle associated with \( \SE \), and 
let \( f \colon Y \to T\) be the \( T \)-scheme 
defined as \( \mathit{Proj}_{T} \, q_{*}( \mathit{Sym}(\SE)) \) for 
the symmetric \( \SO_{P} \)-algebra \( \mathit{Sym}(\SE) \). 
Then, \( Y \) is a normal projective variety and 
there is a birational morphism \( \mu \colon V \to Y \) over \( T \) such that 
the induced morphism \( \mu_{t} \colon V_{t} \to Y_{t} \) of fibers over \( t \in T \) is described as follows:  
\begin{itemize}
\item  \( \mu_{0}  \) is isomorphic to the contraction morphism \( S \to X \) of \( \Gamma \), and 

\item  \( \mu_{t} \) is isomorphic to the identity morphism 
of \( \BPP^{1}_{\Bbbk(t)} \times_{\Bbbk(t)} \BPP^{1}_{\Bbbk(t)} \). 
\end{itemize}
In fact, \( \SL \isom \mu^{*}\SH \) for the tautological invertible sheaf \( \SL \) on \( V \) 
associated with \( \SE \) and for the \( f \)-ample tautological invertible sheaf \( \SH \) 
on \( Y \) associated with 
the graded algebra \( q_{*}(\mathit{Sym}\SE) \). 
In particular, \( f \) is a flat projective morphism whose fibers are all \( \BQQ \)-Gorenstein. 
However, \( f \) is not naively \( \BQQ \)-Gorenstein. 
For, if the canonical divisor \( K_{Y} \) is \( \BQQ \)-Cartier, 
then \( K_{Y_{t}}^{2} \) is constant for \( t \in T \), but we have \( K_{Y_{0}}^{2} = 9 \) 
and \( K_{Y_{t}}^{2} = 8 \) for \( t \ne 0 \). 
\end{exam}

\begin{lem}\label{lem:P114}
In the situation of Example~\emph{\ref{exam:P114}}, 
let \( \Spec A \subset T = \BAA^{1} \) be the closed immersion defined 
by \( \Bbbk[\ttt] \to A = \Bbbk[\ttt]/(\ttt^{2}) \). 
Then, for the base change \( f_{A} \colon Y_{A} \to \Spec A \) of \( f \colon Y \to T \) 
by the closed immersion, 
the reflexive sheaf \( \omega^{[2]}_{Y_{A}/A} \) on \( Y_{A} \) is not invertible. 
Moreover, \( \omega^{[2]}_{Y_{A}/A} \) does not satisfy relative \( \bfS_{2} \) over \( \Spec A \), 
and \( Y_{A} \to \Spec A \) is not a \( \BQQ \)-Gorenstein morphism. 
\end{lem}

\begin{proof}
The last assertion follows from the previous one by Fact~\ref{fact:elem-flat}\eqref{fact:elem-flat:2} 
and Proposition~\ref{prop:BCGor}, since \( \omega^{[2]}_{X/\Bbbk} \) is invertible. 
Let \( g_{A} \colon V_{A} \to \Spec A \) be the base change of \( g := q \circ \pi \colon V \to T \), 
and let \( \mu_{A} \colon V_{A} \to Y_{A} \) and \( \pi_{A} \colon V_{A} \to 
P_{A} := P \times_{\Spec T} \Spec A \isom \BPP^{1}_{A} \) 
be the induced morphisms from the morphisms \( \mu \) and \( \pi \) over \( T \), respectively.  
Note that the further base change by \( \Spec \Bbbk \to \Spec A \) produces the contraction morphism 
\( S \to X \) and the ruling \( S \to \BPP^{1} \) 
from \( \mu_{A} \) and \( \pi_{A} \), respectively. 
Assume that \( \omega^{[2]}_{Y_{A}/A} \) is invertible. 
Then, 
\[ \SM := \omega^{\otimes 2}_{V_{A}/A} \otimes \mu_{A}^{*}(\omega^{[2]}_{Y_{A}/A})^{-1} \]
is invertible on \( V_{A} \). We have canonical homomorphisms  
\[ (\mu_{A})_{*}\omega^{\otimes 2}_{V_{A}/A} \to \omega^{[2]}_{Y_{A}/A} 
\quad \text{and} \quad 
\mu_{A}^{*}\left((\mu_{A})_{*}\omega^{\otimes 2}_{V_{A}/A}\right) \to \omega^{\otimes 2}_{V_{A}/A}. \]
The left one is obtained by taking double-dual. The right one is surjective, since 
\[ \omega_{V_{A}/A} \isom \omega_{V/T} \otimes_{\SO_{V}} \SO_{V_{A}} 
\isom \left(\pi^{*}p^{*}\SO_{\BPP^{1}}(2) \otimes_{\SO_{V}} \SL^{\otimes -2}\right) 
\otimes_{\SO_{V}} \SO_{V_{A}} \]
and since \( \SL \isom \mu^{*}\SH \). Therefore, there is a homomorphism \( \SM \to \SO_{V_{A}} \) 
which is an isomorphism outside \( \Gamma \). Since \( V_{A} \times_{\Spec A} \Spec \Bbbk \isom S \), 
we have \( \depth_{\Gamma} \SM \geq 1\) by Lemma~\ref{lem:relSkCodimDepth}\eqref{lem:relSkCodimDepth:3}, 
and it implies that \( \SM \to \SO_{V_{A}}  \) is injective and 
\( \SM \otimes_{\SO_{V_{A}}} \SO_{S} \to \SO_{S} \) is also injective. 
Therefore, the closed subscheme \( D \) of \( V_{A} \) 
defined by the ideal sheaf \( \SM \) 
is an effective Cartier divisor, and it is flat over \( \Spec A \) by the local criterion of flatness 
(cf.\ Proposition~\ref{prop:LCflat}\eqref{prop:LCflat:2}).  
Moreover, \( D \times_{\Spec A} \Spec \Bbbk \isom \Gamma \) 
by the isomorphism
\[ \omega^{\otimes 2}_{S/\Bbbk} \otimes_{\SO_{S}} \SO_{S}(\Gamma) \isom \mu_{0}^{*}(\omega^{[2]}_{X/\Bbbk}). \]
Hence, the composite \( D \subset V_{A} \to P_{A} \) is a finite morphism, and the corresponding 
ring homomorphism 
\( \SO_{P_{A}} \to \pi_{A*}\SO_{D} \) is an isomorphism, 
since its base change by \( A \to \Bbbk \) is isomorphic to 
the isomorphism \( \SO_{\BPP^{1}} \to \SO_{\Gamma} \) 
and since \( \pi_{A*}\SO_{D} \) is flat over \( A \). 
Therefore, \( D \) is a section of \( \pi_{A} \colon V_{A} \to P_{A} \). Then, the pullback of 
the exact sequence \eqref{eq:2|exam:P114} to \( P_{A} \) is split by 
the surjection
\[ \SE \otimes_{\SO_{P}} \SO_{P_{A}} 
\isom \pi_{A*}(\SL \otimes_{\SO_{V}} \SO_{V_{A}}) 
\to \pi_{A*}(\SL \otimes_{\SO_{V}} \SO_{D}) \isom \pi_{A*}\SO_{D} \isom \SO_{P_{A}}. \]
This means that the morphism \( \Spec A \to T \) factors through \( \Spec \Bbbk \subset T \); 
This is a contradiction. Therefore, \( \omega^{[2]}_{Y_{A}/A} \) is not invertible. 
\end{proof}

\begin{rem}\label{rem:KollarCondition}
Let \( f \colon Y \to T \) be an \( \bfS_{2} \)-morphism of locally Noetherian schemes whose
fibers are all Gorenstein in codimension one.
The \emph{Koll\'ar condition for} \( f \) \emph{along a fiber} \( Y_{t} = f^{-1}(t) \)
is a condition that the base change homomorphism
\[ \phi^{[m]}_{t} \colon \omega_{Y/T}^{[m]} \otimes_{\SO_{Y}} \SO_{Y_{t}}
\to \omega_{Y_{t}/\Bbbk(t)}^{[m]} \]
is an isomorphism for any \( m \in \BZZ \). 
By Proposition~\ref{prop:BCGor}, we see that the Koll\'ar condition is equivalent to 
that \( \omega^{[m]}_{Y/T} \) satisfies relative \( \bfS_{2} \) over \( T \) 
along \( Y_{t} \) for any \( m \in \BZZ \). 
Therefore, when \( Y_{t} \) is \( \BQQ \)-Gorenstein, 
the Koll\'ar condition for \( f \) is satisfied along \( Y_{t} \) if and only if
\( f \) is \( \BQQ \)-Gorenstein along \( Y_{t} \).  
The Koll\'ar condition has been considered for deformations of \( \BQQ \)-Gorenstein algebraic
varieties of characteristic zero in \cite[2.1.2]{KolProj},
\cite[\S 2, Property $\mathbf{K}$]{HaKo}, etc.
\end{rem}

\begin{fact}\label{fact:Pa2}
Some naively \( \BQQ \)-Gorenstein morphisms are
not \( \BQQ \)-Gorenstein.
Koll\'ar gives an example of a naively \( \BQQ \)-Gorenstein morphism which is not \( \BQQ \)-Gorenstein
in the positive characteristic case 
(cf.\ \cite[14.7]{HacKov}, \cite[Exam.\ 7.6]{Kovacs}). See Remark~\ref{rem:naiveQGorCharP} below 
for a detail. 
Patakfalvi has constructed an example of characteristic zero in
\cite[Th.~1.2]{Pa} using some example of projective cones (cf.\ \cite[Prop.~5.4]{Pa}):
This is a projective flat morphism
\( \SH \to B \) of normal algebraic varieties over a field \( \Bbbk \)
of characteristic zero such that
\begin{itemize}
\item  \( B \) is an open subset of \( \BPP^{1}_{\Bbbk} \),

\item  a closed fiber \( \SH_{0} \) has a unique singular point, but other fibers
are all smooth of dimension \( \geq 3 \),

\item  \( \omega_{\SH/B}^{[r]} \) is invertible for some \( r > 0 \), but
\[\omega_{\SH/B} \otimes_{\SO_{\SH}} \SO_{\SH_{0}} \not\isom \omega_{\SH_{0}/\Bbbk}.\]
\end{itemize}
Recently, Altmann and Koll\'ar \cite{AltmannKollar} 
have constructed several examples of natively \( \BQQ \)-Gorenstein morphisms 
which are not \( \BQQ \)-Gorenstein as infinitesimal 
deformations of two-dimensional cyclic quotient singularities. 
\end{fact}

We can construct another example by the following lemma, which is
inspired by Patakfalvi's work \cite{Pa}.

\begin{lem}\label{lem:Lee}
Let \( S \) be a non-singular projective variety of dimension \( \geq 2 \) over
an algebraically closed field \( \Bbbk \) of characteristic zero, and let
\( \SL \) be an invertible \( \SO_{S} \)-module of order \( l > 1 \), 
i.e., \( l\) is the smallest positive integer
such that \( \SL^{\otimes l}\isom \SO_{S}\).
Assume that \( \OH^{1}(S, \SO_{S}) = 0 \), \( \OH^{1}(S, \SL) \ne 0 \), and that \( K_{S} \) is ample.
For an integer \( r \geq 2 \), we set
\[ \SA := \SO_{S}(rK_{S}) \otimes \SL^{-1} = \omega_{S/\Bbbk}^{\otimes r} \otimes \SL^{-1}, \]
and let \( X \) be the affine cone \( \Cone(S, \SA) \) with a vertex \( P \).
Then,
\begin{enumerate}
\item  \label{lem:Lee:1}
\( X \) is a normal \( \BQQ \)-Gorenstein variety with one isolated singularity \( P \)
of Gorenstein index \( lr \),
\end{enumerate}
and the following hold for any non-constant function \( f \colon X \to T := \BAA^{1}_{\Bbbk}\)\emph{:}
\begin{enumerate}
\addtocounter{enumi}{1}
\item \label{lem:Lee:2}
\( f \) is a naively \( \BQQ \)-Gorenstein morphism along
the fiber \( F = f^{-1}(f(P))\)\emph{;}
\item \label{lem:Lee:3}
\( \omega^{[r]}_{X/T} \isom \omega_{X/\Bbbk}^{[r]}\) 
does not satisfy relative \( \bfS_{2} \) over \( T \) at \( P \).
In particular, \( f \) is not \( \BQQ \)-Gorenstein at \( P \).
\end{enumerate}
\end{lem}

\begin{proof}
\eqref{lem:Lee:1}:
The affine cone
\( X \) is \( \BQQ \)-Gorenstein by Corollary~\ref{cor:coneGor}\eqref{cor:coneGor:QGor}.
Here, \( X \setminus P \) is a non-singular variety by Lemma~\ref{lem:cone1}\eqref{lem:cone1:4}.
Therefore, \( X \) is a normal variety.
We have \( \omega_{X/\Bbbk}^{[lr]} \isom \SO_{X} \) by Proposition~\ref{prop:coneGor}\eqref{prop:coneGor:3}.
If \( \omega_{X/\Bbbk}^{[m]} \) is invertible for some \( m > 0 \), then
\( \omega_{S/\Bbbk}^{\otimes m} \isom \SA^{\otimes l'} \) for some integer \( l' \)
by Proposition~\ref{prop:coneGor}\eqref{prop:coneGor:3},
but it implies that \( m = l'r \), and \( \SL^{\otimes l'} \isom \SO_{S} \).
Hence, the Gorenstein index of \( X \) is \( lr \).

\eqref{lem:Lee:2}:
For any \( i > 0 \) and \( m > 0 \), we have
\[ \OH^{i}(S, \SA^{\otimes m}) = \OH^{i}(S, \omega_{S/\Bbbk} \otimes_{\SO_{S}} \SA^{\otimes m}) = 0 \]
by the Kodaira vanishing theorem, since
\[ \SA^{\otimes m} \otimes_{\SO_{S}} \omega_{S/\Bbbk}^{-1} \isom \SA^{\otimes m -1} \otimes_{\SO_{S}}
\omega_{S/\Bbbk}^{\otimes r-1} \otimes \SL^{-1} \]
is ample. Then, we can apply Corollary~\ref{cor:coneGor2}\eqref{cor:coneGor2:0}. As a consequence,
\( X \) satisfies \( \bfS_{3} \), since \( \OH^{1}(S, \SO_{S}) = 0 \).
Now, \( f  \) is a flat morphism, since \( X \) is irreducible and dominates \( T \).
Hence, \( F \) satisfies \( \bfS_{2} \) by the equality
\[ \depth \SO_{F, x} = \depth \SO_{X, x} - \depth \SO_{T, f(x)} = \depth \SO_{X, x} - 1 \]
for any closed point \( x \in F \) (cf.\ \eqref{eq:fact:elem-flat:2} in Fact~\ref{fact:elem-flat}).
Thus, \( f \) is an \( \bfS_{2} \)-morphism along \( F \), and
\( f \) is a naively \( \BQQ \)-Gorenstein morphism along \( F \)
(cf.\ Definition~\ref{dfn:QGorMor}\eqref{dfn:QGorMor:naive}), since
\( \omega_{X/T}^{[lr]} \isom \omega_{X/\Bbbk}^{[lr]} \) is invertible by \eqref{lem:Lee:1}.

\eqref{lem:Lee:3}:
By assumption, we have
\[ \OH^{1}(S, \omega_{S/\Bbbk}^{[r]} \otimes \SA^{-1}) \isom \OH^{1}(S, \SL) \ne 0. \]
Then, \( \depth (\omega^{[r]}_{X/\Bbbk})_{P} = 2 \)
by Proposition~\ref{prop:coneGor}\eqref{prop:coneGor:2}.
Since \( \omega^{[r]}_{X/\Bbbk} \) is flat over \( T \), we have
\[\depth (\omega^{[r]}_{X/\Bbbk} \otimes_{\SO_{X}} \SO_{F})_{P} =
\depth (\omega_{X/\Bbbk}^{[r]})_{P} - \depth \SO_{T, f(P)} = 1\]
by \eqref{eq:fact:elem-flat:2} in Fact~\ref{fact:elem-flat}.
This implies that \( \omega^{[r]}_{X/\Bbbk} \isom \omega^{[r]}_{X/T} \) does not satisfy relative
\( \bfS_{2} \) over \( T \) at \( P \). Therefore, \( f \) is not \( \BQQ \)-Gorenstein at \( P \)
(cf.\ Definition~\ref{dfn:QGorMor}\eqref{dfn:QGorMor:QGor}).
\end{proof}

We have the following example of
non-singular projective varieties \( S \) with invertible \( \SO_{S} \)-module \( \SL \) of order \( l = 2 \)
in Lemma~\ref{lem:Lee}:

\begin{exam}\label{exam:KummerType}
Let \( V \) be an abelian variety of dimension \( d \geq 3 \)
and let \( \iota \colon V \to V\) be the involution defined by \( \iota(v) = -v \) with respect to
the group structure on \( V \).
Let \( W \) be the quotient variety \( V/\langle \iota \rangle \).
Then, \( W \) is a normal projective variety with only isolated singular points, and
\begin{equation}\label{eq:0|exam:KummerType}
\OH^{1}(W, \SO_{W}) = 0,
\end{equation}
since it is isomorphic to
the invariant part of \( \OH^{1}(V, \SO_{V}) \) by the induced action
of \( \iota \), which is just the multiplication map by \( -1 \).
The quotient morphism \( \pi \colon V \to W \)
is a double-cover \'etale outside the singular locus of \( W \), and we have isomorphisms
\( \pi_{*}\SO_{V} \isom \SO_{W} \oplus \omega_{W/\Bbbk} \) and
\( \omega^{[2]}_{W/\Bbbk} \isom \SO_{W}\).
In particular,
\begin{equation}\label{eq:1|exam:KummerType}
\OH^{1}(W, \omega_{W/\Bbbk}) \isom \OH^{1}(V, \SO_{V}) \isom \Bbbk^{\oplus d}
\end{equation}
by \eqref{eq:0|exam:KummerType}.
We can take a smooth ample divisor \( S \) on \( W \) away from the singular locus of \( W \).
Then, \( \dim S = d - 1 \geq 2 \).
By the Kodaira vanishing theorem applied to the ample divisor \( \pi^{*}S \) on \( V \),
we have \( \OH^{i}(V, \pi^{*}\SO_{W}(-S)) = 0  \) for any  \( 0 < i < d = \dim W\). Hence,
\begin{equation}\label{eq:2|exam:KummerType}
\OH^{i}(W, \SO_{W}(-S)) = \OH^{i}(W, \omega_{W/\Bbbk} \otimes_{\SO_{W}} \SO_{W}(-S)) = 0
\end{equation}
for \( i = 1 \) and \( 2 \).
The canonical divisor \( K_{S} \) is ample by
\[ \omega^{\otimes 2}_{S/\Bbbk} \isom 
(\omega^{[2]}_{W/\Bbbk} \otimes_{\SO_{W}} \SO_{W}(2S)) \otimes_{\SO_{W}} \SO_{S}
\isom \SO_{S}(2S).\]
We define \( \SL := \omega_{W/\Bbbk} \otimes_{\SO_{W}} \SO_{S} \).
This is invertible and \( \SL^{\otimes 2} \isom \SO_{S} \).
We have
\[ \OH^{1}(S, \SO_{S}) = 0 \quad \text{and} \quad \OH^{1}(S, \SL) \isom \Bbbk^{\oplus d} \]
by applying \eqref{eq:0|exam:KummerType}, \eqref{eq:1|exam:KummerType},
and \eqref{eq:2|exam:KummerType} to the cohomology long exact sequences
derived from two short exact sequences:
\begin{gather*}
0 \to \SO_{W}(-S) \to \SO_{W} \to \SO_{S} \to 0,  \\
0 \to \omega_{W/\Bbbk} \otimes_{\SO_{W}} \SO_{W}(-S) \to \omega_{W/\Bbbk} \to \SL \to 0.
\end{gather*}
The order of \( \SL \) equals two by \( \OH^{1}(S, \SL) \not\isom \OH^{1}(S, \SO_{S}) \).
Therefore, \( S \) and \( \SL \) satisfy
the conditions of Lemma~\ref{lem:Lee}.
\end{exam}

\begin{dfn}[relative Gorenstein index]
For a naively \( \BQQ \)-Gorenstein morphism \( f \colon Y \to T \)
and for a point \( y \in Y \), the
\emph{relative Gorenstein index}
of \( f \) \emph{at} \( y \)
is the smallest positive integer \( r \) such that \( \omega_{Y/T}^{[r]} \)
is invertible at \( y \).
The least common multiple of relative Gorenstein indices
at all the points is called
the \emph{relative Gorenstein index} of \( f \), which might be \( +\infty \).
\end{dfn}

\begin{prop}\label{prop:GorIndexQGorMor}
Let \( f \colon Y \to T \) be a naively \( \BQQ \)-Gorenstein morphism.
For a point \( y \in Y\), let \( m \) be the relative Gorenstein index of \( f \) at \( y \)
and let \( r \) be the Gorenstein index of \( Y_{t} = f^{-1}(t) \) at \( y \), where \( t = f(y) \).
Then, \( m = r \) in the following three cases\emph{:}
\begin{enumerate}
    \renewcommand{\theenumi}{\roman{enumi}}
    \renewcommand{\labelenumi}{(\theenumi)}
\item \label{prop:GorIndexQGorMor:cond1}
\( f \) is \( \BQQ \)-Gorenstein at \( y \)\emph{;}
\item \label{prop:GorIndexQGorMor:cond2}
\( Y_{t} \) is Gorenstein in codimension two and satisfies \( \bfS_{3} \) at \( y \)\emph{;}
\item \label{prop:GorIndexQGorMor:cond3}
\( m \) is coprime to the characteristic of \( \Bbbk(t) \).
\end{enumerate}
\end{prop}

\begin{proof}
Note that \( m \) is divisible by \( r \). In fact, the base change homomorphism
\[  \omega_{Y/T}^{[m]} \otimes_{\SO_{Y}} \SO_{Y_{t}} \to \omega_{Y_{t}/\Bbbk(t)}^{[m]}  \]
is an isomorphism at \( y \), since the left hand side is invertible at \( y \)
and since \( Y_{t} \) satisfies \( \bfS_{2} \).
We set \( \SM := \omega_{Y/T}^{[r]} \). It is enough to prove that \( \SM \) is invertible at \( y \).
Let \( Z \) be the complement of the relative Gorenstein locus \( \Gor(Y/T) \)
and let \( j \colon Y \setminus Z \injmap Y \) be the open immersion.
Note that \( \Codim(Z \cap Y_{t}, Y_{t}) \geq 2  \)
(\( \geq 3 \) in case \eqref{prop:GorIndexQGorMor:cond2}) and \( \Codim(Z, Y) \geq 2 \).
If \( f \) is \( \BQQ \)-Gorenstein,
then \( \SM \) satisfies relative \( \bfS_{2} \) over \( T \); in particular,
\[ \SM \otimes_{\SO_{Y}} \SO_{Y_{t}} \isom
j_{*}\left( (\SM \otimes_{\SO_{Y}} \SO_{Y_{t}})|_{Y_{t} \setminus Z}\right)
\isom \omega_{Y_{t}/\Bbbk(t)}^{[r]}\]
and hence, \( \SM \) is invertible at \( y \) by Fact~\ref{fact:elem-flat}\eqref{fact:elem-flat:2}.
Thus, it is enough to consider the cases \eqref{prop:GorIndexQGorMor:cond2}
and \eqref{prop:GorIndexQGorMor:cond3}.
By replacing \( Y \) with an open neighborhood of \( y \), we may assume the following:
\begin{enumerate}
\item \( \depth_{Z} \SO_{Y} \geq 2 \) (cf.\ Lemma~\ref{lem:relSkCodimDepth}\eqref{lem:relSkCodimDepth:3});

\item  \( \SM|_{Y \setminus Z} \) is invertible and \( \depth_{Z} \SM \geq 2 \)
(cf.\ Proposition~\ref{prop:BCGor});

\item  \( j_{*}(\SM \otimes_{\SO_{Y}} \SO_{Y_{t}}|_{Y_{t} \setminus Z})
\isom  \omega_{Y_{t}/\Bbbk(t)}^{[r]}\) is invertible;

\item  one of the following holds:
\begin{enumerate}
    \renewcommand{\theenumi}{\alph{enumi}}
    \renewcommand{\labelenumi}{(\theenumi)}
\item  \( \depth_{Z \cap Y_{t}} \SO_{Y_{t}} \geq 3 \);

\item  \( \SM^{[m/r]} \isom \omega_{Y/T}^{[m]}\) is invertible, where \( m/r \)
is coprime to the characteristic of \( \Bbbk(t) \).
\end{enumerate}
\end{enumerate}
Then, \( \SM \) is invertible by Theorem~\ref{thm:invExt}, and we are done.
\end{proof}

\begin{rem}\label{rem:KSh1}
A special case of Proposition~\ref{prop:GorIndexQGorMor} for naively \( \BQQ \)-Gorenstein morphisms
is stated in \cite[Lem.~3.16]{KSh},
where \( T \) is the spectrum of a complete Noetherian local \( \BCC \)-algebra
and the closed fiber \( Y_{t} \) is a normal complex algebraic surface.
However, the proof of \cite[Lem.~3.16]{KSh} has two problems.
We explain them using the notation there, where
\( (X \to S, 0 \in S) \) corresponds to \( (Y \to T, t \in T) \) in our situation,
and \( 0\) is the closed point of \( S \).
The central fiber \( X_{0} \) is only a germ of complex algebraic surface in
\cite[\S3]{KSh}, but here, for simplicity,
we consider \( X_{0} \) as a usual algebraic surface and hence consider
\( X \to S \) as a morphism of finite type.
The authors of \cite{KSh} write \( X^{0} \) for \( \Gor(X/S) \) and
write \( Y^{0} \to X^{0} \) for the cyclic \'etale cover associated with an isomorphism
\( \omega_{X/S}^{[m]} \isom \SO_{X}  \). They want to prove that \( m \) is equal to
the Gorenstein index \( r  \) of the fiber \( X_{0} \) of \( X \to S \) over \( 0 \).

The first problem is in the proof in the case where \( S = \Spec A \) is Artinian.
This is minor and is caused by omitting an explanation of the isomorphism
\( \omega_{X/S}^{[m]} \isom \SO_{X}  \).
In this situation, they assert that it is enough to prove the fiber
\( Y^{0}_{0} \) of \( Y^{0} \to S\) over \( 0 \) to be connected.
However, \( Y^{0}_{0} \) is connected even if \( r \ne m \). In fact, for isomorphisms
\( u \colon \omega_{X_{0}/\BCC}^{[r]} \isom \SO_{X_{0}} \) and
\( v \colon \omega_{X/S}^{[m]} \isom \SO_{X} \), we have an invertible element
\( \theta \) of \( \SO_{X_{0}} \) such that
\[ v|_{X_{0}} = \theta u^{\otimes m/r} \]
as an isomorphism
\( \omega_{X/S}^{[m]} \otimes_{\SO_{X}} \SO_{X_{0}} \isom \SO_{X_{0}}\).
Here, we can take \( v \) so that \( \theta \) can not have \( k \)-th root
in \( \SO_{X_{0}} \) for any integer \( k \) dividing \( r \).
Then, \( Y^{0}_{0} \) is connected for such \( v \).
Of course, this problem is resolved by replacing the isomorphism \( v \)
with \( v (\tilde{\theta})^{-1} \)
for a function \( \tilde{\theta} \in \SO_{X} \) which is a lift of \( \theta \in \SO_{X_{0}} \).

The second problem is in the reduction to the Artinian case.
They set \( A_{n} = A/\GM^{n} \), \( S_{n} = \Spec A_{n} \),
and \( X^{0}_{n} = X^{0} \times_{S} S_{n} \), for \( n \geq 1 \) and for the maximal ideal \( \GM \)
of \( A \), and they obtain an isomorphism
\[ \Phi_{n} \colon \omega_{X^{0}_{n}/S_{n}}^{\otimes r} \isom \SO_{X_{n}^{0}} \]
for any \( n  \) by applying the assertion: \( m = r \), to the Artinian case.
However, just after the isomorphism \( \Phi_{n} \), they deduce
an isomorphism \( \omega_{X^{0}/S}^{\otimes r} \isom \SO_{X^{0}} \) without mentioning any reason.
This is thought of as a lack of the proof, by Remark~\ref{rem:exam:8-3}. For,  their isomorphism above 
induces isomorphisms 
\[  \omega_{X/S}^{[r]} \otimes \SO_{X_{n}} \isom j_{*}(\omega_{X^{0}_{n}/S_{n}}^{\otimes r}) \] 
for all \( n \), while we always have an isomorphism 
\[ \omega_{X/S}^{[r]} \isom j_{*}(\omega_{X^{0}/S}^{\otimes r}),  \]
where \( j \colon X^{0} \injmap X \) denotes the open immersion. 
\end{rem}


\subsection{Virtually \texorpdfstring{$\BQQ$}{Q}-Gorenstein morphisms}
\label{subsect:virQGormor}

\begin{dfn}\label{dfn:vQGorMor}
Let \( f \colon Y \to T \) be a morphism locally of finite type between locally Noetherian schemes.
For a given point \( y \in Y \) and the image \( o = f(y) \),
the morphism \( f \) is said to be \emph{virtually \( \BQQ \)-Gorenstein at} \( y \) if
\begin{itemize}
\item  \( f \) is flat at \( y \),

\item  the fiber \( Y_{o} = f^{-1}(o)\) is \( \BQQ \)-Gorenstein at \( y \),
\end{itemize}
and if there exist an open neighborhood \( U \) of \( y \) in \( Y \) 
and a reflexive \( \SO_{U} \)-module \( \SL \)
satisfying the following conditions:
\begin{enumerate}
    \renewcommand{\theenumi}{\roman{enumi}}
    \renewcommand{\labelenumi}{(\theenumi)}
\item \label{dfn:vQGorMor:cond1}
\( \SL \otimes_{\SO_{U}} \SO_{U_{o}} \isom \omega_{U_{o}/\Bbbk(o)} \),
where \( U_{o} = U \cap Y_{o} \);

\item  \label{dfn:vQGorMor:cond2}
for any integer \( m \), the double-dual \( \SL^{[m]} \) of \( \SL^{\otimes m} \)
satisfies relative \( \bfS_{2} \) over \( T \) at \( y \).
\end{enumerate}
If \( f \) is virtually \( \BQQ \)-Gorenstein at every point of \( Y \), then
it is called a \emph{virtually} \( \BQQ \)-\emph{Gorenstein morphism}.
\end{dfn}

\begin{rem}\label{rem:dfn:vQGorMor}
If the morphism \( f \) above is virtually \( \BQQ \)-Gorenstein at \( y \), then
there exist an open neighborhood \( U \) of \( y \) in \( Y \) and
a reflexive \( \SO_{U} \)-module \( \SL \) such that
\begin{enumerate}
\item  \label{rem:dfn:vQGorMor:1} \( f|_{U} \colon U \to T \) is an \( \bfS_{2} \)-morphism 
of pure relative dimension,

\item \label{rem:dfn:vQGorMor:2} every non-empty fiber \( U_{t} = U \cap Y_{t}\)
of \( f|_{U} \) is Gorenstein in codimension one, i.e., \( \Codim(U_{t} \setminus Y^{\circ}, U_{t}) \geq 2 \)
for any \( t \in f(U) \), where \( Y^{\circ} = \Gor(Y/T) \),

\item \label{rem:dfn:vQGorMor:3} \( \SL \otimes_{\SO_{U}} \SO_{U_{o}} \isom \omega_{U_{o}/\Bbbk(o)} \),

\item \label{rem:dfn:vQGorMor:4} \( \SL|_{U \cap Y^{\circ}} \) is invertible,

\item \label{rem:dfn:vQGorMor:5}\( \SL^{[r]} \) is invertible for some integer \( r > 0 \), and

\item \label{rem:dfn:vQGorMor:6}
\( \SL^{[m]} \) satisfies relative \( \bfS_{2} \) over \( T \) for any integer \( m \).
\end{enumerate}
In fact, we have an open neighborhood \( U \) satisfying
\eqref{rem:dfn:vQGorMor:1} and \eqref{rem:dfn:vQGorMor:2} by Lemma~\ref{lem:S2Codim2}.
By shrinking \( U \) and by Fact~\ref{fact:elem-flat}\eqref{fact:elem-flat:2},
we may assume the existence of \( \SL \) satisfying \eqref{rem:dfn:vQGorMor:3},
\eqref{rem:dfn:vQGorMor:4}, and \eqref{rem:dfn:vQGorMor:5},
where \( r \) is a multiple the Gorenstein index of \( Y_{o} \) at \( y \).
Then, for any point \( t \in f(U) \), the coherent sheaf \( \SL^{[m]}_{(t)} = \SL^{[m]} \otimes \SO_{U_{t}} \)
is locally equi-dimensional by Fact~\ref{fact:S2}\eqref{fact:S2:1}, since
\( \Supp \SL^{[m]} = U \), \( \Supp \SL^{[m]}_{(t)} = U_{t}\), and since
\( U_{t} \) is catenary satisfying \( \bfS_{2} \).
Hence, the relative \( \bfS_{2} \)-locus \( \bfS_{2}(\SL^{[m]}/T) \) is an open subset of \( U \)
by Fact~\ref{fact:dfn:RelSkCMlocus}\eqref{fact:dfn:RelSkCMlocus:2},
and now, \( y \in \bfS_{2}(\SL^{[m]}/T) \) for any \( m \in \BZZ \).
We have  \( \bfS_{2}(\SL^{[m + r]}/T) = \bfS_{2}(\SL^{[m]}/T)\) for any \( m \)
by \( \SL^{[m+r]} \isom \SL^{[r]} \otimes \SL^{[m]}\), 
and hence the intersection of \( \bfS_{2}(\SL^{[m]}/T) \)
for all \( m \) is still an open neighborhood of \( y \).
Thus, we can also assume \eqref{rem:dfn:vQGorMor:6}.
As a consequence of \eqref{rem:dfn:vQGorMor:1}--\eqref{rem:dfn:vQGorMor:6}, we see that
\begin{enumerate}
\addtocounter{enumi}{6}
\item  \label{rem:dfn:vQGorMor:7} \( U_{o} = U \cap Y_{o} \) is \( \BQQ \)-Gorenstein, and

\item \label{rem:dfn:vQGorMor:8}
\( \SL^{[m]} \otimes_{\SO_{U}} \SO_{U_{o}} \isom \omega^{[m]}_{U_{o}/\Bbbk(o)} \)
for any \( m \in \BZZ \).
\end{enumerate}
In fact, \( \SL^{[m]} \otimes_{\SO_{U}} \SO_{U_{o}} \) satisfies \( \bfS_{2} \) by
\eqref{rem:dfn:vQGorMor:6} and its depth along \( U_{o} \setminus Y^{\circ} \)
is \( \geq 2 \) by \eqref{rem:dfn:vQGorMor:1} and \eqref{rem:dfn:vQGorMor:2}
(cf.\ Lemma~\ref{lem:depth+codim+Sk}\eqref{lem:depth+codim+Sk:2}); this implies \eqref{rem:dfn:vQGorMor:8}.
The condition \eqref{rem:dfn:vQGorMor:7} follows from
\eqref{rem:dfn:vQGorMor:5} and \eqref{rem:dfn:vQGorMor:8}.
\end{rem}

\begin{remn}
The set of points \( y \in T \) at which \( f \) is virtually \( \BQQ \)-Gorenstein, is not open in general.
Even if a morphism \( f \colon Y \to T\) is virtually \( \BQQ \)-Gorenstein
at any point of a fiber \( Y_{o} \),
the other fibers \( Y_{t} \) are not necessarily \( \BQQ \)-Gorenstein even if \( t \in T \)
is sufficiently close to the point \( o \). The following gives such an example.
\end{remn}

\begin{exam}\label{exam:vQGorMor}
Let \( X \) be a non-singular projective variety over an algebraically closed field \( \Bbbk \)
of characteristic zero
such that the dualizing sheaf \( \omega_{X/\Bbbk} \) is ample, \( \OH^{1}(X, \SO_{X}) \ne 0\),
and \( \OH^{1}(X, \omega_{X/\Bbbk}) = 0 \). Then, \( n := \dim X \geq 3 \).
As an example of \( X \), we can take the product \( C \times S \)
of a non-singular projective curve \( C \) of genus \( \geq 2 \) and
a non-singular projective surface \( S \) such that \( \omega_{S/\Bbbk} \) is ample and
\( \OH^{1}(S, \SO_{S}) = \OH^{2}(S, \SO_{S}) = 0 \).
Let us take a positive-dimensional nonsingular affine subvariety \( T = \Spec A \) of
the Picard scheme \( \Pic^{0}(X) \)
which contains the origin \( 0 \) of \( \Pic^{0}(X) \).
Then, there is an invertible sheaf \( \SN \) on \( X_{A} := X \times_{\Spec \Bbbk} T \)
such that
\begin{itemize}
\item \( \SN_{(t)} \) is algebraically equivalent to zero for any \( t \in T \), and
\item  \( \SN_{(t)} \isom \SO_{X_{t}} \) if and only if \( t = 0 \),
\end{itemize}
where \( X_{t} = X \times_{\Spec \Bbbk} \Spec \Bbbk(t) \)
and \( \SN_{(t)} = \SN \otimes_{\SO_{X_{A}}} \SO_{X_{t}} \) (cf.\ Notation~\ref{nota:F_(t)}).
We define a \( \BZZ_{\geq 0} \)-graded \( A \)-algebra \( R = \bigoplus\nolimits_{m \geq 0} R_{m} \)
by
\[ R_{m} := \OH^{0}(X_{A}, (p^{*}(\omega_{X/\Bbbk}) \otimes_{\SO_{X_{A}}} \SN)^{\otimes m}) \]
for the projection \( p \colon X_{A} \to X \),
and let \( f \colon Y := \Spec R \to T = \Spec A \) be the induced affine morphism.
We shall prove the following by replacing \( T \) with a suitable open neighborhood of \( 0 \):

\begin{enumerate}
\item \label{exam:vQGorMor:0} \emph{\( f \) is a flat morphism};
\item \label{exam:vQGorMor:1} \emph{for any \( t \in T \), the fiber \( Y_{t} = f^{-1}(t)\)
is isomorphic to the affine cone of the polarized scheme
\( (X_{t}, \omega_{X_{t}/\Bbbk(t)} \otimes \SN_{(t)}) \)};

\item \label{exam:vQGorMor:2}
\emph{the set of points \( t \in T \) such that \( Y_{t} \) is \( \BQQ \)-Gorenstein,
is a countable set};

\item \label{exam:vQGorMor:3}
\emph{\( f \) is virtually \( \BQQ \)-Gorenstein at any point of the fiber \( Y_{0} \).}
\end{enumerate}
For the proof, we consider a graded \( \Bbbk(t) \)-algebra
\( R^{t} = \bigoplus\nolimits_{m \geq 0}R^{t}_{m} \) defined by
\[ R^{t}_{m}
= \OH^{0}(X_{t}, (\omega_{X_{t}/\Bbbk(t)} \otimes_{\SO_{X_{t}}} \SN_{(t)})^{\otimes m}). \]
Then, \( \Spec R^{t} \) is the affine cone associated
with \( (X_{t}, \omega_{X_{t}} \otimes \SN_{(t)}) \).
On the other hand, \( Y_{t} = \Spec (R \otimes_{A} \Bbbk(t)) \), and we have a natural homomorphism
\[ \varphi^{t} \colon R \otimes_{A} \Bbbk(t) \to R^{t} \]
of graded \( \Bbbk(t) \)-algebras,
since \( (p^{*}\omega_{X/\Bbbk}) \otimes_{\SO_{X_{A}}} \SO_{X_{t}} \isom \omega_{X_{t}/\Bbbk(t)} \).
Let \( \varphi^{t}_{m} \) be the homomorphism \( R_{m} \otimes_{A} \Bbbk(t) \to R^{t}_{m} \)
of \( m \)-th graded piece of \( \varphi^{t} \).
Note that
\[ \OH^{1}(X_{t}, (\omega_{X_{t}/\Bbbk(t)} \otimes_{\SO_{X_{t}}} \SN_{(t)})^{\otimes m}) = 0\]
for any \( m \geq 2 \) by the Kodaira vanishing theorem, since \( \omega_{X/\Bbbk} \) is ample and
\( \SN_{(t)} \) is algebraically equivalent to zero.
Moreover, there is an open neighborhood \( U \) of \( 0 \) in \( T \) such that
\[ \OH^{1}(X_{t}, \omega_{X_{t}/\Bbbk(t)} \otimes_{\SO_{X_{t}}} \SN_{(t)}) = 0 \]
for any \( t \in U \) by the upper semi-continuity theorem
(cf.\ \cite[III, Th.~(7.7.5) I]{EGA}, \cite[\S 5, Cor., p.~50]{MumfordAV}),
since we have assumed that \( \OH^{1}(X, \omega_{X/\Bbbk}) = 0 \).
We may replace \( T \) with \( U \).
Then, \( \varphi^{t}_{m} \) is an isomorphism for any \( m \geq 1 \) and for any \( t \in T \)
by \cite[III, Th.~(7.7.5) II]{EGA} (cf.\ \cite[\S 5, Cor.~3, p.~53]{MumfordAV}).
Since \( \varphi^{t}_{0} \) is obviously an isomorphism,
\( \varphi^{t} \) is an isomorphism and \( Y_{t} \isom \Spec R^{t} \) for any \( t \in T \).
Moreover \( R_{m} \) is a flat \( A \)-module for any \( m \geq 0\)
by \cite[III, Cor.\ (7.5.5)]{EGA} (cf.\ \cite[III, Th.~12.11]{HartshorneGTM}),
and it implies that \( Y = \Spec R \) is flat over \( T \).
This proves \eqref{exam:vQGorMor:0} and \eqref{exam:vQGorMor:1}.

By Corollary~\ref{cor:coneGor2}\eqref{cor:coneGor2:2}, \( Y_{t} \) is \( \BQQ \)-Gorenstein
if and only if \( \SN_{(t)}^{\otimes r} \isom \SO_{X_{t}}  \) for some \( r > 0 \).
For an integer \( r > 0 \), let \( F_{r} \) be the kernel of the \( r \)-th power map
\( \Pic^{0}(X) \to \Pic^{0}(X)\) which  sends an invertible sheaf \( \SL \) to \( \SL^{\otimes r} \).
Then, \( F_{r} \) is a finite set, and \( F_{r} \cap T \) is just the set of
points \( t \in T \) such that \( \SN_{(t)}^{\otimes r} \isom \SO_{X_{t}} \).
Thus, \( Y_{t} \) is \( \BQQ \)-Gorenstein if and only if \( t \) is contained in the countable set
\( \bigcup_{r > 0} F_{r} \cap T \). This proves \eqref{exam:vQGorMor:2}.

Note that \( \omega_{Y_{0}/\Bbbk} \isom \SO_{Y_{0}} \)
by Proposition~\ref{prop:coneGor}\eqref{prop:coneGor:3}.
Hence, \( f \colon Y \to T \) is virtually \( \BQQ \)-Gorenstein
at any point of \( Y_{0} \), since
\( \SO_{Y} \) plays the role of \( \SL \) in Definition~\ref{dfn:vQGorMor}.
This proves \eqref{exam:vQGorMor:3}.
\end{exam}

\begin{lem}\label{lem:Hacking}
Let \( f \colon Y \to T \) be a flat morphism locally of finite type between locally
Noetherian schemes
and let \( o \in T \) be a point such that \( Y_{o} = f^{-1}(o) \) is \( \BQQ \)-Gorenstein. 
For a given isomorphism
\( u \colon \omega_{Y_{o}/\Bbbk(o)}^{[r]} \to \SO_{Y_{o}} \) for a positive integer \( r \),
we set
\[ \SR = \bigoplus\nolimits_{i = 0}^{r-1} \omega_{Y_{o}/\Bbbk(o)}^{[i]} \]
to be the \( \BZZ/r\BZZ \)-graded \( \SO_{Y_{o}} \)-algebra defined by the isomorphism \( u \).
Then, the following two conditions are equivalent to each other\emph{:}
\begin{enumerate}
\item \label{lem:Hacking:cond1} Locally on \( Y \),
there exists a \( \BZZ/r\BZZ \)-graded coherent \( \SO_{Y} \)-algebra \( \SR\sptilde \)
flat over \( T \) with an isomorphism
\[ \SR\sptilde \otimes_{\SO_{Y}} \SO_{Y_{o}} \isom \SR \]
as a \( \BZZ/r\BZZ \)-graded \( \SO_{Y_{o}} \)-algebra.

\item \label{lem:Hacking:cond2}
The morphism \( f \) is virtually \( \BQQ \)-Gorenstein along \( Y_{o} \).
\end{enumerate}
\end{lem}

\begin{proof}
We write \( X = Y_{o} \) and \( \Bbbk = \Bbbk(o) \) for short.
First, we shall show \eqref{lem:Hacking:cond1} \( \Rightarrow \) \eqref{lem:Hacking:cond2}.
We may assume that \( \SR\sptilde \) is defined on \( Y \). Thus,
there exist coherent \( \SO_{Y} \)-modules \( \SL_{i} \) for \( 0 \leq i \leq r-1 \) such that
\[ \SR\sptilde = \bigoplus\nolimits_{i = 0}^{r-1} \SL_{i} \]
as a \( \BZZ/r\BZZ \)-graded \( \SO_{Y} \)-algebra. Hence, \( \SL_{i} \) are all flat over \( T \),
and moreover,
\begin{itemize}
\item  \( \SL_{i} \otimes_{\SO_{Y}} \SO_{X} \isom \omega_{X/\Bbbk}^{[i]} \)
for any \( 0 \leq i \leq r-1\),

\item the multiplication map \( \SL_{1}^{\otimes i} \to \SL_{i} \) restricts to
the canonical homomorphism \( \omega_{X/\Bbbk}^{\otimes i} \to \omega_{X/\Bbbk}^{[i]} \)
for any \( 1 \leq i \leq r-1 \), and

\item the multiplication map \( \SL_{1}^{\otimes r} \to \SO_{Y} \) induces
the isomorphism \( u \colon \omega_{X/\Bbbk}^{[r]} \to \SO_{X} \).
\end{itemize}
We shall show that \( \SL_{1}^{[r]} \isom \SO_{Y} \) and \( \SL_{i} \isom \SL_{1}^{[i]} \)
for any \( 1 \leq i \leq r-1\) along \( X = Y_{o} \).
Now, \( \SL_{i} \) satisfies relative \( \bfS_{2} \) over \( T \) along \( X \)
for any \( 0 \leq i \leq r-1 \), since \( \omega_{X/\Bbbk}^{[i]} \) satisfies \( \bfS_{2} \)
(cf.\ Lemma~\ref{lem:dfn:canosheaf2}).
Thus, there is a closed subset \( Z\) of \( Y \) such that
\begin{itemize}
\item  \( \Gor(X) \subset X \setminus Z\),

\item   \( \SL_{i}|_{Y \setminus Z} \) is invertible for any \( 0 \leq i \leq r-1\)
(cf.\ Fact~\ref{fact:elem-flat}\eqref{fact:elem-flat:2}),

\item the multiplication maps \( \SL_{1}^{\otimes i} \to \SL_{i} \)
and \( \SL_{1}^{\otimes r} \to \SO_{Y} \) are isomorphisms on \( Y \setminus Z \).
\end{itemize}
By replacing \( Y \) with its open subset, we may assume that
\( \Codim(Y_{t} \cap Z, Y_{t}) \geq 2 \) for any \( t \in T \) by Lemma~\ref{lem:S2Codim2},
since \( \Codim(Y_{o} \cap Z, Y_{o}) \geq \Codim(X \setminus \Gor(X), X) \geq 2 \)
and may assume that \( \SL_{i} \) satisfies relative \( \bfS_{2} \) over \( T \) for all \( i \)
(cf.\ Fact~\ref{fact:dfn:RelSkCMlocus}\eqref{fact:dfn:RelSkCMlocus:2}).
Then, for any \( m \geq 1 \)  and  any \( 1 \leq i \leq r-1 \), we have
\[  \SL_{1}^{[m]} \isom j_{*}(\SL^{\otimes m}_{1}|_{Y \setminus Z})
\quad \text{and} \quad
\SL_{i} \isom j_{*}(\SL_{i}|_{Y \setminus Z}) \]
for the open immersion \( j \colon Y \setminus Z \injmap Y \)
by \eqref{lem:US2add:2} and \eqref{lem:US2add:3a}
of Lemma~\ref{lem:US2add}, respectively.
This argument shows that \( \SL_{i} \isom \SL_{1}^{[i]} \) and \( \SO_{Y} \isom \SL_{1}^{[r]} \)
along \( X = Y_{o} \).

As a consequence, \( \SL_{1} \) satisfies
the conditions in Definition~\ref{dfn:vQGorMor} for any point of \( Y_{o} \), and
we have proved \eqref{lem:Hacking:cond1} \( \Rightarrow \) \eqref{lem:Hacking:cond2}.

Next, we shall show: \eqref{lem:Hacking:cond2} \( \Rightarrow \) \eqref{lem:Hacking:cond1}.
We may assume the existence of a reflexive \( \SO_{Y} \)-module \( \SL \)
which satisfies the conditions of Remark~\ref{rem:dfn:vQGorMor}
for \( U = Y \) and for the fiber \( Y_{o} = X \).
By replacing \( Y \) with an open neighborhood of an arbitrary point of \( Y_{o} \),
we may assume that there is an isomorphism \( u\sptilde \colon \SL^{[r]} \to \SO_{Y} \)
which restricts to the composite of the isomorphism
\( \SL^{[r]} \otimes_{\SO_{Y}} \SO_{X} \isom \omega_{X/\Bbbk}^{[r]} \)
and the isomorphism \( u \colon \omega_{X/\Bbbk}^{[r]} \to \SO_{X}\).
Then, \( u\sptilde \) defines a \( \BZZ/r\BZZ \)-graded \( \SO_{Y} \)-algebra
\[ \SR\sptilde = \bigoplus\nolimits_{i = 0}^{r-1} \SL^{[i]}, \]
which satisfies the condition \eqref{lem:Hacking:cond1}.
Thus, we are done.
\end{proof}

\begin{rem}\label{rem:Hacking}
The \( \BQQ \)-Gorenstein deformation in the sense of Hacking
\cite[Def.\ 3.1]{Hacking} is considered as a virtually \( \BQQ \)-Gorenstein deformation
by Lemma~\ref{lem:Hacking}.
Hacking's notion is generalized to the notion of \emph{Koll\'ar family of}
\( \BQQ \)-\emph{line bundles} by Abramovich--Hassett (cf.\ \cite[Def.\ 5.2.1]{AH}).
This is related to the notion of virtually \( \BQQ \)-Gorenstein morphism as follows.
Let \( f \colon Y \to T \) be an \( \bfS_{2} \)-morphism between Noetherian schemes 
such that every fiber is connected, reduced, and \( \BQQ \)-Gorenstein.
Let \( \SL \) be a reflexive \( \SO_{Y} \)-module.
Then, \( \SL \) satisfies the conditions
\eqref{dfn:vQGorMor:cond1} and \eqref{dfn:vQGorMor:cond2} of Definition~\ref{dfn:vQGorMor}
for \( U = Y \)
and for any \( y \in Y \),
if and only if \( (Y \to T, \SL) \) is a Koll\'ar family of \( \BQQ \)-line bundles
with \( \SL \otimes \SO_{Y_{t}} \isom \omega_{Y_{t}/\Bbbk(t)} \) for all \( t \in T\).
However, in their study of Koll\'ar families \( (Y \to T, \SL) \) for \( \SL = \omega_{Y/T} \),
every fiber and every \( \omega_{Y/T}^{[m]} \) are assumed to be Cohen--Macaulay
(cf.\ \cite[Rem.~5.3.9, 5.3.10]{AH}).
\end{rem}

A \( \BQQ \)-Gorenstein morphism is always virtually \( \BQQ \)-Gorenstein.
The following theorem shows conversely that
a virtually \( \BQQ \)-Gorenstein morphism
is a \( \BQQ \)-Gorenstein morphism under some mild conditions.
In particular, we see that a virtually \( \BQQ \)-Gorenstein morphism
is \( \BQQ \)-Gorenstein if it is a Cohen--Macaulay morphism.

\begin{thm}\label{thm:wQGvsQG}
Let \( Y \) and \( T \) be locally Noetherian schemes and \( f \colon Y \to T \)
a flat morphism locally of finite type.
For a point \( t \in T \),
assume that \( f \) is virtually \( \BQQ \)-Gorenstein at any point of the fiber \( Y_{t} = f^{-1}(t)\)
and that one of the following two conditions is satisfied\emph{:}
\begin{enumerate}
    \renewcommand{\theenumi}{\alph{enumi}}
    \renewcommand{\labelenumi}{(\theenumi)}
\item \label{thm:wQGvsQG:a}
\( Y_{t} \) satisfies \( \bfS_{3} \)\emph{;}

\item \label{thm:wQGvsQG:b}
there is a positive integer \( r \) coprime to the characteristic of \( \Bbbk(t) \)
such that \( \omega_{Y/T}^{[r]} \) is invertible along \( Y_{t} \).
\end{enumerate}
Then, \( f \) is \( \BQQ \)-Gorenstein along \( Y_{t} \).
\end{thm}

\begin{proof}
Since the assertion is local, by Remark~\ref{rem:dfn:vQGorMor},
we may assume that \( f \) is an \( \bfS_{2} \)-morphism and
there is a reflexive \( \SO_{Y} \)-module
\( \SL \) satisfying the following two conditions:
\begin{enumerate}
\item  \label{thm:wQGvsQG:1}
\( \SL^{[m]} = (\SL^{\otimes m})^{\vee\vee} \) satisfies relative \( \bfS_{2} \) over \( T \)
for any integer \( m \);

\item \label{thm:wQGvsQG:2}
there is an isomorphism
\( \SL \otimes_{\SO_{Y}} \SO_{Y_{t}} \isom \omega_{Y_{t}/\Bbbk(t)} \).
\end{enumerate}
We can prove the following for \( \SM := \SHom_{\SO_{Y}}(\SL, \omega_{Y/T}) \)
applying Theorem~\ref{thm:S2S3crit}:
\begin{enumerate}
    \addtocounter{enumi}{2}
\item  \label{thm:wQGvsQG:3}
\emph{\( \SM \) is an invertible \( \SO_{Y} \)-module along \( Y_{t} \)};

\item  \label{thm:wQGvsQG:4}
\emph{\( \SL \isom \omega_{Y/T} \otimes_{\SO_{Y}} \SM^{-1} \) along \( Y_{t} \)}.
\end{enumerate}
In fact, the condition \eqref{thm:S2S3crit:cond2prime} of Theorem~\ref{thm:S2S3crit}
holds by \eqref{thm:wQGvsQG:1} and \eqref{thm:wQGvsQG:2} above, and
the condition \eqref{thm:S2S3crit:cond1} of Theorem~\ref{thm:S2S3crit}
holds for \( U = \CM(Y/T) \) (resp.\ \( U = \Gor(Y/T) \))
in case \eqref{thm:wQGvsQG:a} (resp.\ \eqref{thm:wQGvsQG:b}).
The remaining condition \eqref{thm:S2S3crit:cond3} of
Theorem~\ref{thm:S2S3crit} is checked as follows.
In case \eqref{thm:wQGvsQG:a}, the condition
\eqref{thm:S2S3crit:cond3}\eqref{thm:S2S3crit:cond3:a}
of Theorem~\ref{thm:S2S3crit} is satisfied for \( U \) above.
In case \eqref{thm:wQGvsQG:b},
\( \SL^{[r]} \) is invertible along \( Y_{t} \)
by \eqref{thm:wQGvsQG:1} and \eqref{thm:wQGvsQG:2}, since
\[ \SL^{[r]} \otimes_{\SO_{Y}} \SO_{Y_{t}} \isom \omega_{Y_{t}/\Bbbk(t)}^{[r]}
\isom \omega_{Y/T}^{[r]} \otimes_{\SO_{Y}} \SO_{Y_{t}}\]
is invertible (cf.\ Fact~\ref{fact:elem-flat}\eqref{fact:elem-flat:2}); Thus,
the condition \eqref{thm:S2S3crit:cond3}\eqref{thm:S2S3crit:cond3:b}
of Theorem~\ref{thm:S2S3crit} is satisfied in this case.
Therefore, we can apply Theorem~\ref{thm:S2S3crit} and obtain \eqref{thm:wQGvsQG:3} and
\eqref{thm:wQGvsQG:4}.

As a consequence, we have an isomorphism
\[ \omega_{Y/T}^{[m]} \isom \SL^{[m]} \otimes_{\SO_{Y}} \SM^{\otimes m} \]
for any \( m \in \BZZ \) along \( Y_{t} \).
Therefore, \( \omega_{Y/T}^{[m]} \) satisfies relative \( \bfS_{2} \) over \( T \)
along \( Y_{t} \) by \eqref{thm:wQGvsQG:1},
and hence \( f \colon Y \to T \) is \( \BQQ \)-Gorenstein along \( Y_{t} \).
\end{proof}

\begin{cor}\label{cor:thm:wQGvsQG}
Let \( Y \) and \( T \) be locally Noetherian schemes and
\( f \colon Y \to T \) a flat morphism locally of finite type.
For a point \( t \in T \), assume that the fiber \( Y_{t} = f^{-1}(t) \)
is quasi-Gorenstein. If \( \omega_{Y/T}^{[r]} \) is invertible for a positive integer \( r \)
coprime to the characteristic of \( \Bbbk(t) \),
then \( f \) is \( \BQQ \)-Gorenstein along \( Y_{t} \).
\end{cor}

\begin{proof}
The morphism \( f \) is virtually \( \BQQ \)-Gorenstein at any point of \( Y_{t} \),
since \( \SO_{Y} \) plays the role of \( \SL \) in Definition~\ref{dfn:vQGorMor}.
Thus, we are done by Theorem~\ref{thm:wQGvsQG} in the case \eqref{thm:wQGvsQG:b}.
\end{proof}


\subsection{Basic properties of \texorpdfstring{$\BQQ$}{Q}-Gorenstein  morphism}
\label{subsect:propQGormor}

We shall prove some basic properties of \( \BQQ \)-Gorenstein morphism and its variants. 
The following is a criterion for a morphism
to be naively \( \BQQ \)-Gorenstein.

\begin{lem}\label{lem:naiveQGor}
Let \( f \colon Y \to T \) be an \( \bfS_{2} \)-morphism of locally Noetherian schemes.
Assume that \( T \) is \( \BQQ \)-Gorenstein and that every fiber of \( f \)
is Gorenstein in codimension one.
Then, \( f \) is a naively \( \BQQ \)-Gorenstein morphism
if and only if \( Y \) is \( \BQQ \)-Gorenstein.
\end{lem}

\begin{proof}
Since the \( \BQQ \)-Gorenstein properties are local, we may assume
that \( T \) and \( Y \) are affine
and that \( f \) is of finite type with pure relative dimension
(cf.\ Lemma~\ref{lem:S2Codim2}).
Since the \( \BQQ \)-Gorenstein scheme
\( T \) satisfies \( \bfS_{2} \) (cf.\ Lemma~\ref{lem:QGorSch}\eqref{lem:QGorSch:2}),
we may assume the following (cf.\ Lemma~\ref{lem:QGorSch3}):
\begin{itemize}
\item  \( T \) admits an ordinary dualizing complex \( \SR^{\bullet} \)
(cf.\ Lemma~\ref{lem:ordinaryDC}) with
the dualizing sheaf \( \omega_{T} := \SH^{0}(\SR^{\bullet}) \);

\item  the double-dual \( \omega^{[m]}_{T}\) of \( \omega_{T}^{\otimes m} \)
satisfies \( \bfS_{2} \) for any integer \( m  \);

\item  \( \omega_{T}^{[r]} \) is invertible for
a positive integer \( r \).
\end{itemize}
For the Gorenstein locus \( T^{\circ} := \Gor(T) \) and the relative Gorenstein locus
\( Y^{\circ} := \Gor(Y/T) \),
we set \( U := f^{-1}(T^{\circ}) \) and \( U^{\circ} := U \cap Y^{\circ} \).
Then, \( \Codim(Y \setminus U, Y) \geq 2 \) by
\eqref{eq:fact:elem-flat:1} in Fact~\ref{fact:elem-flat}
and Property~\ref{ppty:dim-codim}\eqref{ppty:dim-codim:2},
since \( f \) is flat and \( \Codim(T \setminus T^{\circ}, T) \geq 2 \).
Hence, \( \Codim(Y \setminus U^{\circ}, Y) \geq 2 \) by
\( \Codim(Y \setminus Y^{\circ}, Y) \geq 2\),
since \( f \) is an \( \bfS_{2} \)-morphism (cf.\ Lemma~\ref{lem:quasi-flat}).
The twisted inverse image \( \SR^{\bullet}_{Y} := f^{!}(\SR^{\bullet}) \) is
a dualizing complex of \( Y \) (cf.\ Example~\ref{exam:RDConrad})
with a quasi-isomorphism
\[ \SR^{\bullet}_{Y} \isom_{\qis}
f^{!}\SO_{T} \otimes^{\bfL}_{\SO_{Y}} \bfL f^{*}(\SR^{\bullet}) \]
by \eqref{eq:Lipman} in Fact~\ref{fact:Lipman}, where
\[  \omega_{Y^{\circ}/T}[d] \isom_{\qis} f^{!}\SO_{T}|_{Y^{\circ}}\]
for the relative dimension \( d \) of \( f \).
Note that \( Y \) satisfies \( \bfS_{2} \) by Fact~\ref{fact:elem-flat}\eqref{fact:elem-flat:6}.
Thus, \( \SR^{\bullet}_{Y}[-d] \) is an ordinary dualizing complex of \( Y \), and
\( \omega_{Y} := \SH^{-d}(\SR_{Y}^{\bullet}) \) is a dualizing sheaf of \( Y \).
In particular, \( U^{\circ} \) is a Gorenstein scheme with the dualizing sheaf
\[ \omega_{Y}|_{U^{\circ}} = \SH^{-d}(\SR_{Y}^{\bullet})|_{U^{\circ}} \isom
\omega_{Y^{\circ}/T}|_{U^{\circ}} \otimes_{\SO_{U^{\circ}}}
(f|_{U^{\circ}})^{*}(\omega_{T^{\circ}}). \]
By Lemma~\ref{lem:QGorSch3}, we have an isomorphism
\begin{equation}\label{eq:lem:naiveQGor}
\omega_{Y}^{[m]} \isom \omega_{Y/T}^{[m]} \otimes_{\SO_{Y}} f^{*}(\omega_{T}^{[m]})
\end{equation}
for any integer \( m \).
For a point \( y \in Y \), \( Y \) is \( \BQQ \)-Gorenstein at \( y \) if and only if
\( \omega_{Y}^{[m]} \) is invertible at \( y \) for some \( m > 0 \).
On the other hand, \( f \) is naively \( \BQQ \)-Gorenstein at \( y \)
if and only if \( \omega_{Y/T}^{[m]} \) is invertible at \( y \) for some \( m > 0 \).
Since \( \omega_{T}^{[r]} \) is invertible, the isomorphism \eqref{eq:lem:naiveQGor} implies
that \( Y \) is \( \BQQ \)-Gorenstein if and only if \( f \) is naively \( \BQQ \)-Gorenstein.
\end{proof}

The following is a criterion for a morphism to be \( \BQQ \)-Gorenstein.

\begin{prop}\label{prop:QGorCritMm}
Let \( f \colon Y \to T \) be a flat morphism locally of finite type between locally Noetherian schemes.
For a point \( t \in T \),
assume that the fiber \( Y_{t} = f^{-1}(t) \) is a \( \BQQ \)-Gorenstein scheme.
If there exist coherent \( \SO_{Y} \)-modules \( \SM_{m} \) for \( m \geq 1 \)
such that
\[ \SM_{m} \otimes_{\SO_{Y}} \SO_{Y_{t}} \isom \omega^{[m]}_{Y_{t}/\Bbbk(t)}
\quad \text{and} \quad
\SM_{m}|_{Y^{\circ}} \isom \omega^{\otimes m}_{Y^{\circ}/T},\]
where \( Y^{\circ} \) is the relative Gorenstein locus \( \Gor(Y/T) \),
then \( f \) is a \( \BQQ \)-Gorenstein morphism along \( Y_{t} \).
\end{prop}

\begin{proof}
We set \( \SM_{0} = \SO_{Y} \). Then,
\( \SM_{m, (t)} = \SM_{m} \otimes_{\SO_{Y}} \SO_{Y_{t}} \)
satisfies \( \bfS_{2} \) along \( Y_{t} \) for any \( m \geq 0 \).
For the complement \( Z = Y \setminus Y^{\circ} \),
we have \( \Codim(Z \cap Y_{t}, Y_{t}) \geq 2 \), since \( Y_{t} \) is \( \BQQ \)-Gorenstein.
Hence, \( \SM_{m} \) is flat over \( T \) along \( Y_{t} \)
by Lemma~\ref{lem:S2flat(new)}\eqref{lem:S2flat(new):1},
since \( \SM_{m}|_{Y^{\circ}} \isom \omega^{\otimes m}_{Y^{\circ}/T}\) is flat over \( T \) and
\[ \depth_{Z \cap Y_{t}} \SM_{m, (t)} \geq 2 \]
(cf.\ Lemma~\ref{lem:depth+codim+Sk}\eqref{lem:depth+codim+Sk:2}).
As a consequence,
\( \SM_{m} \) satisfies relative \( \bfS_{2} \) over \( T \) along \( Y_{t} \)
for any \( m \geq 0 \).
In particular, \( f \) is an \( \bfS_{2} \)-morphism along \( Y_{t} \)
by considering the case \( m = 0 \).
By replacing \( Y \) with an open neighborhood of \( Y_{t} \),
we may assume that \( f \) is an \( \bfS_{2} \)-morphism and that
\( \Codim(Z \cap Y_{t'}, Y_{t'}) \geq 2\) for any \( t' \in f(Y) \),
by Lemma~\ref{lem:S2Codim2}.

Now, \( \Supp \SM_{m} = Y \), since it contains the dense open subset \( Y^{\circ} \).
Hence, \( \Supp \SM_{m, (t')} = Y_{t'}\) for any \( t' \in T \),
and it is locally equi-dimensional by Fact~\ref{fact:S2}\eqref{fact:S2:1}.
Thus, \( U_{m} := \bfS_{2}(\SM_{m}) \)
is open by Fact~\ref{fact:dfn:RelSkCMlocus}\eqref{fact:dfn:RelSkCMlocus:2}, and
\[ \depth_{Z \cap U_{m}} \SM_{m}|_{U_{m}} \geq 2 \]
by Lemma~\ref{lem:relSkCodimDepth}\eqref{lem:relSkCodimDepth:1}.
It implies that, for the open immersion \( j \colon Y^{\circ} \injmap Y \),
\[ \SM_{m} \to j_{*}(\SM_{m}|_{Y^{\circ}}) \isom j_{*}(\omega^{\otimes m}_{Y^{\circ}/T}) = \omega^{[m]}_{Y/T}\]
is an isomorphism along \( Y_{t} \).
As a consequence, \( \omega_{Y/T}^{[m]} \) satisfies relative \( \bfS_{2} \) over \( T \) along \( Y_{t} \)
for any \( m \geq 0\). Therefore, \( f \) is a \( \BQQ \)-Gorenstein morphism along \( Y_{t} \).
\end{proof}

We have the following base change properties for \( \BQQ \)-Gorenstein morphisms
and for their variants.

\begin{prop}\label{prop:QGormor}
Let \( f \colon Y \to T \) be a flat morphism locally of finite type
between locally Noetherian schemes and let
\[ \begin{CD}
Y' @>{p}>> Y \\ @V{f'}VV @VV{f}V \\ T' @>{q}>> T
\end{CD}\]
be a Cartesian diagram of schemes such that \( T' \) is also locally Noetherian.
\begin{enumerate}
\item  \label{prop:QGormor:1n}
If \( f \) is a naively \( \BQQ \)-Gorenstein morphism, then so is \( f' \).
Here, if \( \omega_{Y/T}^{[r]} \) is invertible, then
\( \omega_{Y'/T'}^{[r]} \isom p^{*}(\omega_{Y/T}^{[r]})  \).

\item \label{prop:QGormor:1nInv} In case \( q \colon T' \to T \) is a flat and surjective morphism,
if \( f' \) is naively \( \BQQ \)-Gorenstein, then so is \( f \).

\item \label{prop:QGormor:1nCor}
If every fiber of \( f \) is \( \BQQ \)-Gorenstein, then
every fiber of \( f' \) is so. The converse holds if \( q \) is surjective.

\item \label{prop:QGormor:1v}
If \( f \) is virtually \( \BQQ \)-Gorenstein at a point \( y \in Y\), then \( f' \)
is so at any point of \( p^{-1}(y) \).

\item \label{prop:QGormor:1}
If \( f\) is \( \BQQ \)-Gorenstein, then \( f' \) is so and
\( \omega_{Y'/T'}^{[m]} \isom p^{*}(\omega_{Y/T}^{[m]}) \)
for any \( m \in \BZZ \).

\item \label{prop:QGormor:2}
In case \( q \colon T' \to T \) is a flat and surjective morphism, if \( f' \) 
is \( \BQQ \)-Gorenstein, then so is \( f \). 
\end{enumerate}
\end{prop}

\begin{proof}
Note that \( Y^{\prime \circ} = p^{-1}(Y^{\circ}) \) for
\( Y^{\prime \circ} := \Gor(Y'/T') \) (cf.\ Corollary~\ref{cor:BC-S2CMGor}) and that
\begin{equation}\label{eq:1|prop:QGormor}
\Codim(Y_{t} \setminus Y^{\circ}, Y_{t}) = \Codim(Y'_{t'} \setminus Y^{\prime \circ}, Y')
\end{equation}
for any \( t' \in T' \) and \( t = q(t') \) (cf.\ Lemma~\ref{lem:bc basic}\eqref{lem:bc basic:1}).

\eqref{prop:QGormor:1n}: The base change \( f' \) is an \( \bfS_{2} \)-morphism by
Lemma~\ref{lem:bc basic}\eqref{lem:bc basic:5}, and we have
an isomorphism \( \omega_{Y'/T'}^{[r]} \isom p^{*}(\omega_{Y/T}^{[r]}) \)
by Corollary~\ref{cor:BC-S2CMGor}\eqref{cor:BC-S2CMGor:2}.
In particular, \( f' \) is a naively \( \BQQ \)-Gorenstein morphism.

\eqref{prop:QGormor:1nInv}: The morphism \( f \) is an \( \bfS_{2} \)-morphism by
Lemma~\ref{lem:bc basic}\eqref{lem:bc basic:3} applied to \( \SF = \SO_{Y} \), since
\( p \colon Y' \to Y \) is surjective. Moreover,
every fiber of \( f \) is Gorenstein in codimension one by \eqref{eq:1|prop:QGormor}.
Now, \( p^{*}(\omega^{[m]}_{Y/T}) \) is reflexive for any \( m \)
by Remark~\ref{rem:dfn:reflexive}, since \( p \) is flat.
Hence, \( p^{*}(\omega^{[m]}_{Y/T}) \isom \omega^{[m]}_{Y'/T'} \)
for any \( m \) by  Corollary~\ref{cor:BC-S2CMGor}\eqref{cor:BC-S2CMGor:2}.
If \( p^{*}(\omega^{[r]}_{Y/T}) \) is invertible, then so is \( \omega^{[r]}_{Y/T} \),
since \( p \) is fully faithful (cf.\ Lemma~\ref{lem:LocFreeDecsent}).
Therefore, \( f \) is naively \( \BQQ \)-Gorenstein.

\eqref{prop:QGormor:1nCor}: This is obtained by applying \eqref{prop:QGormor:1n}
and \eqref{prop:QGormor:1nInv}
to the case where \( T = \Spec \Bbbk(t) \) and \( T' = \Spec \Bbbk(t') \)
and by Lemma~\ref{lem:naiveQGor}.

\eqref{prop:QGormor:1v}:
We may assume that the conditions of Remark~\ref{rem:dfn:vQGorMor} are satisfied for \( U = Y \),
a certain reflexive \( \SO_{Y} \)-module \( \SL \), and for \( o = f(y) \).
Then, the conditions \eqref{rem:dfn:vQGorMor:1} and \eqref{rem:dfn:vQGorMor:2}
of Remark~\ref{rem:dfn:vQGorMor} imply
\[ \depth_{Y_{t} \setminus Y^{\circ}} \SO_{Y_{t}} \geq 2 \]
for any \( t \in f(Y) \), by Lemma~\ref{lem:depth+codim+Sk}\eqref{lem:depth+codim+Sk:2}.
Hence, \( p^{*}(\SL^{[m]}) \) is a reflexive \( \SO_{Y'} \)-module and
\( (p^{*}\SL)^{[m]} \isom p^{*}(\SL^{[m]}) \) for any \( m \),
by Lemma~\ref{lem:bc reflexive} applied to \( Z = Y \setminus Y^{\circ} \)
and to \( \SL^{[m]} \).
Here, \( (p^{*}\SL)^{[m]} \) satisfies relative \( \bfS_{2} \) over \( T' \)
by Remark~\ref{rem:dfn:vQGorMor}\eqref{rem:dfn:vQGorMor:6}
and Lemma~\ref{lem:bc basic}\eqref{lem:bc basic:4}.
Furthermore, for any point \( t' \in T' \) and \( t = q(t') \),
we have isomorphisms
\[ p^{*}\SL \otimes_{\SO_{Y'}} \SO_{Y'_{t'}} \isom
(\SL \otimes_{\SO_{Y}} \SO_{Y_{t}}) \otimes_{\Bbbk(t)} \Bbbk(t')
\isom \omega_{Y_{t}/\Bbbk(t)} \otimes_{\Bbbk(t)} \Bbbk(t')
\isom \omega_{Y'_{t'}/\Bbbk(t')},\]
by applying Lemma~\ref{lem:omegaFlatbc} to \( \Spec \Bbbk(t') \to \Spec \Bbbk(t) \).
Therefore, \( f' \) is virtually \( \BQQ \)-Gorenstein at any point of \( p^{-1}(y) \),
since \( p^{*}\SL \) plays the role of \( \SL \) in Definition~\ref{dfn:vQGorMor}.

\eqref{prop:QGormor:1}:
By \eqref{prop:QGormor:1n},
\( f' \) is an \( \bfS_{2} \)-morphism whose fibers are all \( \BQQ \)-Gorenstein.
If \( \omega^{[m]}_{Y/T} \) satisfies relative \( \bfS_{2} \) over \( T \), then
\( p^{*}\omega^{[m]}_{Y/T} \) does so over \( T' \)
by Lemma~\ref{lem:bc basic}\eqref{lem:bc basic:4},
and \( p^{*}\omega^{[m]}_{Y/T} \isom \omega^{[m]}_{Y'/T'} \)
by Corollary~\ref{cor:BC-S2CMGor}\eqref{cor:BC-S2CMGor:2}.
Therefore, \( f' \) is \( \BQQ \)-Gorenstein
(cf.\ Definition~\ref{dfn:QGorMor}\eqref{dfn:QGorMor:QGor}).

\eqref{prop:QGormor:2}: 
By \eqref{prop:QGormor:1nInv} above, 
\( f \) is naively \( \BQQ \)-Gorenstein. 
By Lemma~\ref{lem:LocDefQGor}, 
it is enough to prove that the base change homomorphism 
\[ \phi^{[m]}_{t} \colon 
\omega^{[m]}_{Y/T} \otimes_{\SO_{Y}} \SO_{Y_{t}} \to \omega^{[m]}_{Y_{t}/\Bbbk(t)} \]
is an isomorphism for any \( m \in \BZZ \) and any point \( t \in T \). 
For any point \( t' \in q^{-1}(t) \), the base change morphism 
\[ \phi^{[m]}_{t'} \colon \omega^{[m]}_{Y'/T'} \otimes_{\SO_{Y'}} \SO_{Y'_{t'}} 
\to \omega^{[m]}_{Y'_{t'}/\Bbbk(t')}  \]
is an isomorphism, since \( f' \) is \( \BQQ \)-Gorenstein. 
Now, \( \phi^{[m]}_{t'} \) is isomorphic to the homomorphism \( p_{t'}^{*}(\phi^{[m]}_{t}) \) for 
the morphism \( p_{t'} \colon Y'_{t'} \to Y_{t} \) induced from \( p \), 
since we have an isomorphism \( \omega^{[m]}_{Y'/T'} \isom p^{*}\omega^{[m]}_{Y/T}\) 
as in the proof of \eqref{prop:QGormor:1nInv}. Since \( p_{t'} \) is faithfully flat, 
\( \phi^{[m]}_{t} \) is an isomorphism for any \( m \in \BZZ\) and \( t \in T\). 
Therefore, \( f \) is \( \BQQ \)-Gorenstein. 
\end{proof}

We have the following properties for compositions of \( \BQQ \)-Gorenstein morphisms
and of their variants.

\begin{prop}\label{prop:QGorCompo}
Let \( f \colon Y \to T \) and \( g \colon X \to Y \)
be flat morphisms of locally Noetherian schemes.
\begin{enumerate}
\item \label{prop:QGorCompo:1}
If \( f \) and \( g \) are naively \( \BQQ \)-Gorenstein, then
\( f \circ g \) is so, and
\[ \omega_{X/T}^{[r]} \isom \omega_{X/Y}^{[r]} \otimes_{\SO_{X}} g^{*}(\omega_{Y/T}^{[r]})  \]
for an integer \( r > 0 \) such that \( \omega_{X/Y}^{[r]} \) and \( \omega_{Y/T}^{[r]} \)
are invertible.

\item \label{prop:QGorCompo:2}
Assume that \( g \) is a \( \BQQ \)-Gorenstein morphism.
If \( f \) is virtually \( \BQQ \)-Gorenstein at a point \( y \),
then \( f \circ g \) virtually \( \BQQ \)-Gorenstein at any point of \( g^{-1}(y) \).

\item \label{prop:QGorCompo:3}
If \( f \) and \( g \) are \( \BQQ \)-Gorenstein morphisms,
then \( f \circ g \) is so, and
\[ \omega_{X/T}^{[m]} \isom \omega_{X/Y}^{[m]} \otimes_{\SO_{X}} g^{*}(\omega_{Y/T}^{[m]}) \]
for any integer \( m \).
\end{enumerate}
\end{prop}

\begin{proof}
\eqref{prop:QGorCompo:1}:
Every fiber of the composite \( f \circ g \)
is \( \BQQ \)-Gorenstein by Lemma~\ref{lem:naiveQGor}
and by Proposition~\ref{prop:QGormor}\eqref{prop:QGormor:1n}.
In particular, \( f \circ g \) is an \( \bfS_{2} \)-morphism.
For the relative Gorenstein loci \( Y^{\circ} := \Gor(Y/T) \) and
\( X^{\circ} := \Gor(X/Y) \),
let \( V \) be the intersection \( X^{\circ} \cap g^{-1}(Y^{\circ}) \).
Then, \( V \subset \Gor(X/T) \) and
\( \Codim (X_{t} \setminus V, X_{t}) \geq 2 \) for any fiber
\( X_{t} = (f \circ g)^{-1}(t)\) of \( f \circ g \).
We set
\[ \SM_{r} := \omega_{X/Y}^{[r]} \otimes g^{*}(\omega_{Y/T}^{[r]})\]
for an integer \( r > 0 \) such that
\( \omega_{X/Y}^{[r]} \) and \( \omega_{Y/T}^{[r]} \) are invertible.
Then, \( \SM_{r}|_{V} \isom \omega_{V/T}^{\otimes r} \) and
\[ \SM_{r} \isom j_{*}(\omega_{V/T}^{\otimes r}) = \omega_{Y/T}^{[r]} \]
for the open immersion \( j \colon V \injmap X \), since \( f \circ g \) is an \( \bfS_{2} \)-morphism.
Thus, \( f \circ g \) is naively \( \BQQ \)-Gorenstein.

\eqref{prop:QGorCompo:2}:
We may assume that the conditions of Remark~\ref{rem:dfn:vQGorMor}
are satisfied for \( U = Y \), a certain reflexive \( \SO_{Y} \)-module \( \SL \),
and for \( o = f(y) \). We set
\[ \SN_{m} := \omega_{X/Y}^{[m]} \otimes_{\SO_{X}} g^{*}(\SL^{[m]})\]
for an integer \( m \).  This is flat over \( T \), since
\( \SL^{[m]} \) is so over \( T \) and \( \omega_{X/Y}^{[m]} \)
is so over \( Y \).
Let \( g_{o} = g|_{X_{o}} \colon X_{o} \to Y_{o} \)
be the induced \( \BQQ \)-Gorenstein morphism
(cf.\ Proposition~\ref{prop:QGormor}\eqref{prop:QGormor:1}).
Then, \( X_{o} = g^{-1}(Y_{o})\)
is \( \BQQ \)-Gorenstein by
Remark~\ref{rem:dfn:vQGorMor}\eqref{rem:dfn:vQGorMor:7} and Lemma~\ref{lem:naiveQGor},
and we have isomorphisms
\begin{align*}
\SN_{m} \otimes_{\SO_{X}} \SO_{X_{o}} &\isom
(\omega_{X/Y}^{[m]} \otimes_{\SO_{X}} \SO_{X_{o}}) \otimes_{\SO_{X_{o}}}
g_{o}^{*}(\omega_{Y_{o}/\Bbbk(o)}^{[m]}) \\
&\isom
\omega_{X_{o}/Y_{o}}^{[m]} \otimes_{\SO_{X_{o}}} g_{o}^{*}(\omega_{Y_{o}/\Bbbk(o)}^{[m]})
\isom \omega_{X_{o}/\Bbbk(o)}^{[m]},
\end{align*}
where the first isomorphism is derived from Remark~\ref{rem:dfn:vQGorMor}\eqref{rem:dfn:vQGorMor:8}
and the last one from \eqref{eq:lem:naiveQGor} in the proof of Lemma~\ref{lem:naiveQGor}.
In particular, \( \SN_{m} \) satisfies relative \( \bfS_{2} \) over \( T \)
along \( X_{o} \).
Then, for \( \SN := \SN_{1} \), we have isomorphisms
\[ \SN_{m} \isom j_{*}(\SN_{m}|_{V}) \isom
j_{*}\left(\omega_{V/Y}^{\otimes m} \otimes_{\SO_{V}} (g^{*}\SL)^{\otimes m}|_{V}\right)
\isom j_{*}(\SN^{\otimes m}|_{V}) = \SN^{[m]}\]
along \( X_{o} \) by Lemma~\ref{lem:US2add}\eqref{lem:US2add:3a}, where \( j \colon V \injmap X \)
is the open immersion in the proof of \eqref{prop:QGorCompo:1}.
Hence, \( \SN^{[m]} \) satisfies relative \( \bfS_{2} \) over \( T \)
along \( X_{o} \) for any \( m \), and
\( \SN \otimes_{\SO_{X}} \SO_{X_{o}} \isom \omega_{X_{o}/\Bbbk(o)} \).
Therefore, \( f \circ g \) is virtually \( \BQQ \)-Gorenstein at any point of \( g^{-1}(y) \),
since \( \SN \) plays the role of \( \SL \) in Definition~\ref{dfn:vQGorMor}.

\eqref{prop:QGorCompo:3}: We can apply the argument in the proof of \eqref{prop:QGorCompo:2}
by setting \( \SL = \omega_{Y/T} \).
Then,
\[ \SN^{[m]} \isom j_{*}(\SN_{m}|_{V}) \isom
j_{*}\left(\omega_{V/Y}^{\otimes m} \otimes_{\SO_{V}} (g^{*}\omega_{Y/T}^{\otimes m})|_{V}\right)
\isom j_{*}(\omega_{V/T}^{\otimes m}) = \omega^{[m]}_{X/T}\]
along \( X_{o} \).
Hence, \( \omega^{[m]}_{X/T} \) satisfies
relative \( \bfS_{2} \) over \( T \) for any \( m \).
Consequently, \( f \circ g \) is \( \BQQ \)-Gorenstein with an isomorphism
\( \omega_{X/T}^{[m]} \isom \omega_{X/Y}^{[m]} \otimes_{\SO_{X}} g^{*}(\omega_{Y/T}^{[m]}) \)
for any \( m \in  \BZZ \). Thus, we are done.
\end{proof}

\begin{cor}\label{cor:QGorMorSmooth}
Let \( Y \) and \( T \) be locally Noetherian schemes and \( f \colon Y \to T  \)
a flat morphism locally of finite type.
Let \( g \colon X \to Y \) be a smooth separated surjective morphism
from a locally Noetherian scheme \( X \).
Then, \( f \) is \( \BQQ \)-Gorenstein
if and only if \( f \circ g \colon  X  \to Y \to T \) is so.
\end{cor}

\begin{proof}
For the relative Gorenstein loci \( Y^{\circ} := \Gor(Y/T) \) and \( X^{\circ} := \Gor(X/T) \),
we have \( X^{\circ} = g^{-1}(Y^{\circ}) \) by Lemma~\ref{lem:QGorScheme:smooth}.
Let \( g^{\circ} \colon X^{\circ} \to Y^{\circ} \) be the induced smooth morphism.
Then,
\begin{equation}\label{eq:cor:QGormorSmooth}
\omega_{X^{\circ}/T} \isom \omega_{X^{\circ}/Y^{\circ}} \otimes_{\SO_{X^{\circ}}}
g^{\circ *}(\omega_{Y^{\circ}/T})
\end{equation}
for the relative canonical sheaves \( \omega_{Y^{\circ}/T} \),
\( \omega_{X^{\circ}/T} \), and \( \omega_{X^{\circ}/Y^{\circ}} \)
(cf.\ \eqref{fact:DeligneVerdierLipman:compo} and
\eqref{fact:DeligneVerdierLipman:smooth} of Fact~\ref{fact:DeligneVerdierLipman}).
For a point \( t \in T \), let \( g_{t} \colon X_{t} \to Y_{t} \) be the smooth morphism
induced on the fibers \( Y_{t} = f^{-1}(t) \) and \( X_{t} = (f \circ g)^{-1}(t) \).

By Proposition~\ref{prop:QGorCompo}\eqref{prop:QGorCompo:2}, it is enough to prove
the ``if'' part.
Assume that \( f \circ g\) is \( \BQQ \)-Gorenstein.
Then, every fiber \( Y_{t} \) is \( \BQQ \)-Gorenstein by
Lemma~\ref{lem:QGorScheme:smooth}. In particular, \( Y_{t} \) satisfies \( \bfS_{2} \) and
\( \Codim(Y_{t} \setminus Y^{\circ}, Y_{t}) \geq 2 \).
Hence, by Lemma~\ref{lem:US2add}\eqref{lem:US2add:2},
\[ \omega_{Y/T}^{[m]} \isom j_{*}(\omega_{Y^{\circ}/T}^{\otimes m}) \]
for any \( m \in \BZZ \),
where \( j \colon Y^{\circ} \injmap Y \) is the open immersion.
For the open immersion \( j_{X} \colon X^{\circ} \injmap X \),
we have an isomorphism
\[ g^{*}(\omega_{Y/T}^{[m]}) \isom
g^{*}(j_{*}(\omega_{Y^{\circ}/T}^{\otimes m})) \isom
j_{X*}(g^{\circ *}(\omega_{Y^{\circ}/T}^{\otimes m})) \]
by the flat base change isomorphism (cf.\ Lemma~\ref{lem:flatbc}).
Thus,
\[ \omega_{X/T}^{[m]} \isom j_{X*}(\omega_{X^{\circ}/T}^{\otimes m})
\isom
j_{X*}(j_{X}^{*}(\omega^{\otimes m}_{X/Y}) \otimes_{\SO_{X^{\circ}}}
g^{\circ *}(\omega^{\otimes m}_{Y^{\circ}/T^{\circ}}))
\isom
\omega_{X/Y}^{\otimes m} \otimes_{\SO_{Y}} g^{*}(\omega_{Y/T}^{[m]}) \]
for any \( m \in \BZZ \) by \eqref{eq:cor:QGormorSmooth}.
In particular,
\( \omega_{Y/T}^{[m]} \) is flat over \( T \), since \( g \) is faithfully flat
(cf.\ Lemma~\ref{lem:fflatDescent}).
Moreover,
\[ g_{t}^{*}(\omega_{Y/T}^{[m]} \otimes_{\SO_{Y}} \SO_{Y_{t}})
\isom \omega_{X_{t}/Y_{t}}^{\otimes -m} \otimes_{\SO_{X_{t}}}
(\omega_{X/T}^{[m]} \otimes_{\SO_{X}} \SO_{X_{t}})
\isom \omega_{X_{t}/Y_{t}}^{\otimes -m} \otimes_{\SO_{X_{t}}}
\omega_{X_{t}/\Bbbk(t)}^{[m]}\]
satisfies \( \bfS_{2} \) for any \( t \in T \)
(cf.\ Lemma~\ref{lem:QGorSch3}\eqref{lem:QGorSch3:2}).
As a consequence,
\( \omega_{Y/T}^{[m]} \otimes_{\SO_{Y}} \SO_{Y_{t}}\) satisfies \( \bfS_{2} \)
by Fact~\ref{fact:elem-flat}\eqref{fact:elem-flat:6}.
Therefore, \( \omega_{Y/T}^{[m]} \) satisfies relative \( \bfS_{2} \)
over \( T \) for any \( m \), and
\( Y \to T \) is a \( \BQQ \)-Gorenstein morphism.
\end{proof}

\begin{remn}
Considering an \'etale morphism \( g \) in Corollary~\ref{cor:QGorMorSmooth},
we see that, for a given flat morphism \( f \colon Y \to T \)
locally of finite type between locally Noetherian schemes,
the \( \BQQ \)-Gorenstein condition at a point of \( Y \)
is not only Zariski local but also \'etale local (cf.\ Remark~\ref{rem:QGorScheme:etale}).
\end{remn}


\subsection{Theorems on \texorpdfstring{$\BQQ$}{Q}-Gorenstein  morphisms}
\label{subsect:thmsQGormor}

First of all, 
we shall prove some infinitesimal criteria for a morphism to be 
\( \BQQ \)-Gorenstein or naively \( \BQQ \)-Gorenstein. 

\begin{thm}[infinitesimal criterion]\label{thm:InfQGorCrit}
Let \( f \colon Y \to T \) be a flat morphism locally of finite type between 
locally Noetherian schemes. For a point \( y \in Y\) and its image \( t = f(y) \), assume that 
the fiber \( Y_{t} = f^{-1}(t)\) is \( \BQQ \)-Gorenstein at \( y \).  
Then, for a positive integer \( m \), 
the following two conditions \eqref{thm:InfQGorCrit:1} and \eqref{thm:InfQGorCrit:2} 
are equivalent to each other\emph{:}
\begin{enumerate}
\renewcommand{\theenumi}{\roman{enumi}}
\renewcommand{\labelenumi}{(\theenumi)}
\item \label{thm:InfQGorCrit:1} 
The sheaf \( \omega^{[m]}_{Y/T} \) satisfies relative \( \bfS_{2} \) over \( T \) at \( y \). 
\item \label{thm:InfQGorCrit:2} Let \( \SO_{T, t} \to A \) 
be a surjective local ring homomorphism to an Artinian local ring \( A \) 
and let \( Y_{A} = Y \times_{T} \Spec A \to \Spec A \) be the base change of \( f \) by 
the associated morphism \( \Spec A \to T \). 
Then, \( \omega^{[m]}_{Y_{A}/A} \) satisfies relative \( \bfS_{2} \) over \( A \) 
at the point \( y_{A} \in Y_{A} \) lying over \( y \). 
\end{enumerate}
Moreover, the following two conditions \eqref{thm:InfQGorCrit:3} and \eqref{thm:InfQGorCrit:4} 
are also equivalent to each other\emph{:}

\begin{enumerate}
\renewcommand{\theenumi}{\roman{enumi}}
\renewcommand{\labelenumi}{(\theenumi)}
\addtocounter{enumi}{2}
\item  \label{thm:InfQGorCrit:3} The sheaf \( \omega^{[m]}_{Y/T} \) 
is invertible at \( y \). 
\item \label{thm:InfQGorCrit:4} 
For the same morphism \( Y_{A} \to \Spec A \) in \eqref{thm:InfQGorCrit:2}, 
\( \omega^{[m]}_{Y_{A}/A} \) is invertible at \( y_{A} \). 
\end{enumerate}
\end{thm}

\begin{proof}
Since the fiber \( Y_{t} \) satisfies \( \bfS_{2} \) at \( y \), 
by localizing \( Y \), we may assume that \( f \) is an \( \bfS_{2} \)-morphism 
(cf.\ Fact~\ref{fact:dfn:RelSkCMlocus}\eqref{fact:dfn:RelSkCMlocus:3}). 
Moreover, we may assume that every fiber of \( f \) is Gorenstein in codimension one 
by Lemma~\ref{lem:S2Codim2}\eqref{lem:S2Codim2:2}, since 
\[ \Codim_{y}(Y_{t} \setminus U, Y_{t}) \geq 2 \]
for the relative Gorenstein locus \( U = \Gor(Y/T) \).  
Let \( p_{A} \colon Y_{A} \to Y \) be the projection for a morphism \( \Spec A \to T \) in 
\eqref{thm:InfQGorCrit:2}. Then, by Corollary~\ref{cor:BC-S2CMGor}\eqref{cor:BC-S2CMGor:2}, 
we have an isomorphism
\[ (p_{A}^{*}\omega^{[m]}_{Y/T})^{\vee\vee} \isom \omega^{[m]}_{Y_{A}/A},\] 
and moreover the base change homomorphism 
\[ p_{A}^{*}\omega^{[m]}_{Y/T} \to \omega^{[m]}_{Y_{A}/A} \] 
is an isomorphism at \( y_{A} \) when \eqref{thm:InfQGorCrit:1} holds. 
In particular, we have \eqref{thm:InfQGorCrit:1} \( \Rightarrow \) \eqref{thm:InfQGorCrit:2}, 
and the converse \eqref{thm:InfQGorCrit:2} \( \Rightarrow \) 
\eqref{thm:InfQGorCrit:1} 
is a consequence of Proposition~\ref{prop:inf+val} in the case \eqref{prop:inf+val:inf} 
applied to \( \SF = \omega^{[m]}_{Y/T} \). 
If \eqref{thm:InfQGorCrit:1} or \eqref{thm:InfQGorCrit:2} holds, then 
the base change homomorphisms 
\[ \omega^{[m]}_{Y/T} \otimes_{\SO_{Y}} \SO_{Y_{t}} \to \omega^{[m]}_{Y_{t}/\Bbbk(t)} 
\quad \text{and} \quad 
\omega^{[m]}_{Y_{A}/A} \otimes_{\SO_{Y_{A}}} \SO_{Y_{t}} \to \omega^{[m]}_{Y_{t}/\Bbbk(t)}\] 
are isomorphisms at \( y \) or \( y_{A} \), 
again by Corollary~\ref{cor:BC-S2CMGor}\eqref{cor:BC-S2CMGor:2}. 
Hence, by Fact~\ref{fact:elem-flat}\eqref{fact:elem-flat:2}, 
we have equivalences \eqref{thm:InfQGorCrit:3} \( \Leftrightarrow \) 
\eqref{thm:InfQGorCrit:1} \( + \) \eqref{thm:InfQGorCrit:5} 
and \eqref{thm:InfQGorCrit:4} \( \Leftrightarrow \) 
\eqref{thm:InfQGorCrit:2} \( + \) \eqref{thm:InfQGorCrit:5} 
for the following condition:
\begin{enumerate}
\renewcommand{\theenumi}{\roman{enumi}}
\renewcommand{\labelenumi}{(\theenumi)}
\addtocounter{enumi}{4}
\item \label{thm:InfQGorCrit:5} 
The sheaf \( \omega^{[m]}_{Y_{t}/\Bbbk(t)} \) is invertible at \( y \). 
\end{enumerate}
Thus, we have \eqref{thm:InfQGorCrit:3} \( \Leftrightarrow \) \eqref{thm:InfQGorCrit:4}. 
\end{proof}

By the equivalence \eqref{thm:InfQGorCrit:1} \( \Leftrightarrow \) \eqref{thm:InfQGorCrit:2} 
in Theorem~\ref{thm:InfQGorCrit} 
for all \( m \in \BZZ \) and for any \( y \in Y_{t} \), we have the following infinitesimal criterion 
for a morphism to be \( \BQQ \)-Gorenstein: 

\begin{cor}[infinitesimal criterion]\label{cor:QGorCrit}
Let \( f \colon Y \to T \) be a flat morphism locally of finite type
between locally Noetherian schemes. 
Then, for a given point \( t \in T \), 
the morphism \( f \) is \( \BQQ \)-Gorenstein along 
the fiber \( Y_{t} = f^{-1}(t)\)  
if and only if the base change 
\( f_{A} \colon Y_{A} = Y \times_{T} \Spec A \to \Spec A \)
is \( \BQQ \)-Gorenstein 
for any morphism \( \Spec A \to T \) defined by a surjective local ring homomorphism 
\( \SO_{T, t} \to A \) to any Artinian local ring \( A \).
\end{cor}

\begin{remn} 
For an Artinian local ring \( A \), a flat morphism \( Y_{A} \to \Spec A \) of finite type 
is not necessarily a \( \BQQ \)-Gorenstein morphism even if \( Y_{A} \) is \( \BQQ \)-Gorenstein 
and \( A \) is Gorenstein. 
For example, let us consider 
a naively \( \BQQ \)-Gorenstein morphism \( f \colon Y \to T = \Spec R \) 
for a discrete valuation ring \( R \) such that \( f \) is not \( \BQQ \)-Gorenstein 
along the closed fiber \( Y_{o} = f^{-1}(o) \), where \( o \) is the closed point of \( T \) 
corresponding to the maximal ideal \( \GM_{R} \). 
See Fact~\ref{fact:Pa2} or Lemma~\ref{lem:Lee} and Example~\ref{exam:KummerType} for 
such an example of \( f  \). 
Let \( \Spec A \to T \) be a closed immersion for a local Artinian ring \( A \). 
Then, \( A \) is Gorenstein, and the base change 
\( f_{A} \colon Y_{A} = Y \times_{T} \Spec A \to \Spec A \) of \( f \) 
is a naively \( \BQQ \)-Gorenstein morphism
by Proposition~\ref{prop:QGormor}\eqref{prop:QGormor:1n}.
Hence, \( Y_{A} \) is \( \BQQ \)-Gorenstein by Lemma~\ref{lem:naiveQGor}. 
However, \( f_{A}  \) is not a \( \BQQ \)-Gorenstein morphism for some \( A \)
by Corollary~\ref{cor:QGorCrit}.
\end{remn}

By the equivalence \eqref{thm:InfQGorCrit:3} \( \Leftrightarrow \) \eqref{thm:InfQGorCrit:4} 
in Theorem~\ref{thm:InfQGorCrit} for any \( y \in Y_{t} \), we have 
the following version of an infinitesimal criterion for naively \( \BQQ \)-Gorenstein morphisms 
with bounded relative Gorenstein index: 

\begin{cor}[infinitesimal criterion]\label{cor:NaiveQGorCrit}
Let \( f \colon Y \to T \) be a flat morphism locally of finite type
between locally Noetherian schemes. 
For a point \( t \in T \) and a positive integer \( m \), 
the following two conditions are equivalent to each other\emph{:}
\begin{enumerate}
\renewcommand{\theenumi}{\roman{enumi}}
\renewcommand{\labelenumi}{(\theenumi)}
\item \label{cor:NaiveQGorCrit:1} 
The morphism \( f \) is naively \( \BQQ \)-Gorenstein along \( Y_{t} \) and 
the relative Gorenstein index of \( f \) along \( Y_{t} \) is a divisor of \( m \). 

\item \label{cor:NaiveQGorCrit:2} 
The sheaf \( \omega^{[m]}_{Y_{A}/A} \) is invertible for the base change 
\( f_{A} \colon Y_{A} = Y \times_{T} \Spec A \to \Spec A \) 
by any morphism \( \Spec A \to T \) defined by a surjective local ring homomorphism 
\( \SO_{T, t} \to A \) to any Artinian local ring \( A \).
\end{enumerate}
\end{cor}

\begin{rem}\label{rem:naiveQGorCharP}
The infinitesimal criterion does not hold
for naively \( \BQQ \)-Gorenstein morphisms without boundedness conditions 
for the relative Gorenstein index: 
Let \( f \colon Y \to T \) be a flat morphism of finite type of Noetherian schemes 
such that \( T = \Spec R\) for a discrete valuation ring \( R \). 
For an integer \( n \), we set \( R_{n} := R/\GM_{R}^{n+1} \), \( T_{n} = \Spec R_{n} \), 
and let \( Y_{n} = Y \times_{T} T_{n} \to T_{n} \) be the base change of \( f \) by the closed immersion 
\( T_{n} \subset T \). 
Assume that the special fiber \( Y_{0} \) is a \( \BQQ \)-Gorenstein scheme 
and the residue field \( \Bbbk = R/\GM_{R} \) has characteristic \( p > 0 \). 
Then, \( Y_{n} \to T_{n} \) is naively \( \BQQ \)-Gorenstein for any \( n \geq 0 \) 
by an argument of Koll\'ar in \cite[14.7]{HacKov} or \cite[Exam.\ 7.6]{Kovacs}. 
But, there is an example of \( f \colon Y \to T \) above 
such that \( f \) is not naively \( \BQQ \)-Gorenstein (cf.\ Example~\ref{exam:P114}). 
Let \( r_{n} \) be the relative Gorenstein index of \( Y_{n} \to T_{n} \); this equals  
the Gorenstein index of \( Y_{n} \). 
Then, \( \{r_{n}\}_{n \geq 0} \) is not bounded by Corollary~\ref{cor:NaiveQGorCrit} 
if \( f \) is not naively \( \BQQ \)-Gorenstein. 
\end{rem}

By a similar argument in the proof of Theorem~\ref{thm:InfQGorCrit} 
and by Proposition~\ref{prop:inf+val} in the case \eqref{prop:inf+val:val} 
instead of \eqref{prop:inf+val:inf}, we have the following 
valuative criterion for a morphism to be \( \BQQ \)-Gorenstein. 

\begin{thm}[valuative criterion]\label{thm:valcritQGor}
Let \( f \colon Y \to T \) be a flat morphism locally of finite type
between locally Noetherian schemes. Assume that \( T \) is reduced.
Then, \( f \) is a \( \BQQ \)-Gorenstein morphism if and only if the base change
\( f_{R} \colon Y_{R} = Y \times_{T} \Spec R \to \Spec R \)
is a \( \BQQ \)-Gorenstein morphism
for any discrete valuation ring \( R \) and
for any morphism \( \Spec R \to T \).
\end{thm}

\begin{proof}
It is enough to check the `if' part by Proposition~\ref{prop:QGormor}\eqref{prop:QGormor:1}.
Then, every fiber \( Y_{t} \) is \( \BQQ \)-Gorenstein,
since we can consider \( R \) as the localization at the prime ideal \( (\xtt) \)
of the polynomial ring \( \Bbbk(t)[\xtt] \) for the residue field \( \Bbbk(t) \) 
and consider the morphism \( \Spec R \to T \) 
defined by the composite \( \SO_{T, t} \to \Bbbk(t) \subset R \). 
Therefore, it is enough to prove that \( \omega_{Y/T}^{[m]} \) satisfies relative \( \bfS_{2} \) over \( T \)
for any \( m \in \BZZ\). 
For the base change morphism \( f_{R} \colon Y_{R} \to \Spec R \) 
and the projection \( p \colon Y_{R} \to Y \), we have an isomorphism 
\[ (p^{*}\omega^{[m]}_{Y/T})^{\vee\vee} \isom \omega_{Y_{R}/R}^{[m]}  \]
for any \( m \) by Corollary~\ref{cor:BC-S2CMGor}\eqref{cor:BC-S2CMGor:2}. 
Therefore, the assertion is a consequence of Proposition~\ref{prop:inf+val}
in the case \eqref{prop:inf+val:val}. 
\end{proof}

The following theorem gives a criterion for a morphism to be \( \BQQ \)-Gorenstein 
only by conditions on fibers.

\begin{thm}\label{thm:S3Gor2}
Let \( Y \) and \( T \) be locally Noetherian schemes and \( f \colon Y \to T \)
be a flat morphism locally of finite type. 
For a point \( t \in T \), if the following three conditions are all satisfied, then 
\( f \) is \( \BQQ \)-Gorenstein along the fiber \( Y_{t} = f^{-1}(t) \)\emph{:} 
\begin{enumerate}
\renewcommand{\theenumi}{\roman{enumi}}
\renewcommand{\labelenumi}{(\theenumi)}
\item \label{thm:S3Gor2:1} \( Y_{t} \) is \( \BQQ \)-Gorenstein\emph{;}

\item  \label{thm:S3Gor2:2} \( Y_{t} \) is Gorenstein in codimension \emph{two;}

\item  \label{thm:S3Gor2:3}
\( \omega_{Y_{t}/\Bbbk(t)}^{[m]} \) satisfies \( \bfS_{3} \) for any \( m \in \BZZ \).
\end{enumerate}
\end{thm}

\begin{proof}
As in the first part of the proof of Theorem~\ref{thm:InfQGorCrit}, 
by \eqref{thm:S3Gor2:1} and \eqref{thm:S3Gor2:2}, we may assume that \( Y \to T \) 
an \( \bfS_{2} \)-morphism and its fibers are all Gorenstein in codimension one. 
Then, it is enough to prove that \( \SF = \omega_{Y/T}^{[m]} \) satisfies
relative \( \bfS_{2} \) over \( T \) along \( Y_{t} \) 
for any \( m \). Since \( \SF \) is reflexive, we can apply Proposition~\ref{prop:key}
and its corollaries to the morphism \( Y \to T \) and the closed subset
\( Z = Y \setminus \Gor(Y/T) \).
Then, \eqref{thm:S3Gor2:2} and \eqref{thm:S3Gor2:3} imply
the inequality \eqref{cor:SurjFlat(new)|Kollar|eq}
of Corollary~\ref{cor:SurjFlat(new)|Kollar}.
Thus, \( \SF \) satisfies relative \( \bfS_{2} \) over \( T \) along \( Y_{t} \) by
Corollaries~\ref{cor0:prop:key} and \ref{cor:SurjFlat(new)|Kollar}.
\end{proof}

\begin{dfn}[$\BQQ$-Gorenstein refinement]\label{dfn:QGorRefinement}
Let \( f \colon Y \to T \) be an \( \bfS_{2} \)-morphism of locally Noetherian schemes 
such that every fiber is \( \BQQ \)-Gorenstein. 
A morphism \( S \to T \) from a locally Noetherian scheme \( S \) 
is called a \emph{\( \BQQ \)-Gorenstein refinement} of \( f \) 
if the following two conditions are satisfied:
\begin{enumerate}
\renewcommand{\theenumi}{\roman{enumi}}
\renewcommand{\labelenumi}{(\theenumi)}
\item  \( S \to T \) is a monomorphism in the category of schemes; 

\item  for any morphism \( T' \to T \) from a locally Noetherian scheme \( T' \), the base change 
\( Y' \times_{T} T' \to T' \) is a \( \BQQ \)-Gorenstein morphism if and only if \( T' \to T \) 
factors through \( S \to T \). 
\end{enumerate}
\end{dfn}

\begin{remn}
If the \( \BQQ \)-Gorenstein refinement \( S \to T \) exists, then it is bijective, since every fiber 
is \( \BQQ \)-Gorenstein. 
\end{remn}

\begin{thm}\label{thm:QGorRef}
Let \( f \colon Y \to T \) be an \( \bfS_{2} \)-morphism
of locally Noetherian schemes whose fibers are all \( \BQQ \)-Gorenstein. 
Assume that 
\begin{enumerate}
\renewcommand{\theenumi}{\roman{enumi}}
\renewcommand{\labelenumi}{(\theenumi)}
\item  \label{thm:QGorRef:ass1} 
\( Y \setminus \Sigma \to T \) is a \( \BQQ \)-Gorenstein morphism 
for a closed subset \( \Sigma \) proper over \( T \), and 

\item  \label{thm:QGorRef:ass2} 
the Gorenstein indices of all the fibers \( Y_{t} = f^{-1}(t)\) are bounded above. 
\end{enumerate}
Then, \( f \) admits a \( \BQQ \)-Gorenstein refinement as a separated morphism \( S \to T \) 
locally of finite type. 
Furthermore, \( S \to T \) is a local immersion of finite type if the assumption \eqref{thm:QGorRef:ass1} 
is replaced with 
\begin{enumerate}
\renewcommand{\theenumi}{\roman{enumi}}
\renewcommand{\labelenumi}{(\theenumi)}
\addtocounter{enumi}{2} 
\item  \label{thm:QGorRef:ass3} \( f \) is a projective morphism locally on \( T \).
\end{enumerate}
\end{thm}

\begin{proof}
By \eqref{thm:QGorRef:ass2}, 
we have a positive integer \( m \) such that \( \omega^{[m]}_{Y_{t}/\Bbbk(t)} \) 
is invertible for any \( t \in T \). 
Let \( S \to T \) be the relative \( \bfS_{2} \) refinement for the reflexive \( \SO_{Y} \)-module 
\[ \SF = \bigoplus\nolimits_{i = 1}^{m} \omega_{Y/T}^{[i]}. \]
It exists as a separated morphism \( S \to T \) locally of finite type 
by Theorem~\ref{thm:dd dec}, since \( \SF|_{U} \) is locally free 
for the relative Gorenstein locus \( U = \Gor(Y/T) \) in which 
\( \Codim(Y_{t} \setminus U, Y_{t}) \geq 2 \) for any \( t \in T \), 
and since \( \SF|_{Y \setminus \Sigma} \) satisfies relative \( \bfS_{2} \) over \( T \) 
by \eqref{thm:QGorRef:ass1}.
Moreover,  \( S \to T \) is a local immersion of finite type in case \eqref{thm:QGorRef:ass3}, 
also by Theorem~\ref{thm:dd dec}. 
It is enough to show that \( S \to T \) is the \( \BQQ \)-Gorenstein refinement. 
Let \( T' \to T \) be a morphism from a locally Noetherian scheme \( T' \). 
Then, \( T' \to T \) factors through \( S \to T \) if and only if 
\( \omega^{[i]}_{Y'/T'} \) satisfies relative \( \bfS_{2} \) over \( T' \) 
for the base change \( Y' = Y \times_{T} T' \to T' \) 
for any \( 0 \leq i \leq m \), since we have an isomorphism 
\[ \omega^{[i]}_{Y'/T'} \isom (p^{*}\omega^{[i]}_{Y/T})^{\vee\vee}\]
by Corollary~\ref{cor:BC-S2CMGor}\eqref{cor:BC-S2CMGor:2} for the projection \( p \colon Y' \to Y \). 
This condition is also equivalent to that \( Y' \to T' \) is \( \BQQ \)-Gorenstein. 
Therefore, \( S \to T \) is the relative \( \BQQ \)-Gorenstein refinement. 
\end{proof}

\begin{exam}\label{exam:P114Part2}
Let \( f \colon Y \to T \) be the morphism in Example~\ref{exam:P114}. 
Then,  the \( \BQQ \)-Gorenstein refinement \( S \to T \) 
of \( f \) is just the disjoint union of \( T \setminus \{0\} \) 
and the closed point \( \{0\} \). 
This is shown as follows. Since \( f \) is smooth over \( T \setminus \{0\} \) and 
since \( K_{Y_{0}}^{2} = 9 \) and \( K_{Y_{t}}^{2} = 8 \) for \( t \ne 0 \), 
we see that 
\( S \isom (T \setminus \{0\}) \sqcup \Spec A \) over \( T \) for a closed immersion \( \Spec A \to T \) 
defined by a surjection \( \Bbbk[\ttt] \to \Bbbk[\ttt]/(\ttt^{n}) = A \). 
On the other hand, if \( n \geq 2 \), 
the base change \( f_{A} \colon Y_{A} \to \Spec A \) is not a \( \BQQ \)-Gorenstein morphism by 
Lemma~\ref{lem:P114}. Therefore, \( A = \Bbbk \). 
\end{exam}

The following theorem is a local version of Theorem~\ref{thm:QGorRef}: 

\begin{thm}\label{thm:QGorRefLocal}
Let \( f \colon Y \to T \) be a flat morphism locally of finite type 
from a locally Noetherian scheme \( Y \) such that \( T = \Spec R \) 
for a Noetherian Henselian local ring \( R \). For the closed point \( o \in T \) 
and for a point \( y \) of the closed fiber \( Y_{o} = f^{-1}(o) \), assume that 
\( Y_{o} \) is \( \BQQ \)-Gorenstein at \( y \). 
Then, there is a closed subscheme \( S \subset T \) 
having the following universal property\emph{:}
Let \( T' = \Spec R' \to T \) be a morphism 
defined by a local ring homomorphism \( R \to R' \) 
to a Noetherian local ring \( R' \). 
Then, \( T' \to T \) factors through \( S \) if and only if 
the base change \( Y' = Y \times_{T} T' \to T' \) 
is a \( \BQQ \)-Gorenstein morphism at any point \( y' \in Y' \) 
lying over \( y \) and the closed point of \( T' \). 
\end{thm}

\begin{proof}
As in the first part of the proof of Theorem~\ref{thm:InfQGorCrit}, 
we may assume that \( f \) is an \( \bfS_{2} \)-morphism and every fiber 
is Gorenstein in codimension one. 
For the Gorenstein index \( m \) of \( Y_{o} \) at \( y \), we consider the reflexive sheaf 
\[ \SF = \bigoplus\nolimits_{i = 1}^{m} \omega^{[i]}_{Y/T}  \]
on \( Y \). This is locally free in codimension one on each fiber. 
Let \( S \subset T \) be the universal subscheme in Theorem~\ref{thm:enhancement} 
for \( \SF \). We shall show that \( S \) satisfies the required condition. 
Let \( T' = \Spec R \to T \) be the morphism above. Then, we have an isomorphism 
\[ (p^{*}\omega^{[i]}_{Y/T})^{\vee\vee} \isom \omega^{[i]}_{Y'/T'} \]
for any \( i \in \BZZ \) by Corollary~\ref{cor:BC-S2CMGor}\eqref{cor:BC-S2CMGor:2}. 
Thus, \( T' \to T \) factors through \( S \) 
if and only if \( \omega^{[i]}_{Y'/T'} \) satisfies relative \( \bfS_{2} \) 
over \( T' \) at any point \( y' \) lying over \( y \) and the closed point \( o' \) of \( T' \), 
for any \( 1 \leq i \leq m \). In this case, \( \omega^{[m]}_{Y'/T'} \) is 
invertible at \( y' \) by Fact~\ref{fact:elem-flat}\eqref{fact:elem-flat:2}, 
since \( \omega^{[m]}_{Y_{o}/\Bbbk(o)} \) is so at \( y \) and since the canonical morphism 
\[ \omega^{[m]}_{Y'/T'} \otimes_{\SO_{Y'}} \SO_{Y'_{o'}} \to \omega^{[m]}_{Y'_{o'}/\Bbbk(o')} 
\isom \omega^{[m]}_{Y_{o}/\Bbbk(o)} \otimes_{\SO_{Y_{o}}} \SO_{Y'_{o'}}\]
is an isomorphism at \( y' \) (cf. Corollary~\ref{cor:BC-S2CMGor}\eqref{cor:BC-S2CMGor:2}).
Therefore, the latter condition at \( y' \) is equivalent to that \( \omega^{[i]}_{Y'/T'} \) 
satisfies relative \( \bfS_{2} \) over \( T \) at \( y' \) for any \( i \in \BZZ \); 
This means that \( Y' \to T' \) 
is a \( \BQQ \)-Gorenstein morphism at \( y' \). 
Thus, \( S \) satisfies the required condition. 
\end{proof}

The following theorem is similar to Theorem~\ref{thm:QGorRef},  
and it links a projective \( \bfS_{2} \)-morphism Gorenstein in codimension one 
in each fiber, to a naively \( \BQQ \)-Gorenstein morphism by a specific base change.

\begin{thm}\label{thm:dec naive}
Let \( f \colon Y \to T \) be a projective \( \bfS_{2} \)-morphism
of locally Noetherian schemes such that
every fiber is Gorenstein in codimension one.
Then, for any positive integer \( r > 0 \), there exists a separated monomorphism \( S_{r} \to T \)
from a locally Noetherian scheme \( S_{r} \) satisfying the following conditions\emph{:}
\begin{enumerate}
\renewcommand{\theenumi}{\roman{enumi}}
\renewcommand{\labelenumi}{(\theenumi)}
\item  The morphism \( S_{r} \to T \) is a local immersion of finite type.

\item  Let \( T' \to T \) be a morphism from a locally Noetherian scheme \( T' \).
Then, it factors through \( S_{r} \to T \) if and only if
\( Y \times_{T} T' \to T' \) is a naively \( \BQQ \)-Gorenstein morphism
whose relative Gorenstein index is a divisor of \( r \).
\end{enumerate}
\end{thm}

\begin{proof}
By Theorem~\ref{thm:dd dec}, there is a relative \( \bfS_{2} \) refinement 
\( S \to T \) for the reflexive \( \SO_{Y} \)-module \( \SF = \omega^{[r]}_{Y/T} \). 
In fact, \( \SF|_{U} \) is locally free for the relative Gorenstein locus \( U = \Gor(Y/T) \) 
and \( \Codim(Y_{t} \setminus U, Y_{t}) \geq 2 \) for any \( t \in T \) by assumption. 
Here, \( S \to T \) is a separated monomorphism and a local immersion of finite type, 
since \( f \) is a projective morphism. 
By the universal property of relative \( \bfS_{2} \) refinement and by 
Corollary~\ref{cor:BC-S2CMGor}\eqref{cor:BC-S2CMGor:2}, we see that, 
for a morphism \( T' \to T \) from a locally Noetherian scheme \( T' \), it factors through \( S \to T \) 
if and only if \( \omega^{[r]}_{Y'/T'} \) satisfies relative \( \bfS_{2} \) over \( T' \) 
for the base change morphism \( Y' = Y \times_{T} T' \to T' \). 
Note that \( \omega^{[r]}_{Y'/T'} \) is invertible if and only if 
\( Y' \to T' \) is a naively \( \BQQ \)-Gorenstein morphism 
whose relative Gorenstein index is a divisor of \( r \). 
Let \( B_{r} \) be the set of points \( P \in Y_{S} := Y \times_{T} S \) such that
\( \omega^{[r]}_{Y_{S}/S} \) is not invertible at \( P \).
Then, \( B_{r} \) is a closed subset of \( Y_{S} \). Let \( S_{r} \subset S\) be the complement
of the image of \( B_{r} \) in \( S \). Then, \( S_{r} \) is an open subset. 
For the morphism \( T' \to T \) above, 
if \( \omega^{[r]}_{Y'/T'} \) is invertible, then \( T' \to T \) 
factors through \( S \to T \), and for the induced morphism
\( h \colon Y' \to Y_{S} \) lying over \( T' \to S \), we have an isomorphism 
\[ \omega^{[r]}_{Y'/T'} \isom h^{*}(\omega^{[r]}_{Y_{S}/S}) \]
by Corollary~\ref{cor:BC-S2CMGor}\eqref{cor:BC-S2CMGor:2}. 
This implies that \( h(Y') \cap B_{r} = \emptyset\)
and that the image of \( T' \to S \) is contained in the open subset \( S_{r} \).
Therefore, the composite \( S_{r} \subset S \to T \) is the required morphism. 
\end{proof}

\begin{remn}
When \( f \colon Y \to T \) is a projective morphism, 
similar results to Theorems~\ref{thm:QGorRef} and \ref{thm:dec naive} are 
found in \cite[Cor.\ 24, 25]{KollarHusk}. 
\end{remn}

\appendix


\section{Some basic properties in scheme theory}
\label{sect:Basics}

For readers' convenience,
we collect here famous results on the local criterion of flatness and
the base change isomorphisms.

\subsection{Local criterion of flatness}
\label{subsect:LCflat}

Here, we summarize results related to the ``local criterion of flatness.''
It is usually considered as
Proposition~\ref{prop:LCflat} below.
But, the subsequent Corollaries \ref{cor:LCflat2}, \ref{cor:LCflat2a}, \ref{cor:LCflat3}
are also useful in the scheme theory.
For the detail, the reader is referred to 
\cite[IV, \S5]{SGA1},
\cite[III, \S5]{BAC},
\cite[$0_{\text{III}}$, \S10.2]{EGA},
\cite[V, \S3]{AltmanKleiman},
\cite[\S22]{Matsumura}, etc.
We  also mention a ``local criterion of freeness'' as Lemma~\ref{lem:LCfree},
and explain two  more results on flatness and local freeness for sheaves on schemes.

\begin{prop}[local criterion of flatness]\label{prop:LCflat}
For a ring \( A \), an ideal \( I \) of \( A \), and for an \( A \)-module \( M \),
assume that
\begin{enumerate}
\renewcommand{\theenumi}{\arabic{enumi}}
\renewcommand{\labelenumi}{(\theenumi)}
\item  \( I \) is nilpotent, or

\item  \label{prop:LCflat:ii}
\( A \) is Noetherian and \( M \) is
\emph{\( I \)-adically ideally separated}, i.e.,
\( \Ga \otimes_{A} M \) is separated for the \( I \)-adic topology for all ideals \( \Ga \) of \( A \).
\end{enumerate}
Then, the following four conditions are equivalent to each other\emph{:}
\begin{enumerate}
\renewcommand{\theenumi}{\roman{enumi}}
\renewcommand{\labelenumi}{(\theenumi)}
\item  \label{prop:LCflat:1} \( M \) is flat over \( A \);

\item \label{prop:LCflat:2} \( M/IM \) is flat over \( A/I \) and \( \Tor_{1}^{A}(M, A/I) = 0 \);

\item \( M/IM \) is flat over \( A/I \) and the canonical homomorphism
\[ M/IM \otimes_{A/I} I^{k}/I^{k+1} \to I^{k}M/I^{k+1}M \]
is an isomorphism for any \( k \geq 0 \);

\item  \label{prop:LCflat:4}
\( M/I^{k}M \) is flat over \( A/I^{k} \) for any \( k \geq 1 \).
\end{enumerate}
\end{prop}

\begin{remn}
The proof is found in \cite[IV, Cor.~5.5, Th.~5.6]{SGA1},
\cite[III, \S5.2, Th.~1]{BAC},
\cite[$0_{\text{III}}$, (10.2.1)]{EGA},
\cite[V, Th.~(3.2)]{AltmanKleiman},
\cite[Th.~22.3]{Matsumura}.
The condition \eqref{prop:LCflat:ii} is satisfied, for example,
when there is a ring homomorphism \( A \to B \) of Noetherian rings such that
\( M \) is originally a finitely generated \( B \)-module and that
\( IB \) is contained in the \emph{Jacobson radical} \( \rad(B)\) of \( B \)
(cf.\ \cite[III, \S5.4, Prop.~2]{BAC},
\cite[$0_{\text{III}}$, (10.2.2)]{EGA},
\cite[p.~174]{Matsumura}).
\end{remn}

\begin{cor}\label{cor:LCflat2}
Let \( A \to B \) be a local ring homomorphism
of Noetherian local rings and let
\( u \colon M \to N \) be a homomorphism of \( B\)-modules such that
\( M \) and \( N \) are finitely generated \( B \)-modules and that \( N \) is flat over \( A \).
Then, the following two conditions are equivalent to each other\emph{:}
\begin{enumerate}
\renewcommand{\theenumi}{\roman{enumi}}
\renewcommand{\labelenumi}{(\theenumi)}
\item  \( u \) is injective and the cokernel of \( u \) is flat over \( A \)\emph{;}

\item \( u \otimes_{A} \Bbbk \colon M \otimes_{A} \Bbbk \to N \otimes_{A} \Bbbk \)
is injective for the residue field \( \Bbbk \) of \( A \).
\end{enumerate}
\end{cor}

The proof is given in \cite[IV, Cor.~5.7]{SGA1}, \cite[$0_{\text{III}}$, (10.2.4)]{EGA},
\cite[VII, Lem.~(4.1)]{AltmanKleiman}, \cite[Th.~22.5]{Matsumura}.

\begin{cor}[{cf.\ \cite[$0_{\text{IV}}$, Prop.~(15.1.16)]{EGA}, \cite[Cor. to Th.~22.5]{Matsumura}}]
\label{cor:LCflat2a}
Let \( A \to B \) be a local ring homomorphism
of Noetherian local rings
and let \( M \) be a finitely generated \( B \)-module.
Let \( \Bbbk \) be the residue field of \( A \) and let \( \bar{x} \) denote
the image of \( x \in B \) in \( B \otimes_{A} \Bbbk \).
For elements \( x_{1} \), \ldots, \( x_{n} \) in the maximal ideal \( \GM_{B} \),
the following two conditions are equivalent to each other\emph{:}
\renewcommand{\theenumi}{\roman{enumi}}
\renewcommand{\labelenumi}{(\theenumi)}
\begin{enumerate}
\item  \( (x_{1}, \ldots, x_{n}) \) is an \( M \)-regular sequence
and \( M/\sum\nolimits_{i = 1}^{n} x_{i}M \) is flat over \( A \)\emph{;}

\item \( (\bar{x}_{1}, \ldots, \bar{x}_{n}) \) is an \( M \otimes_{A} \Bbbk \)-regular sequence
and \( M \) is flat over \( A \).
\end{enumerate}
\end{cor}

\begin{cor}\label{cor:LCflat3}
Let \( A \to B \) and \( B \to C \) be local ring homomorphisms
of Noetherian local rings and
let \( \Bbbk \) be the residue field of \( A \). Assume that \( B \) is flat over \( A \).
Then, for a finitely generated \( C \)-module \( M \),
the following conditions are equivalent to each other\emph{:}
\begin{enumerate}
\renewcommand{\theenumi}{\roman{enumi}}
\renewcommand{\labelenumi}{(\theenumi)}
\item  \( M \) is flat over \( B \)\emph{;}

\item  \( M \) is flat over \( A \) and \( M \otimes_{A} \Bbbk \) is flat over \( B \otimes_{A} \Bbbk \).
\end{enumerate}
\end{cor}

The proof is given in \cite[IV, Cor.~5.9]{SGA1}, \cite[III, \S5.4, Prop.~3]{BAC},
\cite[$0_{\text{III}}$, (10.2.5)]{EGA}, \cite[V, Prop.~(3.4)]{AltmanKleiman}.

Next, we shall give the ``local criterion of freeness'' as Lemma~\ref{lem:LCfree} below,
which is similar to Proposition~\ref{prop:LCflat}.
This result is well known (cf.\ \cite[IV, Prop.~4.1]{SGA1}, \cite[II, \S3.2, Prop.~5]{BAC},
\cite[$0_{\text{III}}$, (10.1.2)]{EGA}), but
is not usually called the ``local criterion of freeness'' in articles.

\begin{lem}[local criterion of freeness]\label{lem:LCfree}
Let \( A \) be a ring, \( I \) an ideal of \( A \), and \( M \) an \( A \)-module such that
\begin{itemize}
\item  \( I \) is nilpotent or

\item  \( A \) is Noetherian, \( I \subset \rad(A) \), and
\( M \) is a finitely generated \( A \)-module.
\end{itemize}
Then, the following conditions are equivalent to each other\emph{:}
\begin{enumerate}
\renewcommand{\theenumi}{\roman{enumi}}
\renewcommand{\labelenumi}{(\theenumi)}
\item  \( M \) is a free \( A \)-module\emph{;}

\item  \( M/IM \) is a free \( A/I \)-module and \( \Tor_{1}^{A}(M, A/I) = 0 \)\emph{;}

\item  \( M/IM \) is a free \( A/I \)-module and the canonical homomorphism
\[ M/IM \otimes_{A/I} I^{k}/I^{k+1} \to I^{k}M/I^{k+1}M \]
is an isomorphism for any \( k \geq 0 \).
\end{enumerate}
\end{lem}

\begin{remn}
Applying Lemma~\ref{lem:LCfree} to the case
where \( A \) is a Noetherian local ring and \( I \) is the maximal ideal,
we have the equivalence of flatness and freeness for finitely generated \( A \)-modules
(cf.\ \cite[IV, Cor.~4.3]{SGA1}, \cite[$0_{\text{III}}$, (10.1.3)]{EGA}).
On the other hand, the equivalence of flatness and freeness can be proved by other methods
(cf.\ \cite[Th.~7.10]{Matsumura}, \cite[Lem.~5.8]{AltmanKleiman}), and using the equivalence,
we obtain Lemma~\ref{lem:LCfree}
for the same local ring \( (A, I) \) and for a finitely generated \( A \)-module \( M \),
as a corollary of Proposition~\ref{prop:LCflat}.
\end{remn}

\begin{remn}
The equivalence explained above implies the following well-known fact:
\emph{For a locally Noetherian scheme \( X \), a coherent flat \( \SO_{X} \)-module
is nothing but a locally free \( \SO_{X} \)-module of finite rank.}
\end{remn}

The following is proved immediately from the definitions of flatness and faithful flatness
(cf.\ \cite[I, \S3, no.\ 2, Prop.\ 4]{BAC}):

\begin{lem}\label{lem:fflatDescent}
Let \( f \colon X \to Y \) and \( g \colon Y \to Z \) be morphisms of schemes such that
\( f \) is faithfully flat, i.e., flat and surjective.
Then, for an \( \SO_{Y} \)-module \( \SG \),
it is flat over \( Z \) if and only if \( f^{*}\SG \) is flat over \( Z \).
\end{lem}

As a corollary in the case where \( Y = Z \),
we have the following descent property of locally freeness
by the relation with flat coherent sheaves.

\begin{lem}\label{lem:LocFreeDecsent}
Let \( f \colon X \to Y \) be a flat surjective morphism of locally Noetherian schemes.
For a coherent \( \SO_{Y} \)-module \( \SG \), it is locally free if and only if \( f^{*}\SG \)
is so.
\end{lem}

\noindent
The authors could not find a good reference for Lemma~\ref{lem:LocFreeDecsent}.
For example, we have a weaker result as a part of \cite[VIII, Prop.~1.10]{SGA1}, where
\( f \) is assumed additionally to be quasi-compact; However, the quasi-compactness
is related to the other part.

\subsection{Base change isomorphisms}

Let us consider a Cartesian diagram

\[ \begin{CD}
X' @>{g'}>> X \\ @V{f'}VV @VV{f}V \\ S' @>{g}>> S
\end{CD}\]
of schemes, i.e., \( X' \isom X \times_{S} S' \).
Then, for any quasi-coherent \( \SO_{X} \)-module \( \SF \),
one has a functorial canonical homomorphism
\[ \theta(\SF) \colon g^{*}(f_{*}\SF) \to f'_{*}(g^{\prime *}\SF) \]
of \( \SO_{S'} \)-modules, and more generally, a functorial canonical homomorphism
\[ \theta^{i}(\SF) \colon g^{*}(R^{i}f_{*}\SF) \to R^{i}f'_{*}(g^{\prime *}\SF) \]
for each \( i \geq 0 \).
We have the following assertions on \( \theta(\SF) \) and \( \theta^{i}(\SF) \).

\begin{lem}[affine base change]\label{lem:affineBC}
If \( f \) is an affine morphism, then \( \theta(\SF) \) is an isomorphism.
\end{lem}

\begin{lem}[flat base change]\label{lem:flatbc}
Assume that \( g \) is flat and that \( f \) is quasi-compact and quasi-separated.
Then, \( \theta^{i}(\SF) \) is an isomorphism for any \( i \).
\end{lem}

A proof of Lemma~\ref{lem:affineBC} is given \cite[II, Cor.~(1.5.2)]{EGA}, and
a proof of Lemma~\ref{lem:flatbc} is given in \cite[III, Prop.~(1.4.15)]{EGA}
(cf.\ \cite[IV, (1.7.21)]{EGA}).
Here, the morphism \( f \colon X \to S\) is said to be ``quasi-separated'' if
the diagonal morphism \( X \to X \times_{S} X \)
is quasi-compact (cf.\ \cite[IV, D\'ef.~(1.2.1)]{EGA}).

We have also the following generalization of Lemma~\ref{lem:flatbc}
to the case of complexes
by \cite[II, Prop.~5.12]{ResDual}, \cite[IV, Prop.~3.1.0]{IllusieSGA6}, and
\cite[Prop.~3.9.5]{Lipman09}.

\begin{prop}\label{prop:flatbc}
In the situation of Lemma~\emph{\ref{lem:flatbc}}, let \( \SF^{\bullet} \)
be a complex of \( \SO_{X} \)-modules
in \( \bfD^{+}_{\qcoh}(X) \). Then, there is a functorial quasi-isomorphism
\[ \bfL g^{*}(\bfR f_{*}(\SF^{\bullet})) \to \bfR f'_{*}(\bfL g^{\prime *}(\SF^{\bullet})).  \]
\end{prop}

\end{document}